%% file: BGlimit.tex
\documentclass[11p,english,pdflatex]{smfbook} 
\usepackage[french,english]{babel} 
\usepackage{smfthm} 
\pdfoutput=1
\usepackage{multicol} 
\usepackage{graphicx, color, multicol}

\renewcommand{\theequation}{\thesection.\arabic{equation}}
\definecolor{myblue}{rgb}{0,0,0.6}
\definecolor{myred}{rgb}{0.8,0,0}

\def\Nt#1{\|\!|{#1}\|\!|}
\def\BigNt#1{\Big \|\! \Big |{#1} \Big\|\! \Big |}
\def\tildeto{\stackrel{\sim}{\longrightarrow}}

\DeclareMathAlphabet{\mathscr}{OT1}{pzc}{m}{it} 

\usepackage{chngcntr}

\font\tenms=msbm10
\font\sevenms=msbm7
\font\fivems=msbm5
\newfam\bbfam \textfont\bbfam=\tenms \scriptfont\bbfam=\sevenms
\scriptscriptfont\bbfam=\fivems

\def\a{\alpha}
\def\o{\omega}
\def\L{\Lambda}
\newtheorem{Thm}{Theorem}
\newtheorem{Def}{Definition}[section]
\newtheorem{Rem}[Def]{Remark}

\newtheorem{Cor}[Def]{Corollary}
\newtheorem{Lem}[Def]{Lemma}
\newtheorem{lem}[Def]{Lemma}
\newtheorem{Rmk}[Def]{Remark}

\newtheorem{ass}[Def]{Assumption}
\newtheorem{nota}[Def]{Notation}
\counterwithout{figure}{chapter}
\def\d {{\partial}}
\def\b{\beta}
\def\k{\kappa}

\let\ds=\displaystyle
\def\C{\mathcal C}
\def\D{\mathcal D}
\def\O{\Omega}

\newcommand{\Tr}{\operatorname{trace}}

\newcommand{\ba}{\begin{aligned}}
\newcommand{\ea}{\end{aligned}}
\newcommand{\be}{\begin{equation}}
\newcommand{\ee}{\end{equation}}
\def\prodetage#1#2{
\prod_{\scriptstyle {#1}\atop\scriptstyle {#2}} }

\renewcommand{\theequation}{\thesection.\arabic{equation}}

\def \N{{\mathbf N}}
\def \R{{\mathbf R}}

\def\cB{\mathcal{B}}

\def\cD{\mathcal{D}}
\def\cE{\mathcal{E}}

\def\cG{\mathcal{G}}

\def\cJ{\mathcal{J}}

\def\cL{\mathcal{L}}
\def\cM{\mathcal{M}}
\def\cN{\mathcal{N}}

\def \indc{ 1 \!\! 1}

\def \eps{{\varepsilon}}
\def \e{{\varepsilon}}

\def\s{\sigma}
\def \AA{{\mathcal A}}

\def \d{{\partial}}
\def\l{\lambda}
\def \iint{\int \! \! \! \int }
\def \iiint{\int \! \! \! \int \! \! \! \int }

\def\t{\tau}

\def \1{ 1 \!\! 1}

\def\g{\gamma}

\def\P{\mathcal P}
\def\dsp{\displaystyle}

\newcommand{{\bv}}{ \bf v}

\makeindex

\begin{document}

\author[I. Gallagher]{Isabelle Gallagher}
\address[I. Gallagher]%
{Institut de Math{\'e}matiques de Jussieu UMR CNRS 7586 \\
      Universit{\'e} Paris-Diderot (Paris 7) \\
175, rue du Chevaleret\\
75013 Paris\\FRANCE}
\email{gallagher@math.univ-paris-diderot.fr}
\author[L. Saint-Raymond]{Laure Saint-Raymond}
\address[L. Saint-Raymond]%
{Universit\'e Paris VI and DMA \'Ecole Normale Sup\'erieure, 45 rue d'Ulm, 75230
Paris Cedex 05\\FRANCE }
\email{Laure.Saint-Raymond@ens.fr}
\author[B. Texier]{Benjamin Texier}
\address[B. Texier]
{Institut de Math{\'e}matiques de Jussieu UMR CNRS 7586 \\
      Universit{\'e} Paris-Diderot (Paris 7) \\
175, rue du Chevaleret\\
75013 Paris\\FRANCE \\
and  
 DMA \'Ecole Normale Sup\'erieure, 45 rue d'Ulm, 75230
Paris Cedex 05\\FRANCE }
\email{Benjamin.Texier@math.jussieu.fr}
\title{From Newton to Boltzmann: hard spheres and   short-range potentials }

 \begin{abstract}
We provide a rigorous derivation of the Boltzmann equation as the mesoscopic limit of systems of hard spheres, or Newtonian particles interacting via a short-range potential, as the number of particles $N$ goes to infinity and the characteristic length of interaction $\e$ simultaneously goes to $0,$ in the Boltzmann-Grad scaling $N \e^{d-1} \equiv 1.$ 

The time of validity of the convergence is a fraction of the average time of first collision, due to a limitation of the time on which one can prove uniform estimates for the BBGKY and Boltzmann hierarchies.

Our proof relies on the fundamental ideas of Lanford, and the important contributions of   King, Cercignani, Illner and Pulvirenti, and Cercignani, Gerasimenko and Petrina. 
The main novelty here is the detailed study of pathological trajectories involving recollisions, which proves the  term-by-term convergence  for the correlation series expansion.
    \end{abstract}

\maketitle

\frontmatter

\tableofcontents

\mainmatter

\newpage

{\bf Acknowledgements}: We thank J. Bertoin, Th. Bodineau, D. Cordero-Erausquin, L. Desvillettes, F. Golse, S. Mischler, C. Mouhot and R. Strain for many helpful discussions on topics addressed in this text. We are particularly grateful to M. Pulvirenti, C. Saffirio and S. Simonella for explaining 
to us
how condition~(\ref{strangeassumption}) makes possible a parametrization of the collision integral by the deflection angle  (see Chapter~\ref{scattering}).

Finally we thank the anonymous referee for helpful suggestions to improve the   manuscript.

\part{Introduction}

\include{low-density}

\include{Boltzmann}

\include{structure}


\part{The case of hard spheres}

\include{BBGKY-HS}

\include{Cauchy-Kowalewski}

\include{Convergence-HS}

\include{Strategy}


\part{The case of short range potentials}

\include{scattering}

\include{BBGKY}

\include{cluster}

\include{Convergence}


\part{Term-by-term convergence}

\include{geom-lem}

\include{recollision}

\include{proof}

\include{conclusion}

\backmatter
\include{reference}
\input notation.tex
\printindex


\end{document}

%% file: low-density.tex
\chapter{The low density limit}
\label{intro} 
\setcounter{equation}{0}

 We are interested in this monograph in the qualitative behavior of systems of  particles with short-range interactions. We study the qualitative behaviour of particle systems with short-range binary interactions, in two cases: hard spheres, that move in uniform rectilinear motion until they undergo elastic collisions, and smooth, monotonic, compactly supported potentials. 
 
 \medskip
 $\bullet$
 For hard spheres, the equations of motion are
 \begin{equation} \label{n1-HS}
\begin{aligned}
 \frac  { d x_i }{d t} = \displaystyle v_i \, , \qquad 
 \frac  { d  v_i }{d  t} = 0 \, ,
\end{aligned}
\end{equation}
for $1 \leq i \leq N,$ where $(x_i,v_i) \in \R^d \times \R^d$ denote the position and velocity of particle $i$,
provided that the exclusion condition $|x_i(t) - x_j(t)|> \sigma$ is satisfied, where $\sigma$ denotes the diameter of the particles.
We further have to prescribe a reflection condition at the boundary~: if there exists $j \neq i $ such that $   |x_i-x_j|=\sigma$
\begin{equation}
\label{HS-BC}
\begin{aligned}v_i^{in} = v_i^{out} - \nu^{i,j} \cdot( v_i^{out} -   v_j^{out} ) \, \nu^{i,j}&  \\
v_j^{in} = v_j^{out} + \nu^{i,j} \cdot ( v_i^{out} -   v_j^{out} )  \,\nu^{i,j} & \,,
\end{aligned}
\end{equation}
where~$\nu^{i,j}:=(x_i-x_j) / |x_i-x_j|$.
Note that it is not obvious to check that (\ref{n1-HS})-(\ref{HS-BC}) defines  global dynamics. This question is addressed in Chapter~\ref{spheres}.

\bigskip

$\bullet$ 
In the case of smooth interactions, the Hamiltonian  equations of motion are
\begin{equation} \label{n1}
\begin{aligned}
 \frac  { d x_i }{d t} = \displaystyle v_i \, , \qquad 
 m_i \frac  { d  v_i }{d  t} = - \sum_{ j \neq i} \nabla \Phi(x_i - x_j) \, ,
\end{aligned}
\end{equation}
 where  $m_i$ is the mass of particle $i$ (which we shall assume equal to~1 to simplify) and  the force exerted by particle $j$ on particle $i$ is $ - \nabla \Phi(x_i - x_j)$.

 When the system is constituted of two elementary particles, in the reference frame attached to the center of mass, the dynamics is two-dimensional. The deflection of the particle trajectories from straight lines can then be described through explicit formulas (which are given in Chapter~\ref{scattering}).

When the system is constituted of     three particles or more, the integrability  is lost, and in general the problem becomes very complicated, as already noted by Poincar\'e \cite{poincare}.

\begin{Rmk}\label{limithardspherespotential}
Note that the dynamics of hard spheres is in some sense a limit of the smooth-forces case with
$$\Phi(x)  = +\infty \hbox{ if }|x|<\sigma \, , \qquad \Phi (x) = 0 \hbox{ if } |x|> \sigma\,.$$
Nevertheless, to our knowledge, there does not exist any mathematical statement concerning these asymptotics.

We will however see in the sequel that the two types of systems exhibit very similar qualitative behaviours in the low density limit. Once the dynamics is defined (i.e. provided that we can discard multiple collisions), the case of hard spheres is actually simpler and we will discuss it in Part II to explain the main ideas and conceptual difficulties. We will then explain, in Part III, how to extend the arguments to the smoother case of Hamiltonian systems.

\end{Rmk}

\section{The Liouville equation}\label{Liouville-PartI}$ $
\setcounter{equation}{0}

In the large $N$ limit,  individual trajectories become irrelevant, and our goal is to describe an average behaviour.

This average will be  of course over particles which are indistiguishable, meaning that we will be only interested in some distribution related to the empirical measure
$$\mu_N \big(t,X_N(0),V_N(0) \big) :=\frac1N \sum_{i=1}^N \delta_{x_i(t), v_i(t)}\,,$$
with~$X_N(0) := (x_1(0),\dots, x_N(0)) \in \R^{dN}$ and~$V_N(0) := (v_1(0),\dots, v_N(0)) \in \R^{dN}$, and~$(x_i(t), v_i(t))$ is the state  at time~$t$ of  particle~$i$ in  the system with initial configuration $\big(X_N(0),V_N(0) \big)$.

But, because we have only a vague knowledge of the state of the system at initial time, we will further average over initial configurations.
At time 0, we thus start with a distribution~$f_N^0(Z_N),$
where we   use the following notation: for any set of~$s$ particles with positions~$X_s:=(x_1,\dots,x_s) \in \R^{ds}$  and velocities~$V_s:=(v_1,\dots,v_s) \in \R^{ds}$, we write~$Z_s := (z_1,\dots, z_s)\in \R^{2ds}$ with~$z_i := (x_i,v_i) \in \R^{2d}$.

We then aim at describing  the evolution of the distribution
$$\int \left( \frac1N \sum_{i=1}^N \delta_{z_i(t)} \right) f_N^0(Z_N) dZ_N\,.$$

We    thus define  the probability $f_N =f_N(t,Z_N)$, referred to as the {\it $N$-particle  distribution function}\index{$N$-particle  distribution function}, and we assume that it satisfies for all permutations $\s$ of $\{1,\dots,N\},$ 
\begin{equation} \label{sym}
 f_N (t,Z_{\sigma(N)}) = f_N (t,Z_N)\, ,
 \end{equation} 
with $Z_{\sigma(N)} = (x_{\sigma(1)}, v_{\sigma(1)}, \dots, x_{\sigma(N)}, v_{\sigma(N)})$. This corresponds to the property that the particles are indistinguishable.

The distribution we are interested in  is therefore nothing else than the first marginal $f^{(1)}_N$ of the distribution function $f_N$, defined by
$$ f^{(1)}_N (t,Z_1) :=\int  f_N (t,Z_N) \,  dz_2 \dots dz_N\,.$$

\bigskip
Since~$f_N$ is an invariant of the particle system, 
the  {\it Liouville equation}\index{Liouville equation} relative to the particle system~\eqref{n1} is
\begin{equation}
\label{liouville}
\partial_t f_N +\sum_{i=1}^N v_i \cdot \nabla_{x_i} f_N  -\sum_{i=1}^N \sum_{j=1 \atop j\neq i}^N \nabla_x \Phi \left(x_i - x_j \right)\cdot \nabla_{v_i} f_N  = 0 \, .
\end{equation}
 
 For hard spheres, provided that we can prove that the dynamics is well defined for almost all initial configurations, we find the Liouville equation
 \begin{equation}\label{liouville-HS}
\partial_t f_N +\sum_{i=1}^N v_i \cdot \nabla_{x_i} f_N  = 0 
\end{equation}
on the domain
   $$
    \cD_N :=\Big \{ Z_N \in \R^{2dN}\,/\, \forall i\neq j,\, |x_i-x_j|>\sigma \Big\}
   $$
with the boundary condition~$f_N(t,Z_N^{in}) = f_N(t,Z_N^{out}) $, meaning that on the part of the boundary such that $   |x_i-x_j|=\sigma$
$$ f_N ( t, \dots, x_i, v_i^{in}, \dots x_j, v_j^{in},\dots)  =  f_N ( t, \dots, x_i, v_i^{out}, \dots x_j, v_j^{out},\dots) $$
where the ingoing and outgoing velocities are related by (\ref{HS-BC}).

\section{Mean field versus collisional dynamics}$ $
\setcounter{equation}{0}

In this framework, in order for the average energy per particle to remain bounded, one has to assume that the energy of each pairwise interaction is small. In other words, one has to consider a rescaled potential $\Phi_\eps$ obtained
\begin{itemize}
\item either by scaling the strength of the force,
\item or by scaling the range of potential.
\end{itemize}
According to the scaling chosen, we expect to obtain different  asymptotics.

\medskip
\noindent
$\bullet$ In the case of a weak coupling, i.e. when the strength of the individual interaction becomes small (of order $1/N$) but the range  remains macroscopic, the convenient scaling in order for the macroscopic dynamics to be sensitive to the coupling is:
$$\partial_t f_N +\sum_{i=1}^N v_i \cdot \nabla_{x_i} f_N  -\frac1N \sum_{i=1}^N \sum_{j=1 \atop j\neq i}^N \nabla \Phi \left(x_i - x_j \right)\cdot \nabla_{v_i} f_N  = 0 \, .
$$
Then each particle feels the effect of the force field created by all the (other) particles
$$F_N (x) =- \frac1N \sum_{j=1}^N \nabla_x \Phi \left(x - x_j \right) \sim -\iint \nabla \Phi (x-y) f^{(1)}_N(t,y,v) dydv\,.$$
In particular, the dynamics seems to be stable under small perturbations of the positions or velocities of the particles.

In the limit $N\to \infty$, we thus get a {\it mean field approximation}, that is an equation of the form
$$\partial_t f +v\cdot \nabla_x f  +F \cdot \nabla_v f  = 0$$
for the first marginal, where the coupling arises only through some average 
$$F:= -\nabla_x \Phi * \int fdv\,.$$

An important amount of literature is devoted to such asymptotics, but this is not our purpose here.  We refer to \cite{braun-hepp,S2} for pioneering results, to~\cite{haurayjabin} for a recent study and  to \cite{golse} for a   review on that topic.

\medskip
\noindent
$\bullet$ The scaling we shall deal with in the present work corresponds to a strong coupling, i.e. to the case when the amplitude of the potential remains of size $O(1)$, but its range becomes small.

 Introduce a small parameter $\e > 0$ corresponding to the typical interaction length of the particles. For hard spheres, $\eps$ is simply the diameter of particles.
In the case of Hamiltonian systems, $\eps$ will be the range of the interaction potential. We shall indeed assume throughout this text the following properties for~$\Phi$  (a {\it short-range} potential).
\begin{ass}\label{propertiesphi}
The potential~$\Phi: \R^d \to \R$  is a radial, nonnegative, nonincreasing function supported  in the unit ball of~$\R^d$, of class~$C^2$ in~$\{x \in \R^d \, , 0< |x|<1 \}$. Moreover it is assumed that~$\Phi$ is unbounded near zero, goes to zero at~$|x|=1$ with bounded derivatives, and  that~$\nabla\Phi$ vanishes only on~$|x|=1$. 
\end{ass}
   Then in the macroscopic spatial and temporal scales, the Hamiltonian system becomes
   \begin{equation} \label{n1.1}
\begin{aligned}
 \frac  { d x_i }{d t} = \displaystyle v_i \, , \qquad 
  \frac  { d  v_i }{d  t} = - \frac1\e\sum_{ j \neq i} \nabla \Phi\left({x_i - x_j \over \eps}\right) \, ,
\end{aligned}
\end{equation}
   and the Liouville equation takes the form 
\begin{equation}\label{Liouville}
\partial_t f_N +\sum_{i=1}^N v_i \cdot \nabla_{x_i} f_N  -\sum_{i=1}^N \sum_{j=1 \atop j\neq i}^N \frac1\eps \nabla_x \Phi \left({x_i - x_j \over \eps}\right)\cdot \nabla_{v_i} f_N  = 0 \, .
\end{equation}

With such a scaling, the dynamics is very sensitive to the positions of the particles.

\begin{figure}[h]
\begin{center}
\scalebox{0.4}{\includegraphics{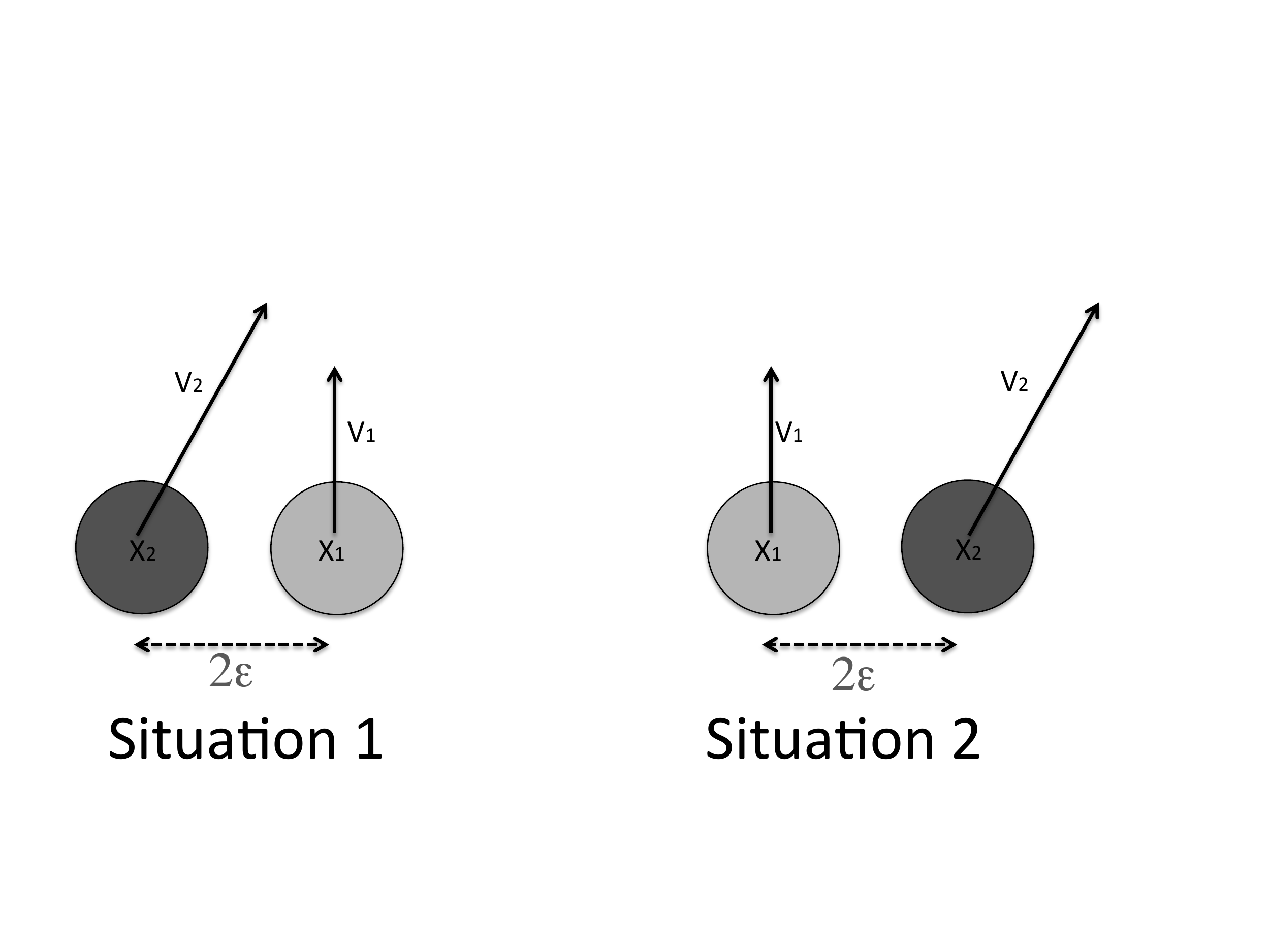}}
\caption{\label{unstability}Instability}
\end{center}
\end{figure}

Situations 1 and 2 on Figure \ref{unstability} are differ by a spatial translation of $O(\eps)$ only. However in Situation~1, particles will interact and be deviated from their free motion, while in Situation 2, they will evolve under free flow.

\section{The Boltzmann-Grad limit}
\setcounter{equation}{0}

Particles move with uniform rectilinear motion as long as they remain at a distance greater than $\eps$ to other particles.
 In the limit $\eps\to 0$, we thus expect trajectories to be almost polylines.
 
 Deflections are due to elementary interactions 
 \begin{itemize}
 \item which occur when two particles are at a distance smaller than $\eps$ (exactly $\eps$ in the case of hard spheres),
 
 \item during a time interval of order $\eps$ (if the relative velocity is not too small) or even instantaneously in the case of hard spheres,
 
 \item which involve generally only two particles~: the probability that a third particle enters a security ball of radius $\eps$ should indeed  tend to 0 as $\eps \to 0$ in the convenient scaling. We are therefore brought back to the case of the two-body system, which is completely integrable (see Chapter~\ref{scattering}).

 \end{itemize}

\medskip

In order for the interactions to have a macroscopic effect on the dynamics, each particle should undergo a finite number of collisions per unit of time. A   scaling argument, giving the mean free path in terms of~$N$ and~$\eps$, then shows that $N\eps^{d-1} =O(1)$: indeed  a particle travelling at speed bounded by~$R$  covers in unit time an area  of size~$R\eps^{d-1}$, and there are~$N$ such particles.  This is the Boltzmann-Grad scaling (see~\cite{grad}).

The Boltzmann equation, which is the master equation in collisional kinetic theory \cite{CIP, villani}, is expected to describe such a dynamics.

%% file: Boltzmann.tex
\chapter{The Boltzmann equation}\label{boltz-chapter}

\section{Transport and collisions}
\setcounter{equation}{0}

 As mentioned in the previous chapter, the state of the system in the low density limit  should be described (at the statistical level) by the kinetic density, i.e. by the probability $f\equiv f(t,x,v)$ of finding a particle with position $x$ and velocity $v$ at time $t$.
 
 This density is expected to evolve under both the effects of transport and binary elastic collisions, which is expressed in  the Boltzmann equation (introduced by Boltzmann in \cite{boltzmann1}-\cite{boltzmann})~:
  \begin{equation}
  \label{boltz-eq}
\underbrace{ \d_t f +v\cdot \nabla_x f}_{\mbox{
\footnotesize{free transport}}} \quad = 
\underbrace{  Q(f,f)}_{\mbox{\footnotesize{localized binary collisions}}}  
\end{equation}
The Boltzmann collision operator, present in the right-hand side of  \eqref{boltz-eq}, is the quadratic form, acting on the velocity variable, associated with the bilinear operator
\begin{equation}
	\label{boltz-operator}
	Q(f,f)=\iint [f'f'_1-ff_1]  \,b(v-v_1,\omega) \, dv_1 d\omega
\end{equation}
where we have used the standard abbreviations
$$
	f=f(v) \,, \quad f'=f(v') \,, \quad f'_1=f(v_1') \,, \quad f_1=f(v_1) \,,
$$
with $(v',v'_1)$ given by
$$
	v'=v +\omega \cdot (v_1-v)  \, \omega \, , 
	\quad v'_1=v_1 - \omega \cdot(v_1-v) \, \omega \,.
$$
One can easily show that the quadruple $(v,v_1,v',v_1')$ parametrized by $\omega\in{\mathbf {S}}_1^{d-1}$ (where~${\mathbf {S}}_\rho^{d-1}$\label{unitsphere} denotes the sphere of radius~$\rho$ in~$\R^d$) provides the family of all solutions to the system of $d+1$ equations
\begin{equation}\label{conservation momentum energy}
	\begin{aligned}
		v+v_1 & = v'+v_1'  \,,\\
		|v|^2+|v_1|^2 & = |v'|^2+|v_1'|^2 \,,
	\end{aligned}
\end{equation}
which, at the kinetic level, express the fact that  collisions are  elastic and thus conserve momentum and energy. Notice that the transformation $(v,v_1,\omega)\mapsto \left(v',v'_1,-\omega\right)$ is an involution.

\begin{figure}[h]
\begin{center}
\scalebox{0.4}{\includegraphics{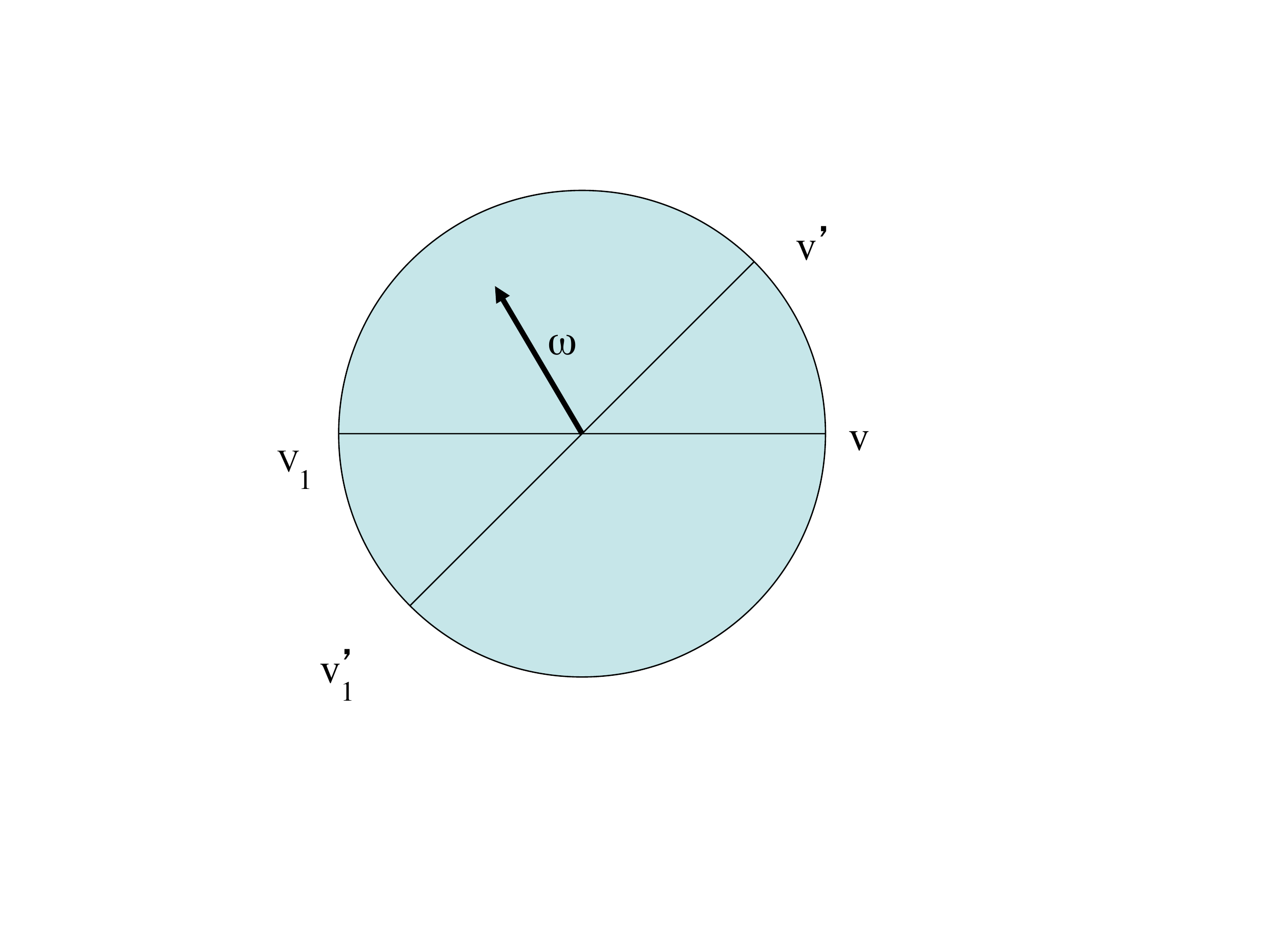}}
\caption{\label{omega-fig}Parametrization of the collision by the deflection angle $\omega$}
\end{center}
\end{figure}

The Boltzmann collision operator can therefore be split, at least formally, into a gain term and a loss term (see \cite{cercignani, villani})
$$Q(f,f)=Q^+(f,f)-Q^-(f,f).$$
The loss term counts all collisions in which a given particle of velocity $v$ will encounter another particle, of velocity $v_1$, and thus will change its velocity leading to a loss of particles of velocity $v$, whereas the gain term measures the number of particles of velocity $v$ which are created due to a collision between particles of velocities $v'$ and $v'_1$.

The collision kernel $b=b(w, \omega)$  is a measurable function positive almost everywhere, which  measures the statistical repartition of post-collisional velocities $(v,v_1)$ given the pre-collisional velocities $(v',v'_1)$. Its precise form depends crucially on the nature of the microscopic interactions, and will be discussed in more details in the sequel. 
Note that, due to the Galilean invariance of collisions, it only depends on the magnitude of the relative velocity $|w|$ and on the deviation angle $\theta$, or deflection (scattering) angle, defined by $\cos\theta=k\cdot\omega$ where $k= {w/|w|}$.

\section{Boltzmann's H theorem and irreversibility}
\setcounter{equation}{0}

From \eqref{conservation momentum energy} and using the well-known facts (see \cite{cercignani}) that transforming $(v,v_1)\mapsto(v_1,v)$ and~$(v,v_1,\omega)\mapsto \left(v',v_1',\omega\right)$ merely induces mappings with unit Jacobian determinants, one can show that formally
\begin{equation}
	\int Q(f,f)\varphi dv=\frac14 \iiint [f'f'_1-ff_1] (\varphi+\varphi_1 - \varphi'-\varphi'_1) \,b(v-v_1,\omega) \, dvdv_1 d\omega\,.
\end{equation}

In particular,
$$	\int Q(f,f)\varphi dv=0
$$
for all $f$ regular enough, if and only if $\varphi(v)$ is a collision invariant, i.e. $\varphi(v)$ is a linear combination of $\left\{1,v_1,\ldots,v_d,|v|^2\right\}$. Thus, successively multiplying the Boltzmann equation \eqref{boltz-eq} by the collision invariants and then integrating in velocity yields formally the local conservation laws
\begin{equation}\label{macroscopic conservation laws}
	\partial_t\int_{\mathbf{R}^d}f
	\left(\begin{array}{c} 1 \\ v \\ \frac{|v|^2}{2} \end{array}\right)
	dv
	+
	\nabla_x\cdot\int_{\mathbf{R}^d}f
	\left(\begin{array}{c} v \\ v\otimes v \\ \frac{|v|^2}{2}v \end{array}\right)
	dv
	=0 \, ,
\end{equation}
which provides the link to a macroscopic description of the gas.

\bigskip

The other very important feature of the Boltzmann equation comes also from the symmetries of the collision operator. Disregarding integrability issues, we choose~$\varphi=\log f$  and use the properties of the logarithm, to find
\begin{equation}
\label{dissipation-def}
\begin{array}{rl}
D(f)&\ds \equiv  -\int Q(f,f) \log f dv\\
&\ds = \frac14 \int_{\R^d\times \R^d \times {\mathbf S}_1^{d-1}}  b(v-v_1,\omega)(f'f'_1-ff_1) \log {f'f'_1\over ff_1} \, dvdv_1 d\omega \geq 0 \, .
\end{array}
\end{equation}
The so-defined entropy dissipation is therefore a nonnegative functional.

This leads to  Boltzmann's H theorem, also known as second principle of thermodynamics, stating that the entropy is (at least  formally) a Lyapunov functional for the Boltzmann equation.
\begin{equation}
\label{H-thm}
\d_t \int_{\mathbf{R}^d} f\log f dv +\nabla_x \cdot \int_{\mathbf{R}^d}  f\log f vdv \leq 0 \, .
\end{equation}

\bigskip
As to the equation $Q(f,f)=0$, it is possible to show that it is only satisfied by the so-called Maxwellian distributions $M_{\rho,u,\theta}$, which are defined by
\begin{equation}\label{defMrhoutheta}
	M_{\rho,u,\theta}(v):=\frac{\rho}{(2\pi\theta)^{\frac{d}{2}}}e^{-\frac{|v-u|^2}{2\theta}},
\end{equation}
where $\rho\in\mathbf{R}_+$, $u\in\mathbf{R}^d$ and $\theta\in\mathbf{R}_+$ are respectively the macroscopic density, bulk velocity and temperature, under some appropriate choice of units. The relation $Q(f,f)=0$ expresses the fact that collisions are no longer responsible for any variation in the density and so, that the gas has reached statistical equilibrium. In fact, it is possible to show that if the density $f$ is a Maxwellian distribution for some $\rho(t,x)$, $u(t,x)$ and $\theta(t,x)$, then the macroscopic conservation laws \eqref{macroscopic conservation laws} turn out to constitute the compressible Euler system.

More generally, the H-theorem~(\ref{H-thm})   together with the conservation laws (\ref{macroscopic conservation laws}) constitute the key elements of the study of hydrodynamic limits. 

\begin{Rmk}
Note that the irreversibility inherent to the Boltzmann dynamics seems at first sight to contradict the possible existence of a connection with the microscopic dynamics which
is reversible and satisfies the Poincar\'e recurrence theorem (while the Boltzmann dynamics predict some relaxation towards equilibrium).

That irreversibility will actually appear in the limiting process as an arbitrary choice of the time direction (encoded in the distinction between pre-collisional and post-collisional particles), and more precisely as an arbitrary choice of the initial time, which is the only time for which one has a complete information on the correlations. The point is  that the joint probability of having particles of velocity~$(v',v'_1)$ (respectively of velocities~$(v,v_1)$) before the collision is assumed to be equal to~$f(t,x,v')f(t,x,v'_1)$ (resp. to~$f(t,x,v)f(t,x,v_1)$), meaning that  particles should be  independent before collision.
\end{Rmk}

\section{The Cauchy problem}
\setcounter{equation}{0}

Let us first describe briefly the most apparent problems in trying to construct a general, good Cauchy theory for the Boltzmann equation. In the full, general situation, known a priori estimates for the Boltzmann equation are only those which are associated with the basic physical laws, namely the formal conservation of mass and energy, and the bounds on entropy and entropy dissipation. Note that, when the physical space is unbounded, the dispersive properties of the free transport operator allow to further expect some control on the moments with respect to~$x$-variables.  Yet the Boltzmann collision integral is a quadratic operator that is purely local in the position and time variables, meaning that it acts as a convolution in the $v$ variable, but as a pointwise multiplication in the $t$ and $x$ variables~: thus, with the only a priori estimates which seem to hold in full generality, the collision integral is even not a well-defined distribution with respect to $x$-variables. This major obstruction is one of the reasons why the Cauchy problem for the Boltzmann equation is so tricky, another reason being the intricate nature of the Boltzmann operator.

For the sake of simplicity, we shall consider here only bounded collision cross-sections $b$. A huge literature is devoted to the study of more singular cross-sections insofar as the presence of long range interactions always creates  singularities associated to grazing collisions. However, at the present time, there is no extension of Lanford's convergence result in this framework.

\subsection{Short time existence of continuous solutions}

The easiest way to construct local solutions to the Boltzmann equation is to use a fixed point argument in the space of continuous  functions.

Remarking that the free transport operator preserves weighted  $L^\infty$ norms
$$ \Big\| f_0(x-vt, v) \exp\left( \frac{\beta}2 |v|^2\right) \Big\| _{L^\infty} =   \Big\| f_0(x,v) \exp( \frac{\beta}2 |v|^2) \Big\|_{L^\infty} \,  ,$$
  and that the following continuity property holds for the collision operator
$$ \Big\| Q(f,f) (v) \exp(\frac{\beta}2 |v|^2) \Big \|_{L^\infty}  \leq C_\beta  \Big \| f(v) \exp(\frac{\beta}2 |v|^2) \Big\| _{L^\infty}  ^2\,  ,$$
we get the existence of continuous solutions, the lifespan of which is inversely proportional to the norm of the initial data.

\begin{Thm}\label{Boltz-existence0}
Let $f_0\in C^0(\R^d\times \R^d)$ such that
\begin{equation}
\Big\|f_0\exp( \frac \beta 2 |v|^2)\Big \|_{L^\infty} <+\infty
\end{equation}
for some $\beta >0$.
 
Then,  there exists $C_\beta>0$ (depending only $\beta$) such that  the  Boltzmann equation {\rm(\ref{boltz-eq})}   with initial data
$ 
f_0\,
$ has a unique continuous solution on $[0,T]$ with
$$
T = \frac{ C_\beta }{\Big \|f_0\exp( \frac \beta 2 |v|^2)\Big\|_{L^\infty}} \, \cdotp
$$ 
\end{Thm}

Note that the weigthed $L^\infty$ norm controls in particular the macroscopic density
$$ \rho(t,x) := \int f(t,x,v) dv \leq C_\beta \| f (t,x,v) \exp( \frac \beta 2 |v|^2) \|_\infty\,,$$
therefore the possible concentrations for which the collision process can become very pathological.
This restriction, even coming from a very rough analysis, has therefore a physical meaning.

\subsection{Fluctuations around some global equilibrium}

Historically the first global existence result for the spatially inhomogeneous Boltzmann equation is due to S. Ukai \cite{ukai1,ukai2}, who considered initial data that are
fluctuations around a global equilibrium, for instance around the reduced centered Gaussian $M:=M_{1,0,1}$ with notation~(\ref{defMrhoutheta}):
$$
f_0=M(1+g_0) \, .
$$
He proved the global existence of a solution to the Cauchy poblem for (\ref{boltz-eq}) under the assumption 
that the initial perturbation $g_0$ is smooth and small enough in a norm that involves derivatives and weights so as to ensure decay for
large $v$.

The convenient  functional space to be considered is indeed
$$
H_{\ell,k}=\{g\equiv g(x,v)\,/\,
\|g\|_{\ell,k}:=\sup_v(1+|v|^k)\|M^{1/2}g(\cdot,v)\|_{H^ \ell_x}<+\infty\}\,.
$$

\begin{Thm}[\cite{ukai1,ukai2}]\label{Boltz-existence1}
Let $g_0\in H_{\ell,k}$ for $\ell>d/2$ and $k>d/2+1$ such that
\begin{equation}
\label{small-init}
\|g_0\|_{\ell,k}\leq a_0
\end{equation}
for some $a_0$  sufficiently small. 
 
Then,  there exists a unique global solution $f=M(1+g)$ with 
$g\in L^\infty(\R^+, H_{\ell,k}) \cap C(\R^+, H_{\ell,k})$ to the  Boltzmann equation {\rm(\ref{boltz-eq})}
 with initial data
$$ 
g_{|t=0}=g_0\,.
$$
\end{Thm}

 Such a  global existence result is based on Duhamel's
formula and on  Picard fixed point theorem. It requires a very precise study of the linearized collision operator~$\cL_M$ defined by
 $$\cL_M g :=-{2\over M} Q(M,Mg) \, ,$$
and more precisely of   the semi-group $U$ generated by
$$v\cdot\nabla_x+ \cL_M\,.
$$
The main disadvantage inherent to that strategy is that one cannot expect to extend such a result to classes of initial data with less regularity.

\subsection{Renormalized solutions}

The theory of renormalized solutions goes back to the late 80s and
is due to  R. DiPerna and  P.-L. Lions \cite{diperna-lions}. It holds for physically admissible initial data of arbitrary sizes,
but does not yield solutions that are known to
solve the Boltzmann equation in the usual weak sense.

Rather, it gives the existence of a global weak solution to
a class of formally equivalent initial-value problems.

\begin{Def}
A renormalized solution     of the Boltzmann equation {\rm(\ref{boltz-eq})}  relatively to the global equilibrium $M$ is a function
$f\in C(\R^+, L^1_{loc}(\R^d \times \R^d))$ such that
$$ H(f|M) (t) := \iint \int  \left( f \log {f\over M} -f+M \right) (t,x,v)
     \ dv \,  dx<+\infty \,,$$
which satisfies in the sense of distributions
\begin{equation}
    \label{boltz-renormalized}
\begin{array}l
   \displaystyle M \Big( \d_t + v \cdot \nabla_x  \Big) \Gamma\left({f\over M}\right)
    =   \Gamma'\left({f\over M}\right) Q(f,f) \quad \hbox{ on } \R^+\times \R^d \times \R^d\,, \\
   \displaystyle f_{|t=0} = f_0 \geq 0 \quad \hbox{ on } \R^d \times \R^d\,.
\end{array}
\end{equation}
 for any
$\Gamma\in C^1(\R^+)$ such that 
$|\Gamma'(z)| \leq { C/ \sqrt{1+z}}$.
 \end{Def}

\bigskip

With the above definition of renormalized solution relatively to $M$,  the following existence result holds~:

\begin{Thm}[\cite{diperna-lions}] \label{Boltz-existence2}
  Given any initial data $f_0$
satisfying
\begin{equation}
    H(f_0|M) = \int \int  \left( f_0 \log {f_0\over M} -f_0+M \right) (x,v)
     \ dv \,  dx<+\infty \, ,
\end{equation}
there exists a renormalized solution $f \in C(\R^+, L_{loc}^1(\R^d \times \R^d))$ relatively to $M$ to the Boltzmann equation {\rm(\ref{boltz-eq})}
with initial data $f_0$.

   Moreover, $f$ satisfies 

- the continuity equation
\begin{equation}
\label{mass-conservation}
 \d_t \int fdv +\nabla_x \cdot \int fvdv =0 \, ;
\end{equation}

- the momentum equation with defect measure
\begin{equation}
\label{momentum-conservation}
 \d_t \int fvdv +\nabla_x \cdot \int fv\otimes v dv +\nabla_x \cdot m=0
\end{equation}
where $m$ is a Radon measure on $\R^+\times \R^d$ with values in the nonnegative symmetric matrices;

- the entropy inequality
\begin{equation}
\label{entropy+defect}
\begin{array}r
\ds H(f|M) (t) + \int \Tr (m )(t) + \int_0^t \int D(f)(s,x)dsdx\\
\ds 
\leq H(f_0|M)
\end{array}
\end{equation}
where $\Tr (m)$ is the trace of the nonnegative symmetric matrix $m$, and  the entropy dissipation $D(f)$ is defined by {\rm(\ref{dissipation-def})}.
\end{Thm}

The weak stability of approximate solutions is inherited from the entropy inequality.
 In order to take limits in the renormalized Boltzmann equation, we have further
to obtain some strong compactness.  The crucial idea
here  is to use the velocity averaging lemma due to F. Golse, P.-L. Lions, B. Perthame
and R. Sentis \cite{GLPS}, stating that the moments in
$v$ of the  solution to some transport equation are  more regular than the function itself.

\begin{Rmk}\label{Linfty-rmk}
As we will see, the major weakness of the convergence theorem describing the Boltzmann equation as the low density limit of large systems of particles is the very short time on which it holds.
However, the present state of the art regarding the Cauchy theory for the Boltzmann equation makes it very difficult to improve.

Because of the scaling of the microscopic interactions, the conditioning on energy surfaces (see Chapter~{\rm\ref{convergenceresult}}) introduces strong spatial oscillations in the initial data. We therefore do not expect to get regularity so that we could take advantage of the perturbative theory of S. Ukai~\cite{ukai1,ukai2}. A coarse graining argument would be necessary to retrieve spatial regularity on the kinetic distribution, but we are not aware of any breakthrough in this direction.

As for using the DiPerna-Lions theory~\cite{diperna-lions}, the first step would be to understand the counterpart of renormalization at the level of the microscopic dynamics, which seems to be also a very challenging problem.
\end{Rmk}

%% file: structure.tex
\chapter{ Main results }

\section{Lanford and King's theorems}
\setcounter{equation}{0}
The main goal of this monograph is to prove the two following statements. 
We give here compact, and somewhat informal, statements of our two main results. Precise statements are given in Chapters~\ref{convergenceresult} and~\ref{convergence}  (see Theorem~\ref{main-thm} page~\pageref{main-thm} for the hard-spheres case, and~\ref{main-thmpotential} page~\pageref{main-thmpotential} for the potential case).  
 
The following statement concerns the case of hard spheres dynamics, and the main
ideas behind its proof go back to the fundamental work of Lanford~\cite{lanford}.
 \begin{Thm}\label{thm-cv-intro-HS} 
 Let $f_0:\R^{2d} \mapsto \R^+$ be a continuous density of probability such that
 $$\big\| f_0(x,v) \exp(\frac\beta 2|v|^2 ) \big\|_{L^\infty(\R^{2d})}<+\infty$$
 for some $\beta>0$.
 
  Consider the system of $N$ hard spheres of diameter $\eps$, initially distributed according to $f_0$ and ``independent",  governed by the system~{\rm(\ref{n1-HS})-(\ref{HS-BC})}. Then, in the Boltzmann-Grad limit~$N \to \infty, \, N\eps^{d-1}= 1$, its distribution function  converges to the solution  to the Boltzmann equation~{\rm(\ref{boltz-eq})} with the cross-section~$b(w,\omega):= ( \omega   \cdot w)_+$ and with initial data $f_0$, in the sense of observables. 
 \end{Thm}

The next theorem concerns the Hamiltonian case (with a repulsive potential), and important steps of the proof can be found in the thesis of King~\cite{K}.
 \begin{Thm}\label{thm-cv-intro} 
Assume that the repulsive potential $\Phi$ satisfies   Assumption~{\rm\ref{propertiesphi}} as well as the technical assumption~{\rm(\ref{strangeassumption})}.  Let $f_0:\R^{2d} \mapsto \R^+$ be a continuous density of probability such that
 $$\big\| f_0 \exp(\frac\beta 2|v|^2 ) \big\|_{L^\infty}<+\infty$$
 for some $\beta>0$.

Consider the system of $N$ particles, initially distributed according to $f_0$ and ``independent", governed by the system~{\rm(\ref{n1.1})}. Then, in the Boltzmann-Grad limit~$N \to \infty, \, N\eps^{d-1} = 1$, its distribution function  converges to the solution  to the Boltzmann equation~{\rm(\ref{boltz-eq})} with a bounded cross-section, depending on~$\Phi$ implicitly,  and with initial data $f_0$, in the sense of observables. 
 \end{Thm}

 \begin{Rmk}
Convergence in the sense of observables means that, for any test function~$\varphi $ in~$ C^0_c(\R^d_v)$, the corresponding observable
 $$\phi_\eps(t,x) :=\int f_\eps (t,x,v) \varphi(v) dv  \longrightarrow \phi(t,x) := \int f(t,x,v) \varphi(v) dv $$
 uniformly in $t$ and $x$. We indeed recall that the kinetic distribution cannot be measured, only averages can be reached by physical experiments~: this accounts for the terminology ``observables".
 
 In mathematical terms, this means that we establish only   weak convergence with respect to the $v$-variable. Such a convergence result does not exclude the existence of pathological behaviors, in particular dynamics  obtained by reversing the arrow of time and which are predicted by the (reversible) microscopic  system. We shall only prove that these behaviors have negligible probability in the limit $\eps \to 0$.
 \end{Rmk}

 \begin{Rmk}\label{independence-rmk}
 The initial  independence assumption has to be understood also asymptotically. It will be discussed with much details in Chapter~{\rm\ref{convergenceresult}} (see also Chapter~{\rm\ref{convergence}} in the case of a potential): it is actually related to some coarse-graining arguments which are rather not intuitive at first sight.
  
  For hard spheres, the exclusion obviously prevents independence for fixed $\eps$, but we expect to retrieve this independence as $\eps \to 0$  if we consider a fixed number $s$ of particles.
  The question is to deal with an infinite number of such particles.
  
  The case of the smooth Hamiltonian system could seem to be simpler insofar as particles can occupy the whole space. Nevertheless, in order to control the decay at large energies, we need to introduce some conditioning on energy surfaces, which is very similar to   exclusion. 
 \end{Rmk}

  \begin{Rmk}\label{potential-rmk}
 The technical assumption~{\rm(\ref{strangeassumption})} will be made explicit in Chapter~{\rm\ref{scattering}}~: it ensures that the deviation angle is a suitable parametrization of the collision, and more precisely that we can retrieve the impact parameter from both the ingoing velocity and the deviation angle. What we will use is the fact that the jacobian of this change of variables is bounded at least locally.
 
 Such an assumption is not completely compulsory for the proof. We can imagine of splitting the integration domain in many subdomains where the deviation angle is a good parametrization of the collision, but then we have to extend the usual definition of the cross-section. The important point is that the deviation angle cannot be a piecewise constant function of the impact parameter.
  \end{Rmk}

 \section{Background and references}
\setcounter{equation}{0}

The problem of asking for a rigorous derivation of the Boltzmann equation from the Hamiltonian dynamics goes back to Hilbert \cite{hilbert}, who suggested to use the Boltzmann equation as an intermediate step between the Hamiltonian dynamics and fluid mechanics, and who described this axiomatization of physics as a major challenge for mathematicians of the twentieth century.

We shall not give an exhaustive presentation of the studies  that have been carried out on this question but indicate some of the fundamental landmarks, concerning for most of them the case of hard spheres. First one should mention N. Bogoliubov~\cite{Bogoliubov}, M. Born, and H. S.   Green~\cite{Born}, J. G. Kirkwood~\cite{Kirkwook} and J. Yvan~\cite{Yvan}, who gave their names to the BBGKY hierarchy on the successive marginals, which we shall be using extensively in this study.   H. Grad was able to obtain  in~\cite{gradphys} a differential equation on the   first marginal which after some manipulations converges towards the Boltzmann equation.

The first mathematical result on this problem goes back to C. Cercignani~\cite{cercignanirigid} and O. Lanford \cite{lanford} who proved that the propagation of chaos should be established by a careful study of trajectories of a hard spheres system, and who exhibited -- for the first time --  the origin of irreversibility. The proof, even though incomplete, is therefore an important breakthrough. The limits of their methods, on which we will comment later on -- especially regarding the short time of convergence -- are still challenging questions.
 
 The argument of O.  Lanford was then revisited and completed in several works. Let us mention especially the contributions of K. Uchiyama~\cite{uchiyama}, C. Cercignani, R. Illner and M. Pulvirenti \cite{CIP} and H. Spohn \cite{S1} who introduced a mathematical formalism, in particular to get uniform a priori estimates for the  solutions to the BBGKY hierarchy which turns out to be  a  theory in the spirit of the Cauchy-Kowalewskaya theorem.

 The term-by-term convergence of the hierarchy in the Boltzmann-Grad scaling was studied in more details by C. Cercignani, V. I. Gerasimenko and D. I. Petrina \cite{CGP}~: they provide for the first time quantative estimates on the set of ``pathological trajectories", i.e. trajectories for which the Boltzmann equation does not provide a good approximation of the dynamics. What is not completely clear in this approach is the stability of the estimates under microscopic spatial translations.

 The method of proof  was then extended
 \begin{itemize}
 \item to the case when the initial distribution is close to vacuum, in which case global in time results may be proved \cite{CIP,IP1,IP};
 
 \item to the case when interactions are localized but not pointwise \cite{K}. Because multiple collisions are no longer negligible, this requires a careful study of clusters of particles.

 \end{itemize}
Many review papers deal with those different results, see \cite{EP,pulvirenti,villani} for instance.
  
\bigskip

    Let us now summarize the strategy of the proofs.  Their are two main steps:
\begin{itemize}
\item[(i)] a short time bound for the series expansion expressing the correlations of the system of $N$ particles  and the corresponding quantities of the Boltzmann equation;
\item[(ii)] the term by term convergence.
 \end{itemize}
 
  In the case of hard spheres, point (i) is just a matter of explicit estimates,
while point (ii) is usually considered as almost obvious (but deep).   Among experts in the field the hard sphere case is therefore considered to be completely solved.
However, we could not find a proof for the measure zero estimates (i.e. the control of recollisions) in the litterature. It might
be that to experts in the field such an estimate is easy, but from our point of view
it turned out to be quite delicate.

\begin{itemize}

\item For the Boltzmann dynamics, it  seems to be correct that a zero measure argument allows to control recollisions inasmuch as particles are pointwise.
\item For fixed $\eps$, we will see that the set of velocities leading to recollisions (even in the case of three particles) is small but not zero~:  this cannot be obtained by a straightforward thickening argument   without any {\bf geometrical information} on the limiting zero measure set. 
\item For the microscopic system of $N$ particles, collisional particles are at a distance $\eps$ from each other, we thus expect that even ``good trajectories"
deviate from trajectories associated to the Boltzmann dynamics. We shall therefore need some {\bf stability of ``pathological sets"} of velocities with respect to microscopic spatial translations, to be able to iterate the process.
\end{itemize}

   \section{New contributions}
   
    Our  goal here is to provide a  self-contained presentation, which includes all the details of the proofs, especially concerning   term-by-term  convergence which to our knowledge is not completely written anywhere, even in the hard-spheres case.

 \bigskip
 
  Part II is a review of known results in the case of hard spheres.
  Following  Lanford's strategy, we shall establish the starting hierarchy of
equations, providing a short time, uniform estimate. 
   
   We   focus especially on the {\bf definition of  functional spaces}: we shall see that the short time estimate is obtained as an analytical type result, meaning that we control all correlation functions together. The functional spaces we   consider are in some sense natural from the point of view of statistical physics, since they involve two parameters $\beta$ and $\mu$ (related to the inverse temperature and chemical potential) to control the growth of energy and of the number of particles. Nevertheless, instead of usual $L^1$ norms, we   use $L^\infty$ norms, which are needed to control collision integrals (see Remark~\ref{Linfty-rmk}).
   
   The second point we   discuss in details is the {\bf notion of  independence}. As noted in Remark~\ref{independence-rmk},  for any fixed $\eps>0$, because of the exclusion, particles cannot be independent. In the $2Nd$-dimensional phase-space, we shall see actually that the Gibbs measure has   support on    only a very small set. Careful estimates on the partition function show however that the marginal of order $s$ (for any fixed $s$) converges to some tensorized distribution, meaning that independence is recovered at the limit $\eps \to 0$.

\bigskip

Part III deals with the case of the Hamiltonian system, with a repulsive potential. It basically follows  King's thesis~\cite{K},  filling in some gaps.

   In the limit $\eps\to 0$ with $N\eps^{d-1}\equiv1$, we would like to obtain a kind of homogeneization result~: we want to average the motion over the small scales in $t$ and $x$, and replace the localized interactions by  pointwise collisions as in the case of hard spheres.
 We   therefore introduce an {\bf artificial boundary} (following~\cite{K}) so that
 \begin{itemize}
 \item
 on the exterior domain, the dynamics reduces to free transport,
 \item
 on the interior domain, the dynamics can be integrated in order to compute outwards boundary conditions
 in terms of the incoming flux. Note that such a scattering operator is relevant only if we can guarantee that there is no other particle involved in the interaction. 
  \end{itemize}
   
  An important point is therefore to control multiple collisions, which - contrary to the case of hard spheres - could happen for a non zero set of initial data. We however expect that they become negligible in the Boltzmann-Grad limit (as the probability of finding three particles having approximately  the same position tends to zero). {\bf Cluster estimates}, based on suitable partitions of the $2Nd$-dimensional phase-space and symmetry arguments, give the required asymptotic bound on multiple collisions.

   \bigskip
   
Part IV is the heart of our contribution, where we   establish the term-by-term convergence. Note that the arguments work  in the same way in both situations (hard spheres and potential case), up to some minor technical points due to the fact that, for the $N$-particle Hamiltonian system, pre-collisional and post-collisional configurations differ by their velocities but also by their microscopic positions and by some microscopic shift in time.
   
   However the two main difficulties are exactly the same:
   \begin{itemize}
   \item describing geometrically the set of ``pathological" velocities and deflection angles leading to possible recollisions, in order to get a {\bf quantitative estimate} of its measure;
   \item proving that this set is {\bf stable under small translations} of positions.
   \end{itemize}
   
   \medskip
   
   Note that the estimates we   establish depend only on the scattering operator, so that we have a rate of convergence which can be made explicit for instance in the case of hard spheres.
   
To control the set of recolliding trajectories by means of explicit estimates,
we   make use of properties of the cross-section which
are not guaranteed a priori for a generic repulsive potential. 
Assumption~(\ref{strangeassumption}) guarantees that these conditions are satisfied.

%% file: BBGKY-HS.tex
\chapter{Microscopic dynamics and BBGKY hierarchy}\label{spheres}
\setcounter{equation}{0}

In this  chapter we define
 the~$N$-particle flow for hard spheres (introduced in Chapter~\ref{intro}), and write down the associated BBGKY hierarchy. Finally we present a formal derivation of the Boltzmann hierarchy, and the Boltzmann equation of hard spheres.  This chapter follows the classical approaches of~\cite{alexanderthesis}, \cite{CGP}, \cite{CIP}, \cite{lanford}, among others.

\section{The $N$-particle flow}\label{Nparticle}
We consider~$N$ particles in the space~$\R^d$, the motion of which is 
described by~$N$ positions~$(x_1,\dots,x_N)$ and~$N$ velocities~$(v_1,\dots,v_N)$, each in~$\R^{d}$. Denoting by~$Z_N:= (z_1,\dots,z_N)$ the set of particles, 
   each particle~$z_i:=(x_i,v_i) \in \R^{2d}$ is submitted to free flow
   \begin{equation}
\label{hard-spheres}
\begin{aligned}
\forall 1 \leq i \leq N \, , \quad {dx_i\over dt} =v_i \, , \quad {dv_i\over dt} = 0 \end{aligned}
\end{equation}
 on the domain\label{def:DNHS}
   $$
    \cD_N :=\Big \{ Z_N \in \R^{2dN}\,/\, \forall i\neq j,\, |x_i-x_j|>\eps \Big\}
   $$
    and bounces off the boundary~$\partial \cD_N$ according to the laws of elastic reflection: if $  |x_i-x_j|=\eps$
\begin{equation}
\label{hard-spheres-BC}
\begin{aligned}
v_i^{in} = v_i^{out} - \nu^{i,j} \cdot ( v_i^{out} -   v_j^{out} ) \, \nu^{i,j}&  \\
v_j^{in} = v_j^{out} + \nu^{i,j} \cdot ( v_i^{out} -   v_j^{out} )  \,\nu^{i,j} & \,  
\end{aligned}
\end{equation}
where~$\nu^{i,j}:=(x_i-x_j) / |x_i-x_j|$, and in the case when~$\nu^{i,j}\cdot (v_i^{in} -   v_j^{in} )<0$ (meaning that the ingoing velocities are precollisional).

Contrary to the potential case studied in Part III, 
it is not obvious to check that (\ref{hard-spheres}) defines  a global dynamics, at least for almost all initial data. 
Note indeed that this is not a simple consequence of the Cauchy-Lipschitz theorem since the boundary condition is not smooth, and even not defined for all configurations. We call  {\em pathological} a trajectory such that

- either there exists a collision involving more than two particles, or the collision is grazing (meaning that~$\nu^{i,j}\cdot (v_i^{in} -   v_j^{in} )=0$) hence the boundary condition is not well defined;

- or there are an infinite number of collisions in finite time so the dynamics cannot be globally defined.

In~\cite[Proposition~4.3]{alexander},  it is stated  that outside a negligible set of initial data there are no pathological trajectories; the complete proof is provided in~\cite{alexanderthesis}. Actually the setting of~\cite{alexanderthesis} is more complicated than ours since an infinite number of particles is considered. The arguments of~\cite{alexanderthesis} can however be easily adapted to our case to yield the following result, whose proof we detail for the convenience of the reader. 

\begin{prop}\label{flowwelldefined}
Let~$N,\eps$ be fixed. The set of initial configurations leading to a pathological trajectory   is of measure zero in~$\R^{2dN}$.
\end{prop}
We first prove  the following elementary lemma, in which we have used the following notation:
for any~$s \in \N^*$ and~$R > 0,$ we denote~$B_R^s  :=\displaystyle \{V_s \in \R^{ds}, \, |V_s| \leq R\}$ where~$|\cdot|$ is the euclidean norm; we often write~$B_R:=B_R^1$.\label{indexdefBR}
\begin{Lem}\label{sizehardspheres}
Let~$\rho,R>0$ be given, and~$\delta < \eps/2$. Define
$$
\begin{aligned}
 I:=\Big\{ Z_N \in B_\rho^{N} \times B_R^{N} \, /  \,  \mbox{one particle will collide with two others on the time interval}\, \, [0,\delta] \Big\} \, .
\end{aligned}
$$
Then~$
 |I  | \leq C(N,\eps,R) \, \rho^{d(N-2)} \delta^2 \, .
$
\end{Lem}
\begin{proof}
We   notice that~$I$ is a subset of 
$$
\begin{aligned}
  \Big\{Z_N \in B_\rho^{N} \times B_R^{N} \, /  \, \exists \{i,j,k\} \, \mbox{distinct}  \, , & \quad |x_i-x_{j}| \in [\eps,\eps+2R\delta] \quad  \mbox{and} \quad |x_i-x_{k}| \in [\eps,\eps+2R\delta] \Big\} \, ,
\end{aligned}
$$
and the lemma follows directly.
\end{proof}
\begin{proof}[Proof of Proposition~\ref{flowwelldefined}]
Let~$R>0$ be given and fix some time~$t>0$. Let~$\delta< \eps/2$  be a parameter such that~$t/\delta$ is an integer. 

Lemma~\ref{sizehardspheres} implies that there is a subset~$I_0(\delta, R)$ of~$B_R^{N} \times B_R^{N} $ of measure at most~$C (N,\eps,R)R^{d(N-2)} \delta^2 $ such that   any initial configuration belonging to~$(B_R^{N} \times B_R^{N}) \setminus I_0( \delta, R)$ generates a solution on~$[0,\delta]$ such that each particle encounters at most one other particle on~$[0,\delta]$.  Moreover up to removing a measure zero set of initial data each collision is non-grazing.

Now let us start again at  time~$\delta$. We recall that in the velocity variables, the ball of radius~$R$ in $\R^{dN}$ is stable by the flow, whereas the positions at  time~$\delta$  lie in the ball~$B_{R+R\delta}^{N}$. Let us apply Lemma~\ref{sizehardspheres} again to that new initial configuration space.
Since  the measure is invariant by the flow, we can construct  a subset~$I_1(\delta, R)$  of the initial positions~$B_R^{N} \times B_R^{N} $,  of size~$C(N,\eps,R) R^{d(N-2)}(1+\delta)^{d(N-2)} \delta^2$  such that outside~$I_0 \cup I_1(\delta, R)$, the flow starting from any initial point in~$B_R^{N} \times B_R^{N} $ is such that  each particle encounters at most one other particle on~$[0,\delta]$, and then  at most one other particle on~$[\delta,2\delta]$, again in a non-grazing collision. 
We repeat the procedure~$t/\delta$ times: we   construct a subset
$$\displaystyle I_\delta (t , R) := \bigcup_{j=0}^{t/\delta-1} I_j(\delta, R)$$ of~$B_R^{N} \times B_R^{N}$, of measure
$$
\begin{aligned}
|I_\delta (t , R) | & \leq C(N,\eps,R) R^{d(N-2)} \delta^2\sum_{j=0}^{t/\delta-1} (1+j\delta)^{d(N-2)} \\
& \leq C(N,R,t,\eps)\delta \, ,
\end{aligned}  
$$
such that for any initial configuration in~$B_R^{N} \times B_R^{N} $ outside that set, the flow is well-defined up to time~$t$. The intersection
$
\displaystyle I (t,R):=\bigcap_{\delta>0} I_\delta(t,R) $ is of measure zero, and   any initial configuration  in~$B_R^{N} \times B_R^{N} $  outside~$I(t,R)$ generates a well-defined flow until time~$t$. 
 Finally we consider the countable union of those zero measure  sets
$\displaystyle
I := \bigcup_n I(t_n,R_n)
$
where~$t_n$ and~$ R_n$ go to infinity, and any initial configuration in~$\R^{2dN}$ outside~$I$ generates a globally defined flow. 
The proposition is proved.
\end{proof}


\section{The Liouville equation and the BBGKY hierarchy}\label{liouvillebbgkyHS}
\setcounter{equation}{0}

According to Part I, Paragraph~\ref{Liouville-PartI}, the Liouville equation 
relative to the particle system~\eqref{hard-spheres} is
\begin{equation}\label{liouvillespheres}
\partial_t f_N +\sum_{i=1}^N v_i \cdot \nabla_{x_i} f_N  = 0 \quad \mbox{on} \quad  \cD_N
\end{equation}
with the boundary condition~$f_N(t,Z_N^{in}) = f_N(t,Z_N^{out}) $.  We recall the assumption that~$f_N$  is invariant by permutation in the sense of~(\ref{sym}),  meaning that the particles are indistinguishable.

%

The classical strategy to obtain asymptotically a kinetic equation such as (\ref{boltz-eq}) is to write the evolution equation for the first marginal of the distribution function $f_N$\label{indexdefmarginalbbgkyHS}, namely
$$
f_N^{(1)} (t,z_1):= \int_{\R^{2d(N-1)}} f_N(t,z_1,z_2, \dots,z_N) \indc_{Z_N \in  \cD_N}  \,  dz_2 \dots dz_N\, .
$$
The point to be noted is that the evolution of $f_N^{(1)}$ depends actually on $f_N^{(2)}$ because of the quadratic interaction imposed by the boundary condition. And in the same way, the equation on $f_N^{(2)}$ depends on $f_N^{(3)}$. Instead of a kinetic equation, we therefore obtain a hierarchy of equations involving all the marginals of $f_N$ 
\begin{equation}
\label{marginalHS}
 f_N^{(s)} (t,Z_s) :=    \int_{\R^{2d(N-s)}} f_N (t,Z_s,z_{s+1},\dots,z_N) \indc_{Z_N \in  \cD_N} \: dz_{s+1} \cdots dz_N \, .
 \end{equation}
 Notice that~$ f_N^{(s)} (t,Z_s) $ is defined on~${\mathcal D}_s$ only, and that
 \begin{equation}
\label{nicemarginalsHS}
  f_N^{(s)} (t,Z_s) = \int_{\R^{2d}} f_N^{(s+1)}  (t,Z_s,z_{s+1})\: dz_{s+1} \, .
\end{equation}
 Finally by integration of the boundary condition on~$f_N$ we find that~$f_N^{(s)} (t,Z_s^{in}) = f_N^{(s)} (t,Z_s^{out}) $.
  An equation for the    marginals is derived in weak form in Section~\ref{weakliouvilleHS}, and from that equation we derive formally the Boltzmann hierarchy in the Boltzmann-Grad limit (see Section~\ref{spheresboltz}).

\section{Weak formulation of Liouville's equation}\label{weakliouvilleHS}
\setcounter{equation}{0}

 Our goal in this section is to find the weak formulation of the system of equations satisfied 
 by the family of   marginals $\big(   f_N^{(s)}\big)_{1 \leq s \leq N}$ defined above in \eqref{marginalHS}. From now on we assume that~$  f_N $ decays at infinity in the velocity variable (the functional setting will be made precise in Chapter~\ref{existence}).

  Given a smooth, compactly supported  function~$\phi$   defined on~$\R_+ \times {\mathcal D}_s$ and satisfying the symmetry assumption~(\ref{sym}) as well as the boundary condition~$\phi (t,Z_s^{in}) = \phi (t,Z_s^{out}) $,  we have 
\begin{equation}\label{integralformulationfHS}
\begin{aligned}
\int_{ \R_+ \times  \R^{2dN} } \bigl(
\partial_t f_N + \sum_{i = 1}^N v_i \cdot \nabla_{x_i} f_N  \bigr)  \phi (t,Z_s)  \indc_{Z_N \in  \cD_N}   \:  d Z_N dt = 0 \, .
\end{aligned}
\end{equation}

We now use  integrations by parts to derive from (\ref{integralformulationfHS}) the weak form of the equation in the marginals~$  f_N^{(s)}.$ 
On the one hand an integration by parts in the time variable gives
$$ \begin{aligned}
\int_{ \R_+ \times   \R^{2dN} }  \partial_t   f_N (t,Z_N)  \phi (t,Z_s)  \displaystyle \indc_{Z_N \in  \cD_N}  \:  d Z_N dt& =  - \int_{  \R^{2dN} } f_N(0,Z_N) \phi(0,Z_s) \displaystyle  \indc_{Z_N \in  \cD_N}   \: dZ_N \\ & \qquad - \int_{ \R_+ \times   \R^{2dN} }  f_N (t,Z_N) \partial_t\phi (t,Z_s)  \displaystyle \indc_{Z_N \in  \cD_N} \:    d Z_N dt \, ,
\end{aligned}$$
hence, by definition of $  f_N^{(s)}$ in~(\ref{marginalHS}),
$$
 \begin{aligned}
 \int_{ \R_+ \times  \R^{2dN} }  \partial_t  f_N (t,Z_N)  \phi (t,Z_s)  \displaystyle \indc_{Z_N \in  \cD_N}  \:  d Z_N dt& =- \int_{\R^{2ds}}    f_N^{(s)}(0,Z_s) \phi(0,Z_s) \: dZ_s  \\
 & \qquad - \int_{ \R_+ \times \R^{2ds}}   f_N^{(s)}  (t,Z_s) \partial_t\phi (t,Z_s) \:    d Z_s dt \, .
\end{aligned}
$$
Now let us compute
$$\displaystyle   \sum_{i = 1}^N \int_{  \R^{2dN} } v_i \cdot \nabla_{x_i} f_N (t,Z_N) \phi  (t,Z_s)   \displaystyle  \indc_{Z_N \in  \cD_N}  \,dZ_N = \int_{  \R^{2dN} } 
  \mbox{div}_{X_N} \big ({V_N} \:  f_N   (t,Z_N) \big)  \phi  (t,Z_s)    \indc_{Z_N \in  \cD_N}  \,dZ_N$$
   using Green's formula. The boundary terms involve configurations with at least one pair $(i,j)$ satisfying~$|x_i - x_j| = \e.$ 
   According to Paragraph~\ref{Nparticle} we may neglect configurations where more than two particles collide at the same time, so the boundary condition is well defined. 
   For any~$i $ and~$ j $ in~$ \{1,\dots ,N\} $  we   denote
   $$
{\Sigma}_N (i,j) := \Big\{
X_N \in  \R^{2dN}, \,  |x_i - x_j| = \eps   \Big\}  \, ,
$$
and~$n^{i,j}$\label{indexdefnij} is the outward normal to~${\Sigma}_N (i,j) $ in~$\R^{dN}$. We   obtain by  Green's formula:
$$ \begin{aligned} 
& \sum_{i= 1}^N \int_{ \R_+ \times \R^{2dN}} 
 v_i\cdot \nabla_{x_i}   f_N  (t,Z_N) \phi  (t,Z_s)   \displaystyle \indc_{Z_N \in  \cD_N}  \,dZ_N \, dt \\
 &\qquad  = - \sum_{i=1}^s    \int_{ \R_+ \times \R^{2dN}}    f_N (t,Z_N)  v_i\cdot \nabla_{x_i} \phi  (t,Z_s)   \displaystyle  \indc_{Z_N \in  \cD_N}  \,dZ_N dt \\
 &\qquad      \quad 
 \displaystyle  {} +     \sum_{1 \leq i\neq j\leq N}  \int_{  \R_+ \times  \R^{dN}\times{\Sigma}_N  (i,j) }  
n^{i,j}  \cdot V_N \:   f_N (t,Z_N)\phi (t,Z_s)\:    d \sigma_N^{i,j} dV_N dt  \,,
\end{aligned}
$$ 
with~$d\sigma_N^{i,j}$ the surface measure on~${\Sigma}_N (i,j) $, induced by the Lebesgue measure. 
Now we split the last term into four parts:
$$
 \begin{aligned} 
&  \sum_{1 \leq i\neq j\leq N}  \int_{  \R_+ \times  \R^{dN}\times{\Sigma}_N  (i,j) }  
n^{i,j}  \cdot V_N \:   f_N (t,Z_N)\phi (t,Z_s)\:    d \sigma_N^{i,j} dV_N dt  \\
&\qquad        
 \displaystyle  {} =     \sum_{i=1}^s \sum_{j=s+1}^N  \int_{  \R_+ \times  \R^{dN}\times{\Sigma}_N  (i,j) }  
n^{i,j}  \cdot V_N \:   f_N (t,Z_N)\phi (t,Z_s)\:    d \sigma_N^{i,j} dV_N dt \\
 &\qquad      \quad 
 \displaystyle  {} +     \sum_{i=s+1}^N \sum_{j= 1}^s  \int_{  \R_+  \times  \R^{dN}\times{\Sigma}_N  (i,j) }  
n^{i,j}  \cdot V_N \:   f_N (t,Z_N)\phi (t,Z_s)\:    d \sigma_N^{i,j} dV_N dt \\
 &\qquad      \quad 
 \displaystyle  {} +   \sum_{1 \leq i\neq j \leq s}  \int_{  \R_+  \times  \R^{dN}\times{\Sigma}_N  (i,j) }  
n^{i,j}  \cdot V_N  \:   f_N (t,Z_N)\phi (t,Z_s)\:    d \sigma_N^{i,j} dV_Ndt  \\
 &\qquad      \quad 
 \displaystyle  {} +   \sum_{s+1 \leq i\neq j \leq N}  \int_{  \R_+ \times  \R^{dN}\times{\Sigma}_N  (i,j) }  
n^{i,j}  \cdot  V_N  \:   f_N (t,Z_N)\phi (t,Z_s)\:    d \sigma_N^{i,j}dV_N dt   \,.
\end{aligned}
$$The boundary condition on~$f_N$ and~$\phi$ imply that the two last terms of on the right-hand side are zero. By symmetry~\eqref{sym} and by definition of~$    f_N^{(s)}$, we can write
$$
\begin{aligned} 
&\sum_{i=1}^s \sum_{j=s+1}^N  \int_{  \R_+  \times  \R^{dN}\times{\Sigma}_N  (i,j) }  
n^{i,j}  \cdot V_N \:   f_N (t,Z_N)\phi (t,Z_s)\:    d \sigma_N^{i,j} dV_N dt \\
 &\qquad      \quad 
 \displaystyle  {} +     \sum_{i=s+1}^N \sum_{j= 1}^s  \int_{  \R_+ \times  \R^{dN}\times{\Sigma}_N  (i,j) }  
n^{i,j}  \cdot V_N \:   f_N (t,Z_N)\phi (t,Z_s)\:    d \sigma_N^{i,j} dV_N dt \\
& = -(N-s) \sum_{i=1}^s   \int_{ \R_+ \times  {\mathbf S}_\eps^{d-1}\times \R^d \times  \R^{2ds}   }  
{(x_{s+1}-x_i)\over |x_{s+1}-x_i|}  \cdot (v_{s+1}-v_i)   \:   f_N^{(s+1)} (t,Z_s,x_{s+1} ,v_{s+1}) \phi (t,Z_s)\:    dZ_s  d\sigma(x_{s+1})  d v_{s+1}dt\, \\
 & = -(N-s) \e^{d-1}\sum_{i=1}^s   \int_{ \R_+ \times  {\mathbf S}_1^{d-1}\times \R^d \times  \R^{2ds}   }  
\omega \cdot (v_{s+1}-v_i)   \:   f_N^{(s+1)} (t,Z_s,x_{i}+ \eps \omega,v_{s+1}) \phi (t,Z_s)\:    dZ_s  d \omega d v_{s+1}dt\, .
\end{aligned}
$$
 Finally we obtain
  $$
 \begin{aligned}
&\displaystyle  \sum_{i=  1}^N \int_{ \R_+ \times  \R^{2dN}} 
 v_i\cdot \nabla_{x_i}     f_N (t,Z_N)  \phi  (t,Z_s)    \displaystyle\indc_{Z_N \in  \cD_N} \,dZ_N \, dt  \\
& \displaystyle     = -\sum_{i=1}^s    \int_{ \R_+ \times \R^{2ds}}     f_N^{(s)} (t,Z_s)   v_i\cdot \nabla_{x_i} \phi  (t,Z_s)   \,dZ_s dt  \\
&\displaystyle   {}  - (N-s) \e^{d-1}\sum_{i=1}^s   \int_{ \R_+ \times  {\mathbf S}_1^{d-1}\times \R^d \times  \R^{2ds}   }  
\omega \cdot (v_{s+1}-v_i)   \:   f_N^{(s+1)} (t,Z_s,x_{i}+ \eps \omega,v_{s+1}) \phi (t,Z_s)\:    dZ_s  d \omega d v_{s+1}dt \,  .
  \end{aligned}
$$
 It remains to define the {\it collision operator}
\begin{equation}
\label{def:collisionop0}
\begin{aligned}
\big( {\mathcal C}_{s,s+1}  f_N^{(s+1)}\big)  (t,Z_{s}):=   (N-s) \e^{d-1}&\sum_{i=1}^s    \int_{ {\mathbf S}_1^{d-1}\times \R^d}  
  \omega \cdot (v_{s+1}-v_i) \\
&\quad {}\times {}   f_N^{(s+1)} (t,Z_s,x_{i}+ \eps \omega,v_{s+1})     d \omega d v_{s+1} \, , 
  \end{aligned}
\end{equation}
where recall that~${\mathbf S}_1^{d-1}$  is the unit sphere of~$\R^d$,   and in the end
   we obtain the weak formulation of  the BBGKY hierarchy
\begin{equation}
\label{BBGKYhierarchyspheres}
  \d_t   f_N^{(s)} +\sum_{1 \leq i \leq s} v_i\cdot \nabla_{x_i}   f_N^{(s)} =  {\mathcal C}_{s,s+1}   f_N^{(s+1)}  \quad   \mbox{in} \quad \R_+ \times {\mathcal D}_s \, ,
\end{equation}
with the boundary conditions~$f_N^{(s)}(t,Z_s^{in}) = f_N^{(s)}(t,Z_s^{out}) $.

In the integrand of the collision operators ${\mathcal C}_{s,s+1}$ defined in \eqref{def:collisionop0}, we now distinguish between pre- and post-collisional configurations, as we decompose
$$
  {\mathcal C}_{s,s+1} =   {\mathcal C}^{+}_{s,s+1} -   {\mathcal C}^-_{s,s+1}
$$
where 
 \begin{equation} \label{css+1+HS}
 {\mathcal C}_{s,s+1}^\pm   f^{(s+1)} = \sum_{ i = 1}^{s}   {\mathcal C}^{\pm,i}_{s,s+1}   f^{(s+1)}
 \end{equation} 
 the index $i$ referring to the index of the interaction particle among the $s$ ``fixed" particles, with the notation
 $$\begin{aligned}
 \big({\mathcal C}^{\pm,i}_{s,s+1}   f^{(s+1)} \big)(Z_s) :=   (N-s ) \eps^{d-1}\int_{{\bf S}_1^{d-1}\times\R^d}
(\omega  \cdot ( v_{s+1}  - v_i) )_\pm   f^{(s+1)}  (Z_s,x_i+\eps \omega  , v_{s+1})   \, d \omega dv_{s+1} \, ,
\end{aligned}$$
the index $+$ corresponding to post-collisional configurations and the index $-$ to pre-collisional configurations.

Denote by~${\bf \Psi}_s(t)$\label{HSflow} the~$s$-particle flow associated with the hard-spheres system, and by~$ {\bf T}_s$ the associated solution operator:
\begin{equation}\label{solutionoperatorHS}
  {\bf T}_s(t): \qquad f \in C^0({\mathcal D}_s;\R) \mapsto f({\bf \Psi}_s(-t,\cdot)) \in C^0({\mathcal D}_s;\R) \, .
 \end{equation}
 The time-integrated form of equation~(\ref{BBGKYhierarchyspheres}) is
\begin{equation}
\label{BBGKY-mildHS}
  f_N^{(s)} (t,Z_s)=   {\bf T}_s(t)  f_N^{(s)}(0,Z_s) +  \int_0^t {\bf T}_s(t - \t)   {\mathcal C}_{s,s+1}   f_N^{(s+1)} (\t,Z_s)\, d\t \, .
\end{equation}
 The {\it total flow} and {\it total collision} operators~${\mathbf  T}  $ and~${\mathbf  C} _N$ are defined on  finite sequences~$G_N= (g_s)_{1\leq s\leq N}$ as follows:
\begin{equation}
\label{bH-defHS}
\left\{\begin{aligned} & \forall s  \leq N \,,\,\,  \left( {\mathbf T}  (t) G_N\right)_s :=  {\mathbf T}_{s}(t) g_{s}\, , \\ & \forall \, s \leq N-1 \,, \,\, \left( {\mathbf C} _N G_N\right)_s :=  {\mathcal C}_{s,s+1} g_{s+1} \,, \quad \big( {\bf C}_N G_N\big)_N := 0\, .\end{aligned}\right.
\end{equation}
We finally define {\it mild solutions} to the BBGKY hierarchy~(\ref{BBGKY-mildHS}) to be solutions of  
\begin{equation} \label{BBGKY-mild-totaHS}   F_N (t) = {\mathbf T}  (t)    F_N (0) +  \int_0^t {\bf T}(t - \t)   {\bf C}_N    F_N (\t) \, d\t\,, \qquad   F_N = (  f^{(s)}_N)_{1 \leq s \leq N}\,.
\end{equation}

%
\section{The Boltzmann hierarchy and the Boltzmann equation}\label{spheresboltz}
\setcounter{equation}{0}
Starting from~(\ref{BBGKY-mild-totaHS}) we   now consider the limit~$N \to \infty$ under the Boltzmann-Grad scaling~$N\eps^{d-1} \equiv 1$, in order to derive formally the expected form of the  Boltzmann hierarchy.

Because of the scaling assumption~$N \eps^{d-1} \equiv 1$, the collision term~$   {\mathcal C}_{s,s+1}   f^{(s+1)} (Z_s) $ is   approximately equal to
$$
 \sum_{ i = 1}^s 
   \int_{  {\bf S}_1^{d-1} \times \R^d} 
\omega  \cdot (v_{s+1}-v_i)   f_N^{(s+1)}  (Z_s, x_i+\eps \omega  , v_{s+1})  \,  d \omega  dv_{s+1} 
 $$
 which we may split into two terms, depending on the sign of~$\omega   \cdot (v_{s+1}-v_i)$, as in~(\ref{css+1+HS}):
 $$
 \begin{aligned}
  &   \sum_{ i = 1}^s  \int_{  {\bf S}_1^{d-1} \times \R^d} 
\Bigl(\omega   \cdot (v_{s+1}-v_i)\Bigr)_+ \,   f_N^{(s+1)}  (Z_s, x_i+\eps \omega ,v_{s+1})  \,  d \omega    dv_{s+1}  \\
 & \quad - \sum_{ i = 1}^s 
   \int_{   {\bf S}_1^{d-1} \times \R^d}  \Bigl(\omega   \cdot (v_{s+1}-v_i)\Bigr)_- \,  f_N^{(s+1)}  (Z_s, x_i+\eps \omega , v_{s+1})  \,  d \omega    dv_{s+1} \, .
\end{aligned}
$$
Changing $\omega$ in $-\omega$ in the second term, we get 
$$
 \begin{aligned}
  &   \sum_{ i = 1}^s  \int_{  {\bf S}_1^{d-1} \times \R^d} 
\Bigl(\omega   \cdot (v_{s+1}-v_i)\Bigr)_+ \,   f_N^{(s+1)}  (Z_s, x_i+\eps \omega ,v_{s+1})  \,  d \omega    dv_{s+1}  \\
 & \quad - \sum_{ i = 1}^s 
   \int_{   {\bf S}_1^{d-1} \times \R^d}  \Bigl(\omega   \cdot (v_{s+1}-v_i)\Bigr)_+ \,  f_N^{(s+1)}  (Z_s, x_i-\eps \omega , v_{s+1})  \,  d \omega    dv_{s+1} \, .
\end{aligned}
$$
Recall  that pre-collisional particles are  particles~$(x_i,v_i)$ and $(x_{s+1}, v_{s+1})$ whose distance is decreasing up to collision time, meaning that for which $$(x_{s+1}-x_i)  \cdot (v_{s+1}-v_i)<0 \, .$$
With the above notation this means that
$$
\omega \cdot (v_{s+1}-v_i)<0 \, .
$$
 On the contrary the case when~$\omega \cdot (v_{s+1}-v_i)>0$  is called the post-collisional case; we recall that grazing collisions, satisyfing~$\omega \cdot (v_{s+1}-v_i)=0$ can be neglected (see Paragraph~\ref{Nparticle} above).
 
  Consider a set of particles~$Z_{s+1} = (Z_s, x_i + \eps \omega, v_{s+1})$ such that~$(x_i,v_i)$ and $(x_i + \eps \omega, v_{s+1})$ are post-collisional. We recall the boundary condition 
    $$
  f_N^{(s+1)}  (t,Z_s, x_i + \eps \omega, v_{s+1})  =   f_N^{(s+1)}  (t ,Z^*_s, x_i + \eps \omega, v^*_{s+1})  $$
where~$Z^*_s = (z_1,\dots,z_i^*, \dots z_s) $ and~$(v^*_i,v^*_{s+1})$ is the pre-image of~$(v_i,v_{s+1})$ by~(\ref{hard-spheres}):
\begin{equation}\label{defpreimagehardspheres}
\begin{aligned}
v_i^{*} := v_i - \omega \cdot ( v_i  -   v_{s+1} ) \, \omega&  \\
v_{s+1}^{*} := v_{s+1}+ \ \omega \cdot ( v_i  -   v_{s+1})  \, \omega \, ,
\end{aligned}
\end{equation}
while~$x_i^{*} := x_i$. In the following writing also~$x_{s+1}^{*} := x_{s+1}$ we shall use the notation
\begin{equation}\label{defscatteringhardspheres}
\begin{aligned}
 \sigma(z_i^* ,z_{s+1}^* )  := (z_i ,z_{s+1}) \, .
 \end{aligned}
\end{equation}
Then neglecting the small spatial translations in the arguments of~$  f_N^{(s+1)}$ and using the fact that~$  f_N^{(s+1)}$ is left-continuous in time  for all~$s$   we obtain the following asymptotic expression for the collision operator at the limit:
\begin{equation}
\label{Boltzmann-ophardspheres}
\begin{aligned}   {\mathcal C}^0_{s,s+1} &   f^{(s+1)} (t,Z_s) :=       \sum_{ i = 1}^s  \int \big( \omega   \cdot (v_{s+1}-v_i)  \big)_+\\
&\quad \times{} \Big(  f^{(s+1)}  (t,x_1,v_1,\dots ,x_i, v^*_i,\dots ,x_s,v_s, x_i, v^*_{s+1})    -     f^{(s+1)}  (t,Z_s, x_i, v_{s+1}) \Big) d \omega dv_{s+1} \, .
\end{aligned}
\end{equation}
The asymptotic dynamics are therefore governed by the following integral form of the Boltzmann hierarchy:
\begin{equation}
\label{mild-Boltzmann}
   f^{(s)} (t) = {\bf S}_s(t)  f^{(s)}_{0} + \int_0^t {\bf S}_{s}(t - \t)     {\mathcal C}^0_{s,s+1}    f^{(s+1)}  (\t) \, d\t\,,
 \end{equation}
 where~$ {\bf S}_s(t)$ denotes the~$s$-particle free-flow.

Similarly to~(\ref{bH-defHS}), we can define the total Boltzmann flow and collision operators~${\mathbf  S}  $ and~${\mathbf  C}  $  as follows:
\begin{equation}
\label{bH-defboltzmann}
\left\{\begin{aligned} & \forall s \geq 1\,,\,\,  \left( {\mathbf S}  (t) G \right)_s :=  {\mathbf S}_{s}(t) g_{s}\, , \\ & 
\forall \, s \geq 1 \,, \,\, \left(  {\bf C^0}    G \right)_s := {\mathcal C}^0_{s,s+1} g_{s+1} \,,  \end{aligned}\right.
\end{equation}
so that {\it mild solutions} to the Boltzmann hierarchy~(\ref{mild-Boltzmann}) are solutions of  
\begin{equation} \label{Boltzmann-mild-totaHS}   F (t) = {\mathbf S}  (t)    F  (0) +  \int_0^t {\bf S}(t - \t)   {\bf C^0}    F  (\t) \, d\t\,, \qquad   F  = (  f^{(s)})_{s \geq 1 }\,.
\end{equation}

 Note that if~$   f^{(s)} (t,Z_s) = \displaystyle \prod_{i = 1}^s f(t,z_i) $ (meaning~$f^{(s)}(t)$ is {\it tensorized}) then~$f$ satisfies the Boltzmann equation~(\ref{boltz-eq})-(\ref{boltz-operator}), where the cross-section is~$b(w,\omega):= \big( \omega   \cdot w \big)_+$.

%% file: Cauchy-Kowalewski.tex
\chapter{Uniform a priori estimates for the BBGKY and Boltzmann hierarchies}
\label{existence} 
\setcounter{equation}{0}

This chapter is devoted to the statement and proof of uniform a priori estimates for mild solutions to the BBGKY hierarchy, defined in~\eqref{BBGKY-mild-totaHS}, which we reproduce here:
 \begin{equation} \label{BBGKY-mild-total-HSbis}   F_N (t) = {\mathbf  T}  (t)    F_N (0) +  \int_0^t {\mathbf  T}(t - \t)   {\bf C}_N    F_N (\t) \, d\t\,, \qquad   F_N = (  f^{(s)}_N)_{1 \leq s \leq N} \, ,
\end{equation}
as well as     for the limit Boltzmann hierarchy defined in~(\ref{Boltzmann-mild-totaHS})
 \begin{equation}
\label{mild-Boltzmannbis}
  F  (t) = {\bf S} (t)  F (0) + \int_0^t {\bf S} (t - \t)    {\bf C^0}  F (\t) \, d\t \, , \qquad   F = (  f^{(s)} )_{s \geq 1  } \, . \end{equation}
Those results are obtained in Paragraphs~\ref{mainsteps} and~\ref{sec:cont-estHS} by use of a Cauchy-Kowalevskaya type argument in some adequate function spaces defined in Paragraph~\ref{sec:ex-BBGKYHS}.


\section{Functional spaces and statement of the results}\label{sec:ex-BBGKYHS}

\setcounter{equation}{0}

In order to obtain uniform a priori bounds  for mild solutions to the BBGKY hierarchy, we need to introduce some norms on the space of sequences $(   g^{(s)})_{s \geq 1}$.
Given $\e > 0,$ $\b > 0,$ an integer~$s \geq 1,$ and a continuous function $g_s:{\mathcal D}_s\to \R,$ we let
\begin{equation} \label{norm:e-b}
| g_s|_{\eps,s,\beta} :=\sup _{Z_s \in{\mathcal D}_s} \left( |g_s(Z_s)| \exp \big(\beta E_0(Z_s)\big)     \right) 
\end{equation}
where  $E_0$ is the free Hamiltonian: 
\begin{equation} \label{def:freehamiltonian} E_0(Z_s) :=\sum_{1 \leq i \leq s} {|v_i|^2\over 2} \, \cdotp
\end{equation}
Note that the dependence on~$\eps$ of the norm is through the constraint~$Z_s \in {\mathcal D}_s
$.
We also define,  for a continuous function $g_s:\R^{2ds}\to \R,$ 
\begin{equation} \label{norm:0-b}
| g_s|_{0,s,\beta} :=\sup _{Z_s \in \R^{2ds}} \left( |g_s(Z_s)| \exp \big(\beta E_0(Z_s)\big)    \right)  \, .
\end{equation}

\begin{Def} \label{def:functional-spacesHS} For $\e > 0$ and $\b > 0,$ we denote $X_{\e,s,\beta}$ the Banach space of continuous functions~$ {\mathcal D}_s\to \R$ with finite $|\cdot|_{\e,s,\b}$ norm, and similarly~$X_{0,s,\beta}$ is the Banach space of continuous functions~$\R^{2ds}\to \R$ with finite $|\cdot|_{0,s,\b}$ norm.
\end{Def}
 
For sequences of continuous functions $G = (g_s)_{s\geq1},$ with $g_s:  {\mathcal D}_s  \to \R,$ we let for $\e > 0,$ $\b > 0,$ and~$\mu \in \R,$ 
$$
\| G \|_{\eps, \beta,\mu}  :=\sup_{s\geq 1}\Big( | g_s|_{\eps,s,\beta} \exp( \mu s)  \Big) \, .
$$
We define similarly for~$G = (g_s)_{s\geq1},$ with $g_s: \R^{2ds}  \to \R,$
$$
\| G \|_{0, \beta,\mu}  :=\sup_{s\geq 1}\Big( | g_s|_{0,s,\beta} \exp( \mu s)  \Big) \, .
$$
\begin{Def} \label{def:functional-spaces2HS} For $\e \geq 0,$ $\b > 0,$ and $\mu \in \R,$ we denote ${\bf X}_{\e,\b,\mu}$ the Banach space of sequences of   functions~$G = (g_s)_{s \geq 1},$ with $g_s \in X_{\e,s,\b}$ and $\| G \|_{\e,\b,\mu} < \infty.$  
\end{Def}
The following inclusions hold:
\begin{equation} \label{inclusionsHS} \mbox{if $\b' \leq \b$ and $\mu' \leq \mu,$ then} \quad X_{\e,s,\b'} \subset X_{\e,s,\b} \, , \quad {\bf X}_{\e,\b',\mu'} \subset {\bf X}_{\e,\b,\mu} \, .
\end{equation}
\begin{Rmk}\label{physinterpretation}
These norms are rather classical in statistical physics (up to replacing the~$L^\infty$ norm by an~$L^1$ norm) , where probability measures are called ``ensembles".

 At the canonical level, the ensemble~$\indc_{Z_s \in {\mathcal D}_s}e^{-\beta E_0(Z_s)} dZ_s$ is a normalization of the Lebesgue measure, where~$\beta \sim \theta^{-1}$ (and~$\theta $ is the absolute temperature) specifies fluctuations of energy. The Boltzmann-Gibbs principle states that
the average value of any quantity in the canonical ensemble is its equilibrium value at temperature~$\theta $.

 The micro-canonical level consists in restrictions of the ensemble to energy surfaces. 
 
 At the grand-canonical level  the number of particles may vary, with variations indexed by chemical potential~$\mu \in \R$.
\end{Rmk}

Existence and uniqueness for \eqref{BBGKY-mild-total-HSbis} comes from the theory of linear transport equations which provides a unique, global solution to the Liouville equation~(\ref{liouvillespheres}). Nevertheless, in order to obtain a similar result for the limiting hierarchy~(\ref{mild-Boltzmannbis}), we need to obtain uniform a priori estimates  with respect to~$N$, on the marginals~$f_N^{(s)}$ for any fixed~$s$. We shall thus deal with both systems~\eqref{BBGKY-mild-total-HSbis} and~(\ref{mild-Boltzmannbis}) simultaneously, using analytical-type techniques which will provide short-time existence in the spaces of ${\bf X}_{\e,\b,\mu}$-valued functions of time (resp.~${\bf X}_{0,\b,\mu}$). Actually the parameters $\b$ and $\mu$ will themselves depend on time: in the sequel we choose for simplicity a linear dependence in time, though other, decreasing functions of time could be chosen just as well.  Such a time dependence on the parameters of the function spaces is a situation which occurs whenever continuity estimates involve a loss, which is the case here since the  continuity estimates on the collision operators lead to a deterioration in the parameters~$\beta$ and~$\mu$.
We refer to Section~\ref{commentscauchy} for some comments.

\begin{Def} \label{deffunctionspacesexistenceHS}
Given $T > 0$, a positive function ${\boldsymbol\beta}$  and a real valued function~${\boldsymbol \mu} $ defined on~$[0,T]$ we denote ${\bf X}_{\e,{\boldsymbol\beta},{\boldsymbol\mu}}$ the space of functions $G: t \in [0,T] \mapsto G(t) = (g_s(t))_{1 \leq s} \in {\bf X}_{\e,\b(t),\mu(t)},$ such that for all $Z_s \in \R^{2ds},$ the map $t \in [0,T] \mapsto g_s(t,Z_s)$ is measurable, and
 \begin{equation} \label{def:normzHS}   | \! \| G    | \! \| _{ \e,{\boldsymbol\beta},{\boldsymbol\mu}}
 := \sup_{0 \leq t \leq T} \| G(t) \|_{\e,\b(t),\mu(t)} < \infty \, .
 \end{equation}
 We define similarly
 $$
   | \! \| G    | \! \| _{ 0,{\boldsymbol\beta},{\boldsymbol\mu}}
 := \sup_{0 \leq t \leq T} \| G(t) \|_{0,\b(t),\mu(t)}  \, .
 $$
\end{Def}
We shall prove the following uniform bounds for the BBGKY hierarchy.
 \begin{Thm}[Uniform estimates for the BBGKY hierarchy]\label{existence-thmHS}
 Let~$\beta_0 >0$ and~$\mu_0 \in \R $ be given. There is a time~$T>0$  as well as
 two nonincreasing functions~${\boldsymbol\beta }>0$ and~${\boldsymbol\mu }$ defined on~$[0,T]$, satisfying~${\boldsymbol\beta }(0) = \beta_0 $ and~${\boldsymbol\mu}(0) = \mu_0 $, 
such that {in the Boltzmann-Grad scaling~$N \eps^{d-1} \equiv 1$}, any family of initial marginals $   F_N(0) =\big(   f_N^{(s)} (0) \big)_{1 \leq s \leq N}$  in~${\bf X}_{\eps,\beta_0,\mu_0}$ gives rise to a unique solution~$   F_N (t)= (   f_N^{(s)} (t))_{1 \leq s \leq N}$  in~${\bf X}_{\e,{\boldsymbol\beta},{\boldsymbol\mu}}$ to    the BBGKY hierarchy~{\rm(\ref{BBGKY-mild-total-HSbis})}    satisfying  the following bound:
         $$
  | \! \|   F_N   | \! \| _{ \e,{\boldsymbol\beta},{\boldsymbol\mu}}
\leq  2  \|{   F_N(0) }\|_{\eps,
      \b_0,\mu_0} \, .$$
    \end{Thm}

\begin{Rem}\label{existencetime}
The proof of Theorem~{\rm\ref{existence-thmHS}} provides a lower bound of the time~$T$ on which one has a uniform bound, in terms of the initial parameters~$\beta_0$, $\mu_0$ and the dimension~$d$: one finds
\begin{equation} \label{lower-bd:T*}
  T \geq C_d e^{\mu_0} (1 + \b_0 ^{\frac12})^{-1} \max_{\beta \in [0, \b_0]} \beta e^{-\beta} (\b_0 - \beta)^\frac{d+1}2 \, , 
 \end{equation}
 where~$C_d$ is a constant depending only on~$d$.
 
  In particular if $d \ll \b_0,$ there holds $\displaystyle \max_{\beta \in [0, \b_0]} \beta e^{-\beta} (\b_0- \beta)^\frac{d+1}2= \b_0^ \frac{d+1}2 \big(1 + o(1)\big),$ hence an existence time of the order of $e^{\mu_0} \b_0 ^{d/2}.$ 
\end{Rem}
The proof  of Theorem~{\rm\ref{existence-thmHS}}  uses neither the  fact that the BBGKY hierarchy is closed by the transport equation satisfied by~$f_N$, nor possible cancellations of the collision operators. It only relies on crude estimates and in particular the limiting hierarchy satisfies the same result, proved similarly.
\begin{Thm}[Existence for the Boltzmann hierarchy]\label{existence-thmboltzmannHS}
 Let~$\beta_0 >0$ and~$\mu_0 \in \R $ be given. There is a time~$T>0$  as well as
 two nonincreasing functions~${\boldsymbol\beta }>0$ and~${\boldsymbol\mu }$ defined on~$[0,T]$, satisfying~${\boldsymbol\beta }(0) = \beta_0 $ and~${\boldsymbol\mu}(0) = \mu_0 $,
such that  any family of initial marginals $F(0) =\big(f^{(s)} (0) \big)_{ s \geq 1} $ in~$ {\bf X}_{0,\beta_0,\mu_0}$  gives rise to a unique solution~$   F  (t)= (   f^{(s)} (t))_{s \geq 1}$  in~${\bf X}_{0,{\boldsymbol\beta},{\boldsymbol\mu}}$ to      the Boltzmann hierarchy~{\rm (\ref{mild-Boltzmannbis})}, satisfying  the following bound:
         $$
   | \! \|   F    | \! \| _{ 0,{\boldsymbol\beta},{\boldsymbol\mu}}\leq  2  \|{   F(0) }\|_{0,
      \b_0,\mu_0} \, .$$
    \end{Thm}

\section{Main steps of the proofs}\label{mainsteps}

\setcounter{equation}{0}

The proofs of Theorems \ref{existence-thmHS} and~\ref{existence-thmboltzmannHS} are typical of analytical-type results, such as the classical Cauchy-Kowalevskaya theorem. We follow here Ukai's approach \cite{ukai}, which turns out to be remarkably short and self-contained.

Let us give the main steps of the proof: we start by noting that the conservation of energy for the~$s$-particle flow is reflected in identities
\begin{equation}\label{continuitytransportHS}
\begin{aligned}|{\mathbf  T}_s (t) g_s |_{\eps,s,\beta} = | g_s |_{\eps,s,\beta}\quad &\mbox{and} \quad 
\| {\mathbf  T} (t) G_N \|_{ \eps,\beta, \mu } = \| G_N \|_{ \eps,\beta, \mu} \,, \\
\end{aligned}
\end{equation}
for all parameters~$\beta >0,$ $\mu \in \R,$ and for all~$g_s \in X_{\eps,s,\beta},$    $G_N = (g_s)_{1 \leq s \leq N} \in {\bf X}_{\eps,\beta,\mu}$, and all $t \geq 0.$
Similarly,
\begin{equation}\label{continuitytransportHS0}
\begin{aligned}
|{\bf S}_s (t) g_{s} |_{0,s,\beta} = | g_s |_{s,\beta}\quad &\mbox{and} \quad 
\| {\bf S} (t) G \|_{ 0,\beta, \mu } = \| G \|_{ 0,\beta, \mu} \,,
\end{aligned}
\end{equation}
for all parameters~$\beta >0,$ $\mu \in \R,$ and for all~$g_s \in X_{0,s,\beta},$ $G =  (g_{s})_{s \geq 1  } \in {\bf X}_{0,\beta,\mu}$, and all $t \geq 0.$

Next assume that    in the Boltzmann-Grad scaling~$N \eps^{d-1} \equiv 1$, there holds the bound
\begin{equation} \label{est:duh-HCHS}
\forall \, 0< \eps \leq \eps_0 \, , \quad \Big   | \! \Big \| \int_0^t {\mathbf  T}(t-\tau ) {\bf C}_N G_N(\tau) \, d \tau\Big   | \! \Big \|_{ \eps,{\boldsymbol\beta},{\boldsymbol\mu}}\leq \frac12\, |\! \| G_N| \! \| _{ \eps,{\boldsymbol\beta},{\boldsymbol\mu}}\, ,
 \end{equation} 
 for some functions~${\boldsymbol\beta}$ and~${\boldsymbol\mu} $ as in the statement of Theorem~\ref{existence-thmHS}.
 Under~(\ref{est:duh-HCHS}),   the linear operator
 $$ {\mathfrak L}: \quad G_N \in {\bf X}_{\e,{\boldsymbol\beta},{\boldsymbol\mu}}\mapsto \left(t \mapsto  \int_0^t  {\mathbf  T}(t-\tau ) {\bf C}_N G_N(\tau) \, d \tau\right) \in {\bf X}_{\e,{\boldsymbol\beta},{\boldsymbol\mu}}$$
 is   linear continuous from ${\bf X}_{\e,{\boldsymbol\beta},{\boldsymbol\mu}}$ to itself with norm strictly smaller than one. In particular, the operator~${\rm Id} -  {\mathfrak L}$ is invertible in the Banach algebra~${\mathcal L}({\bf X}_{\e,{\boldsymbol\beta},{\boldsymbol\mu}}).$  Next given~$  F_N(0) \in {\bf X}_{\e,\b_0,\mu_0},$ by conservation of energy~\eqref{continuitytransportHS}, inclusions~\eqref{inclusionsHS} and decay of ${\boldsymbol\beta}$ and ${\boldsymbol\mu}$, there holds 
 $$
 \big( t \mapsto {\mathbf  T}(t)   F_N(0)\big) \in {\bf X}_{\e,{\boldsymbol\beta},{\boldsymbol\mu}}\, .
 $$  
Hence, there exists a unique solution $  F_N \in {\bf X}_{\e,{\boldsymbol\beta},{\boldsymbol\mu}}$ to $({\rm Id} -  {\mathfrak L})  F_N = {\mathbf  T}(\cdot)   F_N(0),$ an equation which is equivalent to \eqref{BBGKY-mild-total-HSbis}. 

The reasoning is identical for Theorem~\ref{existence-thmboltzmannHS}, replacing~(\ref{est:duh-HCHS})
by
\begin{equation}\label{est:duh-HCboltzmannHS}
\Big   | \! \Big \| \int_0^t {\bf S}(t-\tau) {\bf C}^0  G(\tau) \, d \tau\Big   | \! \Big \|_{ 0,{\boldsymbol\beta},{\boldsymbol\mu}}\leq \frac12\,      | \!   \| G  | \!   \|_{ 0,{\boldsymbol\beta},{\boldsymbol\mu}} \, .
 \end{equation}
The next section is devoted to the proofs of~(\ref{est:duh-HCHS}) and~(\ref{est:duh-HCboltzmannHS}).

\section{Continuity estimates} \label{sec:cont-estHS}
\setcounter{equation}{0}
In order to prove~(\ref{est:duh-HCHS}) and~(\ref{est:duh-HCboltzmannHS}), we first establish bounds, in the above defined functional spaces, for the collision operators defined in~\eqref{def:collisionop0} and~\eqref{Boltzmann-ophardspheres}, and for the total collision operators.
  In $\C_{s,s+1},$ the sum in $i$ over $[1,s]$ will imply a loss in $\mu,$ while the linear velocity factor will imply a loss in $\b.$ The losses are materialized in~\eqref{continuityHS}  below by inequalities $\b' < \b,$ $\mu' < \mu.$  
  
  The next statement concerns the BBGKY collision operator.
\begin{prop}\label{propcontinuityclustersHS}
Given $\b > 0$ and $\mu \in \R,$ for $1 \leq s \leq N-1,$ the collision operator~$\C_{s,s+1}$ satisfies the bound, for all $G_N = (g_s)_{1 \leq s \leq N} \in {\bf X}_{\e,\b,\mu}$ in the Boltzmann-Grad scaling~$N \e^{d-1} \equiv 1$,
\begin{equation} \label{est-col:1HS}
 \big| {\mathcal C}_{s,s+1} g_{s+1} (Z_s)\big| \leq     {C_{d}} \, {\b^{-\frac d2} } \Big( s\b^{-\frac12} + \sum_{1 \leq i \leq s} |v_i|\Big)  e^{- \b E_0(Z_s)}| g_{s+1}|_{\e,s+1,\b} \, ,
 \end{equation}
for some $C_{d} > 0$ depending only on $d$.

In particular, for all~$0< \beta' < \beta$ and $\mu' < \mu,$ the total collision operator ${\bf C}_N$ satisfies the bound 
\begin{equation}
\label{continuityHS}
\| {\mathbf C} _N G_N \|_{\eps,\beta', \mu'} \leq  C_d (1 + \b^{-\frac12}) \Big( \frac1{\b - \b'} + \frac1{\mu - \mu'} \Big)
  \| G_N \|_{\eps,\beta,\mu} 
\end{equation}
 in the Boltzmann-Grad scaling~$N \e^{d-1} \equiv 1$. 
\end{prop}
Estimate \eqref{continuityHS}, a continuity estimate with loss for the total collision operator ${\bf C}_N,$ is not directly used in the following. In the existence proof, we use instead the pointwise bound \eqref{est-col:1HS}.  Note that the more abstract (and therefore more complicated in our particular setting) approach of L. Nirenberg \cite{nirenberg} and T. Nishida \cite{nishida} would require   the loss estimate \eqref{continuityHS}. 

\begin{proof} 
Recall that as in~\eqref{def:collisionop0},
$$
\big( {\mathcal C}_{s,s+1}  g^{(s+1)}\big)  (t,Z_{s}):=\!    (N-s) \eps^{d-1}\sum_{i=1}^s    \int_{ {\mathbf S}_1^{d-1}\times \R^d}  \! \! \! 
  \omega \cdot (v_{s+1}-v_i)  \:  g^{(s+1)} (t,Z_s,x_{i}+ \eps \omega,v_{s+1})  
     d \omega d v_{s+1} \, .
$$
Estimating each term in the sum separately, regardless of 
possible cancellations between ``gain" and ``loss" terms,
it is obvious that
 $$
  | {\mathcal C}_{s,s+1} g_{s+1}(Z_s) |\leq  \k_d \e^{d-1}(N-s) | g_{s+1} |_{\eps,s+1,\beta} \sum_{1 \leq i \leq s} I_{i}(V_s)  \,,
  $$
  where~$\kappa_d$ is the volume\label{defkappad} of the unit ball of~$\R^d$, and where
  $$
   I_{i}(V_s):= \int_{\R^{d}} \big( |v_{s+1}| + |v_i|\big) \exp\Big(-\frac\beta2 \sum_{j=1}^{s+1} |v_j|^2 \Big) dv_{ s+1 } \, .
  $$
 Since a direct calculation gives
  $$
 I_{i}(V_s) \leq  C_d \, \b^{-\frac d2} \big(   \b^{-\frac12}+ |v_i| \big) \exp\Big(-\frac\beta2\sum_{1 \leq j \leq s} |v_j|^2 \Big) \, ,
$$
the result \eqref{est-col:1HS} is deduced directly in the Boltzmann-Grad scaling~$N \e^{d-1} \equiv 1$.

We turn to the proof of \eqref{continuityHS}. From the pointwise inequality (due to Cauchy-Schwarz)
\begin{equation}\label{pointwiseHS}
\sum_{1 \leq i \leq s} |v_i| \exp\Big(- (\g/2) \sum_{1 \leq j \leq s} |v_j|^2\Big) \leq s^{\frac12} (e \g)^{-1/2}\,, \qquad \g > 0\,,
\end{equation}
we deduce for the above velocity integral $I_{i}(V_s)$ the bound, for $0 < \b' < \b,$ 
$$ 
\sum_{1 \leq i \leq s} \exp\Big( (\b'/2) \sum_{1 \leq j \leq s} |v_j|^2\Big) I_{i}(V_s) \leq  C_d \, \b^{-\frac d2} \big( s \b^{-\frac12} + s^{\frac12}(\b - \b')^{-\frac12}\big) \, .
$$
  Since
  $$ \sup_{1 \leq s \leq N} \Big( \big( s\b^{-\frac12} + s^{1/2} (\b- \b')^{-\frac12}\big) e^{-(\mu - \mu') s} \Big) \leq e^{-1} (\mu - \mu')^{-1} \big(\beta^{-\frac12}+ 1 \big) + e^{-1} (\b- \b')^{-1}\, ,$$
we find \eqref{continuityHS}. 
Proposition \ref{propcontinuityclustersHS} is proved.
\end{proof}
 
A similar result holds for the limiting collision operator.

\begin{prop} \label{propcontinuityBHS}
Given $\b > 0,$ $\mu \in \R,$ the collision operator $\C^0_{s,s+1}$ satisfies the following bound, for all $g_{s+1} \in X_{0,s+1,\b}:$
\begin{equation} \label{est-col:BHS}
 \big| {\mathcal C}^0_{s,s+1} g_{s+1} (Z_s)\big| \leq C_{d} \b^{-\frac d2} \Big( s\b^{-\frac12} + \sum_{1 \leq i \leq s} |v_i|\Big)  e^{- \b E_0(Z_s)}| g_{s+1}|_{0,s+1,\b} \, ,
 \end{equation}
for some $C_{d} > 0$ depending only on $d.$ 
\end{prop}

\begin{proof}
 There holds
$$  \big| {\mathcal C}^0_{s,s+1} g_{s+1} (Z_s)\big| \leq \sum_{1 \leq i \leq s} \int_{{\bf S}^{d-1} \times \R^{d}} \big( |v_{s+1}| + |v_i|\big) \big(|g_{s+1}(v_i^*, v^*_{s+1})| +  |g_{s+1}(v_i,v_{s+1})| \big) d\omega dv_{s+1},$$
omitting most of the arguments of $g_{s+1}$ in the integrand. By definition of $|\cdot|_{0,s,\b}$ norms and conservation of energy \eqref{continuitytransportHS}, there holds
 $$ \begin{aligned} |g_{s+1}(v_i^*, v^*_{s+1})| +  |g_{s+1}(v_i,v_{s+1})| & \leq \big( e^{- \b E_0(Z_s^*)} + e^{-\b E_0(Z_s)} \big) |g_{s+1}|_{0,\b} \\ & = 2 e^{-\b E_0(Z_s)} |g_{s+1}|_{0,s+1,\b},
 \end{aligned}$$
 where $Z_s^*$ is identical to $Z_s$ except for $v_i$ and $v_{s+1}$ changed to $v_i^*$ and $v_{s+1}^*.$ This gives
 $$  \big| {\mathcal C}^0_{s,s+1} g_{s+1} (Z_s)\big| \leq C_d |g_{s+1}|_{0,s+1,\b} e^{-\b E_0(Z_s)} \sum_{1 \leq i \leq s} I_{i}(V_s) \, ,$$
 borrowing notation from the proof of Proposition \ref{propcontinuityclustersHS}, and we conclude as above. 
\end{proof}

Propositions~\ref{propcontinuityclustersHS} and~\ref{propcontinuityBHS} are the key to the proof of~(\ref{est:duh-HCHS})  and~(\ref{est:duh-HCboltzmannHS}). Let us first prove a continuity estimate based on Proposition \ref{propcontinuityclustersHS}, which implies directly~(\ref{est:duh-HCHS}).
\begin{lem} \label{lemukai}  
Let~$\beta_0>0$ and~$\mu_0 \in \R$ be given. For all~$\lambda>0$ and~$t >0$ such that~$ \lambda t < \beta_0 $,  there holds the bound
\begin{equation} \label{est:C1}
 e^{s (\mu_0-\lambda t)} \Big| \int_0^t {\mathbf  T}_s(t-\tau) \C_{s,s+1} g_{s+1}(\tau) \, d \tau\Big|_{\e,s,\b_0-\lambda t} \leq \bar c(\b_0,\mu_0, \lambda ,T)  |\! \| G_N| \! \| _{ \eps,{\boldsymbol\beta},{\boldsymbol\mu}}\, ,
 \end{equation}
 for all $G_N = (g_{s+1})_{1 \leq s \leq N} \in {\bf X}_{\e,{\boldsymbol\beta},{\boldsymbol\mu}},$ with   $ \bar c(\b_0,\mu_0, \lambda ,T)$ computed explicitly in~{\rm(\ref{defbarcHS})} below. In particular there is~$T>0$ depending only on~$\beta_0$ and~$\mu_0$ such that for an appropriate choice of~$\lambda $ in~$ (0, \beta_0 /T )$, there holds for all~$t \in [0,T]$
\begin{equation} \label{est:C1bis}
  e^{s (\mu_0-\lambda t)} \Big| \int_0^t {\mathbf  T}_s(t-\tau) \C_{s,s+1} g_{s+1}(\tau) \, d \tau\Big|_{\e,s,\b_0-\lambda t} \leq \frac12 \,  |\! \| G_N| \! \| _{ \eps,{\boldsymbol\beta},{\boldsymbol\mu}}\, .
  \end{equation}
 \end{lem}

\begin{proof}
Let us define, for all~$\lambda>0$ and~$t >0$ such that~$ \lambda t < \beta_0 $,  the functions
\begin{equation}\label{notationweights}
\beta_0^\lambda(t):=\b_0-\lambda t \quad \mbox{and} \quad \mu_0^\lambda(t):=\mu_0-\lambda t \, .
\end{equation}
By conservation of energy \eqref{continuitytransportHS}, there holds the bound 
$$  
\begin{aligned}
 \Big| \int_0^t {\mathbf  T}(t-\tau ) \C_{s,s+1} g_{s+1}(\tau) \, d \tau \Big|_{\e,s,\beta_0^\lambda(t)} \leq \sup_{Z_s \in \R^{2ds}} \int_0^t e^{\beta_0^\lambda(t)E_0(Z_s)} \big| {\mathcal C}_{s,s+1} g_{s+1} (\tau,Z_s)\big| \, d \tau \, .
 \end{aligned}
$$ 
 Estimate \eqref{est-col:1HS} from Proposition \ref{propcontinuityclustersHS} gives
 $$ 
 \begin{aligned} 
& e^{\beta_0^\lambda(t) E_0(Z_s)} 
  \big| {\mathcal C}_{s,s+1} g_{s+1} (\tau,Z_s)\big| \\ & \leq 
  C_d  \, \big(\beta_0^\lambda(\tau)\big) ^{-\frac d2} | g_{s+1}(\tau)|_{\e,s+1,\beta_0^\lambda(\tau)} \Big( s(\beta_0^\lambda(\tau))^{-\frac 12} + \sum_{1 \leq i \leq s} |v_i|\Big)  e^{\lambda(\tau-t) E_0(Z_s)} \, .
 \end{aligned}
 $$
 By definition of norms $\| \cdot \|_{\e,\b,\mu}$ and  $ |\! \| \cdot| \! \| _{ \eps,{\boldsymbol\beta},{\boldsymbol\mu}}$ we have  
 \begin{equation}\label{estimategs+1t'}
\begin{aligned} 
 | g_{s+1}(\tau)|_{\e,s+1,\beta_0^\lambda(\tau)} & \leq e^{-(s+1)\mu_0^\lambda(\tau)} \| G_N(\tau) \|_{\e,\beta_0^\lambda(\tau),\mu_0^\lambda(\tau)}\\
 & \leq e^{-(s+1) \mu_0^\lambda(\tau)}    |\! \| G_N| \! \| _{ \eps,{\boldsymbol\beta},{\boldsymbol\mu}}\, .
 \end{aligned}
 \end{equation}
The above bounds yield, since~$\beta_0^\lambda$ and~$\mu_0^\lambda$ are nonincreasing,
$$
 \begin{aligned} &e^{s \mu_0^\lambda(t)} \Big| \int_0^t {\mathbf  T}(t-\tau)  \C_{s,s+1} g_{s+1}(\tau) \, d \tau \Big|_{\e,s,\b_0^\lambda(t)} \\ 
 &\quad  \leq  C_d |\! \| G_N| \! \| _{ \eps,{\boldsymbol\beta},{\boldsymbol\mu}}  e^{ - \mu_0^\lambda(T)} \big( \beta_0^\lambda(T)\big)^{-\frac d2} \sup_{Z_s \in \R^{2ds}} \int_0^t  \overline C(\tau,t,Z_s) \, d \tau\,, \end{aligned}
$$
 where, for $\tau \leq t,$ 
\begin{equation} \label{def:C2}
  \overline C(\tau,t,Z_s) := \Big( s(\b_0^\lambda(\tau))^{-\frac 12} + \sum_{1 \leq i \leq s} |v_i|\Big)  e^{\lambda (\tau-t) (s+ E_0(Z_s))}\,.
  \end{equation}
Since
 \begin{equation} \label{bd:tildeC2}
 \sup_{Z_s \in \R^{2ds}} \int_0^t \overline C(\tau,t ,Z_s) \, d \tau \leq  \frac{C_d}\lambda \Big( 1 +  \big (\beta_0^\lambda(T)\big)^{-\frac 12} \Big) \,,
  \end{equation}
 there holds finally
 $$
  \begin{aligned}  e^{s \mu_0^\lambda(t)} \Big| \int_0^t {\mathbf  T}(t-\tau)  \C_{s,s+1} g_{s+1}(\tau) \, d \tau \Big|_{\e,s,\b_0^\lambda(t)} 
  \leq \bar c(\b_0,\mu_0,\lambda,T)   |\! \| G_N| \! \| _{ \eps,{\boldsymbol\beta},{\boldsymbol\mu}}\,, \end{aligned}
 $$
 where,  with a possible change of the constant~$C_d$,
 \begin{equation}\label{defbarcHS}
\bar c(\b_0,\mu_0,\lambda,T) :=C_d \,  e^{ - \mu_0^\lambda(T)} \lambda^{-1}  \big(\beta_0^\lambda(T)\big)^{-\frac d2}  \Big( 1 + \big(\beta_0^\lambda(T)\big)^{-\frac 12}  \Big) \, .
 \end{equation} The result~(\ref{est:C1}) follows.
 To deduce~(\ref{est:C1bis}) we  need to find~$T>0$ and $\l > 0$ such that~$\lambda T < \beta_0$ and
  \begin{equation} \label{lambdaHS}   C_d (1+ (\b_0 - \l T)^{-\frac12}\big) e^{-\mu_0 + \l T} (\b_0 - \l T)^{-\frac d2} = \frac{\l }{2} \,\cdotp
 \end{equation}
 With $\beta := \l T \in (0, \b_0),$ condition \eqref{lambdaHS} becomes
 $$ \begin{aligned} 
 T & = C_d e^{\mu_0} \beta e^{-\beta} \frac{(\b_0 - \beta)^{\frac{d+1}2}}{1 + (\b_0 - \beta)^{\frac12}} \\
 &\geq  C_d e^{\mu_0} (1 + \b_0^{\frac12})^{-1} \beta e^{-\beta} (\b_0- \beta)^{\frac{d+1}2} \,, \end{aligned}$$
 up to changing the constant~$C_d$ and~(\ref{est:C1bis}) 
 follows. Notice that 
  \eqref{lower-bd:T*} is a consequence of this computation. 
\end{proof}

The proof of the corresponding result~(\ref{est:duh-HCboltzmannHS}) for the Boltzmann hierarchy is identical, since the estimates for $\C^0_{s,s+1}$ and $\C_{s,s+1}$ are essentially identical (compare estimate \eqref{est-col:1HS} from Proposition \ref{propcontinuityclustersHS} with estimate  \eqref{est-col:BHS} from Proposition \ref{propcontinuityBHS}).
\section{Some remarks on the strategy of proof}\label{commentscauchy}

\setcounter{equation}{0}

 The key in the   proof of~(\ref{est:duh-HCHS}) is \emph{not} to apply Minkowski's integral inequality, which would indeed lead here to 
 $$ \Big\| \int_0^t {\mathbf  T}(t-\tau) {\bf C}_N G_N(\tau) \, d \tau\Big\|_{\e,\b_0^\lambda(t),\mu_0^\lambda(t)} \leq \int_0^t \Big\| {\bf C}_N G_N(\tau)\Big\|_{\e,\b_0^\lambda(t),\mu_0^\lambda(t)} \, d \tau\,,$$
 by \eqref{continuitytransportHS}, and then to a divergent integral of the type
 $$ \begin{aligned} \Big\| \int_0^t {\mathbf  T}(t-\tau) {\bf C}_N G_N(\tau) \, d \tau\Big\|_{\e,\b_0^\lambda(t),\mu_0^\lambda(t)}& \leq  C \big(\b_0^\lambda(T),\mu_0^\lambda(T) \big)  \int_0^t 
 \Big( \frac1{\b_0^\lambda(\tau) - \b_0^\lambda(t)} + \frac1{\mu_0^\lambda(\tau) - \mu_0^\lambda(t)}  \Big)\,d \tau\,.\end{aligned}$$
The difference is that by Minkowski the upper bound appears as the time integral of a supremum in~$s,$ while in the proof of~(\ref{est:duh-HCHS}), the upper bound is a supremum in $s$ of a time integral. 
 
As pointed out in Section~\ref{sec:ex-BBGKYHS}, other proofs of Theorems~\ref{existence-thmHS} and~\ref{existence-thmboltzmannHS} can be devised, using tools inspired by the proof of the Cauchy-Kowalevskaya theorem: we recall for instance the approaches of~\cite{nirenberg} and~\cite{nishida}, as well as~\cite{lanford} and~\cite{K}.

%% file: Convergence-HS.tex
\chapter{Statement of the convergence result}\label{convergenceresult}
\setcounter{equation}{0}

\label{convergence-statement} 

We state here our first main result,   describing convergence of mild solutions to the BBGKY hierarchy~(\ref{BBGKY-mildHS}) to mild solutions of the  Boltzmann hierarchy \eqref{mild-Boltzmann}. 
This result implies in particular Theorem~\ref{thm-cv-intro} stated in the Introduction page~\pageref{thm-cv-intro}.

The first part of this chapter is devoted to a precise description of Boltzmann initial data   which are {\it admissible}, i.e., which can be obtained as the limit of BBGKY initial data satisfying the required uniform bounds. 
This involves discussing  the notion of ``quasi-independence" mentioned in the Introduction, via a conditioning of the initial data.
Then we state the main convergence result (Theorem \ref{main-thm} page~\pageref{main-thm}) and sketch the main steps of its proof. 

\section{Quasi-independence}\label{admissibledata}
\setcounter{equation}{0}
In this paragraph we discuss the notion of ``quasi-independent" initial data. We first
define admissible Boltzmann initial data, meaning data which can be reached from BBGKY initial data (which are bounded families of   marginals) by a limiting procedure, and then
 show how to ``condition" the initial BBGKY initial data so as to converge towards admissible Boltzmann initial data.  Finally we characterize admissible Boltzmann initial data.

\subsection{Admissible Boltzmann data} \label{sec:adm-B}
 In the following we denote
$$\label{defOmegaN}
\Omega_s := \{Z_s \in \R^{2ds} \, , \,\, \forall i \neq j \, , \,  x_i \neq x_j\} \, .
$$
\begin{Def}[Admissible Boltzmann data] \label{def:adm-data}
Admissible Boltzmann data
are defined as families~$F_0 = (f^{(s)}_0)_{s \geq 1},$ with each~$f_0^{(s)}$  nonnegative, integrable and continuous over~$\Omega_s$, such that
\begin{equation} \label{limit:marginals}
 \int_{\R^{2d}} f^{(s+1)}_0(Z_s,z_{s+1})  \, dz_{s+1} = f^{(s)}_0(Z_s) \, , 
\end{equation}
and which are limits of BBGKY initial data~$  F_{0,N} = (  f^{(s)}_{0,N})_{1 \leq s \leq N}  \in {\bf X}_{\e,\b_0,\mu_0}$ in the following sense: for some~$  F_{0,N}$ satisfying
\begin{equation} \label{def:unif-dataHS} \begin{aligned}
 \sup_{N \geq 1} \|   F_{0,N} \|_{\e,\b_0,\mu_0} < \infty\,,\quad \mbox{for some $\b_0  > 0\,,$ $\mu_0 \in \R\,,$ as $N \e^{d-1} \equiv 1\,,$ } \end{aligned}
 \end{equation}
for each given~$s \in [1,N]$, the marginal of order~$s$ defined by
 \begin{equation} \label{def:adm-truncHS} 
   f_{0,N}^{(s)}(Z_s)  = \int_{  \R^{2d(N-s)}}  \indc_{Z_N \in {\mathcal D}_N }   f_{0,N}^{(N)} (Z_N) \,  dz_{s+1} \dots dz_N \,, \quad  1 \leq s < N\,, 
\end{equation}
converges   in the Boltzmann-Grad limit:
  \begin{equation} \label{def:cv-data}   f_{0,N}^{(s)} \longrightarrow f_0^{(s)}   \quad \mbox{as $N \to \infty$ with $N \e^{d-1} \equiv 1 \, ,$ locally uniformly in $\Omega_s \, .$}
 \end{equation}
 \end{Def}
 In this section we shall prove the following result.
\begin{prop} \label{cor:B-data} The set of admissible Boltzmann data, in the sense of Definition~{\rm\ref{def:adm-data}},  is the set of families of marginals $F_0$ as in~\eqref{limit:marginals} satisfying a uniform bound~$\| F_0 \|_{0,\b_0,\mu_0} < \infty$ for some~$\b_0  > 0$ and~$\mu_0 \in \R$.
\end{prop}

\subsection{Conditioning}
We first consider ``chaotic" configurations, corresponding to tensorized initial measures, or initial densities which are products of one-particle distributions:
\begin{equation} \label{def:init-g0}
f_0^{\otimes s}(Z_s)=\prod_{1 \leq i \leq s} f_0 (z_i) \, , \qquad 1 \leq s \leq N \, ,
\end{equation} where $f_0$ is nonnegative, normalized, and belongs to some $X_{0, 1,\b_0}$ space (see Definition \ref{def:functional-spacesHS} page \pageref{def:functional-spacesHS}): 
\begin{equation} \label{cond:f0}
f_0 \geq 0 \, , \quad \int_{\R^{2d}} f_0(z) dz =1 \, ,
\quad 
e^{\mu_0} | f_0|_{0,1,\beta_0}\leq 1  \quad \mbox{for some $\b_0 > 0 \,, \mu_0 \in \R \, .$}
\end{equation}
Such initial data are particularly meaningful insofar as they will produce the Boltzmann equation~(\ref{boltz-eq}), and we shall show  in Proposition~\ref{init-cv1}   that they are admissible.

We then consider the initial data~$\displaystyle \indc_{Z_N \in \D_N} f_0^{\otimes N}(Z_N)$, 
and the property of normalization is   preserved by introduction of the partition function
\begin{equation} \label{def:Z-f} {\mathcal Z}_N  := \int_{\R^{2dN}} 
 \indc_{Z_N \in \D_N} f_0^{\otimes N}(Z_N)\, dZ_{N}\,.
\end{equation}
{\it Conditioned datum built on $f_0$} is then defined as~$ {\mathcal Z}_N ^{-1}  \indc_{Z_N \in \D_N} f_0^{\otimes N}(Z_N).$
This operation is called {\it conditioning on energy surfaces}, and is a classical tool in statistical mechanics (see \cite{Georgii,LL5,LL9} for instance).

The partition function defined in \eqref{def:Z-f} satisfies the next result, which will be useful in the following.
\begin{Lem} \label{lem:bd-f0HS} Given $f_0$ satisfying \eqref{cond:f0}, there holds for $1 \leq s \leq N$ the bound 
$$  1 \leq {\mathcal Z}_N^{-1} {\mathcal Z}_{N-s} \leq \big( 1 - \e \k_d |f_0|_{L^\infty L^1}\big)^{-s},
$$ 
in the scaling $N \e^{d-1} \equiv 1,$ where $|f_0|_{L^\infty L^1}$ denotes the $L^\infty(\R^d_x,L^1(\R^d_v))$ norm of $f_0,$ and $\k_d$ denotes the volume of the unit ball in $\R^d.$\end{Lem}
 
\begin{proof}  
We have  by definition
$$
 {\mathcal Z}_{s+1} =  \int_{\R^{2d(s+1)}}    \indc_{Z_{s+1} \in \D_{s+1}} \Big( \prod_{i=1}^s \indc_{|x_i-x_{s+1}| >\eps} \Big) f_0^{\otimes (s+1)}(Z_{s+1})  \, dZ_{s+1} \, .
$$ 
By Fubini, we deduce
 $$ \begin{aligned}  {\mathcal Z}_{s+1}  = \int_{\R^{2ds}} \left( \int_{\R^{2d}} \Big(\prod_{1 \leq i \leq s} \indc_{|x_i-x_{s+1} |>\eps}\Big)  f_0(z_{s+1})  dz_{s+1}\right)\indc_{Z_s \in \D_s} f_{0}^{\otimes s}(Z_s)  dZ_{s} \, .\end{aligned}$$
Since
 $$ \int_{\R^{2d}} \Big(\prod_{1 \leq i \leq s} \indc_{|x_i-x_{s+1} |>\eps}\Big)  f_0(z_{s+1})  dz_{s+1} \geq \|  f_0\|_{L^1} -  \k_d s \eps^d  |f_0|_{L^\infty L^1} \, ,$$
   we deduce from the above, by nonnegativity of $f^{\otimes s}_{0}$ and the fact that~$ \| f_0\|_{L^1}=1$ the lower bound
$$  {\mathcal Z}_{s+1} \geq
 {\mathcal Z}_{s} \big(1 -  \k_d s \eps^d  | f_0|_{L^\infty L^1}\big) \, ,
$$ 
implying by induction
$$ {\mathcal Z}_N  \geq {\mathcal Z}_{N-s}  \prod_{j = N-s}^{N-1} (1 - j \e^d \k_d |f_0|_{L^\infty L^1})  \geq {\mathcal Z}_{N-s} \big( 1 - \e \k_d |f_0|_{L^\infty L^1}\big)^s \, ,$$
where we used $s \leq N$ and the scaling $N \e^{d-1} \equiv 1.$  That proves the lemma.
\end{proof}

\subsection{Characterization of admissible Boltzmann initial data}
The aim of this paragaph is to prove Proposition~\ref{cor:B-data}.

Let us start by proving the following statement, which  provides examples of admissible Boltzmann initial data, in terms of tensor products.
\begin{prop}\label{init-cv1}  
Given $f_0$ satisfying  \eqref{cond:f0},  the data~$F_0 = (f_0^{\otimes s})_{s \geq 1}$ is admissible  in the sense of Definition~{\rm\ref{sec:adm-B}}.   
\end{prop}
  
 \begin{proof} Let us define, with notation~(\ref{def:Z-f}),
 $$
  f_{0,N}^{(N)} := {\mathcal Z}_N ^{-1}  \indc_{Z_N \in \D_N} f_0^{\otimes N}(Z_N)
 $$
 and let~$  F_{0,N} := \big(  f_{0,N}^{(s)}\big)_{s \leq N}$   be the set of its  marginals. In a first step we   prove they   satisfy uniform bounds as in~(\ref{def:unif-dataHS}).  In a second step, we prove the local uniform convergence to zero of~$ f_{0,N}^{(s)} - f^{\otimes s}_{0}$ in $\Omega_s$, as in~(\ref{def:adm-truncHS}).  

   {\it First step.}  We have clearly 
$$
\begin{aligned}
   f_{0,N}^{(s)}(Z_s) \leq {\mathcal Z}_N^{-1}   \indc_{Z_s \in {\mathcal D}_s}  f^{\otimes s}_{0}(Z_s)\int_{\R^{2d(N-s)}}
 \prod_{s+1\leq i<  j\leq N}   \indc_{|x_i-x_j| >\e}  \prod_{s+1\leq i\leq N}{}  f_0 (z_i) \,   dZ_{(s+1,N)} \,,
 \end{aligned}$$
  where we have used the notation
  $$\label{page:ij-measure}
  dZ_{(s+1,N)} :=dz_{s+1} \dots dz_N \, .
  $$
This gives
$$ 
\begin{aligned}
  f_{0,N}^{(s)} (Z_s) &\leq {\mathcal Z}_N^{-1} {\mathcal Z}_{N-s}   \indc_{Z_s \in {\mathcal D}_s}  f^{\otimes s}_{0}(Z_s)\\
 & \leq \big( 1 - \e \k_d |f_0|_{L^\infty L^1}\big)^{-s}   \indc_{Z_s \in {\mathcal D}_s}  f^{\otimes s}_{0}(Z_s)\,,
   \end{aligned}
   $$
the second inequality by Lemma \ref{lem:bd-f0HS}. 

By $2 x + \ln(1 - x) \geq 0$ for $x \in [0,1/2],$ there holds  \begin{equation} \label{ineq-elem}
 (1 - \e \k_d |f_0|_{L^\infty L^1})^{-s} \leq e^{2 s \e \k_d |f_0|_{L^\infty L^1}}\, , \qquad \mbox{if $2 \e \k_d |f_0|_{L^\infty L^1} < 1\, ,$}
 \end{equation}
so that for $N$ larger than some $N_0$ (equivalently, for $\e$ small enough),
$$   e^{s \mu'_0} \big|   f_{0,N}^{(s)}  \big|_{\e,s,\b_0} \leq e^{s (\mu'_0 + 2 \e \k_d |f_0|_{L^\infty L^1})} \big|   \indc_{Z_s \in {\mathcal D}_s}  f^{\otimes s}_{0}(Z_s) \big|_{\e,s,\b_0} \leq \Big( e^{ 2 \e \k_d |f_0|_{L^\infty L^1}} |f_0|_{0,1,\beta_0}\Big)^s \, 
.$$
Therefore, for any $\mu'_0 < \mu_0$ and for $\eps$ sufficiently small,
$$\begin{aligned} \sup_{N \geq N_1} \|   F_{0,N} \|_{\e,\b_0,\mu'_0}  < \infty\, , \end{aligned} $$
which of course implies the uniform bound $\displaystyle \sup_{N \geq 1} \|   F_{0,N} \|_{\e,\b_0,\mu'_0} < \infty.$ 
  
 \medskip
 
 {\it Second step.}   
  We compute for $s \leq N:$ 
$$  
 \begin{aligned} 
  f_{0,N}^{(s)} = {\mathcal Z}_N^{-1} \indc_{Z_s \in {\mathcal D}_s}f^{\otimes s}_{0} 
  \int_{\R^{2d(N-s)}}  \prod_{s+1 \leq i < j \leq N} \indc_{|x_i - x_j| > \eps}  \prod_{i \leq s <  j}  \indc_{|x_i - x_j| > \eps}   \prod_{s+1\leq i \leq N} f_0(z_i) \, dZ_{(s+1,N)} \, .
  \end{aligned}$$
  We deduce, by   symmetry, 
\begin{equation} \label{dec:tensor}   f_{0,N}^{(s)} = {\mathcal Z}_N^{-1} \indc_{Z_s \in {\mathcal D}_s}f^{\otimes s}_{0} 
  \Big( {\mathcal Z}_{N-s} - {\mathcal Z}^\flat_{(s+1,N)} \Big)
\end{equation} 
 with the notation
$$
 \begin{aligned} {\mathcal Z}^\flat_{(s+1,N)} = \int_{\R^{2d(N-s)}}
  \Big(1 - \prod_{i \leq s < j}\indc_{|x_i - x_j| > \eps}\Big)   \prod_{s+1 \leq i < j \leq N} \indc_{|x_i - x_j| > \eps} \prod_{s+1\leq i \leq N} f_0(z_i)\, dZ_{(s+1,N)} \, ,\end{aligned}
 $$
  so that ${\mathcal Z}^\flat_{(s+1,N)}$ is a function of $X_s$.
  
From there, the difference $\indc_{Z_s \in {\mathcal D}_s}f^{\otimes s}_{0} 
-  f_{0,N}^{(s)}$ decomposes as a sum:
\begin{equation}
\label{marginal-dec}
\begin{aligned} \indc_{Z_s \in {\mathcal D}_s}f^{\otimes s}_{0} 
 -   f_{0,N}^{(s)} & = \Big(1 - {\mathcal Z}_{N}^{-1} {\mathcal Z}_{N-s}\Big) \indc_{Z_s \in {\mathcal D}_s}f^{\otimes s}_{0} 
+  {\mathcal Z}_N^{-1} {\mathcal Z}^\flat_{(s+1,N)} \indc_{Z_s \in {\mathcal D}_s}f^{\otimes s}_{0} 
 \, .
\end{aligned}
\end{equation} 
By Lemma \ref{lem:bd-f0HS}, there holds $1 - {\mathcal Z}_{N}^{-1} {\mathcal Z}_{N-s} \to 0$ as $N \to \infty,$ for fixed $s.$ Since $f^{\otimes s}_0$ is uniformly bounded in $\Omega_s,$ this implies that the first term in the right-hand side of \eqref{marginal-dec} tends to 0 as $N \to \infty,$ uniformly in $\Omega_s.$
Besides, by
 $$\dsp{0 \leq  1 - \prod_{i \leq s < j}\indc_{|x_i - x_j| > \eps} \leq   \sum_{\begin{smallmatrix} 1 \leq i \leq s \\ s+1 \leq j \leq N \end{smallmatrix}} \indc_{|x_i-x_j|<\eps}\,,}$$
 we bound 
  $$ 
  \begin{aligned}
  {\mathcal Z}^\flat_{(s+1,N)} & \leq \sum_{1 \leq k \leq s} \int_{\R^{2d(N-s)}}  \Big(\sum_{s+1 \leq j \leq N} \indc_{|x_k-x_j|<\eps}\Big) \prod_{s+1 \leq i < j \leq N} \indc_{|x_i - x_j| > \eps} \prod_{s+1\leq i \leq N} f_0(z_i)   \,  dZ_{(s+1,N)}
\,   .
  \end{aligned}
$$
 Given $1 \leq k\leq s,$ there holds by symmetry  and Fubini,
 $$ \begin{aligned} \int_{\R^{2d(N-s)}} & \Big(\sum_{s+1 \leq j \leq N} \indc_{|x_k - x_j| < \e}\Big)   \prod_{s+1 \leq i < j \leq N} \indc_{|x_i - x_j| > \eps}  \prod_{s+1\leq i \leq N} f_0(z_i)\,  dZ_{(s+1,N)} \\ & \leq (N-s) \int_{\R^{2d}} \indc_{|x_i - x_{s+1}| < \e} f_0(z_{s+1}) dz_{s+1}  \\
 & \qquad \times \, \int_{\R^{2d(N-s-1)}}   \prod_{s+2\leq i < j \leq N} \indc_{|x_k - x_j| > \eps}\prod_{s+2\leq i \leq N} f_0(z_i) \, dZ_{(s+2,N)} \\ & = (N-s) \int_{\R^{2d}} \indc_{|x_i - x_{s+1}| < \e} f_0(z_{s+1}) dz_{s+1}  \, \times \, {\mathcal Z}_{N-s-1} \, ,
 \end{aligned}$$
 so that
 \begin{equation} \label{bd:flat} {\mathcal Z}^\flat_{(s+1,N)} \leq s (N-s) \e^d \k_d |f_0|_{L^\infty L^1} {\mathcal Z}_{N-s-1} \, ,
 \end{equation}
where $|f_0|_{L^\infty L^1}$ denotes the $L^\infty(\R^d_x,L^1(\R^d_v))$ norm of $f_0.$ By Lemma \ref{lem:bd-f0HS}, we obtain
$$ 
   {\mathcal Z}_N^{-1} {\mathcal Z}^\flat_{(s+1,N)}  \leq \e s \k_d |f_0|_{L^\infty L^1} \big(1 - \e \k_d | f_0|_{L^\infty L^1}\big)^{-(s+1)} \, ,
$$
  and the upper bound tends to 0 as $N \to \infty,$ for fixed $s.$ This implies convergence to 0, uniformly in~$\Omega_s,$ of the second term in the right-hand side of \eqref{marginal-dec}.
  
 We thus proved the uniform convergence $  f^{(s)}_{0,N} - \indc_{Z_s \in {\mathcal D}_s}f^{\otimes s}_{0} \to 0$ in $\Omega_s,$ and hence  $f^{\otimes s}_{0,N} \to f_0^{\otimes s}$ holds locally uniformly in $\Omega_s.$ We conclude that $  f^{(s)}_{0,N}$ converges locally uniformly to tensor products $f_0^{\otimes s}$ in~$\Omega_s.$

Proposition \ref{init-cv1}   is proved. \end{proof}

 By Proposition \ref{init-cv1}, tensor products $(f_0^{\otimes s})_{s \geq 1},$ with $f_0$ satisfying \eqref{cond:f0}, are admissible Boltzmann data. It is easy to generalize that result (see Proposition \ref{prop:gen-data} below) to the convex hull of the set of tensor products. We shall actually also show the converse: all admissible Boltzmann data belong to the convex hull of tensor products, and that will enable us to deduce Proposition~\ref{cor:B-data}.

We first remark that given a Boltzmann datum $F_0,$ and an associated BBGKY datum $F_{0,N},$ there holds 
\begin{equation} \label{bd:B-datum} \| F_0 \|_{0,\b_0,\mu_0} < \infty \, ,
\end{equation}
 with $\b_0$ and $\mu_0$ as in \eqref{def:unif-dataHS}.
 Indeed, let $C_0 = \displaystyle \sup_{N \geq 1} \| F_{0,N} \|_{\e,\b_0,\mu_0} < \infty.$ Given $s$ and $Z_s \in \Omega_s,$ for $\e$ small enough, $\indc_{Z_s \in \D_s}= 1.$ Besides, by \eqref{def:cv-data} there holds the pointwise convergence $f^{(s)}_{0,N}(Z_s) \to f^{(s)}_0(Z_s).$ Hence taking the limit~$\e \to 0$ in the left-hand side of the inequality
 $e^{s \mu_0 + \b_0 E_\e(Z_s)} |f^{(s)}_{0,N}(Z_s)| \leq C_0,$
 we find \eqref{bd:B-datum}.
  
The Hewitt-Savage theorem reveals the specific role played by tensor products: the set of families~$F_0 = (f^{(s)}_0)_{s \geq 1}$ of marginals \eqref{limit:marginals} satisfying the uniform bound \eqref{bd:B-datum} is the convex  hull of tensorized initial data, as described in the following statement.
 We define  ${\mathcal P} = {\mathcal P}(\R^{2d})$ be the set of continuous densities of probability in $\R^{2d}:$ 
 \begin{equation} \label{def:P} {\mathcal P} := \big\{ h \in C^0(\R^{2d};\R) \, , \quad h \geq 0\, , \quad \int_{\R^{2d}} h(z) dz = 1 \big\} \, .\end{equation}
\begin{prop} \label{prop:B-data}
Given $F_0 = (f^{(s)}_0)_{s \geq 1}$ a family of marginals \eqref{limit:marginals} satisfying the uniform bound \eqref{bd:B-datum} with constants $\b_0 > 0$ and $\mu_0 \in \R,$ there exists a probability measure~$\pi$ over the set~${\mathcal P}$, with
\begin{equation} \label{supp:pi-prop}
 {\rm supp}\,\pi \subset \big  \{g  \in {\mathcal P}, \, |g|_{0,1,\b_0}  \leq e^{-\mu_0} \big\} \, ,
 \end{equation}  such that the following representation holds:
\begin{equation}\label{hewittsavage}
f_0^{(s)} = \int_{{\mathcal P}} g^{\otimes s} d\pi (g) \, , \qquad s \geq 1 \, .
\end{equation}
\end{prop}  
\begin{proof} Given a family~$F_0$ satisfying \eqref{limit:marginals} and~\eqref{bd:B-datum}, the existence of $\pi$ satisfying \eqref{hewittsavage} is granted by the Hewitt-Savage theorem~\cite{HS}. The goal is then to prove the inclusion \eqref{supp:pi-prop}.
Assume by contradiction that, for some $\a > 0,$ 
\begin{equation} \label{pi-a}
\pi (A_\alpha)= \kappa_\alpha  >0\,,  \quad \mbox{where $A_\alpha:=\big  \{g  \in {\mathcal P}(\R^{2d}), \quad |g|_{0,1,\b_0} \geq e^{-\mu_0} + \alpha \big\}$} \, . 
\end{equation}
We then have by~(\ref{hewittsavage})  
$$
 f_0^{(s)}  \geq \int_{A_\alpha }g^{\otimes s} d\pi (g),
$$
hence by $ f^{(s)}_0 \leq e^{-s \mu_0} \| F_0 \|_{0,\b_0,\mu_0},$ we infer that
$
\| F_0\|_{0,\b_0,\mu_0} \geq \kappa_\alpha  (1 + \alpha e^{\mu_0})^s,
$
which cannot hold for some $\a > 0$ and all $s,$ since $1 + \alpha e^{\mu_0} > 1.$ Hence \eqref{pi-a} does not hold, which proves the result.
\end{proof}

 We now give the generalization of Proposition \ref{init-cv1} that will be useful in the proof of Proposition~\ref{cor:B-data}.
Let $\pi$ be a probability measure on ${\mathcal P}$ satisfying~(\ref{supp:pi-prop})  for some $\b_0 > 0$ and some $\mu_0 \in \R$. Next we define
 \begin{equation} \label{def:g-data}
 \pi^{(s)} := \int_{{\mathcal P}} h^{\otimes s} d\pi(h)\,.
 \end{equation}
In the case when $\pi = \delta_{f_0},$ then \eqref{def:g-data} reduces to the tensor product \eqref{def:init-g0}-\eqref{cond:f0}.

In the general case, we define
 \begin{equation}
  \label{def:partition-new} 
  \begin{aligned}
  {\mathcal Z}_N(h): = \int_{\R^{2dN}} 
 \indc_{Z_N \in \D_N}  h^{\otimes N}(Z_N)\, dZ_N \, , \quad h \in \P \, ,\\
 \pi_N ^{(N)}:= \int_{{\mathcal P}} {1\over {\mathcal Z}_N(h)} h^{\otimes N} d\pi(h)\,.
 \end{aligned}
 \end{equation}
 generalizing \eqref{def:Z-f}.
   
  The following result is an obvious generalization of Lemma \ref{lem:bd-f0HS}.
\begin{Lem} \label{lem:tildeZ} Given $\pi$ satisfying \eqref{supp:pi-prop} and $h \in \mbox{supp}\,\pi,$ the family of partition functions ${\mathcal Z}_s$ defined in \eqref{def:partition-new} satisfies for $1 \leq s \leq N$ the bound 
$$ 
 1 \leq {\mathcal Z}_N(h)^{-1} {\mathcal Z}_{N-s}(h) \leq \big( 1 - \e C_d e^{-\mu_0}  \b_0^{-1/2} \big)^{-s} \, , 
$$ 
where~$C_d$ depends only on~$d$.
\end{Lem}

The next statement generalizes Proposition \ref{init-cv1}.
Its proof is an immediate extension of the proof of Proposition \ref{init-cv1} thanks to the dominated convergence theorem, using the obvious bound~$ \indc_{Z_s \in \D_s}  h^{\otimes s}(Z_s) \leq e^{-s\mu_0}$.
 \begin{prop}\label{prop:gen-data}   Given $\pi$ satisfying  \eqref{supp:pi-prop}, the 
  data~$(\pi^{(s)})_{s \geq 1}$, with~$\pi^{(s)}$ defined in \eqref{def:g-data},  is admissible  in the sense of Definition~{\rm\ref{sec:adm-B}}. 
  
  It is obtained for instance from the BBGKY data $(\pi_N^{(s)})_{s\leq N}$ defined by (\ref{def:partition-new}).
\end{prop}
%

\begin{proof}[Proof of Proposition~{\rm\ref{cor:B-data}}] We already observed in \eqref{bd:B-datum} that admissible Boltzmann data are bounded families of marginals. Conversely, given a bounded family of marginals $F_0,$ by Proposition~\ref{prop:B-data} representation~\eqref{hewittsavage} holds. Then, by Proposition~\ref{prop:gen-data}, $F_0$ is an admissible Boltzmann datum.  This proves Proposition~\ref{cor:B-data}.
 \end{proof} 
 Combining Propositions  \ref{cor:B-data} and  \ref{prop:B-data}, we see that all admissible Boltzmann data are built on tensor products, in the sense that given an admissible Boltzmann datum, representation \eqref{hewittsavage} holds for some $\pi$ satisfying \eqref{supp:pi-prop}. 

\section{Main result: Convergence of the BBGKY hierarchy to the Boltzmann hierarchy}\label{mainresulconvergence}
\setcounter{equation}{0}

\subsection{Statement of the result}$ $

 Our main result is a {\it weak convergence} result, in the sense of convergence of observables, or averages with respect to the momentum variables. Moreover, since the marginals are defined in~$\D_s$, we must also 
eliminate, in the convergence, the diagonals in physical space. Let us give a precise definition of the convergence we shall be considering. 
 \begin{Def}[Convergence] \label{def:cv} Given a sequence $(h^s_N)_{1 \leq s \leq N}$ of functions $h^s_N \in C^0(\D_s;\R),$ a sequence $(h^s)_{s \geq 1}$ of functions $h^s \in C^0(\Omega_s;\R),$ we say that $(h^s_N)$ converges on average and locally uniformly outside the diagonals to $(h^s),$ and we denote
  $$ (h^s_N)_{1 \leq s \leq N} \stackrel{\sim}{\longrightarrow} (h^s)_{1 \leq s} \, ,$$
  when for any fixed $s$, any test function $\varphi_s \in \C^0_c(\R^{ds};\R),$ there holds
 \label{indexdefobservable}
   $$
  I_{\varphi_s} \big( h^s_N -  h^s\big)\big(X_s):= \int_{\R^{ds}}  \varphi_s(V_s) \big( h^s_N - h^s \big)(Z_s) dV_s  \longrightarrow 0 \, , \quad \mbox{as $N \to \infty \, ,$}
  $$
 locally uniformly in $\dsp{\big\{X_s \in \R^{ds}, \,\,x_i \neq x_j \,\,\mbox{for $i \neq j$}\big\}\,.}$
\end{Def}
With regard to spatial variables, this notion of  convergence is similar to the convergence in the sense of Chacon.

We remark that local uniform convergence in $\Omega_s$ implies convergence in the sense of Definition \ref{def:cv}:

\begin{lem} \label{lem:cv} Given $(f^{(s)}_N)_{1 \leq s \leq N}$ a bounded sequence in ${\bf X}_{\e,{\boldsymbol\beta},{\boldsymbol\mu}}$ with the notation of Definition~{\rm\ref{deffunctionspacesexistence}}, if $f^{(s)}_N \to f^{(s)}$ for fixed $s,$ uniformly in $t \in [0,T]$ and locally uniformly in $\Omega_s,$ then there holds~$f^{(s)}_N \tildeto f^{(s)},$ uniformly in $t \in [0,T].$
\end{lem}

\begin{proof} Let $K_s$ be compact in $\dsp{\big\{X_s \in \R^{ds}, \,\,x_i \neq x_j \,\,\mbox{for $i \neq j$}\big\}\,.}$ There holds 
$$ \big| I_{\varphi_s} \big( f^{(s)}_N -  f^{(s)}\big) (X_s)\big| \leq \|\varphi_s\|_{L^1(\R^d)} \sup_{V_s \in \mbox{\footnotesize supp}\, \varphi_s} 
\big| \big(  f^{(s)}_N - f^{(s)} \big) (X_s,V_s)\big| \, .$$
The set $K_s \times \mbox{supp}\, \varphi_s$ is compact in $\Omega_s.$ Hence the above upper bound converges to 0 as $N \to \infty,$ uniformly in $[0,T] \times K_s.$  
\end{proof}
We can now state our main result.
 \begin{Thm}[Convergence]\label{main-thm}
Let~$\beta_0 >0$ and~$\mu_0 \in \R$ be given. There is a time~$T>0$ such that the following holds. Let~$F_0$ in~${\mathbf X }_{0,\beta_0,\mu_0}$  be an   admissible Boltzmann datum, with    associated   family of   BBGKY datum~$(  F_{0,N})_{N \geq 1}$, in~${\mathbf X }_{\eps,\beta_0,\mu_0}$. Let~$F$ and~$F_N$ be  the solutions  to the Boltzmann and BBGKY hierarchy produced by~$F_0$ and~$F_{0,N}$ respectively. There holds
\begin{equation}\label{convergenceFNtildetoF}  F_N \tildeto F \, , 
\end{equation} 
 uniformly on $[0,T]$,.

 In particular, if~$F_{0} =( f_0^{\otimes s} )_{s \geq 1}$, then the first marginal~$f_N^{(1)}$ converges to the solution~$f$ of the Boltzmann equation~{\rm(\ref{boltz-eq})} with initial data~$f_0$.
 
 Finally in the case when~$F_{0} =( f_0^{\otimes s} )_{s \geq 1}$ with~$f_0$ Lipschitz, then 
 the convergence~{\rm(\ref{convergenceFNtildetoF})} holds at a rate~$O(\eps^\alpha)$ for any~$\alpha<(d-1)/(d+1)$.
 \end{Thm}
  
Solutions  to the Boltzmann hierarchy issued from tensorized initial data are themselves  tensorized. For such data, the Boltzmann  hierarchy then reduces to the nonlinear Boltzmann equation~{\rm(\ref{boltz-eq})}, and Theorem \ref{main-thm} describes an asymptotic form of propagation of chaos, in the sense that an initial property of independence is propagated in time, in the   limit.    This corresponds to Theorem~\ref{thm-cv-intro} stated in the Introduction page~\pageref{thm-cv-intro}.

\subsection{About the proof of Theorem~\ref{main-thm}: outline of Chapter~\ref{convergenceHS} and Part IV}$ $
\setcounter{equation}{0}

The formal derivation presented in Chapter~\ref{spheres}, Section~\ref{spheresboltz}, fails because of a number of incorrect arguments:
\begin{itemize}
\item Since mild solutions to the BBGKY hierarchy are defined by the Duhamel formula (\ref{BBGKY-mildHS}) where the solution itself occurs in the source term, we   need some precise information on the convergence to take limits directly in (\ref{BBGKY-mildHS}).

\item The irreversibility inherent to the Boltzmann hierarchy appears in the limiting process as an arbitrary choice of the time direction (encoded in the distinction between pre-collisional and post-collisional particles), and more precisely as an arbitrary choice of the initial time, which is the only time for which one has a complete information on the family of marginals $F_{0,N}$. This specificity of the initial time does not appear clearly in (\ref{BBGKY-mildHS}).

\item The heuristic argument which allows to neglect pathological trajectories, meaning trajectories for which the reduced dynamics with  $s$-particles does not coincide with the free transport (${\bf T}_s \neq {\bf S}_s$),
requires to be quantified.  Indeed we have more or less to repeat the operation infinitely many times, since mild solutions are defined by a loop process; moreover, the question of the stability with respect to micro-translations in space must be analyzed.

\item  Because of the conditioning, the initial data are not so smooth. The operations such as infinitesimal translations on the arguments require therefore a careful treatment.

\end{itemize}
 
\medskip

To overcome the two first difficulties, the idea is to start from the iterated Duhamel formula, which allows to express any marginal $  f^{(s)}_N(t,Z_s)$ in terms of the initial data $  F_{0,N}$.
By successive integrations in time, we have indeed the following representation of $   f_N^{(s)}$: 
  \begin{equation}
\label{fs1}
\begin{aligned}
   f_N^{(s)} (t) =\sum_{k=0}^\infty  \int_0^t \int_0^{t_1}\dots  \int_0^{t_{k-1}}  {\bf T}_s(t-t_1) {\mathcal C}_{s,s+1}  {\bf T}_{s+1}(t_1-t_2) {\mathcal C}_{s+1,s+2} \dots  \\
\dots  {\bf T}_{s+k}(t_k)    f_N^{(s+k)}(0) \: dt_{k} \dots dt_1
\end{aligned}
\end{equation}
where by convention $   f_N^{(j)}(0) \equiv 0 $ for $j > N$.

Using a {\bf dominated convergence argument}, we shall first reduce (in Chapter~\ref{convergenceHS}) to the study of a functional 
\begin{itemize}
\item   defined as a finite sum of terms (independent of $N$),
\item where the energies of the particles are assumed to be bounded (namely $E_0(Z_{s+k})\leq R^2$),
\item and  where the collision times are supposed to be well separated (namely $|t_j-t_{j+1}|\geq \delta$).
\end{itemize}
The reason for the two last assumptions is essentially technical, and will appear more clearly in the next step.

The heart of the proof, in Part IV, is then to prove the term by term convergence, dealing with pathological trajectories. 
Let us   recall that each collision term is defined as an integral with respect to positions and velocities. The main idea consists then in proving that we cannot build pathological trajectories if we exclude at each step a small domain of integration.
The explicit construction of this ``bad set" lies on
\begin{itemize}
\item a very simple geometrical lemma which ensures that two particles of size $\eps$ have not collided in the past provided that their relative velocity does not belong to a small subset of $\R^d$ (see Lemma~\ref{geometric-lem1}),
\item scattering estimates which tell us how these properties of the transport are modified when a particle is deviated by a collision (see Lemma \ref{geometric-lem3}).
\end{itemize}

This construction, which is the technical part of the proof, will be detailed in Chapter~\ref{pseudotraj}. The conclusion of the convergence proof is presented in Chapters 13 and\ref{convergenceproof}.

%% file: Strategy.tex
\chapter{Strategy of the convergence proof}\label{convergenceHS}
\setcounter{equation}{0}

 The goal of this chapter is to use dominated convergence arguments to reduce the proof of Theorem~\ref{main-thm} stated page~\pageref{main-thm} to the term-by-term convergence of some  functionals involving a finite (uniformly bounded) number of marginals (Section~\ref{finite-nHS}).
 In order to further simplify the convergence analysis, we shall modify these functionals by eliminating some small domains of integration in the time and velocity variables corresponding to pathological dynamics, namely by removing  large energies in Section \ref{large-vHS} and clusters of collision times in Section~\ref{t-clustersHS}.

We consider therefore  families of initial data:   Boltzmann initial data~$F_0 = (f_0^{(s)})_{s \in \N} \in {\bf X} _{0,\beta_0,\mu_0}$ 
and for each $N$,    BBGKY initial data $  F_{N,0} =(  f_{N,0}^{(s)})_{1 \leq s\leq N} \in {\bf X} _{\eps,\beta_0,\mu_0}$   such that
$$
 \sup_N \|   F_{N,0}   \|_{\eps,\beta_0,\mu_0} =\sup_N \sup_{s\leq N} \sup_{Z_s \in \D_s}\big( \exp(\beta_0 E_0(Z_s) +\mu_0 s)   f_{N,0}^{(s)}(Z_s) \big)<+\infty \, .
$$
We then associate the respective unique mild solutions  of the hierarchies
$$
 f^{(s)} (t) = {\bf S}_s(t)  f^{(s)}_{0} +   \int_0^t {\bf S}_{s}(t - \t)     {\mathcal C}^0_{s,s+1}    f^{(s+1)}  (\t) \, d\t 
$$
and
$$
   f_N^{(s)} (t) = {\mathbf  T}_s(t)  f^{(s)}_{N,0} +  \int_0^t {\mathbf  T}_{s}(t - \t)     {\mathcal C}_{s,s+1}    f_N^{(s+1)}  (\t) \, d\t \, .
$$
 In terms of the initial datum, they can be rewritten
$$
\begin{aligned}
    f^{(s)} (t,Z_s)=\sum_{k=0}^\infty   \int_0^t \int_0^{t_1}\dots  \int_0^{t_{k-1}}  {\bf S}_s(t-t_1) {\mathcal C}^0_{s,s+1}  {\bf S}_{s+1}(t_1-t_2) {\mathcal C}^0_{s+1,s+2} \dots  \\
\dots  {\bf S}_{s+k}(t_k)     f^{(s+k)}_0 \: dt_{k} \dots dt_1
\end{aligned}
$$
and
$$\begin{aligned}
    f_N^{(s)} (t,Z_s) =\sum_{k=0}^\infty  \int_0^t \int_0^{t_1}\dots  \int_0^{t_{k-1}}  {\mathbf  T}_s(t-t_1) {\mathcal C}_{s,s+1}  {\mathbf  T}_{s+1}(t_1-t_2) {\mathcal C}_{s+1,s+2} \dots  \\
\dots  {\mathbf  T}_{s+k}(t_k)     f_{N,0}^{(s+k)} \: dt_{k} \dots dt_1 \, .
\end{aligned}
$$
 The observables we are interested in (recall the definition of convergence provided in Definition~\ref{def:cv}) are the following\label{observablesIsHS}:
 $$
   I_{s}  (t)(X_s) := \int  \varphi_s(V_s)    f_{N}^{(s)} (t,Z_s)    dV_s 
  \quad \mbox{and}  \quad   I^0_{s}   (t)(X_s) := \int  \varphi_s(V_s) f^{(s)}(t,Z_s) dV_s \, , 
 $$
 and they therefore involve  infinite sums, as there may be infinitely many particles involved (the sum over $n$ is unbounded).
   

 \section{Reduction to a finite number of collision times}\label{finite-nHS}

\setcounter{equation}{0}
Due to the uniform bounds derived in Chapter~\ref{existence}, the dominated convergence theorem implies that it is enough to consider finite sums of elementary functions
  \begin{equation}
\label{gs1-fs-nHS}
\begin{aligned}
   f_{N}^{(s,k)} (t) :=  \int_0^t \int_0^{t_1}\dots  \int_0^{t_{k-1}}  {\mathbf  T}_s(t-t_1) {\mathcal C}_{s,s+1}  {\mathbf  T}_{s+1}(t_1-t_2) {\mathcal C}_{s+1,s+2} \dots  \\
\dots  {\mathbf  T}_{s+k}(t_k)    f_{N,0}^{(s+k)} \: dt_{k} \dots dt_1\\
f^{(s,k)} (t) :=  \int_0^t \int_0^{t_1}\dots  \int_0^{t_{n-1}}  {\bf S}_s(t-t_1) {\mathcal C}_{s,s+1}  {\bf S}_{s+1}(t_1-t_2) {\mathcal C}_{s+1,s+2} \dots  \\
\dots  {\bf S}_{s+k}(t_k)  f^{(s+k)}_0  \: dt_{k} \dots dt_1 \, , 
\end{aligned}
\end{equation}
and  the associate elementary observables~:
  \begin{equation} \label{observablesHS} 
   I_{s,k}  (t)(X_s) := \int  \varphi_s(V_s)    f_{N}^{(s,k)} (t,Z_s)    dV_s\, ,
  \quad \mbox{and}  \quad   I^0_{s,k}   (t)(X_s) := \int  \varphi_s(V_s) f^{(s,k)}(t,Z_s) dV_s\, ,
    \end{equation}
 and therefore to study the term-by-term convergence (for any fixed $k$), as expressed by   the following statement.

\begin{prop}\label{domcv} Fix $\beta_0>0$ and $\mu_0 \in \R$.
With the notation of Theorems~{\rm\ref{existence-thmHS}} and~{\rm\ref{existence-thmboltzmannHS}} page~{\rm\pageref{existence-thmHS}}, for each given~$s \in \N^*$ and~$t \in [0,T]$ there is a constant~$C>0$  such that for each~$n \in \N^*$,
$$
 \big \| I_{s} (t)  - \sum_{k=0}^n I_{s,k} (t) \big\|_{L^\infty (\R^{ds})} \leq C  \|\varphi_s\|_{L^\infty (\R^{ds})}  \left( \frac12\right)^n   \| F_{N,0}\|_{\eps,\beta_0,\mu_0}
$$
and
$$
 \big \| I^0_{s} (t)  - \sum_{k=0}^n I_{s,k} ^{0}(t) \big \|_{L^\infty (\R^{ds})} \leq C  \|\varphi_s\|_{L^\infty (\R^{ds})} \left( \frac12\right)^n    \| F_{0}\|_{0,\beta_0,\mu_0} \, , 
$$ 
  uniformly in~$ N$ and~$t \leq T$, in the Boltzmann-Grad scaling  $N\eps^{d-1} \equiv1$. 
\end{prop}
\begin{proof}
We use the notation of Chapter~\ref{existence}. Using the continuity estimate (\ref{est:duh-HCHS})  we have
\begin{equation} 
\sup_{t \in [0,T]} \Big\| \int_0^t {\mathbf  T}(t-t') {\bf C}_N F_N(t') \, dt'\Big\|_{\e,\beta(t),\mu(t)} \leq  \frac12  \Nt{ F_N }_{ \eps,{\boldsymbol\beta},{\boldsymbol\mu}} \, .
 \end{equation}
 Recalling    the definition of the Hamiltonian 
 $$
E_0(Z_s) :=\sum_{1 \leq i \leq s} {|v_i|^2\over 2} 
$$
 we then deduce that
 \begin{equation}\label{deducethattruncationnHS}
\begin{aligned}
 e^{\beta(t) E_0(Z_s) + s\mu(t)} &\Big \| \sum_{k=n+1}^\infty  \int_0^t \int_0^{t_1}\dots  \int_0^{t_{k-1}}  {\mathbf  T}_s(t-t_1) {\mathcal C}_{s,s+1}  {\mathbf  T}_{s+1}(t_1-t_2) {\mathcal C}_{s+1,s+2} \dots  \\
&\quad \dots  {\mathbf  T}_{s+k}(t_k)    f_{N,0}^{(s+k)} \: dt_{k} \dots dt_1 \Big\|_{L^\infty} \leq  C \left( \frac12\right)^n\Nt{ F_N }_{ \eps,{\boldsymbol\beta},{\boldsymbol\mu}}  \, .
\end{aligned}
\end{equation}
Combining this estimate together with the uniform bound on $\Nt{ F_N }_{ \eps,{\boldsymbol\beta},{\boldsymbol\mu}} $ given in Theorem~\ref{existence-thmHS} leads to the first statement in Proposition \ref{domcv}.
The second statement is established exactly in an analogous way, using  estimate~\eqref{est:duh-HCboltzmannHS} together with the uniform bound obtained in Theorem \ref{existence-thmboltzmannHS}.
  \end{proof}

  \section{Energy truncation}  \label{large-vHS}
\setcounter{equation}{0}

We introduce a parameter $R > 0$  
and define
  \begin{equation}
\label{gs1-fs-R}
\begin{aligned}
f_{N,R}^{(s,k)} (t): =\sum_{k=0}^n  \int_0^t \int_0^{t_1}\dots  \int_0^{t_{k-1}}  {\mathbf  T}_s(t-t_1) {\mathcal C}_{s,s+1}  {\mathbf  T}_{s+1}(t_1-t_2) {\mathcal C}_{s+1,s+2} \dots  \\
\dots  {\mathbf  T}_{s+k}(t_k)   \indc_{E_{0}(Z_{s+k})\leq R^2}  f_{N,0}^{(s+k)} \: dt_{k} \dots dt_1 \, , \\
f^{(s,k)}_R (t) :=\sum_{k=0}^n  \int_0^t \int_0^{t_1}\dots  \int_0^{t_{k-1}}  {\bf S}_s(t-t_1) {\mathcal C}^0_{s,s+1}  {\bf S}_{s+1}(t_1-t_2) {\mathcal C}^0_{s+1,s+2} \dots  \\
\dots  {\bf S}_{s+k}(t_k)  \indc_{E_0(Z_{s+k})\leq R^2} f^{(s+k)}_0 \: dt_{k} \dots dt_1
\end{aligned}
\end{equation}
and the corresponding observables
 \begin{equation} \label{observablesR} 
   I_{s,k} ^R (t)(X_s) := \int  \varphi_s(V_s)  f_{N,R}^{(s,k)} (t,Z_s)    dV_s 
  \quad \mbox{and}  \quad  I^{0,R}_{s,k}   (t)(X_s) := \int  \varphi_s(V_s) f^{(s,k)}_R  dV_s\, .
    \end{equation}

Using the bounds derived in Chapter~\ref{existence} we find easily that
$\displaystyle
\sum_{k}( I_{s,k} -I_{s,k} ^R)(t)$ and~$ \displaystyle\sum_{k}( I^0_{s,k}   -I_{s,k} ^{0,R})(t)$ can be made arbitrarily small when~$R$ is large. More precisely the following result holds.

\begin{prop}\label{delta1small}
Fix $\beta_0>0$ and $\mu_0 \in \R$.
Let~$s \in \N^*$ and~$t \in [0,T]$ be given. There are two nonnegative constants~$C,C'$  such that for each~$n$,
$$
 \big \| \sum_{k=0}^n ( I_{s,k} -I_{s,k} ^R)(t) \big\|_{L^\infty (\R^{ds})} \leq C  \|\varphi_s\|_{L^\infty (\R^{ds})}  e^{-   C' \beta_0  R^2}  \| F_{N,0}\|_{\eps,\beta_0,\mu_0}\,,
$$
and
$$
\big  \|\sum_{k=0}^n ( I^0_{s,k}   -I_{s,k} ^{0,R})(t)\big\|_{L^\infty (\R^{ds})} \leq C  \|\varphi_s\|_{L^\infty (\R^{ds})}  e^{-   C'  \beta_0 R^2}  \| F_{0}\|_{0,\beta_0,\mu_0}\,.
$$
\end{prop}

\begin{proof}
 Let~$0<\beta_0' < \beta_0$ be given, and define the associate functions~${\boldsymbol \beta'}$ and~${\boldsymbol \beta}$ as in  Theorem~\ref{existence-thmHS} stated in Chapter~\ref{existence}. Choose $\beta'_0<\beta_0$ so that
 $$C_d (1+ (\b'_0 - \l T)^{-\frac12}\big) e^{-\mu_0 + \l T} (\b'_0 - \l T)^{-\frac d2} = \frac{2\l }{3} \,.$$
 (to be compared with   (\ref{lambdaHS}) for $\beta_0$).
 
 Then according to the results of Chapter~\ref{existence} 
 and similarly to~(\ref{deducethattruncationnHS})
 we know that
 $$
 \begin{aligned}
   \int_0^t \int_0^{t_1}\dots  \int_0^{t_{k-1}}  {\bf T}_s(t-t_1) {\mathcal C}_{s,s+1}  {\bf T}_{s+1}(t_1-t_2)   \dots    {\bf T}_{s+k}(t_k)  \indc_{E_0(Z_{s+k})\geq R^2} f^{(s+k)}_{N,0} \: dt_{k} \dots dt_1  \\
  \leq C\,  \left( \frac32\right) ^{-k}e^{-{\boldsymbol \beta'}(T)E_0(Z_s)-s \mu_0(T)} \|  G_{N,0,s}\|_{\eps,\beta'_0,\mu_0} \, ,
  \end{aligned}
 $$
  where we have defined
  $$
   G_{N,0,s} := (g^{s+k}_{N,0})_{0 \leq k \leq N-s} \, , \quad \mbox{with} \quad g^{s+k}_{N,0}(Z_{s+k}) :=  \indc_{E_0(Z_{s+k})\geq R^2}\,  f^{(s+k)}_{N,0}  (Z_{s+k}) \, .
  $$
The result then follows from the fact that
$$
\|  G_{N,0,s}\|_{\eps,\beta'_0,\mu_0} \leq C   e^{(\beta_0' -\beta_0)R^2} \| F_{N,0}\|_{\eps,\beta_0,\mu_0}
\, .
$$
The argument is identical for~$I^0_{s,n} (t)  -I_{s,n} ^{0,R}(t)$. 
\end{proof}

\begin{Rem}\label{colltrunccommuteHS}
It is useful to notice that the collision operators preserve the bound on high energies, in the sense that
$$
\begin{aligned}
 {\mathcal C}_{s,s+1}\indc_{E_0(Z_{s+1})\leq R^2 } &\equiv \indc_{E_0(Z_s)\leq R^2}  \,  {\mathcal C}_{s,s+1}\indc_{E_0(Z_{s+1})\leq R^2}  \\
{\mathcal C}^0_{s,s+1}\indc_{E(Z_{s+1})\leq R^2 } &\equiv \indc_{E(Z_s)\leq R^2}  \,  {\mathcal C}^0_{s,s+1}\indc_{E(Z_{s+1})\leq R^2} \, .
\end{aligned}
$$ 
\end{Rem}
  
\section{Time separation} \label{t-clustersHS}
 \setcounter{equation}{0}
We choose another small parameter~$\delta>0$  and  further restrict the study  to the case when $t_i-t_{i+1}\geq \delta$. That is, we define
$$\label{defTnotdeltaHS} 
{\mathcal T}_k(t) := \Big\{ T_k = (t_1,\dots,t_k) \,  /\,   t_i < t_{i-1} \,  \,  \mbox{ with } \,  \, t_{k+1} = 0\,  \,  \mbox{ and } \,  \,t_0 = t \Big\}\, ,
$$
$$\label{defTdeltaHS} 
{\mathcal T}_{k,\delta}(t) := \Big\{ T_k \in {\mathcal T}_k(t)  \,  /\,  t_i - t_{i+1}\geq \delta \Big\}\, ,
$$
and
\begin{equation}\label{defIsnRdeltaHS}
\begin{aligned}
 I^{R,\delta}_{s,k}(t )(X_s)  :=  \int  \varphi_s(V_s) &\int_{{\mathcal T}_{k,\delta}(t) }   {\mathbf  T}_s(t-t_1) {\mathcal C}_{s,s+1} {\mathbf  T}_{s+1}(t_1-t_2) {\mathcal C}_{s+1, s+2}\\
&  \dots {\mathcal C}_{s+k-1 , s+k}  {\mathbf  T}_{s+k}(t_k- t_{k+1}) \indc_{E_{0}(Z_{s+k})\leq R^2}   f^{(s+k)}_{N,0}   dT_kdV_s  \, , \\
 I^{0,R,\delta}_{s,k}(t )(X_s)  :=  \int  \varphi_s(V_s) &\int_{{\mathcal T}_{k,\delta}(t) }   {\bf S}_s(t-t_1) {\mathcal C}^0_{s,s+1}{\bf S}_{s+1}(t_1-t_2) {\mathcal C}^0_{s+1, s+2}\\
&  \dots {\mathcal C}^{0}_{s+k-1 , s+k}  {\bf S}_{s+k}(t_k- t_{k+1}) \indc_{E_0(Z_{s+k})\leq R^2}   f^{(s+k)}_{0}   dT_kdV_s\, .
 \end{aligned}
 \end{equation}
Again applying the continuity bounds for the transport and collision operators,  the error on the functions~$ \displaystyle\sum_{k }( I_{s,k} ^R -I_{s,k} ^{R,\delta} )(t) $ and~$\displaystyle\sum_{k}( I_{s,k} ^{0,R}   -I_{s,k} ^{0,R,\delta}) (t) $
 can be estimated as follows.
 \begin{prop}\label{delta2smallHS}
Let~$s \in \N^*$ and~$t \in [0,T]$ be given. There is a constant~$C$   such that for each~$n$ and~$R$, 
$$
 \big \|\sum_{k=0}^n( I_{s,k} ^R -I_{s,k} ^{R,\delta} )(t) \big\|_{L^\infty (\R^{ds})} \leq Cn^2{ \delta \over T} \|\varphi\|_{L^\infty (\R^{ds})}
\|F_{N,0}\|_{\eps,\beta_0,\mu_0}$$
  and
  $$
 \big \|\sum_{k=0}^n( I_{s,k} ^{0,R}   -I_{s,k} ^{0,R,\delta}) (t) \big\|_{L^\infty (\R^{ds})} \leq C \delta n^2{ \delta \over T} \|\varphi\|_{L^\infty (\R^{ds})}
\| F_{0}\|_{0,\beta_0,\mu_0} \, .$$
\end{prop}
\section{Reformulation in terms of pseudo-trajectories} \label{conclusionHS}
\setcounter{equation}{0}
Putting together Propositions~\ref{domcv}, \ref{delta1small} and~\ref{delta2smallHS}
 we obtain the following result. 
  \begin{Cor}\label{finalresultHS}
With the notation of Theorem~{\rm\ref{existence-thm}}, given~$s \in \N^*$ and~$t \in [0,T]$, there are two positive constants~$C$ and~$C'$ such that  for each~$n \in \N^*$,
$$
 \big  \| I_{s} (t)  - \sum_{k=0}^n I_{s,k} ^{R,\delta} (t) \big \|_{L^\infty (\R^{ds})} \leq C( 2^{-n} + e^{-C'\beta_0 R^2} +n^2{ \delta \over T} \delta ) \|\varphi\|_{L^\infty (\R^{ds})}
  \|  F_{N,0}\|_{\eps,\beta_0,\mu_0} \, .$$
\end{Cor}

 In the same way as in~(\ref{css+1+HS}) we now decompose the Boltzmann collision operators~(\ref{Boltzmann-ophardspheres}) into
 $$
 \displaystyle
  {\mathcal C}^0_{s,s+1} =   {\mathcal C}^{0,+}_{s,s+1} - {\mathcal C}^{0,-}_{s,s+1} \, ,
$$
where the index $+$ corresponding to post-collisional configurations and the index $-$ to pre-collisional configurations.
By definition of the collision cross-section for hard spheres, we have
$$
 \begin{aligned}\big({\mathcal C}^{0,-,m}_{s,s+1}   f^{(s+1)}  \big)(Z_s)&:=    \int_{ {\bf S}_{1}^{d-1} \times \R^d}  b(v_{s+1}-v_m , \omega )   f^{(s+1)}  (Z_s, x_m, v_{s+1}) \, d\omega  dv_{s+1} \\
 &= \int_{ {\bf S}_{1}^{d-1} \times \R^d}  ((v_{s+1}-v_m)\cdot \omega)_-  f^{(s+1)}  (Z_s, x_m, v_{s+1}) \, d \omega dv_{s+1} \quad \mbox{and}
\\
\big({\mathcal C}^{0,+,m}_{s,s+1}   f^{(s+1)}  \big)(Z_s)& := \int_{ {\bf S}_{1}^{d-1} \times \R^d} \! \! \! \! \! 
b(v_{s+1}-v_m , \omega )   f^{(s+1)}  ( z_1, \dots ,x_m, v^*_m,\dots ,z_s, x_m, v^*_{s+1}) \, d\omega  dv_{s+1} \\
& =  \int_{ {\bf S}_{1}^{d-1} \times \R^d}   ((v_{s+1}-v_m)\cdot \omega)_+   f^{(s+1)}  ( z_1, \dots ,x_m, v^*_m,\dots ,z_s, x_m, v^*_{s+1}) \,d \omega dv_{s+1}\, .
\end{aligned}
$$

\medskip
The elementary BBGKY and Boltzmann observables we are interested in can therefore be decomposed as 
 \begin{equation}\label{formulafEHS}
  \begin{aligned}
  I_{s,k}^{R,\delta}(t)(X_s)  &=   \sum_{J,M} \Big( \prod_{i=1}^k j_i\Big) I_{s,k} ^{R,\delta} (t,J,M)(X_s) \quad \mbox{and} \\
  I_{s,k}^{0, R,\delta}(t)(X_s) &=  \sum_{J,M}  I_{s,k} ^{0,R,\delta}  (t,J,M)(X_s)
    \end{aligned}
\end{equation} 
  where the {\it elementary functionals}  $I_{s,k}^{R,\delta} (t,J,M) $  are defined by
   \begin{equation} \label{elementaryHS} 
   \begin{aligned}
   I_{s,k} ^{R,\delta}  (t,J,M)(X_s) := \int  \varphi_s(V_s) \int_{{\mathcal T}_{k,\delta}(t) }   &{\mathbf  T}_s(t-t_1){\mathcal C}_{s,s+1}^{j_1,m_1} {\mathbf  T}_{s+1}(t_1-t_2) {\mathcal C}^{j_2,m_2}_{s+1, s+2}\\
   &\dots  {\mathbf  T}_{s+k}(t_k- t_{k+1})\indc_{E_0(Z_{s+k})\leq R^2}   f^{(s+k)}_{N,0}   dT_k dV_s\, ,
  \\   I^{0, R,\delta}_{s,k}   (t,J,M)(X_s) := \int  \varphi_s(V_s) \int_{{\mathcal T}_{k,\delta}(t) }   &{\bf S}_s(t-t_1) {\mathcal C}_{s,s+1}^{0,j_1,m_1} {\bf S}_{s+1}(t_1-t_2) {\mathcal C}^{0,j_2,m_2}_{s+1, s+2}\\
   &\dots{\bf S}_{s+k}(t_k- t_{k+1})  \indc_{E_0(Z_{s+k})\leq R^2} f^{(s+k)}_{0}   dT_k dV_s\, ,
   \end{aligned}
    \end{equation}
 with
  $$
 J := (j_1,\dots,j_k) \in \{+,-\}^k \,  \,  \mbox{ and } \,  \, M := (m_1,\dots,m_k)  \,  \,  \mbox{ with } \,  \,m_i \in \{1,\dots , s+i-1\}\, .
$$

\bigskip
\bigskip

Each one of the  functionals $I_{s,k}^{R,\delta} (t,J,M)$ and $I_{s,k}^{0,R,\delta} (t,J,M)$ defined in~(\ref{elementaryHS}) can be viewed as the observable associated with some dynamics, which of course is not the actual dynamics in physical space since
\begin{itemize}
\item the total number of particles is not conserved;
\item
 the distribution does even not remain nonnegative because of the sign of loss collision operators.
 \end{itemize}
 This explains the terminology of ``pseudo-trajectories" we choose to describe the process.
 
In this   formulation, the characteristics associated with the operators~$ {\mathbf  T}_{s+i} (t_{i}-t_{i+1}) $ and~$ {\bf S}_{s+i} (t_{i}-t_{i+1}) $ are followed {\it backwards} in time between two consecutive times~$t_{i}$ and~$t_{i+1}$, and collision terms (associated with~${\mathcal C}_{s+i,s+i+1}$ and~${\mathcal C}_{s+i,s+i+1}^0$) are seen as {\it  source terms}, in which, in the words of Lanford~\cite{lanford}, ``additional particles" are ``adjoined" to the marginal.

\bigskip

The main heuristic idea is that for the BBGKY hierarchy, in the time interval~$[t_{i+1},t_{i}]$ between two collisions~${\mathcal C}_{s+i-1,s+i }$ and~${\mathcal C}_{s+i,s+i+1}$, the particles should not interact in general so trajectories should correspond to the free flow~${\bf S}_{s+i}$. On the other hand at a collision time~$t_i$, the velocities of the two particles in interaction are liable to change. This is depicted
in Figure~\ref{schema}.

\begin{figure}[h] 
   \centering
   \includegraphics[width=4.2in]{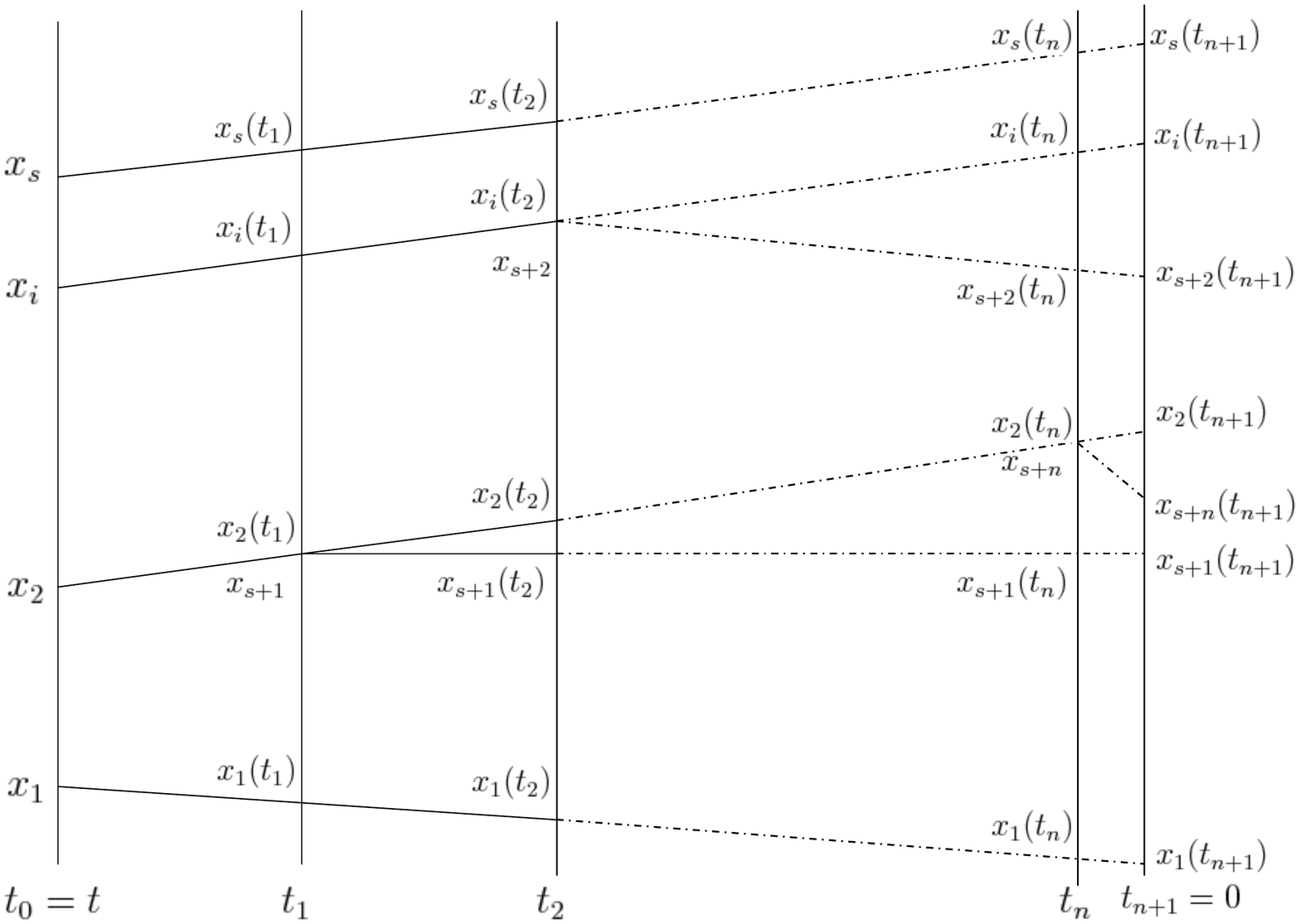} 
\caption{\label{schema}Pseudo-trajectories}
 \end{figure}

 At this stage however, we still cannot study directly the convergence of~$I_{s,k} ^{R,\delta} (t,J,M) -I_{s,k} ^{0,R,\delta} (t,J,M) $ since the transport operators ${\mathbf  T}_k$  do not coincide everywhere with the free transport operators ${\bf S}_k$, which means -- in terms of pseudo-trajectories --  that there are  recollisions.   
 We shall thus prove that these recollisions arise only for a few pathological pseudo-trajectories, which can be eliminated by additional truncations of the domains of integration. This is the goal of Part IV.

%% file: scattering.tex
\chapter{Two-particle interactions}\label{scattering}
\setcounter{equation}{0}


In the case when the microscopic interaction between particles is governed by a short-range repulsive potential, collisions are no more instantaneous and pointwise, and they possibly involve more than two particles. Our analysis in Chapter~\ref{convergence}  shows however that the low density limit $N\eps^{d-1} \equiv 0$ requires only a description of two-particle interactions, at the exclusion of more complicated interactions.

In  this chapter we  therefore study precisely, following the lines of~\cite{cercignani}, the Hamiltonian system \eqref{n1.1} 
 for~$N=2$.
 The study of the reduced motion is carried out in Section~\ref{reducedmotion}, while the  scattering map is introduced in Section~\ref{sec:scattering-map}, and the cross-section, which will play in important role in the Boltzmann hierarchy, is described in Section~\ref{scatteringcrosssection}.
  
 \section{Reduced motion}\label{reducedmotion}
\setcounter{equation}{0}
 We first define a notion of pre- and post-collisional particles, by analogy with the dynamics of hard spheres.
\begin{Def} \label{def:pre-post-col} Two particles~$z_1,z_2$ are said to be {\it pre-collisional} if their distance is~$\eps$ and   decreasing:
 $$|x_1-x_2|=\eps \, , \qquad (v_1 - v_2) \cdot (x_1 - x_2) <0 \, .$$
 Two particles~$z_1,z_2$ are said to be {\it post-collisional} if their distance is~$\eps$ and   increasing:
 $$|x_1-x_2|=\eps \, , \qquad (v_1 - v_2) \cdot (x_1 - x_2) >0 \, .$$
\end{Def}
 We consider here only two-particle systems, and show in Lemma \ref{reduced-lem} that, if $z_1$ and $z_2$ are pre-collisional at time $t_-,$ then there exists a {post-collisional} configuration~$z'_1,z'_2$, attained at $t_+ > t_-.$
 Since $\nabla \Phi (x/\eps)$ vanishes on $\{ |x| \geq \e \},$ the particles $z_1$ and $z_2$ travel at constant velocities~$v'_1$ and~$v'_2$ for ulterior ($t > t_+$) times. 

 
  Momentarily changing back the macroscopic scales of \eqref{n1.1} to the microscopic scales of \eqref{n1} by defining~$\tau :=(t-t_-)/\eps$ and $y(\t) := x(\t)/\eps,$ $w(\t) = v(\t),$ we find that the two-particle dynamics is governed by the equations
 \begin{equation} \label{2part}\left\{
\begin{aligned}
\displaystyle \frac  { d y_1 }{d \tau } & = \displaystyle w_1 \, , \quad  \frac  { d y_2 }{d \tau} = \displaystyle w_2\,,  \\
\displaystyle  \frac  { d w_1 }{d \tau} & = \displaystyle - \nabla \Phi \left( y_1 - y_2 \right) =- \frac  { d w_2 }{d \tau}\,,
\end{aligned}
\right.
\end{equation}
whence the conservations
 \begin{equation} \label{conservation-laws}
\displaystyle  \frac  { d  }{d \tau}(w_1+w_2)= 0\,,\qquad
\displaystyle \frac  { d  }{d \tau} \left( \frac14 (w_1+w_2)^2+\frac14(w_1-w_2)^2 + \Phi(y_1-y_2)\right) =0\,.
\end{equation}

 From \eqref{conservation-laws} we also  deduce
 that the center of mass has a uniform, rectilinear motion:
  \begin{equation} \label{rectiunif}
  ( y_1 +y_2)(\tau) =( y_1 +y_2)(0) + \tau(  w_1+w_2)\,,   \end{equation}
   and 
  that pre- and post-collisional velocities are linked by the classical relations  
\begin{equation}\label{classical}
w'_1+w'_2=w_1+w_2,\quad |w'_1|^2+|w'_2|^2=|w_1|^2+|w_2|^2\,.
\end{equation}

A consequence of \eqref{2part} is that $(\delta y, \delta w) := (y_1 - y_2, w_1 - w_2)$ solves
\begin{equation}
\label{reduced}
\displaystyle  \frac  { d  }{d \tau}\delta y= \delta w\,,\qquad
\displaystyle \frac  { d  }{d \tau} \delta w =-2 \nabla \Phi (\delta y)\,.
\end{equation}
In the following we denote by~$\phi_t: \R^{2d} \to \R^{2d}$ the flow of \eqref{reduced}.

 We notice that, $\Phi$ being radial, there holds
$$ {d\over d\tau} (\delta y \wedge \delta w)=\delta w \wedge \delta w -2 \delta y \wedge \nabla \Phi (\delta y) = 0\,,$$
implying that, if the initial angular momentum $\delta y_0 \wedge \delta w_0$ is non-zero, then $\delta y$ remains for all times in the hyperplane orthogonal to $\delta y_0\wedge \delta w_0.$ 
In this hyperplane, introducing polar coordinates $(\rho,\varphi) $ in~$ \R_+ \times {\bf S}_1^{1},$ such that
$$
\delta y =\rho e_\rho \quad \mbox{and} \quad \delta w = \dot \rho e_\rho + \rho \dot \varphi e_\varphi
$$
the conservations of energy and angular momentum take the form 
$$\begin{aligned}
\frac12 (\dot \rho^2 +(\rho \dot \varphi)^2 )+ 2\Phi(\rho)  = \frac12 |\delta w_0|^2\, ,\\
\rho^2 |\dot \varphi| = |\delta y_0\wedge \delta w_0|\,,
\end{aligned}
$$
implying $\rho > 0$ for all times, and  
\begin{equation} \label{dot-rho}
 \dot \rho^2 + \Psi(\rho, \cE_0, {\mathcal J}_0) = \cE_0\,, \qquad \Psi := \frac{\cE_0 {\mathcal J}_0^2}{\rho^2} + 4 \Phi(\rho)  \,, 
 \end{equation}
 where we have defined
 \begin{equation} \label{def:impact-parameter}
\cE_0 := |\delta w_0|^2 \quad \mbox{and} \quad {\mathcal J}_0:=| \delta y_0\wedge \delta w_0|/|\delta  w_0| =: \sin \a \, , 
 \end{equation}
 which are respectively (twice) the energy and  the impact parameter, $\pi -\alpha$ being the angle between~$\delta w_0$ and~$\delta y_0$ (notice that~$\alpha \geq \pi/2$ for pre-collisional situations). In the limit case when~$\alpha = 0$, the movement is confined to a line since~$\dot \varphi \equiv 0$.

We consider the sets corresponding to pre- and post-collisional configurations:
\begin{equation} \label{def:hemispheres} {\mathcal S}^\pm := \big\{ (\delta y, \delta w) \in {\bf S}_1^{d-1} \times \R^d \,  \big / \,    \pm \delta y \cdot \delta w > 0 \big\}\,.
\end{equation}   
In polar coordinates  
pre-collisional configurations correspond to~$\rho = 1$ and~$\dot \rho <0$ while post-collisional configurations   correspond to~$\rho = 1$ and~$\dot \rho >0$. 


\begin{Lem}[Description of the reduced motion]\label{reduced-lem}
For the differential equation \eqref{reduced} with pre-collisional datum~$(\delta y_0,\delta w_0) \in {\mathcal S}^-,$ there holds $|\delta y(\t)| \geq \rho_*$ for all $\t \geq 0,$ with the notation
\begin{equation} \label{rho*}
 \rho_* = \rho_*(\cE_0, {\mathcal J}_0) := \max\big\{ \rho \in (0,1) \,  \big / \,  \Psi(\rho, \cE_0, {\mathcal J}_0) = \cE_0 \big\}\, ,
\end{equation}
 and for $\t_*$ defined by
 \begin{equation} \label{tau-star} 
 \tau_*:= 2 \int_{\rho_{*}}^1 \left( \cE_0 - \Psi(\rho, \cE_0, {\mathcal J}_0) \right)^{-1/2} d\rho\, , 
 \end{equation}
 the configuration is post-collisional ($\rho = 1,$ $\dot \rho > 0$) at $\t = \t_*.$ 
 \end{Lem}
 
 \begin{proof}  Solutions to \eqref{dot-rho} satisfy $
  \dot \rho = \iota(\rho) \big(\cE_0 - \Psi(\rho)\big)^{1/2},$ with $\iota(\rho) = \pm 1,$ possibly changing values only on $\{ \Psi = \cE_0 \},$ by Darboux's theorem (a derivative function satisfies the intermediate value theorem). The initial configuration being pre-collisional, there holds initially $\iota = -1,$ corresponding to a decreasing radius.  
  The existence of~$\rho_*$ satisfying~(\ref{rho*}) is then easily checked: we have~$|\delta y_0| = 1$ and~$\delta y_0 \cdot \delta w_0 \neq 0,$  so there holds $\Psi(1,\cE_0,{\mathcal J}_0) < \cE_0$, and~$\Psi$ is increasing as~$\rho$ is decreasing.   The set $\{ \t \geq 0, \rho(\t) \geq \rho_*\}$ is closed by continuity. It is also open:
  since~$\Phi$ is nonincreasing, then~$\d_\rho \Psi \neq 0$ everywhere and in particular at~$(\rho_*,\cE_0,{\mathcal J}_0)$. So~$\cE_0 - \Psi$ changes sign at~$\rho_*,$ which forces, by \eqref{dot-rho}, the sign function $\iota$ to jump from $-$ to~$+$ as~$\rho$ reaches the value~$\rho_*$ from above. This proves $\rho \geq \rho_*$ by connexity. The minimal radius~$\rho = \rho_*$ is attained at $\t_*/2,$ where~$\t_*$ is defined by \eqref{tau-star}, the integral being finite since~$\d_\rho \Psi$ does not vanish. Assume finally that for all~$\t > 0,$ there holds~$\rho(\t) < 1.$ Then on~$[\t_*/2,+\infty),$~$\rho$ is increasing and bounded, hence converges to a limit radius, which contradicts the definition of~$\rho_*.$ This proves $\rho = 1$ at~$\t = \t_{*},$ a time at which~$\dot \rho > 0,$ since $\iota$ has jumped exactly once, by definition of $\rho_*.$  
     \end{proof}
 
\begin{figure}[h]
\begin{center}
\scalebox{0.3}{\includegraphics{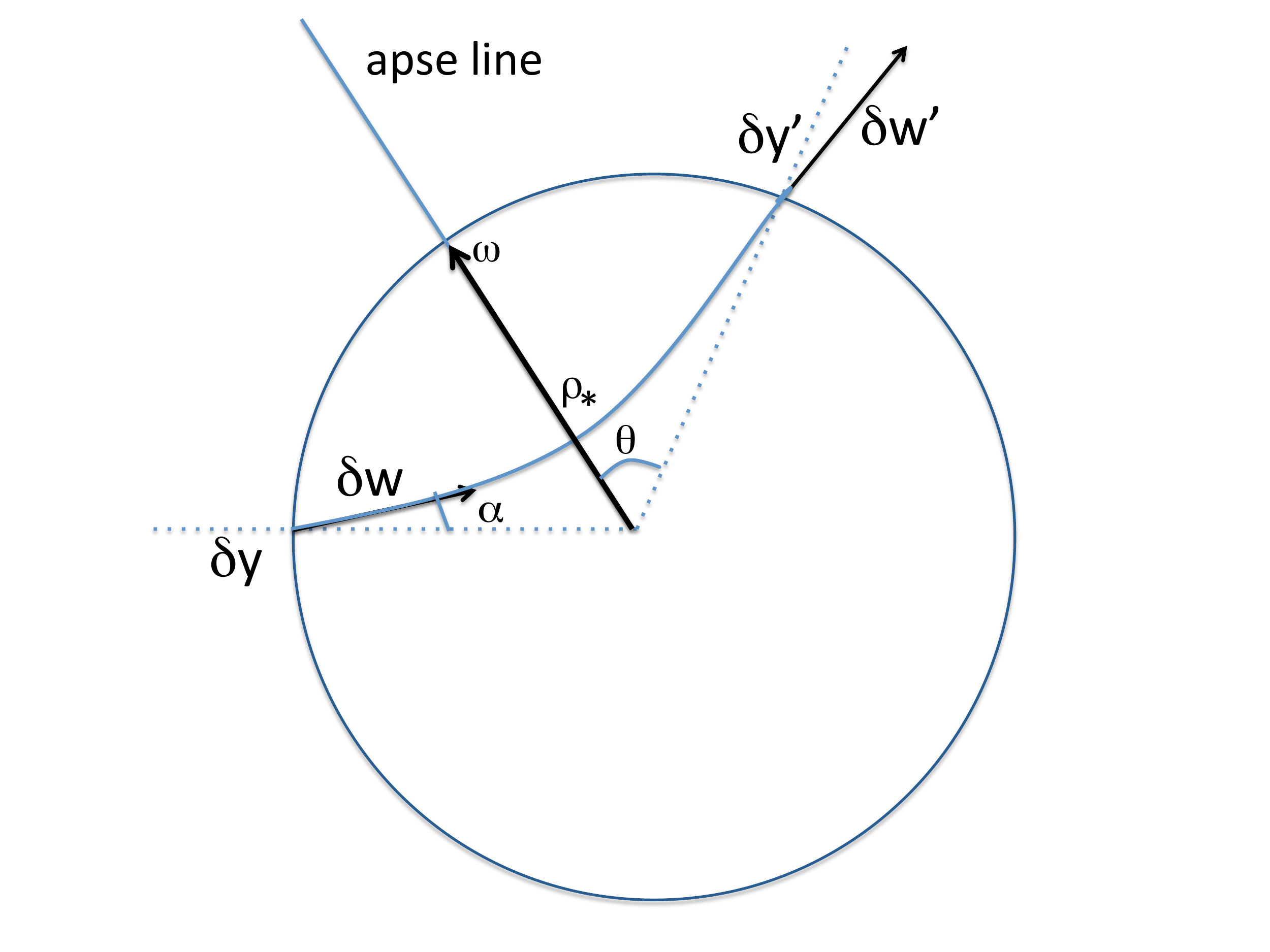}}
\caption{\label{reduced-figure}Reduced dynamics}
\end{center}
\end{figure}


 The reduced dynamics is pictured on Figure \ref{reduced-figure}, where the half-deflection angle $\theta$ is the integral of the angle $\varphi$ as a function of $\rho$ over $[\rho_*,1]:$ 
\begin{equation} \label{theta}
 \theta =\int_{\rho_{*}}^1 \frac{\cE_0^{1/2} {\mathcal I}_0}{\rho^2} \left(\cE_0 - \Psi(\rho,\cE_0,{\mathcal I}_0)\right)^{-1/2} 
d\rho\, ,\end{equation}

With the initialization choice $\varphi_0 = 0,$ the post-collisional configuration is $(\rho, \varphi)(\t_*) = (1, 2 \theta);$ it can be deduced from the pre-collisional configuration by symmetry with respect to the apse line, which by definition is the line through the origin and the point of closest approach~$(\delta y(\t_*/2), \delta w(\t_*/2)).$  The direction of this line is denoted~$\omega \in {\bf S}_1^{d-1}.$\label{indexdefomega}


 \section{Scattering map} \label{sec:scattering-map}
 \setcounter{equation}{0}

We shall now define a microscopic scattering map  that sends pre- to post-collisional configurations:
$$    (\delta y_0, \delta w_0)  \in {\mathcal S}^- \mapsto (\delta y(\t_*), \delta w(\t_*)) = \phi_{\t_*}(\delta y_0,\delta w_0) \in {\mathcal S}^+\, .$$
  By uniqueness of the trajectory of \eqref{reduced} issued from $(\delta y_0, \delta w_0)$ (a consequence of the regularity assumption on the potential, via the Cauchy-Lipschitz theorem), the scattering is one-to-one. It is also clearly onto.

Back in the macroscopic variables, we now define a corresponding scattering operator for the two-particle dynamics. In this view, we introduce the sets
$$ \begin{aligned} {\mathcal S}^\pm_\e := \Big\{ (z_1,&\, z_2) \in \R^{4d} \, \big /\,   |x_1 - x_2 | = \e \, ,   \:  \pm (x_1 - x_2) \cdot (v_1 - v_2) > 0  \Big\}\, .
\end{aligned}$$
We define, as in~(\ref{def:impact-parameter}),
\begin{equation} \label{def:en-impact}
\cE_0 = |v_1 - v_2|^2 \quad \mbox{and} \quad {\mathcal J}_0~:= \frac{|(x_1 - x_2) \wedge (v_1 - v_2)|}{\varepsilon |v_1 - v_2|} =: \sin \a \, .
\end{equation}
 \begin{Def}[Scattering operator] \label{scatteringbbgky}
 The {\em scattering operator} is defined as
 $$ {\sigma}_\e: 
(x_1,v_1,x_2,v_2)  \in {\mathcal S}^-_\e \mapsto (x'_1,v'_1,x'_2,v'_2) \in {\mathcal S}^+_\e\, ,$$
 where
\begin{equation}
\label{scattering-def1}
\begin{aligned}
x'_1&:= \frac12 (x_1+x_2) +\frac{\eps \tau_*}2 (v_1+v_2)+\frac\eps2 \delta y(\tau_*)= -x_1 +  \omega\cdot (x_1-x_2) \omega+\frac{\eps \tau_*}2 (v_1+v_2)\, , \\ 
x'_2& := \frac12 (x_1+x_2) +\frac{\eps \tau_*}2 (v_1+v_2)-\frac\eps2 \delta y(\tau_*) = -x_2  - \omega\cdot (x_1-x_2) \omega+\frac{\eps \tau_*}2 (v_1+v_2)\, , \\
v'_1 & :=\frac12(v_1+v_2) + \frac12 \delta w(\tau_*) =v_1 -\omega\cdot (v_1-v_2) \: \omega\, , \\
 v'_2&:=\frac12(v_1+v_2) - \frac12 \delta w(\tau_*) = v_2+\omega\cdot (v_1-v_2) \: \omega \, ,
\end{aligned}
\end{equation}
where $\t_*$ is the microscopic interaction time, as defined in Lemma {\rm \ref{reduced-lem}}, $(\delta y(\t_*), \delta w(\t_*))$ is the microscopic post-collisional configuration: $(\delta y(\t_*), \delta w(\t_*)) = \phi_{\t_*}( (x_1 - x_2)/\e, v_1 -v_2),$ and $\omega$ is the direction of the apse line. 
Denoting by~$\nu:=(x_1-x_2)/|x_1-x_2|$ we also define
$$ \label{scatteringsigma0}
\sigma_0 (\nu,v_1,v_2):=(\nu',v'_1,v'_2) \, .
$$
\end{Def}

The above description of $(x'_1,v'_1)$ and $(x'_2,v'_2)$ in terms of $\omega$ is deduced from the identities 
 $$
 \delta v(\tau_*) =  \delta v_0  -2 \omega\cdot   \delta v_0\:\omega   \quad \mbox{and} \quad   \delta y(\tau_*)
 = - \delta y_0 + 2 \omega\cdot   \delta y_0\: \omega
  $$
  in the reduced microscopic coordinates.

By $\d_\rho \Psi \neq 0$ in $(0,1)$ and the implicit function theorem, the map $(\cE,{\mathcal J}) \to \rho_*(\cE,{\mathcal J})$ is $C^2$ just like $\Psi.$ Similarly, $\t_* \in C^2.$ By Definition \ref{scatteringbbgky} and $C^1$ regularity of $\nabla \Phi$ (Assumption \ref{propertiesphi}), this implies that the scattering operator $\sigma_\e$ is $C^1,$ just like the flow map~$\phi$ of the two-particle scattering. The scattering $\sigma_\e$ is also bijective, for the same reason that the microsopic scattering  is bijective.

\begin{prop}\label{scatteringestimates}
Let~$R > 0$ be given and
consider 
$$
  {\mathcal S}^\pm_{\e,R}    := \Big\{ (z_1,z_2) \in \R^{4d}\ \,  \big / \,    |x_1-x_2| = \eps \, , \,  |(v_1,v_2)| = R \, , \,    \pm \:  (v_1 - v_2) \cdot (x_1-x_2) >0 \Big\}\, . 
$$
The scattering operator $\sigma_\eps$ is a bijection from~$  {\mathcal S}^-_{\e,R}   $ to $  {\mathcal S}^+_{\e,R}   $. 

The macroscopic time of interaction~$ \e \t_*,$ where $\t_*$ is defined in \eqref{tau-star}, is uniformly bounded  on compact sets of~$\R^+\setminus\{0\} \times [0,1]$, as a function of~${\mathcal E}_0$ and~$\cJ_0$.
  \end{prop}
\begin{proof} 
We already know that $\sigma_\e$ is a bijection from~$  {\mathcal S}^- _\e $ to ~$  {\mathcal S}^+_\e.$ By \eqref{classical}, it also preserves the velocity bound. Hence $\sigma_\e$ is bijective $  {\mathcal S}^-_{\e,R} \to {\mathcal S}^+_{\e,R}.$ 
Now given $\cE_0 > 0$ and~$\cJ_0 \in [0,1],$ we shall show that~$\tau_* $ can be bounded by a constant depending only on~$\cE_0 $. Since~$\Phi(\rho_* ) \leq \cE_0/4$, then~$\rho_*  \geq \Phi^{-1} ( \cE_0/4) 
$. Let us then define~$i_0\in (0,1)$ by
$$
i_0:=\frac1{2 \sqrt 2} \Phi^{-1}\big(\frac{ \cE_0} 4\big)\,  ,
$$ 
so that
$
\rho_*^2 \geq 8 i_0  ^2.
$

On the one hand it is easy to see, after a change of variable in the integral, using
$$
\frac d{d\rho}( \cE_0 - \Psi(\cE_0,{\mathcal J}_0,\rho) ) = \frac{2 \cE_0 {\mathcal J}_0^2}{\rho^3} - 4 \Phi'(\rho) \geq  \frac{2 \cE_0 {\mathcal J}_0^2}{\rho^3}\geq2 \cE_0 {\mathcal J}_0^2 \, ,
$$
 that there holds the bound 
$$
\begin{aligned}
\tau_*  &\leq \frac1{  \cE_0{\mathcal J}_0^2}\int_0^{\cE_0(1-{\mathcal J}_0^2)} \frac{dy}{\sqrt y}  \leq\frac { 2 \sqrt{1-{\mathcal J}_0^2}}{ {\mathcal J}_0^2 \sqrt{\cE_0}} \,\cdotp
\end{aligned}
$$
$-$ So if~${\mathcal J}_0 \geq i_0$, we   find that
$$
\tau_* \leq  \frac2 { \sqrt{\cE_0}i_0^2}  =   \frac{16}{ \sqrt{\cE_0} \big(\Phi^{-1}\big(\frac{ \cE_0} 4\big)\big) ^2} \, \cdotp
$$
$-$ On the other hand for~${\mathcal J}_0 \leq i_0$  we define~$\gamma :=   \Phi^{-1}  ({ \cE_0} / {8} )$ and we cut the integral defining~$\tau_*$ into two parts:
$$
\tau_*= \tau_*^{(1)} + \tau_*^{(2)}  \quad \mbox{with} \quad  \tau_*^{(1)} =2\int_{\rho_{*}}^{\gamma}  \left( \cE_0 - \Psi(\cE_0,{\mathcal J}_0,\rho) \right)^{-1/2}  d\rho \, .
$$
Notice that since~$\rho_*^2 \geq 8 i_0  ^2  $ and~${\mathcal J}_0 \leq i_0$, then~$  {\cE_0} / 4 -  { \cE_0  \cJ_0 ^2}/ {4\rho_*^2} \geq {7\cE_0} /  {32} \geq {\cE_0} /  {8} $ so
$$ 
\rho_* = \Phi^{-1} \Big ( 
\frac {\cE_0} 4 - \frac{ \cE_0  \cJ_0 ^2}{4\rho_*^2}
\Big) \leq   \Phi^{-1}  \Big ( \frac { \cE_0} {8} \Big) = \gamma \, .
$$

The first integral~$ \tau_*^{(1)}$ is estimated using the fact that~$\Phi'$ does not vanish outside~$1$ as stated in Assumption~\ref{propertiesphi}: defining
$$
M(\Phi):= \inf_{   i_0 \leq \rho \leq \gamma } |\Phi'(\rho)|>0 \, ,
$$ 
we find that on~$[i_0,\gamma]$,
$$
\frac d{d\rho}( \cE_0 - \Psi(\cE_0,{\mathcal J}_0,\rho) ) = \frac{2 \cE_0 {\mathcal J}_0^2}{\rho^3} - 4 \Phi'(\rho) \geq    4 M(\Phi)
$$
so
$$
 \tau_*^{(1)} \leq  \frac{ \big( \cE_0/2 - \cE_0{\mathcal J}_0^2/ \gamma^2 \big)^\frac12}{ M(\Phi)} \leq \frac{\sqrt{\cE_0}}{\sqrt 2 M(\Phi)} \,  \cdotp
$$
For the second integral we estimate simply 
$$
\tau_*^{(2)}  \leq \frac2 {\big(  \cE_0/2- \cE_0{\mathcal J}_0^2/ \gamma^2\big)^\frac12} \leq 
\frac2{\big(  \cE_0/2- \cE_0/8  \big)^\frac12} = \frac{4 \sqrt 2}{\sqrt{ 3\cE_0} }   \, \cdotp
$$
The result follows. \end{proof}
\begin{Rem} If~$\Phi$ is of the type~$\displaystyle \frac1{\rho^s} \exp(-\frac 1{1-\rho^2})$ 
then the proof of Proposition~{\rm\ref{scatteringestimates}} shows that~$\tau_*$   may be bounded from above by a constant of the order of~$C /\sqrt{e_0} (1+\log e_0)$ if~$\cE_0 \geq e_0$.
\end{Rem}

  \section{Scattering cross-section and the Boltzmann collision operator}\label{scatteringcrosssection}
  \setcounter{equation}{0}
The scattering operator in Definition~\ref{scatteringbbgky} is parametrized  by the impact parameter and the two ingoing (or outgoing) velocities. However in the Boltzmann limit the impact parameter cannot be observed:  the  observed quantity is the {\it deflection angle} or {\it scattering angle},   defined as the angle between ingoing and outgoing relative velocities. 
The next paragraph defines that angle   as well as the scattering cross-section, and 
the following paragraph defines the Boltzmann collision operators using that formulation.

\subsection{Scattering cross-section}
 With notation from the previous paragraphs, the deflection angle is equal to $\pi - 2 \Theta$ where $\Theta := \alpha + \theta,$ the angle $\alpha$ being defined in~\eqref{def:en-impact} and $\theta$ being defined in~\eqref{theta}, so that  $$
  \Theta = \Theta(\cE_0,\cJ_0) :=     \arcsin  \cJ_0 +  \cJ_0 \int_{\rho_*}^1 \frac{d\rho}
  {
  \sqrt{1-\frac{4 \Phi (\rho)} {\mathcal E_0} 
  - \frac{ \cJ_0^2}{\rho^2}} } \, \cdotp
  $$
The following result, and its proof, are due to~\cite{pulvirentiss}: 

\begin{Lem}\label{assumptiontheta}
Under Assumption {\rm \ref{propertiesphi}}, assume moreover that  for all~$\rho \in (0,1)$,
\begin{equation}\label{strangeassumption}
\rho \Phi''(\rho) + 2 \Phi'(\rho) \geq 0 \, .
\end{equation}
Then for all~$\cE_0>0$, the function~$ \cJ_0   \mapsto \Theta(\cE_0, \cJ_0) \in [0,\pi/2]$ satisfies
$\Theta(\cE_0,0) = 0$ and is strictly monotonic: $\partial_{\cJ_0} \Theta>0$ for all~$\cJ_0 \in (0,1)$. Moreover, it satisfies
$$
\displaystyle\lim_{\cJ_0 \to 0} \,  \partial_{\cJ_0} \Theta \in (0,\infty]\quad \mbox{and} \quad \displaystyle\lim_{\cJ_0 \to 1} \, \partial_{\cJ_0} \Theta = 0 \, .$$
\end{Lem}

\begin{proof} An energy $\cE_0 > 0$ being fixed, the limiting values~$\Theta (\cE_0,0) = 0$ and~$\Theta (\cE_0,1) =   \pi/2$ are found by direct computation. To prove monotonicity, the main idea of~\cite{pulvirentiss} is to use the change of variable  
$$
\sin^2 \varphi:=\frac{4 \Phi (\rho)}{\cE_0} + \frac {\cJ_0^2}{\rho^2} 
$$
which yields
$$
\Theta(\cE_0,\cJ_0) =    \arcsin  \cJ_0 + \displaystyle \int_{   \arcsin  \cJ_0}^\frac\pi2 \frac{ \sin \varphi }{  \frac {\cJ_0}{\rho} - \frac{2 \rho  \Phi'(\rho)}{\cE_0\cJ_0} }\, d\varphi \,.
$$
Computing the derivative of this expression with respect to~$\cJ_0$ gives
$$
\begin{aligned}
\frac{\partial\Theta}{\partial \cJ_0} (\cE_0,\cJ_0) & =  
\frac1{\sqrt{1-\cJ_0^2}} \left(1-\frac{\cE_0\cJ_0^2 }{\cE_0\cJ_0^2 - \Phi'(1)} \right) \\
& \,  + \int_{\arcsin \cJ_0}^{\frac\pi{2}}  \frac {\cE_0^2\cJ_0^2 \rho^4\sin \varphi}{ (\cJ_0^2  \cE_0 - \rho^3 \Phi' (\rho))^3}
\left(
\rho\Phi'' (\rho) +2 \Phi' (\rho) + \frac{\rho^3}{\cE_0 \cJ_0^2}( \Phi' (\rho))^2 
\right)\, d\varphi
\end{aligned}
$$
where~$\varphi$ is defined by
$$
\sin^2 \varphi = \frac{\cJ_0^2}{\rho^2}+ \frac{2 \Phi (\rho)}{\cE_0}  \, \cdotp
$$
In view of the formula giving~$\partial_{\cJ_0} \Theta$, it  turns out assumption \eqref{strangeassumption} implies~$\partial_{\cJ_0} \Theta>0$ for all~$\cJ_0 \in (0,1),$ and also the limits
$$
\displaystyle \lim_{\cJ_0 \to 0} \,  \partial_{\cJ_0} \Theta \in (0,\infty]\quad \mbox{and}
\displaystyle\lim_{\cJ_0 \to 1} \, \partial_{\cJ_0} \Theta = 0
$$
as soon as~$\Phi'(1 )= 0$ (if not then~$\displaystyle\lim_{\cJ_0 \to 1} \partial_{\cJ_0} \Theta = \infty$).
The result follows.
\end{proof}
\begin{Rem}
Note that one can construct examples that violate assumption~\eqref{strangeassumption} and for which monotonicity fails, regardless of convexity properties of the potential $\Phi$ {\rm(\cite{pulvirentiss})}.
\end{Rem}
By Lemma~\ref{assumptiontheta},  for each~$\cE_0$ we can locally invert the map $\Theta(\cE_0,\cdot),$ and thus define~$ \cJ_0$ as a smooth function of~$\cE_0$ and~$\Theta.$ This enables us to  define a scattering cross-section (or {\it collision kernel}), as follows. 

\begin{figure}[h]
\begin{center}
\scalebox{0.5}{\includegraphics{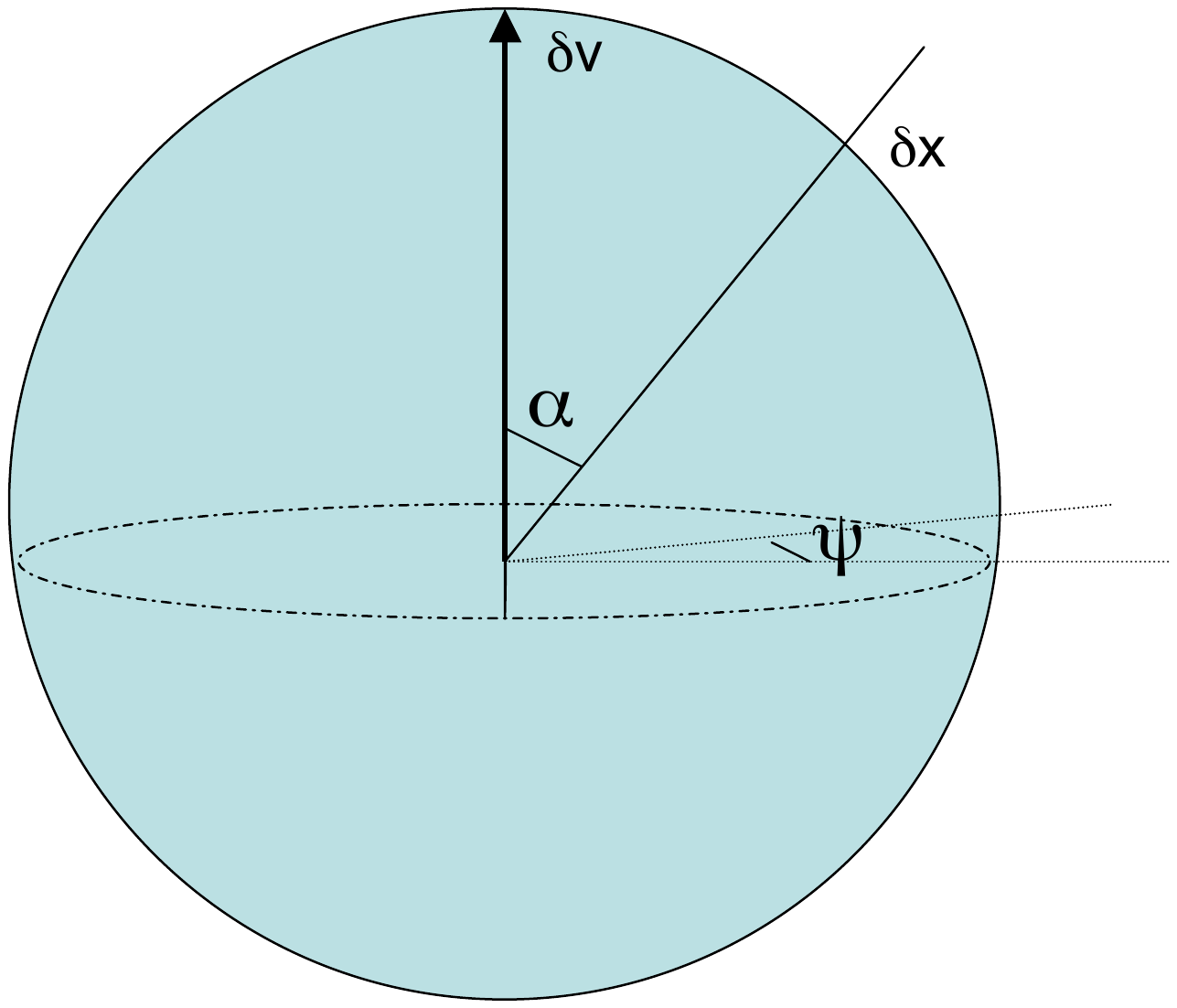}}
\caption{Spherical coordinates\label{coordinates}}
\end{center}
\end{figure}
%
%

For fixed $x_1,$ we denote $d\sigma_1$ \label{def:d-sigma1} the surface measure on the sphere $\{y \in \R^d, |y - x_1| = \e\},$ to which~$x_2$ belongs. We can parametrize the sphere by $(\a,\psi),$ with $\psi \in {\bf S}_1^{d-2},$ where $\a$ is the angle defined in~\eqref{def:en-impact}. There holds
 $$ d\sigma_1 =  \e^{d-1} (\sin \a)^{d-2} d\a d\psi  \, .$$
The direction of the apse line is $\omega = (\Theta,\psi),$ so that, denoting $d\omega$ the surface measure on the unit sphere, there holds 
\begin{equation} \label{domega} d\o = (\sin \Theta)^{d-2} d\Theta d\psi\, . 
\end{equation}
By definition of $\a$ in \eqref{def:en-impact}, there holds
 $$ (x_1-x_2) \cdot (v_1-v_2) =\eps |v_1-v_2|\cos \alpha\, ,$$
 so that
$$\begin{aligned}
\frac1\e \,  (x_1-x_2) \cdot (v_1-v_2) \: d\sigma_1 &=  \eps ^{d-1}|v_1-v_2|\cos \alpha \:  (\sin \alpha )^ {d-2}\: d\alpha d\psi  \\ &= \eps^{d-1} |v_1-v_2| \: \cJ_0^{d-2}d\cJ_0 d\psi\, ,
 \end{aligned}
$$
where in the second equality we used the definition of $\cJ_0$ in \eqref{def:en-impact}. 
This gives
\begin{equation} \label{crosssectionlatlong}
\frac1\e \, (x_1-x_2) \cdot (v_1-v_2) \: d\sigma_1=\eps^{d-1} |v_1-v_2| \cJ_0^{d-2}\d_\Theta \cJ_0 \: d\Theta d\psi\, ,
 \end{equation}
wherever $\d_\Theta \cJ_0$ is defined, that is, according to Lemma \ref{assumptiontheta}, for $\cJ_0 \in [0,1).$ 

\begin{Def} \label{def:b} The scattering cross-section is defined for $|w| > 0$ and $\Theta \in (0,\pi/2]$ by~$\cJ_0^{d-2} \d_\Theta \cJ_0 ( {\sin \Theta})^{2-d}$. In the following we shall use the notation
 \begin{equation}\label{defcross-section}
 b(w, \Theta) :=   |w| \cJ_0^{d-2} \d_\Theta \cJ_0 ( {\sin \Theta})^{2-d}\,  ,
  \end{equation}
and abusing notation we shall write~$b( w, \Theta) = b(w, \omega ).$
 \end{Def}  
By Lemma~{\rm\ref{assumptiontheta}}, the cross-section~$b$ is a locally bounded function of the relative velocities and scattering angle.

\subsection{Scattering cross-section}\label{scatteringcrosssectionpotentialcase}
The relevance of $b$ is made clear in the  derivation of the Boltzmann hierarchy, where we shall   use the identity
 \begin{equation} \label{jac}
  \frac1\e \,  (x_1-x_2) \cdot (v_1-v_2) \: d\sigma_1 = \e^{d-1} b(v_1 - v_2,\o) d\o\,,
 \end{equation}
 derived from \eqref{domega}, \eqref{crosssectionlatlong} and Definition \ref{def:b}. As in Chapter~\ref{spheres} (see in particular Paragraph~\ref{spheresboltz}), we can formally derive the Boltzmann collision operators using this formulation: we thus define
 \begin{equation}
\label{Boltzmann-opscatfacteurdimpactchapter}
\begin{aligned}&  {\mathcal C}^0_{s,s+1}    f^{(s+1)} (t,Z_s) :=       \sum_{ i = 1}^s  \int \indc_{\nu\cdot (v_{s+1}-v_i)>0} \, \nu   \cdot (v_{s+1}-v_i)  \\
&\qquad \times{} \Big(  f^{(s+1)}  (t,x_1,v_1,\dots ,x_i, v^*_i,\dots ,x_s,v_s, x_i, v^*_{s+1})    -     f^{(s+1)}  (t,Z_s, x_i, v_{s+1}) \Big) d\nu dv_{s+1} \, ,
\end{aligned}
\end{equation}
where~$(v^*_i,v^*_{s+1})$ is obtained from~$(v_i,v_{s+1})$ by applying the inverse scattering operator~$\sigma_0^{-1}$:
$$
\sigma_0^{-1} \big(\nu ,v_i,v_{s+1} \big) =  \big( \nu,v^*_i,v^*_{s+1} \big) \, .
$$
This
  can also be written using the cross-section:
 \begin{equation}
\label{Boltzmann-opscatchapter}
\begin{aligned}& {\mathcal C}^0_{s,s+1}    f^{(s+1)} (t,Z_s) :=       \sum_{ i = 1}^s  \int  \,  b(v_1 - v_2,\o)  \\
& \qquad \times{} \Big(  f^{(s+1)}  (t,x_1,v_1,\dots ,x_i, v^*_i,\dots ,x_s,v_s, x_i, v^*_{s+1})    -     f^{(s+1)}  (t,Z_s, x_i, v_{s+1}) \Big) d\omega dv_{s+1} \, .
 \end{aligned}
\end{equation} 
  
 \begin{Rem} It is not possible to define an integrable cross-section if the potential is not compactly supported, no matter how fast it might be decaying. This issue is related to the occurrence of grazing collisions and discussed in particular in \cite{villani}, Chapter~{\rm 1}, Section~{\rm 1.4}. 
 However it is still possible to study the limit towards the Boltzmann equation, if one is ready to change the formulation of the Boltzmann equation by renouncing to the cross-section formulation~{\rm(\cite{pulvirentiss})}.
 
 The question of the convergence to Boltzmann in the case of long-range potentials is a challenging open problem; it was considered by L. Desvillettes and M. Pulvirenti in  \cite{desvillettespulvirenti} and  L. Desvillettes and V. Ricci   in~\cite{desvillettesricci}.
  \end{Rem}

%% file: BBGKY.tex
\chapter{Truncated marginals and the BBGKY hierarchy}
\setcounter{equation}{0}
\label{BBGKY}

Our starting point in this first part is the Liouville equation \eqref{Liouville} satisfied by the $N$-particle distribution function~$f_N.$ We reproduce here equation \eqref{Liouville}: 
\begin{equation} \label{liouville-h}
\partial_t f_N +\sum_{1 \leq i \leq N } v_i \cdot \nabla_{x_i} f_N  -\sum_{1 \leq i \neq j \leq N} \frac1\eps \nabla\Phi \left({x_i - x_j \over \eps}\right)\cdot \nabla_{v_i} f_N  = 0 \, .
\end{equation}
The arguments of $f_N$ in \eqref{liouville-h} are $(t,Z_N) \in \R_+ \times \Omega_N,$ where we recall that
$$
\Omega_N:= \Big \{Z_N \in \R^{2dN} \, , \, \forall i \neq j \, , \,  x_i \neq x_j\Big\} \, .
$$

As recalled in Part II, Chapter~\ref{spheres}, the classical strategy to obtain a kinetic equation  is to write the evolution equation for the first marginal of the distribution function $f_N$\label{indexdefmarginalbbgky}, namely
$$
f_N^{(1)} (t,z_1):= \int_{\R^{2d(N-1)}} f_N(t,z_1,z_2, \dots,z_N) \,  dz_2 \dots dz_N\, ,
$$
which leads to the study of the hierarchy of equations involving all the marginals of $f_N$ 
\begin{equation}
\label{marginal}
 f_N^{(s)} (t,Z_s) :=    \int_{\R^{2d(N-s)}} f_N (t,Z_s,z_{s+1},\dots,z_N)\: dz_{s+1} \cdots dz_N \, .
 \end{equation}
 
 In Section~\ref{truncatedmarginals}
 it is shown that due to the presence of the potential, and contrary to the hard-spheres case described in Paragraph~\ref{liouvillebbgkyHS},
 it is necessary to truncate  those marginals away from the set~$\Omega_N$.    An equation for the {\it truncated marginals} is derived in weak form in Section~\ref{weakliouville}. 
 In order to introduce adequate collision operators, the notion of cluster is introduced and described in Section~\ref{sec:def-closure}, following the work of F. King~\cite{K}. Then
 collision operators are introduced in Section~\ref{collisionliouville}, and finally the integral formulation of the equation is written in Section~\ref{mildliouville}.


 \section{Truncated marginals}\label{truncatedmarginals}
\setcounter{equation}{0}
From \eqref{liouville-h}, we deduce by integration that the {\it untruncated marginals} defined in~(\ref{marginal}) solve
\begin{equation}\label{eq:untruncated-marg}
\begin{aligned}
 \partial_tf_N^{(s)} (t,Z_s)+  \sum_{i=1}^sv_i \cdot \nabla_{x_i}f_N^{(s)}  (t,Z_s)- \frac1{\eps} \sum_{i,j=1 \atop i \neq j}^s  \nabla\Phi \left(\frac{x_i - x_j}\eps \right)\cdot \nabla_{v_i} f_N^{(s)}  (t,Z_s)& \\
= \frac{N-s}\eps \sum_{i=1}^s\int \nabla\Phi \left(\frac{x_i - x_{s+1}}\eps \right)\cdot \nabla_{v_{i}} f_N^{(s+1)}  (t,Z_{s},z_{s+1} )\: dz_{s+1} \, .&
\end{aligned}
\end{equation}
There are several differences between \eqref{eq:untruncated-marg} and the BBGKY hierarchy for hard spheres \eqref{def:collisionop0}-\eqref{BBGKYhierarchyspheres}. One is that the transport operator in the left-hand side of \eqref{eq:untruncated-marg} involves a force term. Another is that the integral term in the right-hand side of \eqref{eq:untruncated-marg} involves velocity derivatives. Also, that integral term  is a linear integral operator acting on higher-order marginals, just like \eqref{def:collisionop0}, but, contrary to~\eqref{def:collisionop0}, is {\it not} spatially localized, in the sense that the integral in~$x_{s+1}$ is over the whole ball $B(x_i,\e),$ as opposed to an integral over a sphere in \eqref{def:collisionop0}.

This leads us to distinguish spatial configurations in which interactions do take place from spatial configurations in which particles are pairwise at a distance greater than $\e,$ by truncating off the interaction domain $\big\{ Z_N, \, \mbox{$|x_i - x_j| \leq \e$ for some $i \neq j$} \big\}$ in the integrals defining the marginals. For the resulting truncated marginals, collision operators will appear as integrals over a piece of the boundary of the interaction domain, just like in the case of hard spheres. The scattering operator of Chapter \ref{scattering} (Section \ref{sec:scattering-map}) will then play the role that the boundary condition plays in the case of hard spheres in Chapter \ref{spheres}.

Suitable quantities to be studied are therefore not the marginals defined in~(\ref{marginal}) but rather the {\it truncated marginals}
 \begin{equation}
 \label{t-marginal0}
 \widetilde  f_N^{(s)} (t,Z_s) :=    \int_{\R^{2d(N-s)}} f_N (t,Z_s,z_{s+1},\dots,z_N) \displaystyle \prod_{ i\in  \{1,\dots, s\} \atop j\in  \{s+1,\dots , N\}} \!\!\!\! \indc_{|x_i - x_j| > \eps}
\: dz_{s+1} \cdots dz_N \, , 
\end{equation}
where~$|\cdot|$ denotes the euclidean norm.
Notice that
 $$
 (\widetilde f_N^{(1)} - f_N^{(1)}) (t,z_1) = \int_{\R^{2d(N-1)}} f_N (t,z_1,z_{2},\dots,z_N)  (1-\prod_{  j\in  \{2,\dots , N\}} \indc_{|x_1 - x_j| > \eps})
\: dz_{2} \cdots dz_N
$$
so that 
\begin{equation}
\label{error-tilde}
 \|( \widetilde f_N^{(1)} - f_N^{(1)} )(t)\| _{L^\infty(\R^{2d})} \leq C (N-1) \eps^d \| f_N^{(2)}(t)\|_{L^\infty(\Omega_2)} \, .
 \end{equation}
We therefore expect both functions to have the same asymptotic behaviour in the Boltzmann-Grad limit~$N \eps^{d-1} \equiv 1$. This is indeed proved in Lemma \ref{lem:trunc-untrunc}.

Given $1 \leq i < j \leq N,$ we   recall that~$dZ_{(i,j)}$ denotes the $2d(j-i+1)$-dimensional Lebesgue measure~$dz_i dz_{i+1} \dots dz_j,$ and~$dX_{(i,j)}$ the~$d(j-i+1)$-dimensional Lebesgue measure~$dx_i dx_{i+1} \dots dx_j.$ \label{page:ij-measureX}We also define
\begin{equation} \label{def:DNS}
{\mathcal D}^s_N:= \Big\{
X_N \in  \R^{dN}, \: \forall (i,j) \in[1,s]\times [s+1,N]\, , \: |x_i - x_j| > \eps 
\Big\}\, ,
\end{equation}
where~$[1,s]$ is short for $[1,s] \cap \N = \{ k \in \N, \, 1 \leq k \leq s\}.$ 
Then the truncated marginals~(\ref{t-marginal0}) may be formulated as follows:
 \begin{equation}
 \label{t-marginal}
 \widetilde  f_N^{(s)} (t,Z_s)  =    \int_{\R^{2d(N-s)}} f_N (t,Z_s,z_{s+1},\dots,z_N) \indc_{X_N \in {\mathcal D_N^s}}  
 \: dZ_{s+1,N }\, . 
\end{equation}

The key in introducing the truncated  marginals \eqref{t-marginal}, following King \cite{K}, is that it allows for a derivation of a hierarchy that is similar to the case of hard spheres. The main drawback is that contrary to the hard-spheres case in~(\ref{nicemarginalsHS}), truncated marginals are not actual {\it marginals}, in the sense that 
\begin{equation} \label{pbhierarchy} \widetilde f_N^{(s)} (Z_s) \neq \int_{\R^{2d}} \indc_{X_{s+1} \in B} \widetilde f_N^{(s+1)} (Z_s, z_{s+1}) \, dz_{s+1} \, ,
\end{equation}
for any $B \subset \R^{d(s+1)},$ in particular if $B = \R^{d(s+1)},$ simply because $\D_N^s$ is {\it not} included in $\D^{s+1}_N.$ Indeed, conditions $|x_j - x_{s+1}| > \e,$ for $j \leq s,$ hold for $X_N \in \D^s_N,$ but not necessarily for $X_N \in \D^{s+1}_N.$ Furthermore, $\D^s_N$ intersects all the $\D^{s+m}_N,$ for $m \in [1,N-s].$ 
A consequence is the existence of higher-order interactions between truncated marginals, as seen below in \eqref{w-BBGKY}. Proposition \ref{propcontinuityclusters} in Chapter~\ref{chaptercluster} states however that these higher-order interactions are negligible in the Boltzmann-Grad limit.

\section{Weak formulation of Liouville's equation}\label{weakliouville}
\setcounter{equation}{0}

 Our goal in this section is to find the weak formulation of the system of equations satisfied 
 by the family of truncated marginals $\big( \widetilde f_N^{(s)}\big)_{s \in [1,N]}$ defined above in \eqref{t-marginal}. The strategy will be similar to that followed in Chapter~\ref{spheres} in the hard-spheres case. From now on we assume that~$  f_N $ decays at infinity in the velocity variable.

 Given a smooth, compactly supported  function~$\phi$   defined on~$\R_+ \times \R^{2ds} $ and satisfying the symmetry assumption~(\ref{sym}),  we have 
\begin{equation}\label{integralformulationf}
\begin{aligned}
\int_{ \R_+ \times  \R^{2dN} } \Bigl(
\partial_t f_N + \sum_{i = 1}^N v_i \cdot \nabla_{x_i} f_N  - \frac1{\eps}\sum_{i = 1}^N  \sum_{j \neq i} \nabla\Phi  \left(\frac{x_i - x_j}\eps \right)\cdot \nabla_{v_i} f_N 
\Bigr)(t,Z_N) \\
{}\times \phi (t,Z_s)  \displaystyle \indc_{X_N \in {\mathcal D_N^s}}  \:  d Z_N dt = 0 \, .
\end{aligned}
\end{equation}
Note that in the above double sum in~$i$ and~$j$, all the terms vanish except when $(i,j) \in[1,s]^2$ and when~$(i,j) \in  [s+1,N]^2,$ by assumption on the support of $\Phi.$

We now use  integrations by parts to derive from (\ref{integralformulationf}) the weak form of the equation in the marginals~$\widetilde f_N^{(s)}.$ 
On the one hand an integration by parts in the time variable gives
$$ \begin{aligned}
\int_{ \R_+ \times   \R^{2dN} }  \partial_t   f_N (t,Z_N)  \phi (t,Z_s)  \displaystyle \indc_{X_N \in {\mathcal D_N^s}} \:  d Z_N dt& =  - \int_{  \R^{2dN} } f_N(0,Z_N) \phi(0,Z_s) \displaystyle \indc_{X_N \in {\mathcal D_N^s}}  \: dZ_N \\ & \qquad - \int_{ \R_+ \times   \R^{2dN} }  f_N (t,Z_N) \partial_t\phi (t,Z_s)  \displaystyle \indc_{X_N \in {\mathcal D_N^s}} \:    d Z_N dt \, ,
\end{aligned}$$
hence, by definition of $\widetilde f_N^{(s)},$
$$
 \begin{aligned}
 \int_{ \R_+ \times  \R^{2dN} }  \partial_t  f_N (t,Z_N)  \phi (t,Z_s)  \displaystyle \indc_{X_N \in {\mathcal D_N^s}} \:  d Z_N dt& =- \int_{\R^{2ds}} \widetilde  f_N^{(s)}(0,Z_s) \phi(0,Z_s) \: dZ_s  \\
 & \qquad - \int_{ \R_+ \times \R^{2ds}} \widetilde f_N^{(s)}  (t,Z_s) \partial_t\phi (t,Z_s) \:    d Z_s dt \, .
\end{aligned}
$$
Now let us compute
$$\displaystyle   \sum_{i = 1}^N \int_{  \R^{2dN} } v_i \cdot \nabla_{x_i} f_N (t,Z_N) \phi  (t,Z_s)   \displaystyle \indc_{X_N \in {\mathcal D_N^s}}  \,dZ_N = \int_{  \R^{2dN} } 
  \mbox{div}_{X_N} \big ({V_N} \:  f_N   (t,Z_N) \big)  \phi  (t,Z_s)    \displaystyle \indc_{X_N \in {\mathcal D_N^s}} \,dZ_N$$
   using Green's formula. The boundary of ${\mathcal D}_N^s$ is made of configurations with at least one pair $(i,j)$, satisfying~$1 \leq i \leq s$ and~$s + 1 \leq j \leq N$, with~$|x_i - x_j| = \e.$ 
   
Let us define, for any couple~$(i,j) \in [1,N]^2$,
\begin{equation}\label{defSigmaNij}
 \begin{aligned}
 {\Sigma}_N^s (i,j) := \Big\{
X_N \in  \R^{dN}, \quad& |x_i - x_j| = \eps\\
& \mbox{and} \quad \forall (k,\ell) \in[1,s]\times [s+1,N] \setminus\{i,j\}, \: |x_k - x_\ell| > \eps  \Big\} \, .
\end{aligned}
\end{equation}
  We 
 notice  that~${\Sigma}_N^s   (i,j)$ is a submanifold of
 $\big\{
X_N \in  \R^{dN}, \,  |x_i - x_j| = \eps \big\},$
which is a smooth,  codimension~1 manifold of $\R^{dN}$ (locally isomorphic to the space~${\mathbf S}_\eps^d \times \R^{d(N-1)}$), and we denote by~$d\sigma_N^{i,j}$ its surface measure, induced by the Lebesgue measure. 
Configurations with more than one {\it collisional} pair, i.e., $(i,j)$ and $(i',j')$ with $1 \leq i,i' \leq s,$ $s + 1 \leq j,j' \leq N$, with $|x_i - x_j| = |x_{i'} - x_{j'}| = \e,$ and~$\{i,j\} \neq \{i',j'\},$ are subsets of submanifols of $\R^{dN}$ of codimension at least two, and therefore contribute nothing to the boundary terms. 
 Denoting~$n^{i,j}$ the outward normal to~${\Sigma}_N^s   (i,j)$ we therefore obtain by  Green's formula:
$$ \begin{aligned} 
& \sum_{i= 1}^N \int_{ \R_+ \times  \R^{2dN}} 
 v_i\cdot \nabla_{x_i}  f_N  (t,Z_N) \phi  (t,Z_s)   \displaystyle \indc_{X_N \in {\mathcal D_N^s}}  \,dZ_N \, dt  \\
& \qquad = - \sum_{i=1}^s    \int_{ \R_+ \times  \R^{2dN}}   f_N   (t,Z_N)  v_i\cdot \nabla_{x_i} \phi  (t,Z_s)   \displaystyle \indc_{X_N \in {\mathcal D_N^s}}   \,dZ_N dt \\
  &    \qquad  \qquad  
 \displaystyle  +   \sum_{1 \leq i \neq j \leq N}  \int_{  \R_+\times \R^{dN} \times{\Sigma}_N^s  (i,j)}  
n^{i,j} \cdot V_N  \:   f_N (t,Z_N)\phi (t,Z_s)\:    d \sigma_N^{i,j} dV_N dt \,.
\end{aligned}
$$
By symmetry \eqref{sym}  and recalling that~$\nu^{i,j} = (x_i-x_j)/|x_i-x_j|$  this gives
$$ \begin{aligned} 
& \sum_{i= 1}^N \int_{ \R_+ \times \R^{2dN}} 
 v_i\cdot \nabla_{x_i}   f_N  (t,Z_N) \phi  (t,Z_s)   \displaystyle \indc_{X_N \in {\mathcal D_N^s}} \,dZ_N \, dt \\
 &\qquad  = - \sum_{i=1}^s    \int_{ \R_+ \times \R^{2dN}}    f_N (t,Z_N)  v_i\cdot \nabla_{x_i} \phi  (t,Z_s)   \displaystyle \indc_{X_N \in {\mathcal D_N^s}}  \,dZ_N dt \\
 &\qquad      \quad 
 \displaystyle  {} +  (N-s)\sum_{i=1}^s  \int_{  \R_+ \times \R^{dN} \times{\Sigma}_N^s  (i,j)}  
\frac{\nu^{i,s+1} }{\sqrt2} \cdot (v_{s+1}-v_i) \:   f_N (t,Z_N)\phi (t,Z_s)\:     d \sigma_N^{i,j} dV_N dt \,,
\end{aligned}
$$ 
so finally by definition of~$ \widetilde  f_N^{(s)}$, we obtain
 \begin{equation} \label{eq:s} 
 \begin{aligned}
&\displaystyle  \sum_{i=  1}^N \int_{ \R_+ \times  \R^{2dN}} 
 v_i\cdot \nabla_{x_i}     f_N (t,Z_N)  \phi  (t,Z_s)    \displaystyle \indc_{X_N \in {\mathcal D_N^s}}\,dZ_N \, dt  \\
& \displaystyle \quad   = -\sum_{i=1}^s    \int_{ \R_+ \times \R^{2ds}}  \widetilde  f_N^{(s)} (t,Z_s)   v_i\cdot \nabla_{x_i} \phi  (t,Z_s)   \,dZ_s dt  \\
&\displaystyle \qquad {}  +   (N-s)  \sum_{i=1}^s  \int_{  \R_+ \times \R^{dN} \times{\Sigma}_N^s  (i,j)}  
\frac{\nu^{i,s+1} }{\sqrt2}   \cdot (v_{s+1}-v_i) \:   f_N (t,Z_N)\phi (t,Z_s)\:     d \sigma_N^{i,j} dV_N dt \,  .
  \end{aligned}
 \end{equation}
 Now let us consider the contribution of the potential in~(\ref{integralformulationf}). 
We split the sum as follows:
  $$
   \begin{aligned}
&  \frac1{\eps} \sum_{i} \sum_{j \neq i}  \int_{\R_+ \times  \R^{2dN}}
\nabla\Phi \left(\frac{x_i - x_j}\eps \right)\cdot \nabla_{v_i}  f_N (t,Z_N) \phi (t,Z_s)   \displaystyle \indc_{X_N \in {\mathcal D_N^s}}\:    d Z_N dt \\
&\quad  =
  \frac1{\eps}  \sum_{i,j = 1 \atop j \neq i}^s  \int_{\R_+ \times  \R^{2dN}}
\nabla\Phi \left(\frac{x_i - x_j}\eps \right)\cdot \nabla_{v_i}  f_N (t,Z_N) \phi (t,Z_s)   \displaystyle \indc_{X_N \in {\mathcal D_N^s}}\:    d Z_N dt \\
  & \qquad +  \frac1{\eps}  \sum_{i,j = s+1 \atop j \neq i}^N  \int_{\R_+ \times  \R^{2dN}}
  \nabla\Phi \left(\frac{x_i - x_j}\eps \right)\cdot \nabla_{v_i}  f_N (t,Z_N) \phi (t,Z_s)  \displaystyle \indc_{X_N \in {\mathcal D_N^s}} \:    d Z_N dt  \, .
 \end{aligned}  $$
We notice that the second term in the right-hand side vanishes identically. It follows that 
   $$
  \begin{aligned}
&  \frac1{\eps} \sum_{i} \sum_{j \neq i}  \int_{\R_+ \times   \R^{2dN}}
\nabla\Phi \left(\frac{x_i - x_j}\eps \right)\cdot \nabla_{v_i}  f_N (t,Z_N) \phi (t,Z_s)   \displaystyle \indc_{X_N \in {\mathcal D_N^s}}\:    d Z_N dt \\
&\quad  =
   -  \frac1{\eps} \sum_{i,j = 1 \atop j \neq i}^s \int_{\R_+ \times \R^{2ds}}  
\nabla\Phi \left(\frac{x_i - x_j}\eps \right)\cdot \nabla_{v_i} \phi (t,Z_s) \widetilde f_N^{(s)}  (t,Z_s) \:       d Z_s dt 
 \end{aligned}
  $$
  so in the end
   we obtain
 \begin{equation} \label{1-BBKGY} \begin{aligned}
& \int_{\R_+ \times \R^{2ds}} \widetilde f_N^{(s)} (t,Z_s)    \Big( \d_t \phi + \mbox{div}_{X_s} \: (V_s  \phi) -
 \displaystyle  \frac1{\eps} \sum_{i,j = 1 \atop j \neq i}^s  \nabla\Phi \left(\frac{x_i - x_j}\eps \right)\cdot \nabla_{v_i} \phi  \Big) (t,Z_s)\:    d Z_s dt  \\ & = -  \int_{\R^{2ds}}  \widetilde f_N^{(s)}(0,Z_s) \phi(0,Z_s) \: dZ_s  \\ & 
 \quad -  (N-s)\sum_{i=1}^s   \int_{ \R_+ \times \R^{dN} \times {\Sigma}_N^s   (i,s+1) }  
\frac{\nu^{i,s+1} }{\sqrt2}  \cdot (v_{s+1}-v_i)  \:   f_N (t,Z_N) \phi (t,Z_s)\:    d \sigma_N^{i,s+1} dV_Ndt \, .\end{aligned}
 \end{equation}
 
 \begin{Rmk}\label{multiple}
Using the weak form of Liouville's equation, we see that configurations in which there would be two pre-or post-collisional pairs,  can be neglected (they occur as a boundary integral on a zero measure subset of $\d {\mathcal D}^s_N$) . 
\end{Rmk}

 \section{Clusters}\label{sec:def-closure}
 \setcounter{equation}{0}
We want to analyze the second term on the right-hand side of~(\ref{1-BBKGY}). 
 We notice that in the   space integration   the variables~$x_{s+2}, \dots, x_N$ are integrated over~$\R^{d(N-s-1)}$ (with the restriction that they must be at a distance at least~$\eps$ from~$X_s$) whereas~$x_{s+1}$
   must lie in the  sphere centered at~$x_i$ and of radius~$\eps$.   It is therefore natural to try to express that contribution in terms  of the marginal~$\widetilde f_N^{(s+1)}(Z_{s+1})$. 
   However as pointed out in~(\ref{pbhierarchy}),   
   $$
   \int \widetilde f_N^{(s+1)}(Z_{s+1}) \,  dz_{s+1} \neq \widetilde f_N^{(s)} (Z_s) \, .
   $$
     The difference between those two terms is that on the one hand 
  $$\forall X_N \in {\mathcal D} _N^{s+1} \, ,\quad \hbox{ one has } |x_j - x_{s+1}| > \eps \hbox{ for all } j\geq s+2 \, ,$$ which is not the case for~$X_N \in {\mathcal D} _N^{s}$, and on the other hand
  $$\forall X_N \in {\mathcal D} _N^{s} \, ,\quad \hbox{ one has } |x_j - x_{s+1}| > \eps \hbox{ for all } j\leq s \, ,$$ a condition which does not appear in the definition of~${\mathcal D} _N^{s+1} $.

  This leads to the following definition.
  \begin{Def}[$\eps$-closure]\label{cluster}
Given a subset~$X_N = \{ x_1,\dots,x_N\} $ of~$ \R^{dN}$ and an integer~$s $ in~$ [1,N],$  the~$\eps$-closure~$E(X_s,X_N) $ of~$X_s$ in~$X_N$ is defined as the intersection of all subsets $Y$ of~$X_N$ which contain~$X_s$ and satisfy the separation condition
\begin{equation} \label{cond:separation}\forall y \in Y \, ,\quad \forall x \in X_N \setminus Y \, , \quad |x-y|>\eps\,.
\end{equation}
We denote $|E(X_s,X_N)|$ the cardinal of $E(X_s,X_N).$\end{Def}

Now let us introduce  the following notation, useful in situations where $X_N$ belongs to~$\Sigma^s_N(i,s+1),$ defined in \eqref{defSigmaNij}.
\begin{nota} If $X_{s+m} = E(X_s,X_{s+m} )$ and if for some  integers~$j_0 \leq s < k_0 \leq s+m,$ there holds~$|x_j - x_k| > \e$ for all $(j,k) \in [1,s] \times [s+1,s+m] \setminus \{(j_0,k_0)\},$ then we say that $E(X_s,X_{s+m})$ has a {\it weak link} at $(j_0,k_0),$ and we denote $X_{s+m} = E_{\langle j_0, k_0 \rangle}(X_s,X_{s+m}).$
\end{nota}
Moreover the following notion, following King \cite{K}, will turn out to be very useful.
 \begin{Def}[Cluster] \label{definitionclusterking}
A cluster of base~$X_s=\{x_1,\dots,x_s\} $  and length $m$ is any point~$\{x_{s+1},\dots,x_{s+m}\} $ in~$\R^{dm}$  such that
$ E(X_s, X_{s+m}) =X_{s+m}\,.$ We denote $\Delta_m(X_s)$ the set of all such clusters.
\end{Def}
The  proof of the following lemma  is completely elementary.
\begin{lem} \label{lem:cluster-eq}
The following equivalences hold, for $m \geq 1:$
\begin{equation} \label{equiv-cluster1}
 \Big( E(X_s,X_N) = X_{s+m} \Big) \iff \Big( \mbox{$E(X_s,X_{s+m}) = X_{s+m}$ and $X_N  \in \D^{s+m}_N$} \Big) \, ,
 \end{equation}
  \begin{equation} \label{equiv-cluster2}
 \left( \begin{aligned} E(X_s,X_N) & = X_{s+m} \\ X_N & \in \Sigma_N^s(i,s+1) \end{aligned} \right) \iff \left( \begin{aligned} E_{\langle i,s+1\rangle}(X_s,X_{s+m}) & = X_{s+m} \\ X_N & \in \D^{s+m}_N \\ |x_i - x_{s+1}| & = \e  \end{aligned}\right) \, ,
 \end{equation}
 as well as the implication, for $m \geq 2,$
 \begin{equation} \label{equiv-cluster3}
 \Big( E_{\langle i,s+1\rangle}(X_s,X_{s+m}) = X_{s+m} \Big) \implies \Big( \big\{x_{s+2},\dots,x_{s+m}\big\} \in \Delta_{m-1}(x_{s+1})\Big) \, .
  \end{equation}\end{lem}

\begin{figure}[h]
\begin{center}
\scalebox{0.45}{\includegraphics{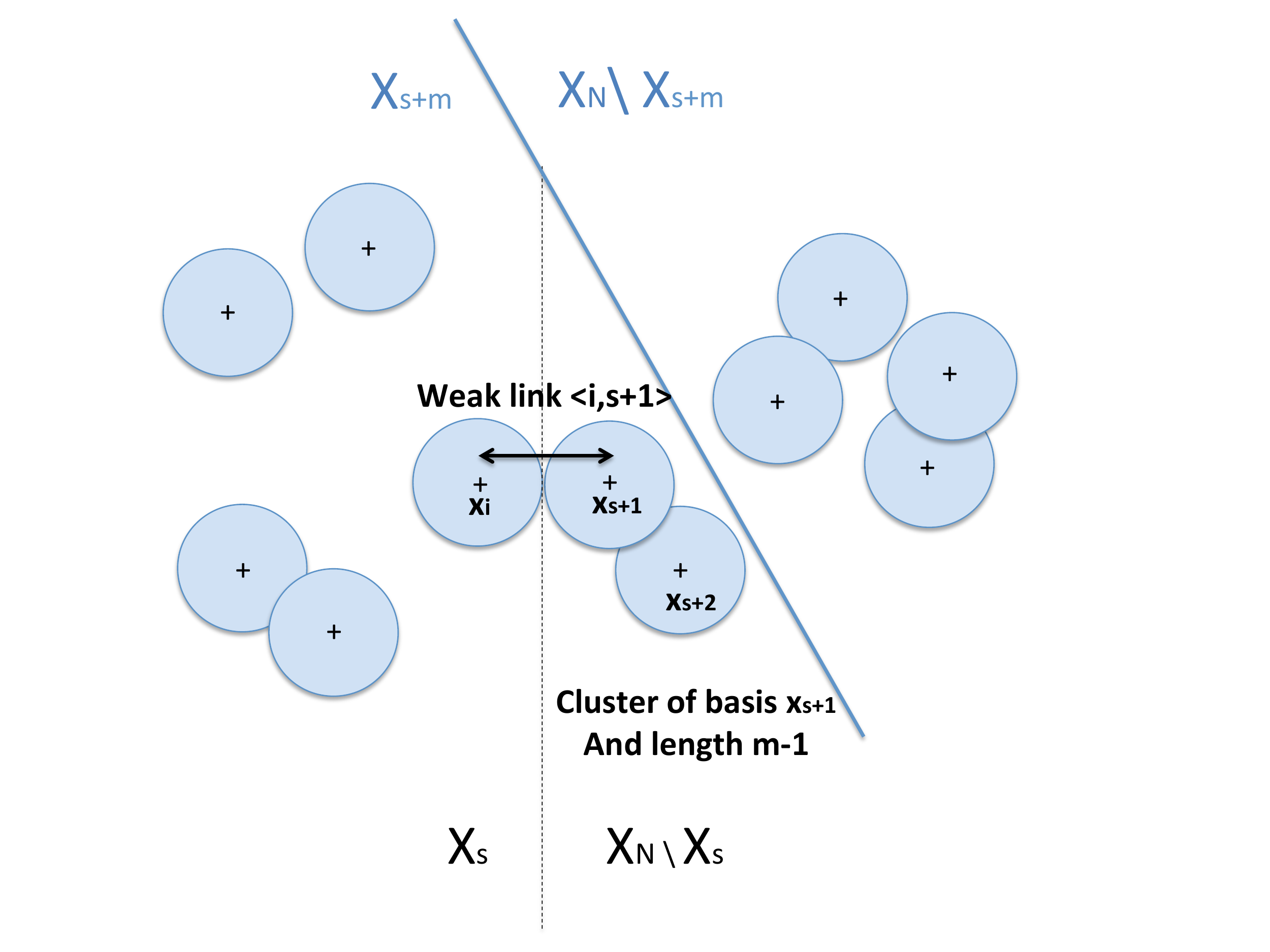}}
\caption{\label{cluster-figure}Clusters with weak links}
\end{center}
\end{figure}

\section{Collision operators}\label{collisionliouville}
 \setcounter{equation}{0}

With the help of the notions introduced in Section \ref{sec:def-closure}, we now can reformulate the boundary integral in~(\ref{1-BBKGY}).

Given $1 \leq s \leq N-1$ and $X_N$ in~$  {\Sigma}^s_N(i,s+1)$, there holds $|x_{s+1} - x_i| = \eps,$ so that~$x_{s+1}$ belongs to~$E(X_s,X_N),$ implying $|E(X_s,X_N)| \geq s+1.$ We decompose $\Sigma^s_N(i,s+1)$ into a disjoint union over the possible cardinals of the $\e$-closure of $X_s$ in $X_N:$  
 \begin{equation} \label{partition} \Sigma^s_N(i,s+1) = \dsp{  \displaystyle \bigcup_{1 \leq m \leq N-s}} \Big( \Sigma^s_N(i,s+1) \bigcap \big\{ Y_N, \,\, |E(Y_s,Y_N)| = s+m \big\}\Big) \, , 
 \end{equation}
 implying
$$
 \begin{aligned} 
& \int_{\R^{dN} \times {\Sigma}^s_N   (i,s+1)}   
\nu^{i,{s+1}}  \cdot (v_{s+1}-v_i)  \:   f_N (Z_N) \phi (Z_s)\:    d \sigma_N^{i,s+1}dV_N
 \\ & = \sum_{1 \leq m \leq N-s} \int_{\R^{dN}\times {\Sigma}^s_N   (i,s+1) }
  \indc_{ |E(X_s,X_N)| = s+m} \, \nu^{i,{s+1}}  \cdot (v_{s+1}-v_i)  \: f_N (Z_N) \phi (Z_s)\:    d \sigma_N^{i,s+1} dV_N\, .
   \end{aligned}
   $$
By assumption of symmetry \eqref{sym} for $f_N$ and $\phi,$ if $|E(X_s,X_N)| = s+m,$ we can index the particles so that $E(X_s,X_N) = X_{s+m}:$  we obtain
\begin{equation} \label{use:sym} 
\begin{aligned} 
& \int_{\R^{dN} \times {\Sigma}^s_N   (i,s+1) } \hspace{-10pt} \indc_{|E(X_s,X_N)| = s+m} \,  \nu^{i,{s+1}}  \cdot (v_{s+1}-v_i)  \:  f_N (Z_N) \phi (Z_s)\:    d \sigma_N^{i,s+1}dV_N \\ 
 &\quad = C_{N-s-1}^{m-1} \int_{\R^{dN} \times {\Sigma}^s_N   (i,s+1) } \indc_{E(X_s,X_{N}) = X_{s+m}} \nu^{i,{s+1}}  \cdot (v_{s+1}-v_i)  \:   f_N (Z_N) \phi (Z_s)\:    d \sigma_N^{i,s+1} dV_N\, .
 \end{aligned}
\end{equation}
We use equivalence \eqref{equiv-cluster2} from Lemma \ref{lem:cluster-eq} and Fubini's theorem to write
$$ \begin{aligned} & \int_{\R^{dN} \times {\Sigma}^s_N (i,s+1)} \hspace{-10pt}
 \indc_{E(X_s, X_N) = X_{s+m}} 
\nu^{i,{s+1}}  \cdot (v_{s+1}-v_i) f_N (Z_{N}) \phi (Z_s)      d\sigma_N^{i,{s+1}}dV_N \\ 
& \quad = \sqrt 2 \int_{S_\e(x_i) \times \R^d } 
 \nu^{i,{s+1}}  \cdot (v_{s+1}-v_i) \phi(Z_s) \\ & \qquad \qquad \qquad \times \left( \int _{\R^{2d(m-1)}}
 \indc_{E_{\langle i,s+1\rangle}(X_s,X_{s+m}) = X_{s+m}} f^{(s+m)}_N(Z_{s+m}) dZ_{(s+1,s+m)}\right) d\s_i(x_{s+1}) \, ,\end{aligned}
$$
with~$ d\sigma_i$\label{dsigmaSeps} the surface measure on~$S_\eps(x_i)  :=\big \{x \in \R^d, \, |x-x_i| = \eps\big\}.$ With \eqref{equiv-cluster3}, if $m \geq2,$ then the above integral over $\R^{2d(m-1)}$ appears as an integral over $\Delta_{m-1}(x_{s+1}).$ 
We also remark that in the case~$m=1,$ we have a simple description of $E_{\langle i,s+1\rangle}(X_s,X_{s+1}) = X_{s+1}:$
\begin{equation} \label{for-c1}
 \Big( \indc_{E_{\langle i,s+1\rangle}(X_s,X_{s+1}) = X_{s+1}} \neq 0 \Big) 
 \iff \left(\begin{aligned}
 |x_i - x_{s+1}| \leq \e & \quad \mbox{} \\  |x_j - x_{s+1}| > \e & \quad \mbox{for $j \in [1,s] \setminus \{i\}$}
\end{aligned}\right).
\end{equation}

This leads to the following definition of the collision term of order $m \geq 1,$ for $s + m \leq N:$   we define
   \begin{equation} \label{css+k} \begin{aligned}
   {\mathcal C}_{s,s+m}  \widetilde f_N^{(s+m)} (Z_s)
    & := m C_{N-s}^m \sum_{ i = 1}^s \int_{S_\e(x_i) \times \R^d }
     \nu^{{s+1},i}  \cdot (v_{s+1}-v_i)  \\
    & \qquad \times G^{(m-1)}_{\langle i,s+1\rangle}(f^{(s+m)}_N)(Z_{s+1})\, d\s_i(x_{s+1}) dv_{s+1}\,,
   \end{aligned}\end{equation}
   where for $m = 1,$ by \eqref{for-c1}:
\begin{equation} \label{def:g0} 
  G^{(0)}_{\langle i,s+1\rangle}(\widetilde f^{(s+1)}_N)(Z_{s+1}) := \Big(\prod_{\begin{smallmatrix} 1 \leq j \leq s \\ j \neq i \end{smallmatrix}} \indc_{|x_{s+1}- x_j| > \e} \Big) \widetilde f^{(s+1)}_N(Z_{s+1})\,,
\end{equation}
and for $m \geq 2:$  
 \begin{equation} \label{def:g} \begin{aligned}
  G^{(m-1)}_{\langle i,s+1\rangle}& (\widetilde f^{(s+m)}_N)(Z_{s+1}) \\& := \int _{\Delta_{m-1}(x_{s+1}) \times \R^{d(m-1)}} \indc_{E_{\langle i, s+1\rangle}(X_s,X_{s+m}) = X_{s+m}} \widetilde f^{(s+m)}_N(Z_{s+m}) dZ_{(s+2,s+m)}\,. \end{aligned}
  \end{equation}
The complex-looking indicator function $\indc_{E_{\langle i, s+1\rangle}(X_s,X_{s+m}) = X_{s+m}}$ will, in the   estimates of the next chapters, be simply bounded from above by one. This will be the case for instance in an estimate showing that higher-order collision operators \eqref{def:g} are negligible in the thermodynamical limit; this estimate is \eqref{est-col:1} in Proposition \ref{propcontinuityclusters}. 
One should notice on the other hand that the operator~${\mathcal C}_{s,s+1} $ is very similar to the corresponding collision operator~(\ref{def:collisionop0}) in the hard-spheres situation. 

With $(N-s) C_{N-s-1}^{m-1} = m C_{N-s}^m,$ we can now reformulate \eqref{1-BBKGY} into 
\begin{equation} \label{2-BBKGY} \begin{aligned}
\int_{\R_+ \times \R^{2ds}} & \widetilde f_N^{(s)} (t,Z_s)    \Big( \d_t \phi + \mbox{div}_{X_s} \: (V_s  \phi)  -
 \displaystyle  \frac1{\eps} \sum_{i,j = 1 \atop j \neq i}^s  \nabla \Phi\left(\frac{x_i - x_j}\eps \right)\cdot \nabla_{v_i} \phi  \Big) (t,Z_s)\:    d Z_s dt  \\ & + \int_{\R^{2ds}} \widetilde f_N^{(s)}(0,Z_s) \phi(0,Z_s) \: dZ_s =  \sum_{m = 1}^{N-s} \int_{\R^+ \times \R^{2ds}} \phi(t,Z_s) {\mathcal C}_{s,s+m} \widetilde f_N^{(s+m)}(t,Z_s) \, dt dZ_s\,,
 \end{aligned}
 \end{equation}
 so that $\widetilde f^{(s)}_N$ appears as a (formal) weak solution to
 \begin{equation} \label{w-BBGKY}
  \d_t \widetilde f_N^{(s)} +\sum_{1 \leq i \leq s} v_i\cdot \nabla_{x_i} \widetilde f_N^{(s)} - \frac1{\eps} \sum_{1 \leq i \neq j \leq s} \nabla \Phi\left(\frac{x_i - x_j}\eps \right)\cdot \nabla_{v_i} \widetilde f_N^{(s)} =  \sum_{m = 1}^{N-s} {\mathcal C}_{s,s+m} \widetilde f_N^{(s+m)}  \,  .
  \end{equation}

\section{Mild solutions}\label{mildliouville}
\setcounter{equation}{0}

We   now  define the integral formulation of \eqref{w-BBGKY}.
Denote by~${\bf \Phi}_s(t)$ the~$s$-particle Hamiltonian flow, and by~$ {\bf H}_s$ the associated solution operator:
\begin{equation}\label{solutionoperator}
  {\bf H}_s(t): \qquad f \in C^0(\Omega_s;\R) \mapsto f({\bf \Phi}_s(-t,\cdot)) \in C^0(\Omega_s;\R) \, .
 \end{equation}
 The time-integrated form of equation~(\ref{w-BBGKY}) is
\begin{equation}
\label{BBGKY-mild}
\widetilde f_N^{(s)} (t,Z_s)=   {\bf H}_s(t)\widetilde f_N^{(s)}(0,Z_s) +  \sum_{m = 1}^{N-s} \int_0^t {\bf H}_s(t - \t)   {\mathcal C}_{s,s+m}\widetilde  f_N^{(s+m)} (\t,Z_s)\, d\t \, .
\end{equation}
 The {\it total flow} and {\it total collision} operators~${\mathbf  H}  $ and~${\mathbf  C} _N$ are defined on  finite sequences~$G_N= (g_s)_{1\leq s\leq N}$ as follows:
\begin{equation}
\label{bH-def}
\left\{\begin{aligned} & \forall s  \leq N \,,\,\,  \left( {\mathbf H}  (t) G_N\right)_s :=  {\mathbf H}_{s}(t) g_{s}\, , \\ & \forall \, s \leq N-1 \,, \,\, \left( {\mathbf C} _N G_N\right)_s :=\sum_{m=1}^{N-s} {\mathcal C}_{s,s+m} g_{s+m} \,, \quad \big( {\bf C}_N G_N\big)_N := 0\, .\end{aligned}\right.
\end{equation}
We define {\it mild solutions} to the BBGKY hierarchy~(\ref{BBGKY-mild}) to be solutions of  
\begin{equation} \label{BBGKY-mild-total} \widetilde F_N (t) = {\mathbf H}  (t)  \widetilde F_N (0) +  \int_0^t {\bf H}(t - \t)   {\bf C}_N  \widetilde F_N (\t) \, d\t\,, \qquad \widetilde F_N = (\widetilde f^{(s)}_N)_{1 \leq s \leq N}\,.
\end{equation}

\begin{Rmk}
At this stage, the use of weak formulations could seem a little bit suspicious since they are used essentially as a technical artifice to go from the Liouville equation {\rm(\ref{Liouville})} to the mild form of the BBGKY hierarchy {\rm(\ref{BBGKY-mild})}. In particular, this allows to ignore pathological trajectories as mentioned in Remark~{\rm\ref{multiple}}.
Nevertheless, the existence of mild solutions to the BBGKY hierarchy  provides the existence of weak solutions to the BBGKY hierarchy, and in particular to the Liouville equation (which is nothing else than the last equation of the hierarchy). The classical uniqueness result for kinetic transport equations then implies that the object we consider, that is the family of truncated marginals, is uniquely determined (almost everywhere). 
\end{Rmk} 

\section{The limiting Boltzmann hierarchy}
\setcounter{equation}{0}

The limit of the BBGKY collision operators~(\ref{css+k}) was obtained formally in Section~\ref{scatteringcrosssectionpotentialcase}, following the formal derivation of the hard-spheres case in  Paragraph~\ref{spheresboltz}, assuming higher-order interactions can be neglected. We recall the form of the collision operator as given in~(\ref{Boltzmann-opscatchapter}):
\begin{eqnarray*}
 &&   {\mathcal C}^0_{s,s+1}    f^{(s+1)} (t,Z_s) :=       \sum_{ i = 1}^s  \int  \,  b(v_1 - v_2,\o) \nonumber\\
&&\quad \times{} \Big(  f^{(s+1)}  (t,x_1,v_1,\dots ,x_i, v^*_i,\dots ,x_s,v_s, x_i, v^*_{s+1})    -     f^{(s+1)}  (t,Z_s, x_i, v_{s+1}) \Big) d\omega dv_{s+1} \, .
\end{eqnarray*} 
where~$(v^*_i,v^*_{s+1})$ is obtained from~$(v_i,v_{s+1})$ by applying the inverse scattering operator~$\sigma_0^{-1}$ defined in Definition~\ref{scatteringbbgky} and~$b(w,\omega)$ is the cross-section given by Definition~\ref{def:b}.

The asymptotic dynamics are therefore governed by the following integral form of the Boltzmann hierarchy:
\begin{equation}
\label{mild-Boltzmannpotential}
   f^{(s)} (t) = {\bf S}_s(t)  f^{(s)}_{0} + \int_0^t {\bf S}_{s}(t - \t)     {\mathcal C}^0_{s,s+1}    f^{(s+1)}  (\t) \, d\t\,,
 \end{equation}
 where~$ {\bf S}_s(t)$ denotes the~$s$-particle free-flow.

Similarly to~(\ref{bH-defHS}), we can define the total Boltzmann flow and collision operators~${\mathbf  S}  $ and~${\mathbf  C}  $  as follows:
\begin{equation}
\label{bH-defboltzmannpotential}
\left\{\begin{aligned} & \forall s \geq 1\,,\,\,  \left( {\mathbf S}  (t) G \right)_s :=  {\mathbf S}_{s}(t) g_{s}\, , \\ & 
\forall \, s \geq 1 \,, \,\, \left(  {\bf C^0}    G \right)_s := {\mathcal C}^0_{s,s+1} g_{s+1} \,,  \end{aligned}\right.
\end{equation}
so that {\it mild solutions} to the Boltzmann hierarchy~(\ref{mild-Boltzmannpotential}) are solutions of  
\begin{equation} \label{Boltzmann-mild-tota}   F (t) = {\mathbf S}  (t)    F  (0) +  \int_0^t {\bf S}(t - \t)   {\bf C^0}    F  (\t) \, d\t\,, \qquad   F  = (  f^{(s)})_{s \geq 1 }\,.
\end{equation}

 Note that if~$   f^{(s)} (t,Z_s) = \displaystyle \prod_{i = 1}^s f(t,z_i) $ (meaning~$f^{(s)}(t)$ is {\it tensorized}) then~$f$ satisfies the Boltzmann equation~(\ref{boltz-eq})-(\ref{boltz-operator}), with the cross-section~$b(w,\omega)$ given by Definition~\ref{def:b}.

%% file: cluster.tex
\chapter{Cluster estimates and uniform a priori estimates}\label{chaptercluster}
\setcounter{equation}{0}

\label{continuity-collision}

In view of proving the existence of mild solutions to the BBGKY hierarchy (\ref{BBGKY-mild}), we need continuity estimates on the linear collision operators ${\mathcal C}_{s,s+m}$ defined in~(\ref{css+k})-\eqref{def:g0}-\eqref{def:g}, and the total collision operator ${\bf C}_N$ defined in \eqref{bH-def}. 

We first note that, by definition,  the operator ${\mathcal C}_{s,s+m}$ involves  only configurations with clusters of length $m$.
Classical computations of statistical mechanics, presented  in Section \ref{sec:clusters},  show that the probability of finding such clusters is exponentially decreasing with $m.$ 

It is then natural to introduce functional spaces encoding the decay with respect to energy and the growth with respect to the order of the marginal (see Section~\ref{sec:norms}, where norms are introduced, generalizing the norms introduced in Chapter~\ref{existence} for the hard spheres case). In these appropriate functional spaces, we can establish uniform continuity estimates for the BBGKY collision operators (Section~\ref{sec:cont-est}).  These will enable us in Section~\ref{sec:continuity-col-B}   to obtain directly   uniform bounds for the hierarchy   as in Chapter~\ref{existence}.


\section{Cluster estimates} \label{sec:clusters}
\setcounter{equation}{0}
A point $X_s \in \R^{ds}$ being given, we recall that $\Delta_m(X_s)$ is the set of all clusters of base $X_s$ and length~$m$ (this notation is introduced in Definition \ref{definitionclusterking} page~\pageref{definitionclusterking}).
\begin{Lem}\label{estimationsgrandcanoniques}
 For any symmetric function $\varphi$ on $\R^{Nd},$ any $s \in [1,N-1],$ any $X_s \in \R^{ds},$ the following identity holds:
\begin{equation}
\label{cluster-dec}
\begin{aligned}
  \int_{\R^{(N-s)d}} & \varphi(X_N)  dX_{(s+1,N)}  = \int_{\R^{d(N-s)}} \indc_{X_N \in \D_N^s} \,  \varphi(X_N) \, dX_{(s+1,N)} \\
& + \sum_{m=1}^{N-s} C^m_{N-s}\int_{\Delta_m(X_s)} \left(\int_{\R^{d(N-s-m)}} \indc_{X_N \in\D_N^{s+m}}   \, \varphi (X_N) \,  dX_{(s+m+1,N)} \right) dX_{(s+1,s+m)} \, ,
\end{aligned}
\end{equation}
implying, for $\zeta>0,$
\begin{equation}
\label{cluster-est1}
{1\over m!} \int_{\Delta_m(X_s)} dX_{(s+1,s+m)}  \leq \zeta^{-m} \exp \big(\zeta \k_d (s+m) \eps^d \big)
\end{equation}
and
\begin{equation}
\label{cluster-est2}
\sum_{m\geq 1} \frac{\zeta^{m+1} \exp \big(-\zeta \k_d (m+1)  \eps^d\big)}{m!} \int _{\Delta_m(x_1)} dX_{(2,m+1)} \leq \zeta \big(1- \exp \big(-\zeta \k_d \eps^d \big)\big) \, ,
\end{equation}
where $\k_d$ is the volume of the unit ball in $\R^d.$ \label{ref:kappa-d}
\end{Lem}

\begin{proof} The first identity (\ref{cluster-dec}) is obtained by a simple partitioning argument, which extends the splitting used to define ${\mathcal C}_{s,s+m}$ in~(\ref{css+k})  in the previous chapter.
We recall that, given any $X_s \in \R^{ds}$,  the family
$$\Big \{ (x_{s+1},\dots, x_N)\,,\,  | E(X_s,X_N) |=s+m \Big\} \quad \hbox{ for } \,  0\leq m\leq N-s\, ,$$
is a partition of $\R^{(N-s)d}$.
Then we use the symmetry assumption, as we did in \eqref{use:sym}, to find
 $$ \int_{\R^{(N-s)d}} \varphi(X_N)  dX_{(s+1,N)} = \sum_{0 \leq m \leq N-s} C_{N-s}^m \int_{\R^{(N-s)d}} \indc_{E(X_s,X_{N}) = X_{s+m}} \varphi(X_N) dX_{(s+1,N)}\, .$$
  It then suffices to use equivalence \eqref{equiv-cluster1} from Lemma \ref{lem:cluster-eq}, noting that the set of all $(x_{s+1}, \dots, x_{s+m}) $ in~$ \R^{md}$ such that $E(X_s,X_{s+m}) = X_{s+m}$ coincides with $\Delta_m(X_s).$  This proves~(\ref{cluster-dec}).

Estimates (\ref{cluster-est1}) and (\ref{cluster-est2}) come from the counterpart of (\ref{cluster-dec}) at the grand canonical level, i.e. when the activity $\zeta^{-1} \sim e^\mu$ is fixed, rather than the total number $N$ of particles (we refer to~Remark~\ref{physinterpretation} for comments on this terminology).

For any bounded $\Lambda \subset \R^d$, the associated grand-canonical ensemble for $n$ non-interacting particles is defined as the probability measure with density 
$$
\varphi_n(X_n) := {\zeta^n \exp(-\zeta |\Lambda|) \over n!} \prod_{1 \leq i \leq n}\indc_{x_i \in \Lambda}  \, .
$$
The $s$-point correlation function $g_s$ and  the truncated $s$-point correlation function $\widetilde g_s$ are defined by
$$\begin{aligned}
g_s(X_s)& :=\sum_{n=s}^\infty{n!\over (n-s)!} \int_{\R^{(n-s)d} } \varphi_n (X_n)dX_{(s+1,n)} \, , \\ 
\widetilde g_s(X_s)&:=\sum_{n=s}^\infty{n!\over (n-s)!} \int_{\R^{(n-s)d}} \indc_{X_n \in \D_n^s} \varphi_n (X_n) dX_{(s+1,n)}\, .
\end{aligned}
$$
We compute
$$ \int_{\R^{(n-s)d}} \varphi_n (X_n)dX_{(s+1,n)} = \zeta^s \exp\big(-\zeta |\Lambda|\big) \frac{(\zeta |\Lambda|)^{n-s}}{n!} \prod_{1 \leq i \leq s} \indc_{x_i \in \Lambda} \, ,$$
so that
\begin{equation}
\label{tilde-gj-eq}
\begin{aligned}
  g_s(X_s)
= \zeta^s \exp\big(-\zeta |\Lambda|\big) \sum_{k=0}^\infty \frac{(\zeta |\Lambda|)^k}{k!} \prod_{1 \leq i \leq s} \indc_{\Lambda}(x_i)=\zeta^s \prod_{1 \leq i \leq s} \indc_{x_i \in \Lambda} \, . 
\end{aligned}
\end{equation}
Similarly, by definition of $\D_n^s$ in \eqref{def:DNS}, 
 $$ 
 \int_{\R^{(n-s)d}} \indc_{X_n \in \D_n^s} \prod_{s+1 \leq j \leq n} \indc_{x_i \in \Lambda} \, dX_{(s+1,n)}   = \big| \Lambda \cap  {}^cB_\e(X_s) \big| \, ,
 $$
where we denote $\displaystyle B_\eps (X_s):=\bigcup_{1 \leq i\leq s} B_\eps(x_i),$ with $B_\eps(x_i) := \big \{y \in \R^{d},\, \, |y-x_i| \leq \e \big\}$.\label{indexdefBepsx}
 This implies 
$$
\begin{aligned} 
\widetilde g_s(X_s) &= \zeta^s \exp\big( - \zeta |\Lambda| \big) \sum_{n \geq s} \frac{\big( \zeta |\Lambda \cap  {}^c B_\e(X_s)\big|\big)^{n-s}}{(n-s)!} \prod_{1 \leq i \leq s} \indc_{x_i \in \Lambda} \, .
\end{aligned}
$$
Since~$|\Lambda| - |\Lambda \cap  {}^cB_\e(X_s)| = | \Lambda \cap B_\e(X_s)|$, we obtain
\begin{equation} \label{gj-eq}
 \widetilde g_s(X_s) = \zeta^s \exp\big(-\zeta | \Lambda \cap B_\eps (X_s)|\big) \, .  \end{equation} 
Besides, by \eqref{cluster-dec}, 
$$
\begin{aligned}
 g_s (X_s) 
& = \widetilde g_s(X_s) \\  
 & \quad + \sum_{n =s}^{ \infty} \sum_{m=1} ^{ n-s} \frac{n! C^m_{n-s}}{(n-s)!} \int_{\Delta_m(X_s)} \left(\int_{\R^{(n-s-m)d}} \indc_{X_n \in \D_n^{s+m}} g_s(X_n)  \, dX_{(s+m+1,n)} \right) dX_{(s+1,s+m)} \, .
\end{aligned}
$$ 
By Fubini, we get
$$ \begin{aligned}
 \sum_{n =s}^{ \infty} &\sum_{m=1} ^{ n-s} \frac{n! C^m_{n-s}}{(n-s)!} \int_{\Delta_m(X_s)} \left(\int_{\R^{(n-s-m)d}} \indc_{X_n \in \D_n^{s+m}} \varphi_n(X_n)  \, dX_{(s+m+1,n)} \right) dX_{(s+1,s+m)} \\ & =  \sum_{n =s}^{ \infty} \sum_{m=1} ^{ n-s}  \frac{n!}{(k-s)!(n-k)!} \int_{\Delta_{k-s}(X_s)} \left(\int_{\R^{(n-k)d}} \indc_{X_n \in \D_n^{k}} \varphi_n(X_n)  \, dX_{(k+1,n)}\right) dX_{(s+1,k)} \\ 
 & = \sum_{k = s+1}^\infty \frac{1}{(k-s)!} \sum_{n = k}^\infty \frac{n!}{(n-k)!}  \int_{\Delta_{k-s}(X_s)} \left(\int_{\R^{(n-k)d}} \indc_{X_n \in \D_n^{k}} \varphi_n(X_n) \, dX_{(k+1,n)}\right) dX_{(s+1,k)} \\ & =  \sum_{k = s+1}^\infty\frac{1}{(k-s)!}  \int_{\Delta_{k-s}(X_s)} \widetilde g_k(X_k) dX_{(s+1,k)} \, .
 \end{aligned}
 $$
We have proved that
 \begin{equation}
\label{cluster-dec-gc}
\begin{aligned}
 g_s (X_s)
&= \widetilde g_s  (X_s)+ \sum_{k=s+1}^\infty  {1\over (k-s)!}  \int_{\Delta_{k-s}(X_s)} g_k(X_k) dX_{(s+1,k)}\,  .
  \end{aligned}
\end{equation}
We now show how identities \eqref{tilde-gj-eq}-(\ref{gj-eq})-(\ref{cluster-dec-gc}) imply the bounds \eqref{cluster-est1}-\eqref{cluster-est2}. 

We first retain only the contribution of $k = s+m$ in the right-hand side of \eqref{cluster-dec-gc}. 
We have 
 $$ \zeta^s \geq \frac{1}{m!} \int_{\Delta_m(X_s)} \zeta^{s+m} \exp\big(-\zeta |\Lambda \cap B_\eps (X_{s+m})|\big) \, dX_{(s+1,s+m)}\, ,$$
and now $|\Lambda \cap B_\e(X_{s+m})| \leq \k_d \e^d (s+m)$ implies \eqref{cluster-est1}. 
 
We finally fix an integer $K \geq 2$ and choose $s = 1$ in \eqref{cluster-dec-gc}. Then $$ \zeta - \zeta \exp\big(-\zeta |\Lambda \cap B_\e(x_1)|\big) \geq \sum_{ k =2}^{ K} \int_{\Delta_{k-1}(x_1)} \zeta^k \exp\big(-\zeta |B_\eps (X_k)|\big)\, dX_{(2,k)}\, ,$$
 and bounding the volumes of balls from above, we find
 $$ \zeta\big(1 - \exp(-\zeta \k_d \e^d)\big) \geq \sum_{ k =1}^{ K-1}\frac{\zeta^{k+1}}{k!} \exp\big(-\zeta \k_d (k+1) \e^{d}\big) \int_{\Delta_k(x_1)} dX_{(2,k+1)}\, .$$
 It then suffices to let $K \to \infty$ to find \eqref{cluster-est2}. 
This ends the proof of Lemma~\ref{estimationsgrandcanoniques}.
\end{proof}


\section{Functional spaces} \label{sec:norms}
\setcounter{equation}{0}

To show the convergence of the series defining mild solutions (\ref{BBGKY-mild}) to the BBGKY hierarchy, we need to introduce some norms on the space of sequences $(\widetilde  f^{(s)})_{s \geq 1}$.
Given $\e > 0,$ $\b > 0,$ an integer $s \geq 1,$ and a continuous function $g_s: \Omega_s \to \R,$ we let
\begin{equation} \label{norm:e-bpot}
| g_s|_{\eps,s,\beta} :=\sup _{Z_s \in \Omega_s} \left( |g_s(Z_s)| \exp \big(\beta E_\eps(Z_s)\big)\right) 
\end{equation}
where for $\e > 0,$ the function $E_\eps$ is the $s$-particle Hamiltonian\label{indexdefPhieps}
\begin{equation} \label{def:hamiltonian}
 E_\eps(Z_s) :=\sum_{1 \leq i \leq s} {|v_i|^2\over 2}  +\sum_{ 1\leq i<k \leq s} \Phi_\eps (x_i-x_k) \, , \quad \mbox{with} \quad  \Phi_\eps (x) :=   \Phi \left(\frac x \eps\right) \,.
\end{equation}
Notice that this norm does coincide with its counterpart defined in Paragraph~\ref{sec:ex-BBGKYHS} in the limit described in Remark~\ref{limithardspherespotential}.

\begin{Def} \label{def:functional-spaces} For $\e > 0$ and $\b > 0,$ we denote $X_{\e,s,\beta}$ the Banach space of continuous functions~$\Omega_s \to \R$ with finite $|\cdot|_{\e,s,\b}$ norm.\end{Def}

By Assumption \ref{propertiesphi}, for $\e > 0$ (and $\b > 0$) there holds $\exp(\b E_\e(Z_s)) \to \infty$ as $Z_s$ approaches $\d \Omega_s.$ This implies for $g_s \in X_{\e,s,\b}$ the existence of an extension by continuity: $\bar g_s \in C^0(\R^{2ds};\R)$ such that~$\bar g_s \equiv 0$ on $\d\Omega_s,$ and $\bar g_s \equiv g$ on $\Omega_s.$

For sequences of functions $G = (g_s)_{s\geq1},$ with $g_s: \Omega_s \to \R,$ we let for $\e > 0,$ $\b > 0,$ $\mu \in \R,$
 \label{normepsbetamu}
$$
\| G \|_{\eps, \beta,\mu}  :=\sup_{s\geq 1}\Big( | g_s|_{\eps,s,\beta} \exp( \mu s)  \Big) \, .
$$

\begin{Def} \label{def:functional-spaces2} For $\e \geq 0,$ $\b > 0,$ and $\mu \in \R,$ we denote ${\bf X}_{\e,\b,\mu}$ the Banach space of sequences~$G = (g_s)_{s \geq 1},$ with $g_s \in X_{\e,s,\b}$ and $\| G \|_{\e,\b,\mu} < \infty.$  
\end{Def}
As in~(\ref{inclusionsHS}), he following inclusions hold:
\begin{equation} \label{inclusions} \mbox{if $\b' \leq \b$ and $\mu' \leq \mu \, ,$ then} \quad X_{\e,s,\b'} \subset X_{\e,s,\b} \, , \quad {\bf X}_{\e,\b',\mu'} \subset {\bf X}_{\e,\b,\mu} \, .
\end{equation}
Finally similarly to Definition~\ref{deffunctionspacesexistenceHS} we define norms of time-dependent functions as follows.
\begin{Def} \label{deffunctionspacesexistence}
Given $T > 0$, a positive function ${\boldsymbol\beta}$  and a real valued function~${\boldsymbol \mu} $ defined on~$[0,T]$ we denote ${\bf X}_{\e,{\boldsymbol\beta},{\boldsymbol\mu}}$ the space of functions $G: t \in [0,T] \mapsto G(t) = (g_s(t))_{1 \leq s} \in {\bf X}_{\e,\b(t),\mu(t)},$ such that for all $Z_s \in \R^{2ds},$ the map $t \in [0,T] \mapsto g_s(t,Z_s)$ is measurable, and
 \begin{equation} \label{def:normz}   | \! \| G    | \! \| _{ \e,{\boldsymbol\beta},{\boldsymbol\mu}}
 := \sup_{0 \leq t \leq T} \| G(t) \|_{\e,\b(t),\mu(t)} < \infty \, .
 \end{equation}
\end{Def}
Notice that the following conservation of energy properties hold, as for~(\ref{continuitytransportHS}):
\begin{equation}\label{continuitytransport}
\begin{aligned}|{\mathbf  H}_s (t) g_s |_{\eps,s,\beta} = | g_s |_{\eps,s,\beta}\quad &\mbox{and} \quad 
\| {\mathbf  H} (t) G_N \|_{ \eps,\beta, \mu } = \| G_N \|_{ \eps,\beta, \mu}\,,
\end{aligned}
\end{equation}
for all parameters~$\beta >0,$ $\mu \in \R,$ and for all~$g_s \in X_{\eps,s,\beta},$   $G_N = (g_s)_{1 \leq s \leq N} \in {\bf X}_{\eps,\beta,\mu}$, and all $t \geq 0.$


\section{Continuity estimates} \label{sec:cont-est}
\setcounter{equation}{0}

We now establish bounds, in the above defined functional spaces, for the collision operators defined in~\eqref{css+k}-\eqref{def:g}, and for the total collision operator ${\bf C}_N$ defined in~\eqref{bH-def}.  

Notice that in the case when~$m=1$ the estimates are the same as in Chapter~\ref{existence}: in particular thanks to~(\ref{continuitytransport}) the following bound holds:
\begin{equation} \label{est:C1pot}
 e^{s (\mu_0-\lambda t)} \Big| \int_0^t {\bf H}_s(t-\tau) \C_{s,s+1} g_{s+1}(\tau) \, d \tau\Big|_{\e,s,\b_0-\lambda t} \leq \bar c(\b_0,\mu_0, \lambda ,T)  |\! \| G_N| \! \| _{ \eps,{\boldsymbol\beta},{\boldsymbol\mu}}\, ,
 \end{equation}
 for all $G_N = (g_{s+1})_{1 \leq s \leq N} \in {\bf X}_{\e,{\boldsymbol\beta},{\boldsymbol\mu}},$ with   $ \bar c(\b_0,\mu_0, \lambda ,T)$ computed explicitly in~{\rm(\ref{defbarcHS})}.

The following statement is the analogue of Proposition~\ref{propcontinuityclustersHS} in the hard spheres case, but in the present situation higher order correlations must be taken into account.
\begin{prop}\label{propcontinuityclusters}
Given $\b > 0$ and $\mu \in \R,$ for $m \geq 1$ and $1 \leq s \leq N-m,$ the collision operators~$\C_{s,s+m}$ satisfy the bounds, for all $G_N = (g_s)_{1 \leq s \leq N} \in { \bf X}_{\e,\b,\mu},$ 
\begin{equation} \label{est-col:1}
 \big| {\mathcal C}_{s,s+m} g_{s+m} (Z_s)\big| \leq \e^{m-1} C_{d}  e^{m \k_d}  (\b/C_{d} )^{-\frac{md}2} \Big( s\b^{-\frac12} + \sum_{1 \leq i \leq s} |v_i|\Big)  e^{- \b E_\e(Z_s)}| g_{s+m}|_{\e,s+m,\b} \, ,
 \end{equation}
for some $C_{d} > 0$ depending only on $d.$ 

If $\e < C_d e^{\mu} \beta^{\frac d2},$ then for all~$0< \beta' < \beta$ and $\mu' < \mu,$ the total collision operator ${\bf C}_N$ satisfies the bound 
\begin{equation}
\label{continuity}
\| {\mathbf C} _N G_N \|_{\eps,\beta', \mu'} \leq  C_d (1 + \b^{-\frac12}) \Big( \frac1{\b - \b'} + \frac1{\mu - \mu'} \Big)
  \| G_N \|_{\eps,\beta,\mu} \, .
\end{equation}
\end{prop}

Considering the case $m > 1$ in \eqref{est-col:1}, for which the upper bound is $O(\e),$ we see that higher-order interactions are negligible in the Boltzmann-Grad limit (provided~(\ref{est-col:1},) can be summed over~$m$, which is possible for~$\eps$ small enough). 

\begin{proof} We shall only consider the case $m \geq 2,$ as the case~$m=1$ is dealt with exactly as in the proof of Proposition~\ref{propcontinuityclustersHS}. From the definition of $G^{(m-1)}_{\langle i,s+1\rangle}$ in \eqref{def:g}, we obtain 
$$  
   \big| G^{(m-1)}_{\langle i,s+1\rangle}(g_{s+m})(Z_{s+1})   \big| \leq | g_{s+m}|_{\e,s+m,\b} \int_{\Delta_{m-1}(x_{s+1}) \times \R^{d(m-1)}} \exp\big(-\b E_\e(Z_{s+m})\big) dZ_{(s+2,s+m)} \, ,
$$ 
  where the norm $|\cdot |_{\e,s,\b}$ is defined in \eqref{norm:e-bpot}, and the Hamiltonian $E_\e$ is defined in \eqref{def:hamiltonian}. For the collision operator defined in \eqref{css+k}, this implies the bound 
\begin{equation}\label{pointwisevelocity}
  | {\mathcal C}_{s,s+m} g_{s+m}(Z_s) |\leq   m C_{N-s }^{m}  | g_{s+m} |_{\eps,s+m,\beta} \times \sum_{1 \leq i \leq s} I_{i,m}(V_s) \times J_{i,m}(X_s) \,,
  \end{equation}
  where $I_{i,m}$ is the velocity integral
  $$ I_{i,m}(V_s) := \int_{\R^{dm}} \big( |v_{s+1}| + |v_i|\big) \exp\Big(-\frac\beta2 \sum_{j=1}^{s+m} |v_j|^2 \Big) dV_{(s+1,s+m)} \, ,$$
  and $J_{i,m}$ is the spatial integral
  $$ J_{i,m}(X_s) := \int_{S_\e(x_i) \times \Delta_{m-1}(x_{s+1})} \exp\Big(-\beta \sum_{1 \leq j < k \leq s+m} \Phi_\e(x_j - x_k)\Big) d\s(x_{s+1}) dX_{(s+2,s+m)} \, .$$
The velocity integral is a product of Gaussian integrals and can be exactly computed, as in the hard-spheres case:
 \begin{equation} \label{gaussians} 
 I_{i,m}(V_s) \leq ( \b/C_{d})^{-\frac{md}2}\Big(  |v_i| +\beta^{-\frac12}\Big) \exp\Big(-\frac\beta2\sum_{1 \leq j \leq s} |v_j|^2 \Big) \, .
 \end{equation}
For the spatial integral, there holds 
$$ \begin{aligned}
 J_{i,m}(X_s) & \leq  \exp \Big(-\beta  \sum_{1\leq j<k\leq s}\Phi_\eps (x_j-x_k) \Big) |S_\eps(x_i)| \times  \sup_{x} \int_{\Delta_{m-1} (x )} dX_{(1,m-1)} \\
 &\leq \exp \Big(-\beta  \sum_{1\leq j<k\leq s}\Phi_\eps (x_j-x_k) \Big) \times \k_d \eps^{d-1} \times \Big( (m-1)! \, \eps^{(m-1)d} \exp(m \k_d)\Big)\,, 
 \end{aligned}$$
where in the last bound we used identity \eqref{cluster-est1} from Lemma \ref{estimationsgrandcanoniques} with $s = 1$ and $\zeta = \e^{-d}.$
This implies
$$
 \begin{aligned} | {\mathcal C}_{s,s+m} g_{s+m} (Z_s) | \leq C_{d}  \e^{m-1} & \big( (N-s) \e^{d-1} \big)^m e^{m \k_d}( \b/C_{d})^{-\frac{md}2} \Big( s \b^{-\frac12} + \sum_{1 \leq i \leq s} |v_i|\Big) \\ & \quad \times e^{- \b E_\e(Z_s)} | g_{s+m}|_{\e,s+m,\b}\,.\end{aligned}
 $$
In the Boltzmann-Grad scaling~$N \e^{d-1} \equiv 1,$ this gives \eqref{est-col:1}. Above and in the following, $C_{d}$ denotes a positive constant which depends only on $d,$ and which may change from line to line. 

We turn to the proof of \eqref{continuity}, which is similar to the proof of~(\ref{continuityHS}) up to the control of higher correlations. From the pointwise inequality~(\ref{pointwiseHS}) we deduce for the above velocity integral $I_{i,m}(V_s)$ the bound, for $0 < \b' < \b,$ 
$$ 
\sum_{1 \leq i \leq s} \exp\Big( (\b'/2) \sum_{1 \leq j \leq s} |v_j|^2\Big) I_{i,m}(V_s) \leq 
C_{d} ( \b/C_{d})^{-\frac{md}2}\Big( s \b^{-\frac12} + s^{\frac12}(\b - \b')^{-\frac12}\big) \, .
$$
  From the above bound in $J_{i,m}(X_s),$ we deduce immediately, for $0 < \b' < \b,$ 
  $$ \max_{1 \leq i \leq s} \exp\Big(\b' \sum_{1 \leq j < k \leq s} \Phi_\e(x_j - x_k)\Big) J_{i,m}(X_s) \leq \k_d (m-1)! \, e^{m \k_d} \e^{md - 1}\,.$$ 
    With \eqref{pointwisevelocity}, these bounds yield, in the Boltzmann-Grad scaling,
 $$ \begin{aligned}
 e^{\b' E_\e(Z_s) + \mu's} \big| \C_{s,s+m} g_{s+m} (Z_s)\big|  & \leq \e^{m-1} C_{d}( \b/C_{d})^{-\frac{md}2} e^{m \k_d}  e^{\mu's}  \big(s \b^{-\frac12} + s^{\frac12} (\b - \b')^{-\frac12}\big) \\ & \qquad \times |g_{s+m}|_{\e,s+m,\b} \, .\end{aligned}$$
Summing over $m,$ we finally obtain, for ${\bf C}_N$ defined in \eqref{bH-def},
   $$ \begin{aligned}
  \| {\bf C}_N G_N \|_{\e,\b',\mu'} \leq C_d \| G_N \|_{\e,\b,\mu} & \sup_{1 \leq s \leq N} \Big( \big( s\b^{-\frac12} + s^{\frac12} (\b- \b')^{-\frac12}\big) e^{-(\mu - \mu') s} \Big) \\ & \qquad \times \sum_{1 \leq m \leq N-s} e^{-m(\mu - \k_d)} \e^{m-1} \ ( \b/C_{d})^{-\frac{md}2} \, . 
  \end{aligned}
  $$
  If $\e$ is small enough so that $\e e^{\k_d - \mu} (C_{d}/\b)^{d/2} < 1,$ then the above series is convergent, and
$$ \sum_{1 \leq m \leq N-s} e^{-m(\mu - \k_d)} \e^{m-1} (C_{d}/\b)^{md/2} \leq \frac{e^{\k_d - \mu} (C_{d}/\b)^{d/2}}{1 - \e e^{\k_d - \mu} (C_{d}/\b)^{d/2}} \, \cdotp$$
We conclude as in the proof of Proposition \ref{propcontinuityclustersHS}. 
Proposition \ref{propcontinuityclusters} is proved.
\end{proof}

\begin{Rmk} We do not use the extra decay provided by the contribution of the potential in the exponential of the Hamiltonian. This is quite obvious in the bound for $J_{i,m}(X_s)$ in the proof of Proposition {\rm \ref{propcontinuityclusters},} where we bound $e^{-\beta \sum_{1 \leq j < k \leq s+m} \Phi_\e(x_j - x_k)}$ by $e^{-\beta \sum_{1 \leq j < k \leq s} \Phi_\e(x_j - x_k)}.$ Then, we might be tempted to replace $E_\e$ by the free Hamiltonian $E_0$ in the definition of the functional spaces. The kinetic energy, however, is not a conserved quantity, so that in $X_{0,s,\b}$ there is no analogue of~{\rm(\ref{continuitytransport})}. 
\end{Rmk} 

This leads to the following lemma, which is the key to the proof of the uniform bound stated in Theorem {\rm \ref{existence-thm}} in the next paragraph. It is the analogue of Lemma \ref{lemukai}.
 \begin{lem}\label{lemukaipotential}
 Let~$\beta_0>0$ and~$\mu_0 \in \R$ be given.  There is~$T>0$ depending only on~$\beta_0$ and~$\mu_0$ such that for an appropriate choice of~$\lambda $ in~$ (0, \beta_0 /T )$, there holds for all~$t \in [0,T]$
\begin{equation} \label{est:C1bispotential}
    \Big\| \int_0^t {\mathbf  H}(t-\tau) {\mathbf C}_NG_N   (\tau) \, d \tau\Big\|_{\e,\b_0-\lambda t,\mu_0-\lambda t} \leq \frac12 \,  |\! \| G_N| \! \| _{ \eps,{\boldsymbol\beta},{\boldsymbol\mu}}\, .
  \end{equation}
\end{lem}
 \begin{proof} 
We follow closely the proof of Lemma \ref{lemukai}. The difference is that here we take into account higher-order collision operators $\C_{s,s+m},$ with $m \geq 2.$ 
Using notation~(\ref{notationweights}),
 Estimate \eqref{est-col:1} from Proposition \ref{propcontinuityclusters} gives
 $$
  \begin{aligned}
 & e^{\beta_0^\lambda(t)E_\e(Z_s)}  \big| {\mathcal C}_{s,s+m} g_{s+m} (t',Z_s)\big| \\ & \leq \e^{m-1} C_d  e^{m \k_d} (2\pi/\beta_0^\lambda(t'))^{md/2} | g_{s+m}(t')|_{\e,s+m,\beta_0^\lambda(t')} \Big( s(\beta_0^\lambda(t'))^{-d/2} + \sum_{1 \leq i \leq s} |v_i|\Big)  e^{\lambda(t'-t) E_\e(Z_s)} \,.
 \end{aligned}
 $$
 Using also~(\ref{estimategs+1t'}) with~$s+1$ replaced by~$s+m$, we get
 \begin{equation} \label{min2}
 \begin{aligned} & \Big\| \int_0^t {\bf H}(t-t') {\bf C}_N G_N(t') \, dt'\Big\|_{\e,\beta_0^\lambda(t),\mu_0^\lambda(t)} \\ &\qquad \leq   |\! \| G_N| \! \| _{ \eps,{\boldsymbol\beta},{\boldsymbol\mu}} \Big(\sum_{1 \leq m \leq N-s}  C_m\Big) \sup_{Z_s \in \R^{2ds}} \int_0^t \overline C(t,t',Z_s) \, dt'\,, \end{aligned}
 \end{equation}
 where
 $  C_m := C_d \e^{m-1} e^{m (\k_d -\mu_0^\lambda(T))} (C_d/\beta_0^\lambda(T))^{md/2},$
 and $  \overline C$ is defined in \eqref{def:C2} and satisfies \eqref{bd:tildeC2} which we recall here:
  \begin{equation} \label{bd:tildeC2pot}
 \sup_{Z_s \in \R^{2ds}} \int_0^t \overline C(\tau,t ,Z_s) \, d \tau \leq  \frac{C_d}\lambda \Big( 1 +  \frac 1 {   (\beta_0^\lambda(T))^{d/2} } \Big) \,,
  \end{equation}
Under the assumption that 
\begin{equation}\label{estimateeps0}
\e_0 e^{\k_d -\mu_0^\lambda(T)} (2\pi/\beta_0^\lambda(T))^{d/2} < 1 /2\, ,
\end{equation}
 we find
  \begin{equation} \label{bd:tildeC1} \sum_{1 \leq m \leq N-s} C_m \leq2 C_d e^{- \mu_0^\lambda(T)}(\beta_0^\lambda(T))^{-d/2}  \, .
 \end{equation}
The upper bounds in \eqref{bd:tildeC2pot}  and \eqref{bd:tildeC1}  are independent of $s,$ and their product is equal to~$2\bar c( \b_0,\mu_0,\lambda,T).$ It then suffices to choose~$\lambda$ so that~$2\bar c( \b_0,\mu_0,\lambda,T) \leq 1/2$ and taking the supremum in $s$ in \eqref{min2} then yields the result.
\end{proof}  


 \section{Uniform bounds for the BBGKY and Boltzmann hierarchies} \label{sec:continuity-col-B}  
The results of the previous section enable us, exactly as in the hard spheres case page~\pageref{existence-thmHS}, to deduce directly the following bounds on the BBGKY hierarchy defined in~(\ref{BBGKY-mild-total}) page~\pageref{BBGKY-mild-total}. 
\begin{Thm}[Uniform bounds for the BBGKY hierarchy]\label{existence-thm}
  Let~$\beta_0 >0$ and~$\mu_0 \in \R $ be given. There is a time~$T>0$  as well as
 two nonincreasing functions~${\boldsymbol\beta }>0$ and~${\boldsymbol\mu }$ defined on~$[0,T]$, satisfying~${\boldsymbol\beta }(0) = \beta_0 $ and~${\boldsymbol\mu}(0) = \mu_0 $, 
such that {in the Boltzmann-Grad scaling~$N \eps^{d-1} \equiv 1$}, any family of initial marginals $ \widetilde  F_N(0) =\big( \widetilde  f_N^{(s)} (0) \big)_{1 \leq s \leq N}$  in~${\bf X}_{\eps,\beta_0,\mu_0}$ gives rise to a unique solution~$ \widetilde  F_N (t)= (  \widetilde f_N^{(s)} (t))_{1 \leq s \leq N}$  in~${\bf X}_{\e,{\boldsymbol\beta},{\boldsymbol\mu}}$ to    the BBGKY hierarchy~{\rm(\ref{BBGKY-mild-total})}    satisfying  the following bound:
         $$
  | \! \| \widetilde  F_N   | \! \| _{ \e,{\boldsymbol\beta},{\boldsymbol\mu}}
\leq  2  \|{  \widetilde F_N(0) }\|_{\eps,
      \b_0,\mu_0} \, .$$
    \end{Thm}
    In the case of the Boltzmann hierarchy associated with the collision operator~(\ref{Boltzmann-opscatfacteurdimpactchapter}), the same existence result as in Theorem~\ref{existence-thmboltzmannHS} holds, again with the same proof.  
\begin{Thm}[Existence for the Boltzmann hierarchy]\label{existence-thmboltzmann}
Assume the potential satisfies Assumption~{\rm\ref{propertiesphi}}.  Let~$\beta_0 >0$ and~$\mu_0 \in \R $ be given. There are a time~$T>0$  as well as
 two nonincreasing functions~${\boldsymbol\beta }>0$ and~${\boldsymbol\mu }$ defined on~$[0,T]$, satisfying~${\boldsymbol\beta }(0) = \beta_0 $ and~${\boldsymbol\mu}(0) = \mu_0 $,
such that  any family of initial marginals $F(0) =\big(f^{(s)} (0) \big)_{ s \geq 1} $ in~$ {\bf X}_{0,\beta_0,\mu_0}$  gives rise to a unique solution~$   F  (t)= (   f^{(s)} (t))_{s \geq 1}$  in~${\bf X}_{0,{\boldsymbol\beta},{\boldsymbol\mu}}$ to      the Boltzmann hierarchy~{\rm (\ref{mild-Boltzmannbis})}, satisfying  the following bound:
         $$
   | \! \|   F    | \! \| _{ 0,{\boldsymbol\beta},{\boldsymbol\mu}}\leq  2  \|{   F(0) }\|_{0,
      \b_0,\mu_0} \, .$$
    \end{Thm}

%% file: Convergence.tex
\chapter{Convergence result and strategy of proof}\label{convergence}\label{Convergence}
\setcounter{equation}{0}

 The main goal of this chapter is to 
reduce the proof of Theorem~\ref{thm-cv-intro} stated page~\pageref{thm-cv-intro}  to the term-by-term convergence of some  functionals involving a finite (uniformly bounded) number of marginals  with only first-order collisions, bounded energies and a finite number of collision times, exactly as was performed 
in Chapter~\ref{convergenceHS} (see Section~\ref{finite-m}). 

Before doing so we define, as in the hard spheres case, the notion of admissible initial data in Section~\ref{admissible-pot}.  We  give the precise version of  Theorem~\ref{thm-cv-intro}  in Section~\ref{potentialboltz}.

\section{Admissible initial data}\label{admissible-pot}
\setcounter{equation}{0}

The characterization of admissible initial data  is   very similar  to the  hard spheres case studied in Paragraph~\ref{sec:adm-B}. The only new aspect concerns the fact that marginals have been truncated, and that feature will be dealt with in this section.
\begin{Def}[Admissible Boltzmann data] \label{def:adm-datapot}
Admissible Boltzmann data
are defined as families~$F_0 = (f^{(s)}_0)_{s \geq 1},$ with each~$f_0^{(s)}$  nonnegative, integrable and continuous over~$\Omega_s$, such that
\begin{equation} \label{limit:marginalspot}
 \int_{\R^{2d}} f^{(s+1)}_0(Z_s,z_{s+1})  \, dz_{s+1} = f^{(s)}_0(Z_s) \, , 
\end{equation}
and which are limits of BBGKY initial data~$\widetilde F_{0,N} = (\widetilde f^{(s)}_{0,N})_{1 \leq s \leq N}  \in {\bf X}_{\e,\b_0,\mu_0}$ in the following sense: it is assumed that
\begin{equation} \label{def:unif-datapot} \begin{aligned}
 \sup_{N \geq 1} \| \widetilde F_{0,N} \|_{\e,\b_0,\mu_0} < \infty\,,\quad \mbox{for some $\b_0  > 0\,,$ $\mu_0 \in \R\,,$ as $N \e^{d-1} \equiv 1\,,$ } \end{aligned}
 \end{equation}
and that for each given~$s \in [1,N]$, the truncated marginal of order~$s$ defined by
 \begin{equation} \label{def:adm-truncpot} 
  \widetilde f_{0,N}^{(s)}(Z_s)  = \int_{  \R^{2d(N-s)}}  \indc_{\D^s_N}(X_N)   f_{0,N}^{(N)} (Z_N) dZ_{(s+1,N)}\,, \quad  1 \leq s < N\,, 
\end{equation}
converges in the Boltzmann-Grad limit:
\begin{equation} \label{def:cv-datapot} \widetilde f_{0,N}^{(s)} \longrightarrow f_0^{(s)} \, , \quad \mbox{ as $N \to \infty$ with $N \e^{d-1} \equiv 1 \, ,$ locally uniformly in $\Omega_s \, .$}
 \end{equation}
   \end{Def}
 The following result is proved very similarly to Proposition~\ref{cor:B-data}.
\begin{prop} \label{cor:B-datapot} The set of admissible Boltzmann data, in the sense of Definition~{\rm\ref{def:adm-datapot}},  is the set of families of marginals $F_0$ as in~\eqref{limit:marginalspot} satisfying a uniform bound~$\| F_0 \|_{0,\b_0,\mu_0} < \infty$ for some~$\beta_0>0$ and~$\mu_0 \in \R$.
\end{prop}
We shall not give the proof of that result, as the only difference with Proposition~\ref{cor:B-data} lies in the presence of a truncation in the marginals, whose effect disappears asymptotically as stated in the following lemma. 

\begin{lem} \label{lem:trunc-untrunc} Given $\widetilde F_{0,N} = (\widetilde f^{(s)}_{0,N})_{1 \leq s \leq N}$ satisfying \eqref{def:unif-datapot} and \eqref{def:adm-truncpot} from Definition~{\rm \ref{def:adm-datapot}}, with associated family $  F_{0,N} = (  f^{(s)}_{0,N})_{1 \leq s \leq N}$ of untruncated marginals: 
 \begin{equation} \label{def:lem-untrunc}  
  f_{0,N}^{(s)}(Z_s) = \int_{  \R^{2d(N-s)}} f_{0,N}^{(N)} (Z_N) \,  dZ_{(s+1,N)}\,, \quad  1 \leq s < N \, , \quad Z_s \in \Omega_s\,, \qquad \widetilde f^{(N)}_{0,N} = f^{(N)}_{0,N}\,,
 \end{equation}
there holds the convergence
 $$ f^{(s)}_{0,N} -\widetilde f^{(s)}_{0,N} \longrightarrow 0 \, ,\qquad \mbox{ for fixed $s \geq 1 \, ,$ as $N \to \infty$ with $N \e^{d-1} \equiv 1 \, ,$ uniformly in $\Omega_s \, .$}$$
\end{lem}

\begin{proof}We apply identity \eqref{cluster-dec} from Lemma \ref{estimationsgrandcanoniques} to $f^{(N)}_{0,N},$ and obtain after integration in the velocity variables
 \begin{equation} \label{f-tilde-f}
\begin{aligned}
f_{0,N}^{(s)}(Z_s) - \widetilde f_{0,N}^{(s)}(Z_s) = \sum_{m=1}^{N-s} C^m_{N-s}\int_{\Delta_m(X_s) \times \R^{dm}} \widetilde f_{0,N}^{(s+m)} (Z_{s+m}) dZ_{(s+1,s+m)} \, .
\end{aligned}
\end{equation}
 Then, denoting $C_0 = \displaystyle \sup_{M \geq 1} \| F_{0,M}\|_{\e,\b_0,\mu_0},$ a finite number by assumption, from
  $$ \begin{aligned} f^{(s+m)}_{0,N}(Z_{s+m}) & \leq \exp\big(- \mu_0 (s+m) -\b_0 E_\e(Z_{s+m})\big) C_0 \\ & \leq \exp\Big( -\mu_0 (s+m) - (\b_0/2) \sum_{1 \leq i \leq s} |v_i|^2\Big) C_0\,,
  \end{aligned}$$ we deduce, first by integrating the velocity gaussians and then by using the cluster bound \eqref{cluster-est1} in Lemma \ref{estimationsgrandcanoniques} with $\zeta = \eps^{-d},$ the bound
 $$ \begin{aligned} \int_{\Delta_m(X_s) \times \R^{dm}} f_{0,N}^{(s+m)} (Z_{s+m}) dZ_{(s+1,s+m)} & \leq (C_d/\b_0)^{md/2} e^{-\mu_0(s + m)} C_0 \int_{\Delta_m(X_s)} dX_{(s+1,s+m)} \\ & \leq  m!  (C_d/\b_0)^{md/2}\e^{md}  e^{(\k_d - \mu_0)(s + m)} C_0 \, . \end{aligned}$$
If $2 \e e^{\k_d - \mu_0} (C_d/\b_0)^{d/2} < 1,$ then 
$$  \sum_{m=1}^{N-s} C^m_{N-s}  m!  (C_d/\b_0)^{md/2}\e^{md}  e^{(\k_d - \mu_0)(s + m)} \leq  e^{(\k_d - \mu_0)s} \sum_{m\geq 1} \big(2 \e e^{\k_d - \mu_0} (C_d/\b_0)^{d/2}\big)^m \longrightarrow 0$$
as $\e \to 0,$ implying $f_{0,N}^{(s)} -\widetilde  f_{0,N}^{(s)} \longrightarrow 0$ for fixed $s,$ uniformly in $\Omega_s.$ 
\end{proof}
\begin{Rmk} \label{rem:for-cor} We can reproduce the above proof in the case of a time-dependent family of bounded marginals, i.e., $F_N \in {\bf X}_{\e,{\boldsymbol\beta},{\boldsymbol\mu}},$ with $\displaystyle \sup_{N \geq 1}    |\! \| F_N| \! \| _{ \eps,{\boldsymbol\beta},{\boldsymbol\mu}}< \infty,$  with the notation of Definition~{\rm\ref{def:functional-spaces}}.
This gives uniform convergence to zero, in time $t \in [0,T]$ and in space $X_s \in \Omega_s,$ of the difference between truncated and untruncated marginals:~$\widetilde f^{(s)}_N - f^{(s)}_N \to 0.$ 
\end{Rmk}

We consider therefore  families of initial data:   Boltzmann initial data~$F_0 = (f_0^{(s)})_{s \in \N}$ such that
$$
\|  F_{0} \|_{0,\beta_0,\mu_0} = \sup_{s\in \N} \sup_{Z_s}\big( \exp(\beta_0 E_0 (Z_s) +\mu_0 s)   f_{0}^{(s)}(Z_s) \big)<+\infty
$$
and for each $N$,    BBGKY initial data $\widetilde F_{N,0} =(\widetilde f_{N,0}^{(s)})_{1 \leq s\leq N}$   such that
$$
\sup_N\| \widetilde F_{N,0} \|_{\eps,\beta_0,\mu_0} =\sup_N \sup_{s\leq N} \sup_{Z_s}\big( \exp(\beta_0 E_\eps(Z_s) +\mu_0 s) \widetilde f_{N,0}^{(s)}(Z_s) \big)<+\infty \, ,
$$
satisfying~(\ref{def:adm-truncpot}) and~(\ref{def:cv-datapot}). 
These give rise to a unique, uniformly bounded solution~$\widetilde F_N$ to the BBGKY hierarchy thanks to Theorem~\ref{existence-thm} page~\pageref{existence-thm}, and to  a unique solution~$F$ to the Boltzmann hierarchy thanks to  Theorem~\ref{existence-thmboltzmann} page~\pageref{existence-thmboltzmann}.


\section{Convergence to the Boltzmann hierarchy}\label{potentialboltz}
\setcounter{equation}{0}
  Our main result is the following.
 \begin{Thm}[Convergence]\label{main-thmpotential}
 Assume the potential satisfies Assumption~{\rm\ref{propertiesphi}} as well as~{\rm(\ref{strangeassumption})}.  Let~$\beta_0 >0$ and~$\mu_0 \in \R$ be given. There is a time~$T>0$ such that the following holds. For any admissible Boltzmann datum~$F_0$ in~${\mathbf X }_{0,\beta_0,\mu_0}$     associated with a family~$( \widetilde F_{0,N})_{N \geq 1}$ of   BBGKY data in~${\mathbf X }_{\eps,\beta_0,\mu_0}$,   the solution~$\widetilde F_N$ to the BBGKY hierarchy satisfies, in the sense of Definition~{\rm\ref{def:cv}},
 $$  \widetilde F_N \tildeto F   $$  
 uniformly on $[0,T]$,
 where~$F$ is the solution to the Boltzmann hierarchy with data~$F_0$.
 \end{Thm}
  
\begin{Cor}\label{cortotheoremcv}
 Assume the potential satisfies Assumption~{\rm\ref{propertiesphi}} as well as~{\rm(\ref{strangeassumption})}.  Let~$\beta_0 >0$ and~$\mu_0 \in \R$ be given. There is a time~$T>0$ such that the following holds. For any admissible Boltzmann datum~$F_0$ in~${\mathbf X }_{0,\beta_0,\mu_0}$     associated with a family~$( \widetilde F_{0,N})_{N \geq 1}$ of   BBGKY data in~${\mathbf X }_{\eps,\beta_0,\mu_0}$,   the associate family of untruncated marginals~$F_N$   satisfies
 $$    F_N \tildeto F \, , $$  
 uniformly on $[0,T]$,
 where~$F$ is the solution to the Boltzmann hierarchy with data~$F_0$.
   \end{Cor}

\begin{proof} By Proposition \ref{cor:B-datapot}, the family $F_0$ is an admissible Boltzmann datum. Denoting $\widetilde F_{0,N}$ an associated BBGKY   datum, let $T > 0$ be a   time for which the solution  the BBGKY hierarchy $\widetilde F_N$ with datum~$\widetilde F_{0,N}$ has a uniform bound,  as given by Theorem \ref{existence-thm}. 

By Theorem \ref{main-thmpotential}, the convergence~$I_{\varphi_s}\big(\widetilde f^{(s)}_N - f^{(s)}\big) \to 0$ holds uniformly in $[0,T]$ and locally uniformly in $\Omega_s.$ 
 Then, by Lemma \ref{lem:trunc-untrunc} and Remark \ref{rem:for-cor}, there holds $f^{(s)}_N - \widetilde f^{(s)}_N \to 0,$ for fixed $s,$ uniformly in~$[0,T] \times \Omega_s.$ By Lemma \ref{lem:cv}, this implies $I_{\varphi_s}\big(  f^{(s)}_N -\widetilde f^{(s)}_N\big) \to 0,$ uniformly in $[0,T]$ and locally uniformly in $\Omega_s.$ 
 We conclude that $  f^{(s)}_N \tildeto f^{(s)},$ uniformly in $[0,T].$ \end{proof}

In the next paragraph we shall prove that in the   sum defining~$\widetilde  f_N^{(s)} (t) $ one can neglect all higher-order interactions and restrict our attention to the case when~$m_i = 1$ for each~$i \in [1,n]$ and each~$n \in \N$. Then we can, exactly as in the hard spheres case discussed in Chapter~\ref{convergenceHS}, consider only a finite number of collisions, and reduce the study to bounded energies and  well separated collision times.
 

\section{Reductions of the BBGKY hierarchy, and pseudotrajectories}\label{finite-m}
\setcounter{equation}{0}

In this paragraph, we  first prove that the estimates obtained in Chapter~\ref{chaptercluster}  enable us to reduce the  study of the BBGKY hierarchy to the equation
\begin{equation}
\label{BBGKY-mild2}
 \widetilde g_N^{(s)} (t,Z_s)=  {\bf H}_s(t  ) \widetilde f_{N}^{(s)}(0,Z_s) +   \int_0^t {\bf H}_s(t - \t)   {\mathcal C}_{s,s+1} \widetilde g_N^{(s+1)} (\t,Z_s)\, d\t \, , \qquad 1 \leq s \leq N-1 \, . 
 \end{equation}

Estimate \eqref{est-col:1} in Proposition \ref{propcontinuityclusters} shows indeed that higher-order collisions are negligible in the Boltzmann-Grad limit. For the solution to the BBGKY hierarchy, this translates as follows.

\begin{prop} \label{modified-h} Let~$\beta_0>0$ and~$\mu_0$ be given. Then with the same notation as Theorem~{\rm\ref{existence-thm}}, in the Boltzmann-Grad scaling $N \eps^{d-1} \equiv 1,$ any   family of initial marginals~$\widetilde F_N(0) =\big(\widetilde f_N^{(s)} (0) \big)_{1 \leq s \leq N}  $ in~$ {\bf X}_{\eps,\beta_0,\mu_0}$
gives rise to a unique solution~$\widetilde G_N \in {\bf X}_{\e,{\boldsymbol\beta},{\boldsymbol\mu}}$ of~{\rm(\ref{BBGKY-mild2})} and there holds the bound 
    $$\Nt{ \widetilde G_N}_{ {\eps,{\boldsymbol\beta},{\boldsymbol\mu}}} \leq 2 \| \widetilde F_N(0)\|_{\e,\b_0,\mu_0} \, .$$
    Besides, the solution $\widetilde G_N$ to the modified hierarchy \eqref{BBGKY-mild2} is asymptotically close to the solution $\widetilde F_N$ to the BBGKY hierarchy \eqref{BBGKY-mild-total}:
\begin{equation}\label{F1closetoF}
\Nt{\widetilde G_N - \widetilde F_N} _{  {\eps,{\boldsymbol\beta},{\boldsymbol\mu}}}  \leq 2 \e    \|  \widetilde  F_{N}(0) \|_{\e,\b_0,\mu_0} \, .
\end{equation}
 \end{prop}
 
\begin{proof} 
We find    the bound for $\widetilde G_N,$   in the same way as for Theorem \ref{existence-thm}. 
We turn to the proof of~\eqref{F1closetoF}. There holds
 $$ \begin{aligned}
\Nt{\widetilde G_N - \widetilde F_N} _{  {\eps,{\boldsymbol\beta},{\boldsymbol\mu}}}  & \leq
 \BigNt{
 \int_0^t \Big({\bf H}_s(t-t') \C_{s,s+1} (\widetilde g^{(s+1)}_N -\widetilde f_N^{(s+1)})(t')\Big)_{1 \leq s \leq N} \, dt' }_{ {\eps,{\boldsymbol\beta},{\boldsymbol\mu}}} \\
  & \quad + \BigNt{\int_0^t \Big( {\bf H}_s(t - t') \sum_{2 \leq m \leq N-s} \C_{s,s+m} f^{(s+m)}_N(t') \Big)_{1 \leq s \leq N} \, dt' }_{  {\eps,{\boldsymbol\beta},{\boldsymbol\mu}}}  \,.
 \end{aligned}$$
 With \eqref{est:C1pot}, this implies 
 $$ \begin{aligned}
 \Nt{\widetilde G_N - \widetilde F_N} _{  {\eps,{\boldsymbol\beta},{\boldsymbol\mu}}} \leq c_0\BigNt{ \int_0^t \Big( {\bf H}_s(t - t') \sum_{2 \leq m \leq N-s} \C_{s,s+m} f^{(s+m)}_N(t') \Big)_{1 \leq s \leq N} \, dt' }_{ {\eps,{\boldsymbol\beta},{\boldsymbol\mu}}}  \,,
 \end{aligned}$$
 with $c_0 := \big(1 - \bar c(\b_0,\mu_0,\lambda,T)\big)^{-1},$ which is indeed strictly positive by assumption. We conclude as in the proof of Proposition~\ref{propcontinuityclusters} and Lemma \ref{lemukaipotential}. \end{proof}

  One has the following formulation for~$   \widetilde g_{N}^{(s)} $ in  terms of the initial datum:
$$  \widetilde  g_N^{(s)} (t) =\sum_{k=0}^\infty  \int_0^t \int_0^{t_1}\dots  \int_0^{t_{k-1}}  {\bf H}_s(t-t_1) {\mathcal C}_{s,s+1}  {\bf H}_{s+1}(t_1-t_2)  \dots   {\bf H}_{s+k}(t_k)  \widetilde  f_N^{(s+k)}(0) \: dt_{k} \dots dt_1 \, .
$$
 We define the functional\label{observablepotential}
$$
\begin{aligned} I_{s}(t )(X_s)  := \sum_{k=0}^\infty \int  \varphi_s(V_s) &\int_{{\mathcal T}_{k}(t) }   {\bf H}_s(t-t_1) {\mathcal C}_{s,s+1} {\bf H}_{s+1}(t_1-t_2) {\mathcal C}_{s+1, s+2}\\
&  \dots {\mathcal C}_{s+k-1 , s+k}  {\bf H}_{s+k}(t_k- t_{k+1}) \widetilde f^{(s+k)}_{N,0}   dT_kdV_s
 \end{aligned}
 $$
and following Chapter~\ref{convergenceHS}, the reduced elementary functional
 \begin{equation}\label{defIsnRdelta}
\begin{aligned}
 I^{R,\delta}_{s,k}(t )(X_s)  :=  \int  \varphi_s(V_s) &\int_{{\mathcal T}_{k,\delta}(t) }   {\bf H}_s(t-t_1) {\mathcal C}_{s,s+1} {\bf H}_{s+1}(t_1-t_2) {\mathcal C}_{s+1, s+2}\\
&  \dots {\mathcal C}_{s+k-1 , s+k}  {\bf H}_{s+k}(t_k- t_{k+1}) \indc_{E_{\e}(Z_{s+k})\leq R^2} \widetilde f^{(s+k)}_{N,0}   dT_kdV_s \, .
 \end{aligned}
 \end{equation}

We can reproduce the proofs of Propositions~\ref{domcv}, \ref{delta1small} and~\ref{delta2smallHS}
 to obtain the following result, as in Corollary~\ref{finalresultHS}.
   \begin{prop}\label{delta2small}
With the notation of Theorem~{\rm\ref{existence-thm}}, given~$s \in \N^*$ and~$t \in [0,T]$, there are two positive constants~$C$ and~$C'$ such that  for all~$n \in \N^*$,
$$
 \big  \| I_{s} (t)  - \sum_{k=0}^n I_{s,k} ^{R,\delta} (t) \big \|_{L^\infty (\R^{ds})} \leq C \left(  2^{-n} + e^{-C'\beta_0R^2} + \frac{n^2}T\delta \right) \|\varphi\|_{L^\infty (\R^{ds})}
\|\widetilde F_{N,0}\|_{\eps,\beta_0,\mu_0} \, .$$
\end{prop}

As in the hard-spheres case, in the integrand of the collision operators ${\mathcal C}_{s,s+1}$ defined in \eqref{css+k}, we can distinguish between pre- and post-collisional configurations, as we decompose
$$
  {\mathcal C}_{s,s+1} =   {\mathcal C}^{+}_{s,s+1} -   {\mathcal C}^-_{s,s+1}
$$
where 
 \begin{equation} \label{css+1+}
 {\mathcal C}_{s,s+1}^\pm \widetilde g^{(s+1)} = \sum_{ m = 1}^{s}   {\mathcal C}^{\pm,m}_{s,s+1} \widetilde g^{(s+1)}
 \end{equation} 
 the index $m$ referring to the index of the interacting particle among the $s$ ``fixed" particles, with the notation
 $$\begin{aligned}
 \big({\mathcal C}^{\pm,m}_{s,s+1} \widetilde g^{(s+1)} \big)(Z_s) :=   (N-s ) \eps^{d-1}\int_{{\bf S}_1^{d-1}\times\R^d}
(\nu  \cdot ( v_{s+1}  - v_m) )_\pm \widetilde g^{(s+1)}  (Z_s,x_m+\eps \nu, v_{s+1}) \\
\times \prodetage{1 \leq j \leq s }{j \neq m} \indc_{|x_j-x_{s+1}| \geq \eps} \, d\nu dv_{s+1} \, ,
\end{aligned}$$
the index $+$ corresponding to post-collisional configurations and the index $-$ to pre-collisional configurations, according to terminology set out in  Chapter~\ref{scattering}.

\medskip
The elementary BBGKY observables we are interested in  can therefore be decomposed as 
 \begin{equation}\label{formulafE}
  \begin{aligned}
  I_{s,k}^{R,\delta}(t,X_s)  &=   \sum_{J,M} \Big( \prod_{i=1}^k j_i\Big) I_{s,k} ^{R,\delta} (t,J,M)(X_s) 
    \end{aligned}
\end{equation} 
  where the   elementary functionals  $I_{s,k}^{R,\delta} (t,J,M) $  are defined by
  $$
   \begin{aligned}
   I_{s,k} ^{R,\delta}  (t,J,M)(X_s) := \int  \varphi_s(V_s) \int_{{\mathcal T}_{k,\delta}(t) }   &{\bf H}_s(t-t_1){\mathcal C}_{s,s+1}^{j_1,m_1} {\bf H}_{s+1}(t_1-t_2) {\mathcal C}^{j_2,m_2}_{s+1, s+2}\\
   &\dots  {\bf H}_{s+k}(t_k- t_{k+1})\indc_{E_\eps(Z_{s+k})\leq R^2} \widetilde f^{(s+k)}_{N,0}   dT_kdV_s\,  ,
   \end{aligned}
 $$
  with
  $$
 J := (j_1,\dots,j_k) \in \{+,-\}^k \,  \,  \mbox{ and } \,  \, M := (m_1,\dots,m_k)  \,  \,  \mbox{ with } \,  \,m_i \in \{1,\dots , s+i-1\}\, .
$$
 
 \bigskip

As in the hard spheres case,  we still cannot study directly the convergence of~$I_{s,n} ^{R,\delta} (t,J,M) -I_{s,n} ^{0,R,\delta} (t,J,M) $ since the transport operators ${\bf H}_k$  do not coincide everywhere with the free transport operators ${\bf S}_k$, which means -- in terms of pseudo-trajectories --  that there are  recollisions. Note that, because the interaction potential is compactly supported,  recollisions happen only for characteristics such that
 there exist $i,j \in [1,k]$ with $i\neq j$, and $\tau >0$ such that
 $$|(x_i-\tau v_i)  -(x_j-\tau v_j)|\leq  \eps\,.$$

 We shall thus prove that these recollisions arise only for a few pathological pseudo-trajectories, which can be eliminated by additional truncations of the domains of integration. This is the goal of  Part IV, which deals with the hard-spheres and the potential case simultaneously.

%% file: geom-lem.tex
\chapter{Elimination of recollisions}\label{pseudotraj}
\setcounter{equation}{0}

This last part is the heart of our contribution. We prove the term-by-term convergence of the series giving the observables, both in the case of hard spheres and in the case of smooth hamiltonian systems.

We have indeed seen in Corollary~\ref{finalresultHS} (for the hard-spheres case) and Proposition~\ref{delta2small} (for the potential case) that the convergence of observables  reduces to the convergence to zero of the elementary functionals~$I_{s,k}^{R,\delta} -I_{s,k}^{0,R,\delta}  $, where~$I_{s,k}^{R,\delta}$ is defined in~(\ref{defIsnRdeltaHS}) in the hard-spheres case and  in~(\ref{defIsnRdelta}) for the potential case, and~$I_{s,k}^{0,R,\delta}  $ is defined  in~(\ref{defIsnRdeltaHS}). These functionals correspond to  dynamics
\begin{itemize}
\item 
 involving only a finite number $s+k$ of particles, 
 \item with bounded energies (at most $R^2$), 
 \item   such that the $k$ additional particles are adjoined through binary collisions, 
 \item at times separated at least by $\delta$.
 \end{itemize}
 
 \bigskip

 What we shall establish is that recollisions can occur only for very pathological pseudo-trajectories, in the sense that the velocities and impact parameters of the additional particles in the collision trees have to be chosen in small measure sets.
 
 We point out the fact that, even in the case of hard spheres, these bad sets are generally not of zero measure because they are built as non countable unions of zero measure sets. The arguments are actually very similar whatever the precise nature of the microscopic interaction.
 
 The only differences we shall see between the case of hard spheres and the case of smooth potentials are the following:
 \begin{itemize}
 \item the parametrization of collisions by the deflection angle is trivial in the case of hard spheres since it coincides exactly with the impact parameter;
 \item there is no time shift between pre-collisional and post-collisional configurations in the case of hard spheres since the reflection is instantaneous.
 \end{itemize}
 These two simplifications will enable us to obtain   explicit estimates on the convergence rate in the case of hard spheres. For more general interactions, this convergence rate can be expressed as an implicit function depending on the potential.

%

\section{Stability of good configurations by adjunction of collisional particles}\label{proposition-geometry}
\setcounter{equation}{0}

In this paragraph we momentarily forget the BBGKY and Boltzmann hierarchies, and focus on the study of  pseudo-trajectories. 

 \begin{Def}[Good configuration] For any constant~$c>0$, we denote by\label{indexdefGk} $\cG_k(c)$ the set of ``good configurations" of $k$ particles, separated by at least~$c$ through backwards transport: that is
 the set of  $(X_k,V_k) \in \R^{dk} \times B_R^k$ such that
 the image of $(X_k, V_k)$ by the backward free transport satisfies the separation condition
 $$\forall \tau\geq 0 ,\quad \forall i\neq j ,\quad  |x_i-x_j-\tau(v_i-v_j)| \geq c\,,$$
 in particular it  is never collisional.
 \end{Def}
  We recall that~$ B_R^k := \Big\{V_k \in \R^{dk} \, / \, |V_k| \leq R\Big\}$ and in the following we write~$B_R:=B_R^1$.

Our aim is to show that ``good configurations" are stable by adjunction of a collisional particle provided that the deflection angle and the velocity of the additional particle do not belong to a small pathological set. Furthermore the set to be excluded can be chosen in a uniform way with respect to the initial positions of the particles in a small neighborhood of any fixed ``good configuration".

\begin{nota} \label{ssmall} In all the sequel, given two positive parameters $\eta_1$ and $\eta_2$, we shall say that  $$\eta_1 \ll \eta_2 \,  \hbox{ if } \,\eta_1 \leq C \eta_2$$
for some large constant $C$ which does not depend  on any parameter.
\end{nota}
In the following we shall fix three parameters~$ a,\eps_0,\eta\ll 1$     such that
\begin{equation}\label{sizeofparameters}
 a \ll \eps_0 \ll \eta \delta \, .
\end{equation}
We recall that the parameter~$\delta$     scales like  time while we shall see that~$\eta$, like~$R$, scales like a velocity. The parameters~$a$ and~$\eps_0$, just like~$\eps$,  will have the scaling of a distance.

\begin{prop}\label{geometric-prop}
Let~$a,\eps_0,\eta\ll 1$ satisfy~{\rm(\ref{sizeofparameters})}. 
Given~$\overline Z_k \in \cG_k(\eps_0)$, there is a subset~$\cB_k(\overline Z_k)$ of~${\mathbf S}_1^{d-1} \times B_R$  of small measure: for some fixed constant~$C>0$ and some constant~$C(\Phi, \eta ,R)>0$,
\begin{equation} 
\label{pathological-size}
\begin{aligned}
\big |\cB_k(\overline Z_k) \big | & \leq Ck \Big(R  \eta^{d-1} + R^d \Big(
  \frac{a}{\eps_0} \Big)^{d-1} + R\Big(  \frac{\eps_0}{\delta} 
\Big)^{d-1} \Big) \\
&\quad \hbox{ in the case of hard spheres}\,\\
\big |\cB_k(\overline Z_k) \big |  &\leq C k \Big(R  \eta^{d-1} + C(\Phi,R, \eta)   \, R^d \Big(
  \frac{a}{\eps_0} \Big)^{d-1} +  C(\Phi, R, \eta )  \, R\Big(  \frac{\eps_0}{\delta} 
\Big)^{d-1} \Big) \\
&\quad \hbox{ in the case of a smooth interaction potential }\Phi \, , \\
 \end{aligned}
\end{equation}
and  such that good configurations close to $\overline Z_k$  are stable by adjunction of a collisional particle close to~$\bar x_k$ and not belonging to $\cB_k(\overline Z_k)$, in the following sense. 

Consider~$(\nu,v) \in ({\mathbf S}_1^{d-1} \times B_R) \setminus  \cB_k(\overline Z_k)$ and let~$Z_k$ be a configuration of $k$ particles such that~$V_k = \overline V_k$ and~$
| X_k- \overline X_k  | \leq a $. 

$ \bullet $ $ $ If~$\nu \cdot (v-\bar v_k)< 0$ then for all~$\eps > 0$ sufficiently small,
\begin{equation}\label{taugeq0precoll}
\forall \tau \geq 0 \, , \quad \left\{ \begin{aligned}
\forall i \neq j \in [1,k] \, , \quad  |  (x_i - \tau \bar v_i) - (x_j - \tau \bar v_j ) | \geq \eps \, ,\\
\forall j \in [1,k] \, , \quad |  (x_k + \eps \nu - \tau v) - (x_j - \tau \bar v_j ) | \geq \eps \, .
\end{aligned}
\right.
\end{equation}
Moreover after the time~$\delta$, the~$k+1$ particles are in a good configuration:
\begin{equation}\label{taugeqdelta2precoll}
(X_k - \delta \overline V_k   ,\overline V_k , x_k + \eps \nu - \delta v , v) \in \cG_{k+1} (\eps_0/2)\, .
\end{equation}

$ \bullet $ $ $ If~$\nu \cdot (v-\bar v_k)> 0$ then define for~$j \in [1,k-1]$ 
$$
(z_k^{\eps*},z^{\eps*}):=\sigma^{-1} \big (z_k,(x_k + \eps \nu , v) \big) \quad \mbox{and} \quad z_j^{\eps*}:=(x_j- \bar v_j, \bar v_j) $$
in the hard-spheres case, where~$\sigma$ is defined in~{\rm(\ref{defscatteringhardspheres})},
and
$$
(z_k^{\eps*},z^{\eps*}):=\sigma_\eps^{-1} \big (z_k,(x_k + \eps \nu , v) \big) \quad \mbox{and} \quad z_j^{\eps*}:=(x_j-t_\eps \bar v_j, \bar v_j) $$
in the potential case,
 where~$\sigma_\eps$ is
the scattering operator  as  in Definition~{\rm\ref{scatteringbbgky}} and where~$t_\eps <\delta $ denotes the scattering time between~$z_k$ and~$(x_k + \eps \nu , v)$.
Then  for all~$\eps > 0$ sufficiently small,
\begin{equation}\label{taugeq0postcoll}
\forall \tau \geq 0 \, , \quad
\left\{ \begin{aligned}
\forall i \neq j \in [1,k ] \, , \quad   |  (x_i^{\eps*}  - \tau v_i^{\eps*} ) - (x_j^{\eps*}  - \tau v_j^{\eps*} ) | \geq \eps \, ,\\
\forall j \in [1,k] \, , \quad  |   ( x^{\eps*}   - \tau  v^{\eps*} ) - (x_j^{\eps*}  - \tau v_j^{\eps*} )| \geq \eps \, .
\end{aligned}\right.\end{equation}
Moreover after the time~$\delta$, the~$k+1$ particles are in a good configuration:
\begin{equation}\label{taugeqdelta2postcoll}
\Big(X_k^{\eps*} -(\delta  -t_\eps) V_k^{\eps*}   ,  V_k^{\eps*}   , x^{\eps*}    -  (\delta  -t_\eps)  v^{\eps*}  , v^{\eps*} \Big) \in \cG_{k+1}(\eps_0/2) \, ,
\end{equation}
with~$t_\eps := 0$ in the hard-spheres case.
\end{prop}
The proof of the proposition may be found in Section~\ref{proofgeometric-prop}. It relies on some elementary geometrical lemmas, stated and proved in the next section. The first one describes the bad trajectories associated with  (free) transport. The other ones explain how they are modified by collisions, both in the case of hard spheres and in the case of smooth interactions.

\begin{Rem}\label{betterresult}
For the sake of simplicity, we have assumed in the statement of Proposition~{\rm\ref{geometric-prop}} that the additional particle collides with the particle numbered~$k$. Of course, a simple symmetry argument shows that an analogous statement holds if the new particle  is added close to any of the particles in~$Z_k$.

The proof of Proposition~{\rm\ref{geometric-prop}} shows that if~$Z_k = \overline Z_k$ then the factor~$\eps_0/2$ in~{\rm(\ref{taugeqdelta2precoll})} and~{\rm(\ref{taugeqdelta2postcoll})}  may be replaced by~$\eps_0$. The loss if~$Z_k \not= \overline Z_k$ comes from the fact that the set to be excluded has to be chosen in a uniform way with respect to the initial positions of the particles in a small neighborhood of $\overline X_k$.

\end{Rem}

\section{Geometrical lemmas}\label{geometry}

We first consider the case of two particles moving freely, and describe the set of velocities $v_2$ leading possibly to collisions (or recollisions).

   Here and in   the sequel, we   denote by~$K(w,y,\rho)$\label{defindexKyeta} the cylinder of origin $w \in \R^d$, of axis $y\in \R^d$ and radius $\rho>0$ and by~$B_{\rho}(y)$ the ball centered at~$y$ of radius~$\rho$.

 \subsection{Bad trajectories associated to   free transport}$ $
 
   \begin{Lem}\label{geometric-lem1}
 Given~$ \bar a>0$ satisfying~$\eps \ll \bar  a \ll \eps_0 $,
   consider $\bar x_1,\bar x_2$ in~$\R^d$ such that $| \bar x_1-\bar x_2|\geq \eps_0$,  and~$ v_1 \in B_R$.
 Then for any~$x_1 \in B_{\bar a}(\bar x_1 )$, any~$x_2 \in B_{\bar a}(\bar x_2 )$  and any $ v_2\in B_R$, the following results hold.
 
 $ \bullet $ $ $ If   $v_2 \notin K(v_1, \bar x_1-\bar x_2, 6 R\bar a /\eps_0)$, then
  $$
   \forall \tau\geq 0 \,, \quad |(x_1-v_1 \tau)-(  x _2-v_2\tau)| >\eps \, ;
   $$
 $ \bullet $ $ $ If  ~$v_2 \notin K(v_1, \bar x_1-\bar x_2, 6 \eps_0/\delta)$
   $$
   \forall \tau\geq \delta  \,, \quad |(x_1-v_1 \tau)-(  x _2-v_2\tau)| >\eps_0 \, .
 $$
 \end{Lem}

   \begin{proof}
    $ \bullet $ $ $Assume that there exists $\tau_*$ such that
$$ |(x_1-v_1 \tau_*)-(  x _2-v_2\tau_*) |\leq \eps\,.$$
Then, by the triangular inequality and provided that $\eps$ is sufficiently small,
$$
|(\bar x_1 -\bar x_2) -  \tau_*(v_1 -v_2)| \leq    \eps+2\bar a \leq 3\bar a \,.
 $$
This means that $( v_1 -v_2)$ belongs to the cone of vertex 0 based on the ball centered at~$\bar x_1 -\bar x_2$ and of radius $  3\bar a $,   which is a cone of solid angle $( 3\bar a /|\bar x_1 -\bar x_2|)^{d-1}$ (since $\bar a \ll \eps_0$).

The intersection of this  cone and of the sphere of radius $2R$  is obviously embedded in the cylinder of axis $\bar x_1 -\bar x_2$ and radius $6R\bar a /\eps_0$, which proves the first result.

 $ \bullet $ $ $ Similarly   assume that there exists $\tau^* \geq \delta $ such that
$$|(x_1-v_1 \tau_*)-(  x _2-v_2\tau_*) |\leq \eps_0 \,.$$
Then, by the triangular inequality again,
$$|(\bar x_1 -\bar x_2) -  \tau_*(v_1 -v_2) |\leq \eps_0+2\bar a    \leq 3  \eps_0 \,.$$
In particular, for any unit vector $n$ orthogonal to $\bar x_1-\bar x_2$,
$$\tau^* | n\cdot (v_1-v_2) |=|n\cdot \left( (\bar x_1 -\bar x_2) -  \tau_*(v_1 -v_2)\right) |   \leq 3 \eps_0 \, .$$
This  tells us exactly  that $v_1-v_2$ belongs to the cylinder of axis $\bar x_1 -\bar x_2$ and radius $3\eps_0/\delta$.
 
 The lemma is proved.
\end{proof}

\subsection{Modification of bad trajectories by hard sphere reflection}$ $

We now consider the case when particles 1 and 2 undergo a hard sphere collision before being transported, and look at impact parameters $\nu$ and velocities $v_2$ leading possibly to collisions (or recollisions).

 \begin{Lem}\label{geometric-lem2}
Consider   $\rho\ll R$, and $(y,w) \in \R^d \times B_R$. 
For any~$v_1 $ in~$ B_R$, define  
$$
\begin{aligned}
\cN^*( w,y,\rho) (v_1) :=\big\{ (\nu ,v_2 )\in {\mathbf S}_1^{d-1}\times  B_R  \,/\, & (v_2-v_1) \cdot \nu > 0 \, ,  \\
&   v^*_1\in K(w,y,\rho) \hbox{ or }  v^*_2 \in K(w, y,\rho) \big\}\, ,
\end{aligned}
$$
where
$$
v_1^{*} := v_1 - \nu \cdot ( v_1  -   v_2 ) \, \nu \quad \mbox{and} \quad v_2^{*} := v_2+ \ \nu \cdot ( v_1  -   v_2 ) \, \nu \, .
$$
Then
$$|\cN^*( w,y,\rho) (v_1) | \leq C_d R \rho^{d-1} \,,$$
where the constant $C_d$ depends only on the dimension $d$.
\end{Lem}

\begin{proof} 
Denote by 
$r =|v_1-v_2|=|v_1^*-v_2^*|$.
The reflection condition shows that, as $\nu$ varies in ${\mathbf S}_1^{d-1}$, the velocities~$v^*_1$ and $v^*_2$ range over a sphere of diameter $r$.

The solid angle of the intersection of such a  sphere with the cylinder $ K(w,y,\rho)$  is less than 
$$C_d \min\left(1, \left( {\rho\over r}\right)^{d-1}\right) $$
which implies that 
$$
\begin{aligned}
\Big | \big \{ (\nu, v_2) \,/\,  v_1^*\in K(w,y,\rho) \hbox{ or }  v_2^* \in K(w, y,\rho) \big\} \Big | &\leq C_d \int r^{d-1} \min \left(1, \left( {\rho\over r}\right)^{d-1}\right) dr\\
& \leq C_d R \rho^{d-1} \,.
\end{aligned}
$$
This proves Lemma~\ref{geometric-lem2}.
\end{proof}

\subsection{Modification of bad trajectories by the scattering associated to $\Phi$}$ $

 The last  geometrical lemma   requires the use of notation coming from scattering theory, introduced in Chapter~\ref{scattering}: it states that if two particles~$z_1, z_2 $ in~$ \R^{2d}$ are in a post-collisional configuration and if~$v_1$ or~$v_2$ belong to a cylinder as in Lemma~\ref{geometric-lem1}, then the pre-image $z_2^*$ of~$z_2$ through the scattering operator belongs to a small set of~$\R^{2d}$.
 
 \begin{Lem}\label{geometric-lem3}
Consider  two parameters $\rho\ll R$  and~$\eta \ll 1$, and $(y,w) \in \R^d \times B_R$. 
For any~$v_1 $ in~$ B_R$, define  
$$
\begin{aligned}
\cN^*( w,y,\rho) (v_1) :=\big\{ (\nu ,v_2 )\in {\mathbf S}_1^{d-1}\times  B_R  \,/\, & (v_2-v_1) \cdot \nu > \eta \, ,  \\
&   v^*_1\in K(w,y,\rho) \hbox{ or }  v^*_2 \in K(w,y,\rho) \big\}\, ,
\end{aligned}
$$
where~$(\nu^*, v^*_1, v_2^*)=\sigma_0^{-1}(\nu,v_1,v_2)$ with the notations of Chapter~{\rm\ref{scattering}}.
Then
$$|\cN^*( w,y,\rho) (v_1) | \leq C (\Phi,R,\eta) R \rho^{d-1} $$
where the constant  depends on  the potential~$\Phi$ through the $L^\infty$ norm of the cross-section $b$ on the compact set~$B_{2R} \times [\eta/2R, \pi/2]$ defined in Chapter~{\rm\ref{scattering}}.
\end{Lem}

\begin{proof} 
Denote by 
$r =|v_1-v_2|=|v_1^*-v_2^*|$, and by $\omega$ the deflection angle.

From the proof of the previous lemma, we deduce that
$$
\begin{aligned}
\Big| \big \{ (\omega, v_2) \,/\,  v_1^*\in K(w,y,\rho) \hbox{ or }  v_2^* \in K(w,y,\rho) \big\} \Big | &\leq C_d \int r^{d-1} \min \left(1, \left( {\rho\over r}\right)^{d-1}\right) dr\\
& \leq C_d R \rho^{d-1}  \, .
\end{aligned}
$$
According to Chapter~\ref{scattering}, the change of variables
$ (\nu,v_1-v_2) \mapsto  (\omega,v_1- v_2)$
is a Lipschitz diffeomorphism away from $\nu \cdot (v_1-v_2) = 0.$ We therefore get the expected estimate.
 \end{proof}
 
\begin{Rem}
Note that the geometrical Lemmas~{\rm\ref{geometric-lem1}} to~{\rm\ref{geometric-lem3}}  consist in eliminating  sets in the  velocity variables and deflection angles only, and do not concern the position variables.
\end{Rem}


\section{Proof of the geometric proposition }\label{proofgeometric-prop} 
In this section we prove Proposition~\ref{geometric-prop}.
  We fix a good configuration~$\overline Z_k \in \cG_k(\eps_0)$, and we consider a configuration~$ Z_k \in \R^{2dk}$, with the same velocities as~$\overline Z_k$, and neighboring positions: $| X_k - \overline X_k | \leq a$. In particular we notice that
  for all~$ \tau\geq 0 $ and all~$ i\neq j$,
 \begin{equation}\label{almostgood}
  |x_i-x_j-\tau(\bar v_i-\bar v_j)| \geq  | \bar x_i-\bar x_j-\tau(\bar v_i-\bar v_j)| - 2a \geq \eps_0 /2
  \end{equation}
   since~$a \ll \eps_0$. This implies that~$ Z_k \in \cG_k(\eps_0/2)$.
   Next we consider an additional particle $(x_k+\eps \nu,v_{k+1})$  and we shall separate the analysis into two parts, depending on whether the situation is pre-collisional (meaning~$  \nu \cdot (v_{k+1}-\bar v_k) < 0$) or post-collisional (meaning~$  \nu \cdot (v_{k+1}-\bar v_k) >0$).

 \subsection{The pre-collisional case}\label{decloss}$ $
We assume that
 $$  \nu \cdot (v_{k+1}-\bar v_k) < 0\,,$$
meaning that~$(x_k + \eps\nu,v_{k+1})$ and~$z_k$ form a pre-collisional pair. In particular we have for all times~$\tau \geq 0$ and all~$\eps >0$
$$
\big | \big (
x_k+\eps \nu - v _{k+1}\tau
\big) - \big(
x_k - \bar v_k \tau
\big) \big | \geq \eps \, .
$$
Furthermore up to excluding the ball~$B_\eta(\bar v_k)$ in the set of admissible~$v_{k+1}$, we may assume that
$$
|v_{k+1}-\bar v_k| > \eta \, .
$$
Under that assumption we have for all~$\tau \geq \delta$ and all~$\eps>0$ sufficiently small,
$$
\begin{aligned}
\big | \big (
x_k+\eps \nu - v _{k+1}\tau
\big) - \big(
x_k - \bar v_k \tau
\big) \big | & \geq \tau |v_{k+1}- \bar v_k| - \eps \\
& \geq  {\delta  }\eta- \eps > \eps_0/2 \, . 
\end{aligned}
$$
Furthermore we know that~$Z_{k }$ belongs to~$\cG_{k }(\eps_0/2)$ thanks to~(\ref{almostgood}).  

\noindent Now let~$j \in [1,k-1]$ be given.  According to Lemma~\ref{geometric-lem1}, we
find that
for any $v_{k+1}  $ belonging to the set~$ B_R \setminus K(\bar v_j, \bar x_j-\bar x_k, 6 Ra /\eps_0+6 \eps_0/\delta) ,$
we have
$$ \forall \tau \geq0 \, , \quad |(x_{k }+ \eps \nu-v _{k+1} \tau)-(x_j-\bar v_j \tau)| >\eps \,,$$
and
$$ \forall \tau \geq \delta \, , \quad |(x_{k }+ \eps \nu-v_{k+1}  \tau)-(x_j-\bar v_j \tau)| >\eps_0 \,.$$
Notice that
 $$
 \Big|B_R \cap K(\bar v_j, \bar x_j-\bar x_k, 6 Ra /\eps_0+6 \eps_0/\delta)  \Big|  \leq  C \Big( R^d \Big(
  \frac{a}{\eps_0} \Big)^{d-1} + R\Big(  \frac{\eps_0}{\delta} 
\Big)^{d-1} \Big)\, .
 $$
Defining~$\displaystyle \cM^{-} (\overline Z_k):= \bigcup_{j \leq k-1} K(\bar v_j, \bar x_j-\bar x_k, 6 Ra /\eps_0+6 \eps_0/\delta)  $ and
$$
{\mathcal B}_k^-(\overline Z_k):= {\mathbf S}_1 ^{d-1} \times \Big(B_\eta(\bar v_k) \cup\cM^{-} (\overline Z_k)\Big)
$$
we find that
$$
\Big |{\mathcal B}_k^-(\overline Z_k) \Big | \leq   C k \Big(\eta^d + R^d \big({a\over\eps_0} \big)^{d-1} + R\big({\eps_0\over\delta} \big)^{d-1}\Big)
$$
and~(\ref{taugeq0precoll}) and~(\ref{taugeqdelta2precoll}) hold as soon as~$(\nu,v_{k+1} ) \notin {\mathcal B}_k^-(\overline Z_k)$.

   \subsection{The post-collisional case with hard sphere reflection}\label{gaindec-HS}$ $
   
   We now assume that
 $$  \nu \cdot (v_{k+1} -\bar v_k) > 0\,,$$
 meaning that~$(x_k + \eps\nu,v_{k+1})$ and~$z_k$ form a post-collisional pair. In particular, at time $\tau = 0+$, the configuration is changed and we have the pre-collisional pair $(x_k + \eps\nu,v_{k+1}^*)$ and~$(x_k, v_k^*)$ where $v_k^*$ and~$v_{k+1}^*$ are defined by the usual reflection condition. Furthermore, we have for all times~$\tau \geq 0$ and all~$\eps >0$
$$
\big | \big (
x_k+\eps \nu - v _{k+1}^*\tau
\big) - \big(
x_k -  v_k^* \tau
\big) \big | \geq \eps \, .
$$

We can then repeat the same arguments as inthe pre-collisional case replacing $\bar v_k, v_{k+1}$ by $v_k^*, v_{k+1}^*$.

   Excluding the ball~$B_\eta(\bar v_k)$ in the set of admissible~$v_{k+1}$, we find that 
   $$
\begin{aligned}
\big | \big (
x_k+\eps \nu - v^* _{k+1}\tau
\big) - \big(
x_k - v_k^* \tau
\big) \big | & \geq \tau |v_{k+1}- \bar v_k| - \eps \\
& \geq  {\delta  }\eta- \eps > \eps_0/2 \, . 
\end{aligned}
$$

According to Lemma~\ref{geometric-lem1}, if $v_k^*, v_{k+1}^*  $ belong to the set~$ B_R \setminus K(\bar v_j, \bar x_j-\bar x_k, 6 Ra /\eps_0+6 \eps_0/\delta) ,$
we have
$$
\begin{aligned}
 \forall \tau \geq0 \, , \quad& |(x_{k }+ \eps \nu-v _{k+1}^* \tau)-(x_j-\bar v_j \tau)| >\eps \,,\\
 & |(x_{k }-v _{k} ^*\tau)-(x_j-\bar v_j \tau)| >\eps \,,
 \end{aligned}
 $$
and
$$
\begin{aligned}
 \forall \tau \geq \delta \, , \quad &|(x_{k }+ \eps \nu-v_{k+1} ^* \tau)-(x_j-\bar v_j \tau)| >\eps_0\,\\
 &|(x_{k }-v_k^*  \tau)-(x_j-\bar v_j \tau)| >\eps_0 \,.
 \end{aligned}$$
 
 Combining Lemmas~\ref{geometric-lem1}and ~\ref{geometric-lem2}, we therefore obtain that
 (\ref{taugeq0precoll}) and~(\ref{taugeqdelta2precoll}) hold as soon as~$(\nu,v_{k+1} ) \notin {\mathcal B}_k^+(\overline Z_k)$ where 
 $$
{\mathcal B}_k^+(\overline Z_k):= {\mathbf S}_1 ^{d-1} \times B_\eta(\bar v_k) \cup  \bigcup_{j \leq k-1} \cN^*(\bar v_j, \bar x_j-\bar x_k, 6 Ra /\eps_0+6 \eps_0/\delta)(\bar v_k)\,.
$$
In particular,
$$
\Big |{\mathcal B}_k^+(\overline Z_k) \Big | \leq   C k \Big(\eta^d + R^d \big({a\over\eps_0} \big)^{d-1} + R\big({\eps_0\over\delta} \big)^{d-1}\Big)\,.
$$
 
    \subsection{The post-collisional case with smooth scattering}\label{gaindec}$ $
    
    In the case of a smooth  interaction potential, dealing with the post-collisional case is a little bit more intricate because of the time shift. Furthermore, using Lemma~\ref{geometric-lem3} instead of Lemma~\ref{geometric-lem2},   we   lose the explicit estimate for the bad set $ {\mathcal B}_k^+(\overline Z_k)$.

Let us first  define
  \begin{equation}\label{defcone}
 C(\overline Z_k):=\Big \{(\nu,v_{k+1} ) \in {\mathbf S}_{1}^d \times B_R \, , \,  \nu \cdot (v_{k+1} - \bar v_k)  \leq \eta  \Big\} \, ,
 \end{equation}
 which satisfies
 $$
 |C(\overline Z_k)| \leq CR \eta^{d-1}\,  . 
 $$
 Choosing~$(\nu,v_{k+1} ) \in  ({\mathbf S}_{1}^d \times B_R) \setminus C(\overline Z_k)$   ensures that the cross-section is well defined (see Definition~\ref{def:b}), and that the scattering time~$t_\eps$ is  of order $C(\Phi,R,\eta)\eps$ by Proposition~\ref{scatteringestimates}.
   
  Considering the formulas~(\ref{scattering-def1})
expressing~$(z^{\eps*}_{k},z^{\eps*}_{k+1} )$ in terms of~$\big(z_k,(x_{k }+\eps \nu, v_{k+1} )\big )$, we know    that
 \begin{equation}\label{distancex*x}
 \begin{aligned}
  | x_k ^{\eps*} -  x_k| & \leq  \frac12 | x_k ^{\eps*} -   x_{k+1}^{\eps^*}|  + \frac12 |(x_k ^{\eps*} +   x_{k+1}^{\eps^*})-(x_k+x_{k+1})|+\frac12  |(x_k-x_{k+1})  |\\&  \leq Rt_\eps +\eps  \leq  C(\Phi , R,\eta)  \eps \, ,  \\
 | x_{k+1} ^{\eps*} - (x_k+\eps \nu) | &\leq  \frac12 | x_k ^{\eps*} -   x_{k+1}^{\eps^*}|  + \frac12 |(x_k ^{\eps*} +   x_{k+1}^{\eps^*})-(x_k+x_{k+1})|+\frac12  |(x_k-x_{k+1})  | \\
 &\leq Rt_\eps +\eps  \leq C(\Phi , R,\eta ) \eps\, .
 \end{aligned}
\end{equation}
 Note that due to~(\ref{almostgood}), all particles $x_j$ with $j\leq k-1$ are at a distance at least $\eps_0/2-\eps\geq \eps_0/3$ of the particles $x_k$ and~$ x_{k }+\eps \nu $. Since  they have bounded velocities, they cannot enter the protection spheres of these post-collisional particles  during the interaction time $t_\eps$, provided that $\eps$ is small enough:
 $$R t_\eps \ll \eps_0/3 \, .$$
   Since the dynamics of the particles $j \leq k-1$ is not affected by the scattering, we get    that~$Z_{k-1}^{\eps*}$  
belongs to~${\mathcal G}_{k-1}(\eps_0/2)$: 
\begin{equation}\label{k-1good}
  \forall \tau\geq 0 \, , \, \forall (i,j) \in [1, k-1] ^2\hbox{ with } i\neq j \, ,\quad  |x_i^{\eps*} -x_j^{\eps*} -\tau(v_i^{\eps*} -v_j^{\eps*} )| \geq \eps_0/2\,. 
  \end{equation}

The pair~$(z_k^{\eps*},z^{\eps*}_{k+1})$ is a pre-collisional pair by definition, so we know that for all~$\tau \geq 0$,
$$
 |  (x_k^{\eps*}  - \tau v_k^{\eps*} ) - ( x^{\eps*}_{k+1}   - \tau  v_{k+1}^{\eps*} ) | \geq \eps \, .
$$
 Excluding the ball~$B_\eta(\bar v_k)$ in the set of admissible~$v_{k+1}$, we find as above that
   $$ \forall \tau\geq \delta \, , \quad  |x_k^{\eps*}-x^{\eps*}-\tau(v_k^{\eps*}-v_{k+1}^{\eps*})| \geq \eta \delta - \e \geq   \eps_0 \,,
 $$
 for $\eps$ sufficiently small, since~$ \eps_0 \ll  \eta \delta$.

 Next for~$j \leq k-1$ we have for $\eps$ sufficiently small, recalling that the uniform, rectilinear motion of the center of mass as described in~(\ref{rectiunif}),
  $$
 \begin{aligned}
  | x_j ^{\eps*} - \bar x_j| \leq  | x_j ^{\eps*} -   x_j|  + | x_j   - \bar x_j| \leq Rt_\eps + a \leq 2a \\
  | x_k ^{\eps*} - \bar x_k| \leq   | x_k ^{\eps*} -   x_k|  + | x_k   - \bar x_k| \leq Rt_\eps + \eps + a \leq 2 a\\
   | x_{k+1} ^{\eps*} - \bar x_k| \leq  | x_{k+1} ^{\eps*} -   x_{k+1}|  + | x_k+\eps \nu    - \bar x_k| \leq Rt_\eps + 2\eps + a \leq 2 a\, .
 \end{aligned}
 $$
By Lemma \ref{geometric-lem1}, provided  $v_k^{\eps*}$ and~$v_{k+1}^{\eps*} $ do not belong to
$$K( \bar v_j, \bar x_j -\bar x_k, 12Ra/\eps_0+12\eps_0/\delta) \cap B_R \, , 
$$
we get since~$v_j^{\eps*}  = \bar v_j$,
$$\begin{aligned}
  \forall \tau\geq 0 \, , \quad  |x_k^{\eps*}-x_j^{\eps*}-\tau(v_k^{\eps*}-v_j^{\eps*})| &\geq  \eps\,,\\
 \mbox{and}\quad  |x_{k+1}^{\eps*}-x_j^{\eps*}-\tau(v_{k+1}^{\eps*}-v_j^{\eps*})| &\geq \eps
 \end{aligned}  $$
as well as
$$\begin{aligned}
  \forall \tau\geq \delta/2 \, , \quad  |x_k^{\eps*}-x_j^{\eps*}-\tau(v_k^{\eps*}-v_j^{\eps*})| &\geq  \eps_0/2\,,\\
 \mbox{and}\quad  |x_{k+1}^{\eps*}-x_j^{\eps*}-\tau(v_{k+1}^{\eps*}-v_j^{\eps*})| &\geq \eps_0/2\,.
 \end{aligned}  $$

Lemma \ref{geometric-lem3} bounds from the above the size of the set ${\mathcal N}^*(\bar v_j, \bar x_j - \bar x_k,\rho)$ of all $(\nu,v_{k+1})$ belonging to~$  ({\mathbf S}_{1}^d \times B_R) \setminus C(\overline Z_k)$ such that 
$v_k^{\eps*}$ or $v_{k+1}^{\eps*}$ belongs to $K(\bar v_j, \bar x_j -\bar x_k, \rho).$ We let $\rho = 12Ra/\eps_0+12\eps_0/\delta,$ and define
%
$$
{\mathcal B}_k^+(\overline Z_k):= C(\overline Z_k) \cup \big( {\mathbf S}_1 ^{d-1} \times  B_\eta( \bar v_k)  \big)  \bigcup_{j \leq k-1} \cN^*(\bar v_j, \bar x_j-\bar x_k, 12 Ra /\eps_0+12 \eps_0/\delta)(\bar v_k)\, .$$
By Lemma \ref{geometric-lem3},
$$
\Big |{\mathcal B}_k^+(\overline Z_k) \Big | \leq   Ck  R \eta^{d-1} + C(\Phi , R,\eta) R\big(R{a\over\eps_0}  +  {\eps_0\over\delta} \big)^{d-1}
$$
and~(\ref{taugeq0postcoll}) and~(\ref{taugeqdelta2postcoll}) hold as soon as~$(\nu,v) \notin {\mathcal B}_k^+(\overline Z_k)$.
 Proposition~\ref{geometric-prop} is proved.\qed
  
  Note that, in order to prove that pathological sets have vanishing measure as $\eps\to 0$, we have to choose~$\eta$ small enough, and then $a$ and $\eps_0$ even smaller in order that (\ref{sizeofparameters}) is satisfied and that (\ref{pathological-size}) is small.
Moreover, if we want to get a rate of convergence, we need to have more precise  bounds on the cross-section $b$ in terms of the truncation parameters $R$ and $\eta$.

%% file: recollision.tex
\chapter{Truncated collision integrals}
\setcounter{equation}{0}

 Our goal in the present chapter is to slightly modify (in a uniform way) the  functionals~$I_{s,k}^{R,\delta} $ (defined in~(\ref{defIsnRdeltaHS}) in the hard-spheres case and  in~(\ref{defIsnRdelta}) for the potential case)  and~$I_{s,k}^{0,R,\delta} $, defined  in~(\ref{defIsnRdeltaHS}),  
  in order for the corresponding   pseudo-trajectories to be decomposed as a succession of free transport and binary collisions, without any recollision.  This will be possible thanks to Proposition~\ref{geometric-prop}.
We then expect to be able to compare these approximate observables, which will be done in the next chapter.
 
%

  \section{Initialization}\label{initialization}
  
 The first step consists in preparing the initial configuration~$Z_s$ so that it is a good configuration. We define
$$
\Delta_s (\eps_0):= \Big \{Z_s \in \R^{ds} \times B_R^s \, / \, \inf_{1 \leq \ell  < j \leq s}   |x_\ell-x_j| \geq \e_0 \Big \} \, ,
$$
and we shall  assume from now on that~$ Z_s $ belongs to~$ \Delta_s (\eps_0)$.  We also define for convenience
$$
\Delta_s^X (\eps_0):= \Big \{X_s \in \R^{ds}   \, / \, \inf_{1 \leq \ell  < j \leq s}   |x_\ell-x_j| \geq \e_0 \Big \} \, .
$$

 \begin{prop}\label{straightHs}
 For all~$X_s \in\Delta_s^X (\eps_0) $, there is a subset~${\mathcal M}_s (X_s)$ of~$\R^{ds}$ such that
 $$
\big |{\mathcal M}_s (X_s) \big |\leq CR s^2 \left(  \Big (R{\eps\over\eps_0}\Big)^{d-1} + \Big ({\eps_0\over\delta}\Big)^{d-1}  \right) \,,
 $$
 and defining $\displaystyle
 {\mathcal P}_s := \Big\{
 Z_s \in \Delta_s  (\eps_0) \, / \, V_s \notin  {\mathcal M}_s (X_s)
 \Big\}  $,
 then \begin{equation}\label{goodinitialization}
 \begin{aligned}
 \forall \tau \geq 0 \,     , \quad    \indc_{ {\mathcal P}_s}\circ  {\bf T}_s(\tau) &  \equiv \indc_{ {\mathcal P}_s} \circ{\bf S}_s(\tau) \\
&  \!\!\!\! \mbox{in the hard-spheres case} \, , \\
 \forall \tau \geq 0 \,     , \quad    \indc_{ {\mathcal P}_s}\circ  {\bf H}_s(\tau) &  \equiv \indc_{ {\mathcal P}_s} \circ{\bf S}_s(\tau) \\
&\!\!\!\! \mbox{in the potential case, and}   \\
  \forall \tau \geq \delta  \, , \quad   \indc_{ {\mathcal P}_s} \circ {\bf S}_s(\tau) &  \equiv  \indc_{ {\mathcal P}_s} \circ {\bf S}_s(\tau) \circ  \indc_{ \cG_s(\eps_0)} \, .
   \end{aligned}
 \end{equation}
 denoting abusively by $\indc_A$ the operator of multiplication by the indicator of $A$.
\end{prop}
\begin{proof}
The proof is very similar to the arguments of the previous chapter. For any~$Z_s$ in~$\Delta_s (\eps_0)$, we apply Lemma~\ref{geometric-lem1} which   shows that  outside a   small measure set~${\mathcal M}_s (X_s)\subset \R^{ds}$   of velocities~$(v_1,\dots,v_s)$, with
$$
|\cM_{s}(X_s) |\leq CR s^2 \left(  \Big (R{\eps\over\eps_0}\Big)^{d-1} + \Big ({\eps_0\over\delta}\Big)^{d-1}  \right)  \, ,
$$
 the backward nonlinear flow is actually the free flow and the particles remain at a distance larger than~$\eps$ to one  another for all times:
  $$ \forall \tau>0, \quad \forall \ell\neq \ell' \in \{1,\dots , s\} \, ,\quad |(x_ \ell-v_ \ell \tau)-(x_{\ell'} - v_{\ell'} \tau)| >\eps \, ,$$
and that
$$
 \forall \tau\geq\delta, \quad \forall \ell\neq \ell' \in \{1,\dots , s\} \, ,\quad |(x_ \ell-v_ \ell \tau)-(x_{\ell'} - v_{\ell'} \tau)| \geq \eps_0  \, .
$$
By construction, $\cM_s(X_s)$ depends continuously on $X_s$; the result follows by definition of~${\mathcal P}_s$. \end{proof}


 \section{Approximation of the Boltzmann functional}\label{approxboltzmann}

 We recall  that we consider a family of initial data~$F_{0} = (f_0^{(s)})$ satisfying
  $$
\|   F_{0} \|_{0,\beta_0,\mu_0} := \sup_{s\in \N} \sup_{Z_s}\big( \exp(\beta_0 E (Z_s) +\mu_0 s)   f_{0}^{(s)}(Z_s) \big)<+\infty
$$
and after the reductions of Chapters~\ref{convergenceHS} and \ref{Convergence}, the observable we are interested in is the following:
 \begin{equation} \label{elementary} 
   \begin{aligned}
    I^{0, R,\delta}_{s,k}   (t,J,M)(X_s) := \int  \varphi_s(V_s) \int_{{\mathcal T}_{k,\delta}(t) }   &{\bf S}_s(t-t_1) {\mathcal C}_{s,s+1}^{0,j_1,m_1} {\bf S}_{s+1}(t_1-t_2) {\mathcal C}^{0,j_2,m_2}_{s+1, s+2}\\
   &\dots   {\bf S}_{s+k}(t_k- t_{k+1})  \indc_{E_0(Z_{s+k})\leq R^2} f^{(s+k)}_{0}   dT_k dV_s\, ,
   \end{aligned}
    \end{equation}
By Proposition \ref{straightHs}, up to an error term of order $CR s^2\displaystyle \big(  \big (R{\eps\over\eps_0}\Big)^{d-1} + \big ({\eps_0\over\delta} \big)^{d-1}  \big) $, we can assume that the initial configuration $Z_s$ is a good configuration, meaning that 
$$
\begin{aligned}
  I^{0,R,\delta}_{s,k}(t,J,M )(X_s)  = \int_{B_R \setminus \cM_s(X_s)}  \varphi_s(V_s) &\int_{{\mathcal T}_{k,\delta}(t) }   {\bf S}_s(t-t_1) {\mathcal C}_{s,s+1}^{0,j_1,m_1} {\bf S}_{s+1}(t_1-t_2) {\mathcal C}^{0,j_2,m_2}_{s+1, s+2}\\
&  \dots {\mathcal C}^{0,j_k,m_k}_{s+k-1 , s+k}  {\bf S}_{s+k}(t_k- t_{k+1}) \indc_{|E_0(Z_{s+k})|\leq R^2}   f^{(s+k)}_{0}   \,dT_kdV_s \\
&+  O\left( c_{k,J,M}R s^2 \left(  \Big (R{\eps\over\eps_0}\Big)^{d-1} + \Big ({\eps_0\over\delta}\Big)^{d-1}  \right)\| F_0\|_{0,\beta_0,\mu_0}  \right) \, ,
\end{aligned}
$$
where~$\displaystyle \sum_k \sum_{J,M} c_{k,J,M} = 1$ and
$$
 \begin{aligned}\big({\mathcal C}^{0,-,m}_{s,s+1}   f^{(s+1)}  \big)(Z_s)
 &= \int_{ {\bf S}_{1}^{d-1} \times \R^d}  ((v_{s+1}-v_m)\cdot \nu_{s+1})_-  f^{(s+1)}  (Z_s, x_m, v_{s+1}) \, d\nu_{s+1} dv_{s+1} \quad \mbox{and}
\\
\big({\mathcal C}^{0,+,m}_{s,s+1}   f^{(s+1)}  \big)(Z_s)
& =  \int_{ {\bf S}_{1}^{d-1} \times \R^d} \! \! \! \! \! 
 ((v_{s+1}-v_m)\cdot \nu_{s+1})_+   f^{(s+1)}  ( z_1, \dots ,x_m, v^*_m,\dots ,z_s, x_m, v^*_{s+1}) \,d\nu _{s+1}dv_{s+1}\, .
\end{aligned}
$$

 Now let us introduce some notation which we shall be using constantly from now on: 
given~$Z_s \in \Delta_s (\eps_0)$, we call~$Z_s^0(\tau)$ the position of the backward free flow initiated from~$Z_s$, at time~$t_1\leq \tau \leq t$. Then given~$j_1 \in \{+,-\}$, $m_1 \in [1,s]$, a deflection angle~$\nu_{s+1}$ and a velocity~$v_{s+1} $ we call~$Z_{s+1}^0(\tau )$ the   position at time~$   t_2 \leq \tau < t_1$ of the Boltzmann pseudo-trajectory initiated by the adjunction of the particle~$(\nu_{s+1},v_{s+1})$ to the particle~$z_{m_1}^0(t_1)$ (which is simply free-flow in the pre-collisional case~$j_1 = -$, and free-flow after scattering of particles~$z_{m_1}^0(t_1)$ and~$(\nu_{s+1},v_{s+1})$ in the post-collisional case~$j_1 = +$). 

Similarly by induction given~$Z_s \in \Delta_s (\eps_0)$, $T, J$ and~$M$ we denote for each~$1 \leq k \leq n$ by~$Z_{s+k}^0(\tau )$ the position at time~$ t_{k+1} \leq   \tau < t_k $ of the pseudo-trajectory initiated by the adjunction of the particle~$(\nu_{s+k},v_{s+k})$ to the particle~$z_{m_k}^0(t_k)$ (which is simply free-flow in the pre-collisional case~$j_k = -$, and free-flow after scattering of particles~$z_{m_k}^0(t_k)$ and~$(\nu_{s+k},v_{s+k})$ in the post-collisional case~$j_k= +$). 

Notice that~$\tau \mapsto~Z_{s+k}^0(\tau )$ is pointwise right-continuous on~$[0,t_k]$.
 
With this notation, the elementary functional~$   I^{0,R,\delta}_{s,k}$ may be reformulated as 
 $$
  \begin{aligned}
     I^{0,R,\delta}_{s,k} (t,J,M)(X_s)&   =  \int_{B_R \setminus \cM_s(X_s)} dV_s  \varphi_s(V_s) 
     \int_{{\mathcal T}_{k,\delta}(t) }  dT_k
  \int_{{\mathbf S}^{d-1}_{1} \times B_R} \!  \!  \!  \!  \!  d\nu_{s+1} dv_{s+1}
  ((v_{s+1} - v_{m_1}^0(t_1)\cdot \nu_{s+1} )_+ \\
       & \quad\dots    \int_{{\mathbf S}^{d-1}_{1} \times B_R}   \!  \!  \!  \!  \! d\nu_{s+k} dv_{s+k}
  ((v_{s+k} - v_{m_k}^0(t_k)\cdot \nu_{s+k} )_+  \indc_{E_0(Z_{s+k}^0 (0))\leq R^2}    f^{(s+k)}_{0}   (Z_{s+k}^0 (0)) \\
  &\qquad +  O\left( c_{k,J,M}R s^2 \left(  \Big (R{\eps\over\eps_0}\Big)^{d-1} + \Big ({\eps_0\over\delta}\Big)^{d-1}  \right)\| F_0\|_{0,\beta_0,\mu_0}\right) \, ,
   \end{aligned}
  $$
  where~$\displaystyle \sum_k \sum_{J,M} c_{k,J,M} = 1$.
  Let $ a,\eps_0,\eta\ll 1$  be  such that
$$
 a \ll \eps_0 \ll \eta \delta \, .
$$
According to  Proposition~\ref{geometric-prop},  for any good configuration~$\overline Z_{s+k-1} \in \R^{2d(s+k-1)}$, we can define a set  $$
 ^c{} {\mathcal B}_{s+k-1} ^{m_k}(\overline Z_{s+k-1}):=\big( {\mathbf S}^{d-1}_{1} \times B_R \big)\setminus {\mathcal B}_{s+k-1} ^{m_k} (\overline Z_{s+k-1} )\, ,
$$
such that good configurations $Z_{s+k-1}=(X_{s+k-1}, \overline V_{s+k-1})$ with $|X_{s+k-1}-\overline X_{s+k-1}|\leq Ca$  are stable by adjunction of a collisional particle $z_{s+k} = (x_{m_k} +\eps \nu_{k+s}, v_{k+s})$ with $(\nu_{k+s}, v_{k+s}) \in {} ^c {\mathcal B}_{s+k-1}^{m_k}  (\overline Z_{s+k-1})$.

We further notice that
thanks to  Remark~\ref{betterresult}, if the adjoined pair~$(\nu_{s+k},v _{s+k})$ belongs to the set~$ ^c{} {\mathcal B}_{s+k-1}^{m_k} (  Z_{s+k-1}^0(t_k))$ with~$ Z_{s+k-1}^0(t_k) \in \cG_{s+k-1}(\eps_0)$, then~$Z_{s+k}^0(t_{k+1}) $ belongs to~$\cG_{s+k }(\eps_0)$.

As a consequence we may define recursively the approximate Boltzmann functional
 \begin{equation}\label{boltzmannfunctionalapprox}
  \begin{aligned}
 J^{0,R,\delta}_{s,k} (t,J,M)(X_s)   & =  \int_{B_R \setminus \cM_s(X_s)} dV_s  \varphi_s(V_s) 
     \int_{{\mathcal T}_{k,\delta}(t) }  dT_k\\
  &   \quad 
    \int_{^c{} {\mathcal B}_{s} ^{m_1}(  Z_s^0(t_1))} \!  \!  \!  \!  \!  d\nu_{s+1} dv_{s+1}
  (v_{s+1} - v_{m_1}^0(t_1)\cdot \nu_{s+1} )_{j_1} \\
& \qquad       \dots    \int_{^c{} {\mathcal B}_{s+k-1}^{m_k} (  Z_{s+k-1}^0(t_k))}   \!  \!  \!  \!  \! d\nu_{s+k} dv_{s+k}
  (v_{s+k} - v_{m_k}^0(t_k)\cdot \nu_{s+k} )_{j_k} \\
  & \qquad\qquad\times \indc_{E_0(Z_{s+k}^0 (0))\leq R^2}    f^{(s+k)}_{0}   (Z_{s+k}^0 (0))  \, .
   \end{aligned}
  \end{equation}
The following result is an immediate consequence of Proposition~\ref{geometric-prop}, together with the continuity estimates for the Boltzmann collision operator in Proposition \ref{propcontinuityBHS}.

\begin{prop}\label{lastapproxb}
  Let $a,\eps_0,\eta\ll 1$  satisfying {\rm(\ref{sizeofparameters})}.
  Then, we have the following error estimates for the observables associated to the Boltzmann dynamics:
  \begin{itemize}
  \item with the cross-section associated to hard-spheres,
  $$
\begin{aligned}
\Big| \sum_{k=0}^n \sum_{J,M}\1_{\Delta_s (\eps_0)} \big (   I_{s,k}^{0,R,\delta}&-   J_{s,k}^{0,R,\delta } \big) (t,J,M )   \Big|  \leq  C n^2 (s+n)  \\
&  \times{}\Big(R  \eta^{d-1} + R^d \Big(
  \frac{a}{\eps_0} \Big)^{d-1} + R\Big(  \frac{\eps_0}{\delta} 
\Big)^{d-1} \Big) \| F_{0}\|_{0,\beta_0,\mu_0} \, ;
\end{aligned}
$$
\item with the cross-section associated with a smooth short-range potential $\Phi$,
 $$
\begin{aligned}
&\Big| \sum_{k=0}^n \sum_{J,M}\1_{\Delta_s (\eps_0)} \big (   I_{s,k}^{0,R,\delta}-   J_{s,k}^{0,R,\delta } \big) (t,J,M )   \Big|  \leq  C n^2 (s+n) \\
&\qquad \times{}\Big(R  \eta^{d-1} + C(\Phi, \eta, R) R^d \Big(
  \frac{a}{\eps_0} \Big)^{d-1} + C(\Phi, \eta, R)  R\Big(  \frac{\eps_0}{\delta} 
\Big)^{d-1} \Big) \|  F_{0}\|_{0,\beta_0,\mu_0} \, .
\end{aligned}
$$
\end{itemize}
\end{prop}
  
 \section{Approximation of the BBGKY functional}\label{approxbbgky}

We recall that after the reductions of Chapters~\ref{convergenceHS} and \ref{Convergence}, the elementary functionals we are interested in are 
\begin{itemize}
\item in the case of hard spheres:
$$
 \begin{aligned}
I_{s,k} ^{R,\delta}  (t,J,M)(X_s) := \int  \varphi_s(V_s) \int_{{\mathcal T}_{k,\delta}(t) }   &{\bf T}_s(t-t_1){\mathcal C}_{s,s+1}^{j_1,m_1} {\bf T}_{s+1}(t_1-t_2) {\mathcal C}^{j_2,m_2}_{s+1, s+2}\\
   &\dots{\mathcal C}^{j_k,m_k}_{s+k-1 , s+k}  {\bf T}_{s+k}(t_k- t_{k+1})\indc_{E_\eps(Z_{s+k})\leq R^2}  f^{(s+k)}_{N,0}   dT_k dV_s\, ,
    \end{aligned}
$$ 
where~$  F_{N,0} = ( f_{N,0}^{(s)})_{1 \leq s \leq N}$ satisfies 
  $$
\|   F_{N,0}  \|_{\eps,\beta_0,\mu_0} := \sup_{s\in \N} \sup_{Z_s \in \cD_s }\big( \exp(\beta_0 E_0 (Z_s) +\mu_0 s)   f_{N,0}^{(s)}(Z_s) \big)<+\infty \, ;
$$
\item in the case of a smooth interaction potential $\Phi$:
 $$
 \begin{aligned}
I_{s,k} ^{R,\delta}  (t,J,M)(X_s) := \int  \varphi_s(V_s) \int_{{\mathcal T}_{k,\delta}(t) }   &{\bf H}_s(t-t_1){\mathcal C}_{s,s+1}^{j_1,m_1} {\bf H}_{s+1}(t_1-t_2) {\mathcal C}^{j_2,m_2}_{s+1, s+2}\\
   &\dots{\mathcal C}^{j_k,m_k}_{s+k-1 , s+k}  {\bf H}_{s+k}(t_k- t_{k+1})\indc_{E_\eps(Z_{s+k})\leq R^2} \widetilde f^{(s+k)}_{N,0}   dT_k dV_s\, ,
    \end{aligned}
$$ 
where~$\widetilde  F_{N,0} = (\widetilde f_{N,0}^{(s)})_{1 \leq s \leq N}$ satisfies
  $$
\| \widetilde  F_{N,0}  \|_{\eps,\beta_0,\mu_0} := \sup_{s\in \N} \sup_{Z_s}\big( \exp(\beta_0 E_\eps (Z_s) +\mu_0 s) \widetilde  f_{N,0}^{(s)}(Z_s) \big)<+\infty \, .
$$
\end{itemize}

Since both formulas are quite similar, we shall deal  with the case of smooth potentials and will indicate --  if need be --  simplifications arising in the case of hard spheres.

\bigskip
Thanks to Proposition~\ref{straightHs}, we have
$$
 \begin{aligned}
        I^{R,\delta}_{s,k}  (t,J,M)(X_s) &= \int _{B_R \setminus \cM_s(X_s)} \varphi_s(V_s) \int_{{\mathcal T}_{k,\delta}(t) }   { \bf S}_s(t-t_1) \indc_{{\mathcal G}_s(\eps_0)}{\mathcal C}_{s,s+1}^{j_1,m_1} {\bf H}_{s+1}(t_1-t_2) {\mathcal C}^{j_2,m_2}_{s+1, s+2}\\
   &\quad \dots {\mathcal C}^{ j_k,m_k}_{s+k-1 , s+k}  {\bf H}_{s+k}(t_k- t_{k+1})  \indc_{E_\eps(Z_{s+k}(0))\leq R^2}     \widetilde f_{N,0}^{(s+k)}  dT_k  dV_s\\
  & \qquad +{} O\left(  c_{k,J,M} R s^2 \left(  \Big (R{\eps\over\eps_0}\Big)^{d-1} + \Big ({\eps_0\over\delta}\Big)^{d-1}  \right)\| \widetilde  F_{N,0}  \|_{\eps,\beta_0,\mu_0}  \right)\, ,
   \end{aligned}
$$
where recall that~$c_{k,J,M}$ denotes a sequence of positive real numbers satisfying~$\displaystyle \sum_k \sum_{J,M} c_{k,J,M} = 1$.

Then using the notation introduced in the previous paragraph for the Boltzmann pseudo-trajectory, let us define
 the approximate functionals
$$
 \begin{aligned}
       J_{s,k} ^{R,\delta} (t,J,M)(X_s) := \int _{B_R \setminus \cM_s(X_s)} \varphi_s(V_s) \int_{{\mathcal T}_{k,\delta}(t) }   { \bf S}_s(t-t_1) \indc_{{\mathcal G}_s(\eps_0)}\widetilde {\mathcal C}_{s,s+1}^{j_1,m_1} {\bf H}_{s+1}(t_1-t_2) \\
   \dots \widetilde{\mathcal C}^{ j_k,m_k}_{s+k-1 , s+k}  {\bf H}_{s+k}(t_k- t_{k+1}) \indc_{E_\eps(Z_{s+k}(0))\leq R^2}      \widetilde f_0^{(s+k)}  dT_k  dV_s \, ,
   \end{aligned}
$$
where the modified collision operators are obtained by elimination of the pathological  set of impact parameters and velocities 
$$
 \begin{aligned}
\big(\widetilde {\mathcal C}_{s+k-1,s+k}^{\pm, m_k} g^{(s+k)} \big) (Z_{s+k-1})  :=   (N- s-k+1) \eps^{d-1} 
 \int_{  ^c{} {\mathcal B}_{s+k-1} ^{m_k} (  Z_{s+k-1}^0(t_k))
}   (\nu_{s+k}   \cdot ( v_{s+k}  - v_{m_k}(t_k))) _\pm \\ 
   \times  {} g^{(s+k)}  (\cdot, x_{m_k} (t_k)+  \eps \nu_{s+k}  ,v_{s+k}(t_k))  \prodetage{1 \leq j \leq s +k-1}{j \neq m_k} \indc_{|(x_j-x_{m_k})(t_k)-\eps \nu_{s+k}| \geq \eps}   \, d\nu_{s+k}  dv_{s+k} \, .
\end{aligned}
$$

\bigskip
By construction, we know that the remaining collision trees are nice, in the sense that collisions involve only two particles and are well-separated in time. 
Using the pre/post-collisional change of variables, we can rewrite the gain terms as follows
$$
 \begin{aligned}
 &\indc_{{\mathcal G}_{s+k-1} (\eps_0/2)}\big(\widetilde {\mathcal C}_{s+k-1,s+k}^{+, m_k} {\bf H}_{s+k} (t_k-t_{k+1}) g^{(s+k)} \big) (Z_{s+k-1}) \\
 & :=   (N- s-k+1) \eps^{d-1}  \indc_{{\mathcal G}_{s+k-1} (\eps_0/2)}
 \int_{  ^c{} {\mathcal B}_{s+k-1} ^{m_k} (  Z_{s+k-1}^0(t_k))
}   (\nu_{s+k}   \cdot ( v_{s+k}  - v_{m_k}(t_k))) _+ \\ 
 &\quad  \times  {} {\bf H}_{s+k} (t_k-t_{k+1}-t_\eps (Z_{s+k}))g^{(s+k)}  (\cdot, x_{m_k}^*, v_{m_k}^*,\dots  ,x_{s+k}^*, v_{s+k}^*) \\ &
 \qquad \times  {}\prodetage{1 \leq j \leq s +k-1}{j \neq m_k} \indc_{|(x_j-x_{m_k})(t_k)-\eps \nu_{s+k}| \geq \eps}   \, d\nu_{s+k}  dv_{s+k} .
\end{aligned}
$$
denoting as previously by $(x_{m_k}^*, v^*_{m_k}, x^*_{s+k}   ,v_{s+k}^*)$ the pre-image  by the scattering operator~$\sigma_\eps$ of the point~$ (x_{m_k}, v_{m_k}(t_k), x_{m_k} (t_k)+  \eps \nu_{s+k}  ,v_{s+k}(t_k))$.

Note that this last step is obvious in the case of hard spheres since there is no time shift~: $t_\eps \equiv 0$.

\bigskip
As in the Boltzmann case described above, the following result is an immediate consequence of Proposition~\ref{geometric-prop} together with the continuity estimates for the BBGKY collision operator in Propositions~\ref{propcontinuityclustersHS} and~\ref{propcontinuityclusters}.

\begin{prop}\label{lastapproxbbgky}
 Let $a,\eps_0,\eta\ll 1$  satisfying {\rm(\ref{sizeofparameters})}.
  Then, for $\eps$ sufficiently small,
 we have the following error estimates for the observables associated to the BBGKY dynamics:
  \begin{itemize}
  \item in the case of  hard-spheres
  $$
\begin{aligned}
&\Big| \sum_{k=0}^n \sum_{J,M}\1_{\Delta_s (\eps_0)} \big (   I_{s,k}^{R,\delta}-   J_{s,k}^{R,\delta } \big) (t,J,M )   \Big|  \leq  C n^2 (s+n) \Big(R  \eta^{d-1} + R^d \Big(
  \frac{a}{\eps_0} \Big)^{d-1} + R\Big(  \frac{\eps_0}{\delta} 
\Big)^{d-1} \Big) \|  F_{N,0}\|_{\eps,\beta_0,\mu_0} \, ,
\end{aligned}
$$
\item in the case of  some smooth short-range potential $\Phi$
 $$
\begin{aligned}
\Big|\sum_{k=0}^n \sum_{J,M} \1_{\Delta_s (\eps_0)} \big (   I_{s,k}^{0,R,\delta}&-   J_{s,k}^{0,R,\delta } \big) (t,J,M )   \Big|  \leq  C n^2 (s+n) \\
&\times{}\Big(R  \eta^{d-1} + C(\Phi, \eta, R) R^d \Big(
  \frac{a}{\eps_0} \Big)^{d-1} +C(\Phi, \eta, R)  R\Big(  \frac{\eps_0}{\delta} 
\Big)^{d-1} \Big)\|  \widetilde F_{N,0}\|_{\eps,\beta_0,\mu_0} \, .
\end{aligned}
$$
\end{itemize}
\end{prop}

   The functional~$J_{s,k} ^{R,\delta}$ can be written in terms of pseudo-trajectories, as in~(\ref{boltzmannfunctionalapprox}).  Let us therefore  introduce some notation which we shall be using constantly from now on: 
given~$Z_s \in \Delta_s (\eps_0)$, we call~$Z_s^0(\tau)$ the position of the backward free flow initiated from~$Z_s$, at time~$t_1 \leq \tau \leq t$. Then given~$j_1 \in \{+,-\}$, $m_1 \in [1,s]$, an angle~$\nu_{s+1}$ (or equivalently a position~$x_{s+1} = x_{m_1}^0(t_1) + \eps \nu_{s+1}$) and a velocity~$v_{s+1} $ we call~$Z^\eps_{s+1}(\tau)$ the   position at time~$ t_2 \leq \tau < t_1$ of the BBGKY pseudo-trajectory initiated by the adjunction of the particle~$z_{s+1}$ to the particle~$z_{m_1}^0(t_1)$. 

Similarly by induction given~$Z_s \in \Delta_s (\eps_0)$, $T, J$ and~$M$ we denote for each~$1 \leq k \leq n$ by~$Z^\eps_{s+k}(\tau )$ the position at time~$  t_{k+1} \leq \tau< t_k $ of the BBGKY pseudo-trajectory initiated by the adjunction of the particle~$z_{s+k}$ to the particle~$z_{m_k}(t_k)$. 
We have
\begin{equation}\label{firstbbgkypseudotraj}
 \begin{aligned}
  & J_{s,k} ^{R,\delta} (t,J,M)(X_s) = \frac{(N-s)! }{(N- s-k)!}\eps^{k(d-1)} \int_{B_R \setminus \cM_s(X_s)} dV_s \varphi_s(V_s) \int_{{\mathcal T}_{k,\delta}(t) }   dT_k    \\
&\quad \int_{  ^c{} {\mathcal B}_{s}^{m_1}  (  Z_{s}^0(t_1))
}   d\nu_{s+1}  dv_{s+1} \,   (\nu_{s+1}   \cdot ( v_{s+1}  - v_{m_1}(t_1)))_{j_1}  \prodetage{1 \leq j \leq s}{j \neq m_1} \indc_{|(x_j-x_{m_1})(t_1)-\eps \nu_{s+1}| \geq \eps}   \\ 
 & \quad \dots  \int_{  ^c{} {\mathcal B}^{j_k}_{s+k-1}  (  Z_{s+k-1}^0(t_k))
}  d\nu_{s+k}  dv_{s+k}  \, (\nu_{s+k}   \cdot ( v_{s+k}  - v_{m_k}(t_k)))_{j_k}   \\ 
  &\qquad \times   \prodetage{1 \leq j \leq s +k-1}{j \neq m_k} \indc_{|(x_j-x_{m_k})(t_k)-\eps \nu_{s+k}   | \geq \eps}  
 \indc_{E_\eps(Z_{s+k}(0))\leq R^2}      \widetilde f_{N,0}^{(s+k)} (Z^\eps_{s+k}(0))   \, .
\end{aligned}
\end{equation}

Thanks to Propositions~\ref{lastapproxb} and~\ref{lastapproxbbgky} the proof of Theorems~\ref{main-thm} and \ref{main-thmpotential}
 reduces to the proof of the convergence to zero of~$ J^{R,\delta}_{s,k}-  J^{0,R,\delta}_{s,k}$. This is the object of the next chapter.

%% file: proof.tex
  \chapter{Convergence proof}\label{convergenceproof}
\setcounter{equation}{0}
In this chapter we conclude the proof of Theorems~\ref{main-thm} and \ref{main-thmpotential} by proving that~$ J^{R,\delta}_{s,k}-  J^{0,R,\delta}_{s,k}$ goes to zero in the Boltzmann-Grad limit, with the notation of the previous chapter, namely~(\ref{boltzmannfunctionalapprox}) and~(\ref{firstbbgkypseudotraj}). 
The main difficulty lies in the fact that in contrast to the Boltzmann situation, collisions in the BBGKY configuration are not pointwise in space (nor in time in the case of the smooth Hamiltonian system). At each collision time~$t_k$ a small error is therefore introduced, which needs to be controlled.

We recall that, as in the previous chapters, we consider dynamics
\begin{itemize}
\item 
 involving only a finite number $s+k$ of particles, 
 \item with bounded energies (at most $R^2\gg1$), 
 \item   such that the $k$ additional particles are adjoined through binary collisions  at times separated at least by $\delta\ll1$.
 \end{itemize}
 The additional truncation parameters $a,\eps_0,\eta\ll 1$  satisfy (\ref{sizeofparameters}).

\section{Proximity of Boltzmann and BBGKY trajectories}\label{finalinduction}

This paragraph  is   devoted to the proof, by induction, that the BBGKY and Boltzmann
pseudo-trajectories remain close for all times, in particular that there is no recollision for the BBGKY dynamics. 

We recall that the notation~$Z_k^0(t)$ and~$Z_k(t)$ were defined in Paragraphs~\ref{approxboltzmann}
 and~\ref{approxbbgky}
 respectively.
 \begin{Lem}\label{translation-lem}
Fix $T\in {\mathcal T}_{n,\delta}(t)$, $J$, and~$M$ and given~$Z_s $ in~$ \Delta_s (\eps_0)$, consider  for all $i\in \{1,\dots n\}$, an impact parameter $\nu_{s+i}$ and a velocity~$v_{s+i}$ such that~$(\nu_{s+i},v_{s+i}) \notin {\mathcal B}_{s+i-1}(Z_{s+i-1}^0(t_i))$.
Then, for $\eps$ sufficiently small, for all $i\in [1,n]$, and all $ k\leq s+i$,
\begin{itemize}
\item for the hard sphere dynamics
\begin{equation}\label{theresult-HS}
| x^\eps_k(t_{i+1}) - x_k^0(t_{i+1}) |   \leq \eps i
\quad \mbox{and}\quad v_k(t_{i+1 }) = v_k^0(t_{i +1})  \,,
\end{equation}
\item for the hamiltonian dynamics associated to $\Phi$
\begin{equation}\label{theresult}
| x^\eps_k(t_{i+1}) - x_k^0(t_{i+1}) |   \leq C(\Phi, R,\eta)\eps i
\quad \mbox{and}\quad v_k(t_{i+1 }) = v_k^0(t_{i +1})  \,,
\end{equation}
where the constant $C(\Phi, R,\eta)$ depends only on $\Phi$, $R$, and~$\eta$.
\end{itemize}
\end{Lem}

\begin{proof}
We proceed by induction on~$i$, the index of the time variables~$t_{i+1}$ for~$0\leq i \leq n$.

We first notice that by construction, $ Z_{s}(t_{1 }) -Z_s^0(t_{1 })=0$, 
so~(\ref{theresult}) holds for~$i=0$. The initial configuration being a good configuration, we indeed know -- by definition -- that there is no possible  recollision.

Now let~$i\in [1,n]$ be fixed, and assume that for all~$\ell \leq i$
\begin{equation}\label{theresult2}
\forall k\leq s +\ell-1,\qquad | x^\eps_k(t_{\ell}) - x_k^0(t_{\ell}) |   \leq C\eps (\ell-1)
\quad \mbox{and}\quad v_k(t_{\ell }) = v_k^0(t_{\ell })  \,,
\end{equation}
with $C=1$ for hard spheres.

Let us prove that~(\ref{theresult2}) holds for~$\ell= i+1$.
We shall consider two cases depending on whether the particle adjoined at time~$t_{i}$ is pre-collisional or post-collisional. 

\medskip
\noindent
$\bullet$
As usual, the case of pre-collisional velocities $(v_{s+i}, v_{m_i}(t_i))$ at time $t_i$ is the most simple to handle. We indeed have $\forall \tau \in [t_{i+1},t_i]$
$$ 
\begin{aligned}
\forall k< s+i\,,\quad x^0_k(\tau) &=x^0_k (t_i) +(\tau-t_i) v^0_k(t_i)\,,\qquad v^0_k(\tau) =v^0_k(t_i)\,,\\
x^0_{s+i}(\tau) &= x^0_{m_i} (t_i)  + (\tau-t_i) v_{s+i}\,,\qquad v_{s+i}^0(\tau) =v_{s+i}\,.
\end{aligned}
$$
Now let us study the BBGKY trajectory.
We recall that the particle is adjoined in such a way that~$(\nu_{s+i}, v_{s+i})$  belongs to~$^c{} {\mathcal B}_{s+i-1}  (  Z_{s+i-1}^0(t_i))$. Provided that $\eps$ is sufficiently small, by the induction assumption~(\ref{theresult2}), we have
$$\forall k\leq s+i-1,\quad  | x_{k}^\eps (t_i) -x_{k}^0 (t_i) |\leq C\eps (i-1) \leq a \,,$$
with $C=1$ for hard spheres.

Since~$Z_{s+i-1}^0(t_i)$ belongs to~$\cG_{s+i-1}(\eps_0)$ (see Paragraph~\ref{approxboltzmann}), we can apply Proposition~\ref{geometric-prop}
which implies  that backwards in time, there is free flow for~$Z^\eps_{s+i}$. In particular,
$$ 
\begin{aligned}
&\forall k< s+i\,,\quad x_k(\tau) =x_k (t_i) +(\tau-t_i) v_k(t_i)\,,\qquad v_k(\tau) =v_k(t_i)\,,\\
&x_{s+i}(\tau) = x_{m_i} (t_i) + \eps \nu_{s+i}  + (\tau-t_i) v_{s+i}\,,\qquad v_{s+i}(\tau) =v_{s+i}\,.
\end{aligned}
$$
We therefore obtain
\begin{equation}\label{induction1}
\forall k\leq s+i\,,\quad \forall \tau \in [t_{i+1},t_i] \, , \quad v_k(\tau ) -v_k^0(\tau) = v_k(t_i ) -v_k^0(t_i) = 0\, ,
\end{equation}
and  
\begin{equation}\label{induction2}\forall k\leq s+i\,,\quad  \forall \tau \in [t_{i+1},t_i] \, , \quad  
 |x_k(\tau)- x_k^0(\tau)| \leq C\eps (i-1) +\eps \,,
 \end{equation}
 with $C=1$ in the case of hard spheres.
 
 \medskip
\noindent
$\bullet$
The case of post-collisional velocities  $(v_{s+i}, v_{m_i}(t_i))$ at time $t_i$ for the hard sphere dynamics is very similar.
We indeed have $\forall \tau \in [t_{i+1},t_i[$
$$ 
\begin{aligned}
&\forall k< s+i, \quad k\neq m_i \,,\quad x^0_k(\tau) =x^0_k (t_i) +(\tau-t_i) v^{0*}_k(t_i)\,,\qquad v^0_k(\tau) =v^{0}_k(t_i)\,,\\
& x^0_{m_i}(\tau) =x^0_{m_i}  (t_i) +(\tau-t_i) v^{0*}_{m_i}(t_i)\,,\qquad v^0_k(\tau) =v^{0*}_{m_i}(t_i)\,,\\
&x^0_{s+i}(\tau) = x^0_{m_i} (t_i)  + (\tau-t_i) v_{s+i}^*\,,\qquad v_{s+i}^0(\tau) =v_{s+i}^*\,.
\end{aligned}
$$
Now let us study the BBGKY trajectory.
We recall that the particle is adjoined in such a way that~$(\nu_{s+i}, v_{s+i})$  belongs to~$^c{} {\mathcal B}_{s+i-1} ^{m_i} (  Z_{s+i-1}^0(t_i))$. Provided that $\eps$ is sufficiently small, by the induction assumption~(\ref{theresult2}), we have
$$\forall k\leq s+i-1,\quad  | x_{k}^\eps (t_i) -x_k^0 (t_i) |\leq \eps (i-1) \,.$$
Since~$Z_{s+i-1}^0(t_i)$ belongs to~$\cG_{s+i-1}(\eps_0)$ (see Paragraph~\ref{approxboltzmann}), we can apply Proposition~\ref{geometric-prop}
which implies  that backwards in time, there is free flow for~$Z^\eps_{s+i}$. In particular,
\begin{equation}\label{induction11}
\forall k\leq s+i\,,\quad \forall \tau \in [t_{i+1},t_i[ \, , \quad v_k(\tau ) -v_k^0(\tau) = v_k(t_i ^-) -v_k^0(t_i^-) = 0\, ,
\end{equation}
and  
\begin{equation}\label{induction21}\forall k\leq s+i\,,\quad  \forall \tau \in [t_{i+1},t_i[\, , \quad  
 |x_k(\tau)- x_k^0(\tau)| \leq \eps (i-1) +\eps \leq i\eps \,.
 \end{equation}

\medskip
\noindent
$\bullet$
The case of post-collisional velocities is a little more complicated since there is a (small) time interval during which  interaction occurs.

Let us start by describing the Boltzmann flow. By definition of the post-collisional configuration, we know that   the following identities hold: 
$$
\forall t_{i+1} \leq \tau < t_i \, , \,  \left \{\begin{aligned}
 (v_{m_i}^0,v_{s+i}^0 )(\tau) = (v_{m_i}^{0*}(t_i) , v_{s+i} ^* ) \hbox{ with } (\nu_{s+i}^*,v_{m_i}^{0*}(t_i) , v_{s+i} ^* ) := \sigma_0^{-1}( \nu_{s+i}, v_{m_i}^{0}(t_i)  , v_{s+i} ) \\
  x_{m_i}^0 (\tau) = x_{m_i}^0 (t_i) + (\tau - t_i) v_{m_i}^{0*}(t_i) \, ,  \, x_{s+i}^0 (\tau) = x_{s+i}^0 (t_i) + (\tau - t_i) v_{s+i}^{*} \\
 \forall j \notin \{ m_i, s+1\} \,, \quad v_{j}^0(\tau) =v_{j}^0(t_i)  \, , \,  x_j^0(\tau) = x_j^0(t_i) + (\tau - t_i) v_{j}^0(t_i)  \, ,
\end{aligned}
\right.
$$
where~$\sigma_0$ denotes the scattering operator defined in Definition~\ref{scatteringbbgky}
 in Chapter~\ref{scattering}.

First, by Proposition~\ref{geometric-prop}, we know that for $j\notin \{m_i, s+i\}$ and $\forall \tau \in [t_{i+1},t_i]$,
$$ 
 x_j(\tau) =x_j (t_i) +(\tau-t_i) v_j(t_i)\, ,\qquad v_j(\tau) =v_j(t_i)\,,
$$
so that by the induction assumption~(\ref{theresult2}) we obtain
\begin{equation}\label{induction3}
\begin{aligned}
\forall j\notin \{m_i, s+i\} \, , \,  \forall \tau \in [t_{i+1},t_i] \, , \quad  |x_j(\tau)- x_j^0(\tau)| =  |x_j(t_i)- x_j^0(t_i)|  \leq C \eps (i-1)\\
 \mbox{and} \quad v_j(\tau) =  v^0_j(\tau)\,.
 \end{aligned}
\end{equation}

We now have  to focus on the pair $(s+i, m_i)$. According to Chapter~\ref{scattering}, the relative velocity evolves under the nonlinear dynamics on a  time interval~$[t_i-t_\eps,t_i]$ with $t_\eps \leq C(\Phi,R,\eta) \eps$ (recalling that by construction, the relative velocity~$|v_{s+i} -  v_{m_i}(t_i)|$ is bounded from above  by $R$ and from below by $\eta$, and that the impact parameter is also bounded from below by $\eta$). Then,    for all $\tau \in [t_{i+1}, t_i-t_\eps ]$,
\begin{equation}\label{thevelocities}
v_{s+i}(\tau) =v_{s+i}^* =  v_{s+i}^0(\tau)  \, , \quad v_{m_i}(\tau) =v_{m_i} ^*(t_i) = v_{m_i}^{0*}(t_i) =v_{m_i}^0  (\tau)  \, .
\end{equation}
In particular,
\begin{equation}\label{induction4}
 v_{s+i}(t_{i+1}) = v_{s+i}^0(t_{i+1}) \quad \mbox{and} \quad  v_{m_i}(t_{i+1})  = v_{m_i}^0(t_{i+1})  \, .
\end{equation}
\noindent
The conservation of total momentum  as in Paragraph \ref{gaindec} shows that
$$
 \begin{aligned}
  \Big|\frac12 (x_{m_i} ^{\eps} (t_i-t_\eps)+x_{s+i} ^{\eps} (t_i-t_\eps)) - \frac12(x_{m_i} ^0 (t_i-t_\eps)+x_{s+i} ^0 (t_i-t_\eps))\Big| \\=\Big|\frac12 (x_{m_i} ^{\eps} (t_i)+x_{s+i} ^{\eps} (t_i) - \frac12(x_{m_i} ^0 (t_i)+x_{s+i} ^0 (t_i))\Big| \\
  = \Big|x_{s+i} ^{\eps} (t_i) - x_{s+i} ^0 (t_i)\Big| +\frac\eps2 \leq C\eps (i-1) +\frac\eps 2 \, \cdotp
 \end{aligned}
 $$
 On the other hand, by definition of the scattering time $t_\eps$, 
 $$
 \begin{aligned}
 |x_{m_i} ^{\eps} (t_i-t_\eps)-x_{s+i} ^{\eps} (t_i-t_\eps)|=\eps \, ,\\
  |x_{m_i} ^0 (t_i-t_\eps)-x_{s+i} ^0 (t_i-t_\eps)| = t_\eps | v_{m_i}^*-v_{s+i}^*|\leq C(\Phi,R,\eta)\,  \eps \, .
 \end{aligned}
 $$
 We obtain finally
 \begin{equation}\label{outsidescatteringclose}
  |x_{m_i} ^{\eps} (t_i-t_\eps)-x_{m_i} ^0 (t_i-t_\eps)|\leq C\eps i
 \hbox{ and } |x_{s+i} ^{\eps} (t_i-t_\eps)-x_{s+i} ^0 (t_i-t_\eps)| \leq C\eps i
 \end{equation}
  provided that $C$ is chosen sufficiently large (depending on $\Phi$, $R$ and $\eta$).

Now let us apply  Proposition~\ref{geometric-prop}, which implies that
for all $\tau \in [t_{i+1}, t_i-t_\eps]$ the backward in time evolution of the two particles~$x^\eps_{s+i}(t_i-t_\eps )  $ and~$ x^\eps_{m_i} (t_i-t_\eps )  $,  is that of free flow:
 we have therefore, using~(\ref{thevelocities}),
 $$ 
 \begin{aligned}
 x^\eps_{m_i} (t_{i+1})-x^0_{m_i} (t_{i+1})  =x_{m_i}^\eps (t_i-t_\eps ) -x_{m_i}^0  (t_i-t_\eps ) \,,\\
 x^\eps_{s+i} (t_{i+1})-x^0_{s+i} (t_{i+1})  =x_{s+i}^\eps (t_i-t_\eps ) -x_{s+i}^0  (t_i-t_\eps ) \,.
\end{aligned}
$$
From~(\ref{outsidescatteringclose}) we therefore deduce that the induction assumption is  satisfied at time step~$t_{i+1}$, and the proposition is proved.
     \end{proof}
     
     Note that, by construction,
     $$Z_{s+k}^0(0)\in \cG_{s+k}(\eps_0)\,,$$
so that an obvious application of the triangular inequality leads to
 $$Z_{s+k}^\eps(0)\in \cG_{s+k}(\eps_0/2)\,.$$
 Note also that the indicator functions are identically equal to 1 for good configurations.
 We therefore have the following
 
     \begin{Cor}\label{secondbbgkypseudotraj}
    Under the assumptions of Lemma~{\rm\ref{translation-lem}}, the functional~$J^{R,\delta}_{s,n}  (t,J,M)$ defined in~{\rm(\ref{firstbbgkypseudotraj})} may be written as follows:
       $$
\begin{aligned}
   J_{s,k} ^{R,\delta} (t,J,M)(X_s)& = \frac{(N-s)! }{(N- s-k)!}\eps^{k(d-1)} \int_{B_R \setminus \cM_s(X_s)} dV_s \varphi_s(V_s) \int_{{\mathcal T}_{k,\delta}(t)}   dT_k    \\
&\quad \int_{  ^c{} {\mathcal B}^{m_1}_{s}  (  Z_{s}^0(t_1))
}   d\nu_{s+1}  dv_{s+1} \,   (\nu_{s+1}   \cdot ( v_{s+1}  - v_{m_1}(t_1)))_{j_1}    \\ 
 & \quad \dots  \int_{  ^c{} {\mathcal B}_{s+k-1}^{m_k}  (  Z_{s+k-1}^0(t_n))
}  d\nu_{s+k}  dv_{s+k}  \, (\nu_{s+n}   \cdot ( v_{s+k}  - v_{m_k}(t_k)))_{j_k}   \\ 
  &\qquad \times  
 \indc_{E_\eps(Z_{s+k}(0))\leq R^2}    \indc_{ Z_{s+k}(0) \in \cG_{s+k}(\eps_0/2)}
  \widetilde f_{N,0}^{(s+k)} (Z^\eps_{s+k}(0))   \, .
\end{aligned}
$$
     \end{Cor}
     
\section{Proof of convergence for the hard sphere dynamics: proof of Theorem 8}\label{theend-HS}
In this section we   prove Theorem~\ref{main-thm}, which concerns the case of hard spheres. The potential case will be treated  in the following section.

From Corollary \ref{finalresultHS}, we know that any observable associated to the BBGKY hierarchy  can be approximated by a finite sum~:  more precisely, given~$s$ and~$t \in [0,T]$, there are two positive constants~$C$ and~$C'$ such that 
\begin{equation}
\label{est1}
 \big  \| I_{s} (t)  - \sum_{k=0}^n  I_{s,k} ^{R,\delta} (t) \big \|_{L^\infty (\R^{ds})} \leq C\left(2^{-n} + e^{-C'\beta_0R^2} +\frac{n^2}T \delta \right) \|\varphi\|_{L^\infty (\R^{ds})}
  \|  F_{N,0}\|_{\eps,\beta_0,\mu_0} \, .
  \end{equation}
 Similarly, for the Boltzmann hierarchy, we get
\begin{equation}
\label{est1-0}
 \big  \| I^0_{s} (t)  - \sum_{k=0}^n  I_{0,s,k} ^{R,\delta} (t) \big \|_{L^\infty (\R^{ds})} \leq  C\left(2^{-n} + e^{-C'\beta_0R^2} +\frac{n^2}T \delta \right)\|\varphi\|_{L^\infty (\R^{ds})}
  \|  F_{0}\|_{0,\beta_0,\mu_0} \, .
  \end{equation}

\bigskip

Then, 
from Propositions \ref{lastapproxb} and \ref{lastapproxbbgky}, we obtain  the error terms corresponding to the elimination of pathological velocities and impact parameters
 \begin{equation}
\label{est2}
\begin{aligned}
\Big| \1_{\Delta_s (\eps_0)} \sum_{k=0}^n\sum_{J,M} \big (   I_{s,k}^{0,R,\delta}&-   J_{s,k}^{0,R,\delta } \big) (t,J,M )   \Big|  \leq  C n^2 (s+n) \\
&\quad \times{}\Big(R  \eta^{d-1} + R^d \Big(
  \frac{a}{\eps_0} \Big)^{d-1} + R\Big(  \frac{\eps_0}{\delta} 
\Big)^{d-1} \Big) \|  F_{0}\|_{0,\beta_0,\mu_0} \|\varphi\|_{L^\infty (\R^{ds})}\, ,
\end{aligned}
\end{equation}
and
\begin{equation}
\label{est2-0}
\begin{aligned}
\Big| \1_{\Delta_s (\eps_0)} \sum_{k=0}^n\sum_{J,M}  \big (   I_{s,k}^{R,\delta}&-   J_{s,k}^{R,\delta } \big) (t,J,M )   \Big|  \leq  C n^2 (s+n) \\
&\quad \times{}\Big(R  \eta^{d-1} + R^d \Big(
  \frac{a}{\eps_0} \Big)^{d-1} + R\Big(  \frac{\eps_0}{\delta} 
\Big)^{d-1} \Big) \|  F_{N,0}\|_{\eps,\beta_0,\mu_0} \|\varphi\|_{L^\infty (\R^{ds})} \,.
\end{aligned}
\end{equation}

\bigskip
The end of the proof of Theorem~\ref{main-thm} consists  in estimating the error terms in~$ J^{R,\delta}_{s,k}-  J^{0,R,\delta}_{s,k}$ coming essentially  from the micro-translations described in the previous paragraph and from the initial data.

\bigskip
\subsection{Error coming from  the initial data}$ $

Let us     replace the initial data in $J^{R,\delta}_{s,k} $ by that of the Boltzmann hierarchy, defining:
 $$
     \begin{aligned}
    \widetilde J^{R,\delta}_{s,k}  (t,J,M)(X_s) &:= \frac{(N-s)! }{(N- s-k)!}\eps^{k(d-1)} \int_{B_R \setminus \cM_s(X_s)} dV_s \varphi_s(V_s) \int_{{\mathcal T}_{k,\delta}(t)}   dT_k    \\
&\quad  \int_{  ^c{} {\mathcal B}_{s} ^{m_1} (  Z_{s}^0(t_1))
}  \!  \!  \!  \!  \!  d\nu_{s+1}  dv_{s+1} \,   (\nu_{s+1}   \cdot ( v_{s+1}  - v_{m_1}(t_1)) )_{j_1}  \\
 &\qquad   \dots  \int_{  ^c{} {\mathcal B}_{s+k-1}  ^{m_k} (  Z_{s+k-1}^0(t_k))
}  \!  \!  \!  \!  \!  d\nu_{s+k}  dv_{s+k}  \, (\nu_{s+k}   \cdot ( v_{s+k}  - v_{m_k}(t_k)))_{j_k}   \\ 
 & \qquad \times    \indc_{E_0(Z_{s+k}(0))\leq R^2}   \indc_{ Z^\eps_{s+k}(0) \in \cG_{s+k}(\eps_0/2)}        f_{0}^{(s+k)} (Z_{s+k}(0))   \, .
\end{aligned}
   $$

   Since, by definition of admissible Boltzmann data, we have for any fixed $s$
    $$ f_{0,N}^{(s)} \longrightarrow f_0^{(s)}   \quad \mbox{as $N \to \infty$ with $N \e^{d-1} \equiv 1 \, ,$ locally uniformly in $\Omega_s \, ,$}$$
we expect that 
$$J^{R,\delta}_{s,k}  (t,J,M)(X_s) - \widetilde J^{R,\delta}_{s,k}  (t,J,M)(X_s) \to 0 $$
 as $N \to \infty$ with $N \e^{d-1} \equiv 1 $,  locally uniformly in $\Omega_s$.

    \begin{Lem}\label{data}
    Let $F_0$ be an admissible Boltzmann datum and $F_{0,N}$ an associated BBGKY datum. Then, in the Boltzmann-Grad scaling $N \eps^{d-1} = 1$,
    for all fixed $s, k\in \N$ and $t<T$,
    $$J^{R,\delta}_{s,k}  (t,J,M)(X_s) - \widetilde J^{R,\delta}_{s,k}  (t,J,M)(X_s) \to 0 \, ,$$
locally uniformly in $\Omega_s$.

For tensorized initial data
     $$f^{(N)}_{0,N} (Z_N)= {\mathcal Z}_N ^{-1}  \indc_{Z_N \in \D_N} f_0^{\otimes N}(Z_N)\quad  \hbox{ with } \quad \big\| f_0 \exp (\beta_0 |v|^2) \big \|_{L^\infty} <+\infty \, ,$$
we further have the following error estimate~:
$$
  \big |  \1_{\Delta^X_s (\eps_0)} \sum_{k=0}^n \sum_{J,M} ( J^{R,\delta}_{s,k}-  \widetilde J^{R,\delta}_{s,k} )  (t,J,M)(X_s) \big | \leq  C \eps (s+n) \|  F_{0}\|_{0,\beta_0,\mu_0} \|\varphi\|_{L^\infty (\R^{ds})} \, .
  $$
   \end{Lem}
 
\begin{proof}

By definition of the good sets~${\mathcal G}_k(c)$, the positions in the argument of~$   f^{(s+k)}_{N,0} -  f^{(s+k)}_{0}$ satisfy the separation condition~$|x_{i} - x_j | \geq \eps_0/2> \eps $ for~$i \neq j$~:
$$\indc_{{\mathcal G}_{s+k}(\eps_0/2)}    (  f^{(s+k)}_{N,0}   - f^{(s+k)}_{0}) =\indc_{{\mathcal G}_{s+k}(\eps_0/2)}  \1_{\Delta^X_{s+k} (\eps_0/2)}   (   f^{(s+k)}_{N,0}   - f^{(s+k)}_{0}) \,.$$
So we can write
$$
     \begin{aligned}
      (J^{R,\delta}_{s,k}(t,J,M)-  \widetilde J^{R,\delta}_{s,k}(t,J,M))(X_s) & = \frac{(N-s)! }{(N- s-k)!}\eps^{k(d-1)} \int_{B_R \setminus \cM_s(X_s)} \! \! \!  dV_s \varphi_s(V_s) \int_{\mathcal T_{k,\delta}(t)}   dT_k  \\
     &\quad    \int_{  ^c{} {\mathcal B}_{s}^{m_1}  (  Z_{s}^0(t_1))
}   d\nu_{s+1}  dv_{s+1}   (\nu_{s+1}   \cdot ( v_{s+1}  - v_{m_1}(t_1))  )_{j_1}  
  \\ &\qquad \dots  \int_{  ^c{} {\mathcal B}_{s+k-1} ^{m_k} (  Z_{s+k-1}^0(t_k))
}   d\nu_{s+k}  dv_{s+k}  \, (\nu_{s+k}   \cdot ( v_{s+k}  - v_{m_k}(t_k)) )_{j_k}  \\ 
    &\qquad    \times   \indc_{E_\eps(Z^\eps_{s+k}(0))\leq R^2} \indc_{\Delta_{s+k}(\eps_0/2)}    ( f^{(s+k)}_{N,0} -f^{(s+k)}_0 )   \, ,\end{aligned}
   $$
and
 we find directly that
  $$
    \Big| \1_{\Delta^X_s(\eps_0)}(J^{R,\delta}_{s,k}(t,J,M)-  \widetilde J^{R,\delta}_{s,k}(t,J,M))(X_s)\Big|      \leq C {R^{k(d+1)}t^k \over k!} \left\|\indc_{\Delta_{s+k}(\eps_0/2)}    ( f^{(s+k)}_{N,0} -f^{(s+k)}_0 )   \right\|_{L^\infty}
\, .
  $$

 Note that, summing  all the elementary contributions (i.e. summing over $J$, $M$ and $k$), we get the convergence to 0, but with a very bad dependence with respect to $R$ and $n$.

  \bigskip
  In the case of tensorized initial data, this estimate can be improved using some explicit control on the convergence of the initial data.
Looking at the proof of Proposition \ref{init-cv1}, we indeed see that
$$\indc_{Z_s \in {\mathcal D}_s}f^{\otimes s}_{0} 
 -   f_{0,N}^{(s)} = \Big(1 - {\mathcal Z}_{N}^{-1} {\mathcal Z}_{N-s}\Big) \indc_{Z_s \in {\mathcal D}_s}f^{\otimes s}_{0} 
+  {\mathcal Z}_N^{-1} {\mathcal Z}^\flat_{(s+1,N)} \indc_{Z_s \in {\mathcal D}_s}f^{\otimes s}_{0} $$
with 
$$ 
 \Big|1 - {\mathcal Z}_{N}^{-1} {\mathcal Z}_{N-s}\Big| \leq (1-\eps \kappa_d |f_0|_{L^\infty L^1})^{-s} - 1\leq \e s \k_d |f_0|_{L^\infty L^1} \big(1 - \e \k_d | f_0|_{L^\infty L^1}\big)^{-(s+1)}$$
 according to Lemma \ref{lem:bd-f0HS}, and
 $$  {\mathcal Z}_N^{-1} {\mathcal Z}^\flat_{(s+1,N)}  \leq \e s \k_d |f_0|_{L^\infty L^1} \big(1 - \e \k_d | f_0|_{L^\infty L^1}\big)^{-(s+1)} \,.$$
 Using the continuity estimate in Proposition~\ref{propcontinuityclustersHS}, we then deduce that 
  $$
     \begin{aligned}
  & \Big| \1_{\Delta^X_s(\eps_0)}(J^{R,\delta}_{s,k}(t,J,M)-  \widetilde J^{R,\delta}_{s,k}(t,J,M))(X_s)\Big|    \\
  &\quad \leq \e (s+k) \k_d |f_0|_{L^\infty L^1}  \|  F_{0}\|_{0,\beta_0,\mu_0} \|\varphi\|_{L^\infty (\R^{ds})}c_{k,J,M} \, .
 \end{aligned}
  $$
  denoting by $(c_{k,J,M})$  a sequence of nonnegative real numbers such that $\sum_{k} \sum_{J,M}  c_{k,J,M} = 1$. This concludes the proof of Lemma~\ref{data}.
 \end{proof}

\bigskip
\subsection{Error coming from the prefactors in the collision operators}$ $

As $\eps \to 0$ in the Boltzmann-Grad scaling, we have 
$$
 \frac{(N-s)! }{(N- s-k)!}\eps^{k(d-1)} \to 1 \, .
$$
Defining
\begin{equation}
\label{lastfunctional}
     \begin{aligned}
   \overline J^{R,\delta}_{s,k}  (t,J,M)(X_s) &= \int_{B_R \setminus \cM_s(X_s)} dV_s \varphi_s(V_s) \int_{{\mathcal T}_{k,\delta}(t)}   dT_k\\
   &\quad   \int_{  ^c{} {\mathcal B}_{s} ^{m_1} (  Z_{s}^0(t_1))
}   d\nu_{s+1}  dv_{s+1} \,  ( \nu_{s+1}   \cdot ( v_{s+1}  - v_{m_1}(t_1)))_{j_1}    \\ 
 & \qquad  \dots  \int_{  ^c{} {\mathcal B}_{s+k-1} ^{m_k} (  Z_{s+k-1}^0(t_k))
}  d\nu_{s+k}  dv_{s+k}  \, (\nu_{s+k}   \cdot ( v_{s+k}  - v_{m_k}(t_k)) )_{j_k}  \\ 
  & \qquad \times    \indc_{E_0(Z_{s+k}(0))\leq R^2}   \indc_{ Z^\eps_{s+k}(0) \in \cG_{s+k}(\eps_0/2)}        f_{0}^{(s+k)} (Z_{s+k}(0))   \, ,
\end{aligned}
\end{equation}
  and using again the continuity estimate in Proposition~\ref{propcontinuityclustersHS}, we have the following obvious convergence.   
   
    \begin{Lem}\label{cross-section}
In the Boltzmann-Grad scaling $N \eps^{d-1} = 1$,
$$
 \big |  \1_{\Delta^X_s(\eps_0)} \sum_{k=0}^n \sum_{J,M}(\widetilde J^{R,\delta}_{s,k}-\overline  J^{R,\delta}_{s,k}  ) (t,J,M)(X_s) \big |\leq  C{ (s+n)^2 \over N} \|\varphi\|_{L^\infty (\R^{ds})} \| F_{0}\|_{0,\beta_0,\mu_0}\,.
$$

\end{Lem}

\bigskip
\subsection{Error coming from the divergence of trajectories}$ $

We can now compare the definition (\ref{boltzmannfunctionalapprox}) of $J^{0,R,\delta}_{s,k}  (t,J,M)$: 
$$  \begin{aligned}
     J^{0,R,\delta}_{s,k} (t,J,M)(X_s)&   =  \int_{B_R \setminus \cM_s(X_s)} dV_s  \varphi_s(V_s) 
     \int_{{\mathcal T}_{k,\delta}(t)}  dT_k  \int_{^c{} {\mathcal B}_{s} ^{m_1}(  Z_s^0(t_1))} \!  \!  \!  \!  \!  d\nu_{s+1} dv_{s+1}
  ((v_{s+1} - v_{m_1}^0(t_1)\cdot \nu_{s+1} )_{j_1} \\
       &\quad\dots    \int_{^c{} {\mathcal B}_{s+k-1}^{m_k} (  Z_{s+k-1}^0(t_k))}   \!  \!  \!  \!  \! d\nu_{s+k} dv_{s+k}
  ((v_{s+k} - v_{m_k}^0(t_n)\cdot \nu_{s+k} )_{j_k} \\
  & \qquad \times{} \indc_{E_0(Z_{s+k}^0 (0))\leq R^2}    f^{(s+k)}_{0}   (Z_{s+k}^0 (0))  \, .
   \end{aligned}
$$
and the formulation (\ref{lastfunctional}) for the approximate BBGKY hierarchy.

 Lemma \ref{translation-lem} implies  that at time~$0$
we have
$$ | X_{s+k}(0 )- X_{s+k}^0(0)| \leq  Ck \eps\, ,\quad \mbox{and} \quad  V_{s+k}(0 )= V_{s+k}^0(0)\, .$$
Since $f_0^{(s+k)}$  is continuous, we  obtain the expected convergence as stated in the following lemma.

\begin{Lem}\label{lipschitz-estimate}
In the Boltzmann-Grad scaling $N \eps^{d-1} = 1$,
    for all fixed $s, k\in \N$ and $t<T$,
    $$\bar J^{R,\delta}_{s,k}  (t,J,M)(X_s) - J^{0,R,\delta}_{s,k}  (t,J,M)(X_s) \to 0 \, .$$

For tensorized Lipschitz initial data, we further have the following error estimate~:
$$
  \big |  \1_{\Delta^X_s (\eps_0)} \sum_{k=0}^n \sum_{J,M} ( \bar J^{R,\delta}_{s,k}-   J^{0,R,\delta}_{s,k} )  (t,J,M)(X_s) \big | \leq  C \eps n \| \nabla_x f_0\|_\infty  \|  F_{0}\|_{0,\beta_0,\mu_0} \|\varphi\|_{L^\infty (\R^{ds})}\,.$$
\end{Lem}
Notice that putting together Lemmas~\ref{data}, \ref{cross-section} and~\ref{lipschitz-estimate}, along with  the estimates~(\ref{est1})-(\ref{est1-0}) and~(\ref{est2})-(\ref{est2-0}),  end the proof of Theorem~\ref{main-thm} up to the rate of convergence. This is the object of the next paragraph.

\bigskip
\subsection{Optimization for tensorized Lipschitz initial data}$ $
We can now conclude the proof of Theorem~\ref{main-thm}.
Gathering the results of Lemmas \ref{data}, \ref{cross-section} and \ref{lipschitz-estimate}, together with the estimates~(\ref{est1})-(\ref{est1-0}) and~(\ref{est2})-(\ref{est2-0}), we get
$$
\begin{aligned}
 \big  \| I_{s} (t)  -  I_s^0 (t) \big \|_{L^\infty (\R^{ds})} \leq & C\left( 2^{-n} + e^{-C'\beta_0R^2} +\frac{n^2}T \delta \right)\|\varphi\|_{L^\infty (\R^{ds})}
 \sup_N \|  F_{N,0}\|_{\eps,\beta_0,\mu_0} \\
 & +C n^2 (s+n) \Big(R  \eta^{d-1} + R^d \Big(
  \frac{a}{\eps_0} \Big)^{d-1} + R\Big(  \frac{\eps_0}{\delta} 
\Big)^{d-1} \Big)\|  F_{N,0}\|_{\eps,\beta_0,\mu_0} \|\varphi\|_{L^\infty (\R^{ds})} \\
&+ C \eps (s+n)\|  F_{0}\|_{0,\beta_0,\mu_0} \|\varphi\|_{L^\infty (\R^{ds})}\\
& + C{ (s+n)^2 \over N} \|\varphi\|_{L^\infty (\R^{ds})} 
  \| F_{0}\|_{0,\beta_0,\mu_0}\\
  & + C n \eps \| \nabla_x f_0 \|_{L^\infty} \|\varphi\|_{L^\infty (\R^{ds})}
\| F_{0}\|_{0,\beta_0,\mu_0}\
  \end{aligned}
  $$
  
Therefore, choosing
$$n\sim C_1 |\log \eps|, \quad R^2 \sim C_2|\log \eps|$$
for some sufficiently large constants $C_1$ and $C_2$, and
$$\delta = \eps^{(d-1)/(d+1)},\quad \eps_0= \eps^{d/(d+1)} $$
we find that the total error is smaller than $C\eps^\alpha$ for any $\alpha <(d-1)/(d+1)$.

This ends  the proof of Theorem~\ref{main-thm}.

\section{Convergence in the case of a smooth interaction potential: proof of Theorem 11}\label{theend}
Let us now prove Theorem~\ref{main-thmpotential}.

The same arguments as in the previous section provide the convergence for any smooth short-range potential satisfying (\ref{strangeassumption}). Let us only sketch the proof and point out how to deal with the following minor differences.
\begin{itemize}
\item The elimination of multiple collisions gives an additional error term~: from Propositions \ref{modified-h} and~\ref{delta2small}, we indeed deduce the analogue of~(\ref{est1}):
\begin{equation}
\label{est1-pot}
 \big  \| I_{s} (t)  -  I_{s,n} ^{R,\delta} (t) \big \|_{L^\infty (\R^{ds})} \leq C\left( \eps + 2^{-n} + e^{-C'\beta_0R^2} +\frac{n^2}T \delta \right ) \|\varphi\|_{L^\infty (\R^{ds})}
\|\widetilde F_{N,0}\|_{\eps,\beta_0,\mu_0} \, .
\end{equation}

\medskip

\item The  error term coming from the elimination of pathological velocities and impact parameters depends (in a non trivial way) on the local $L^\infty$ norm of the cross-section: estimate~(\ref{est2}) becomes
 $$
\begin{aligned}
&\Big| \1_{\Delta_s (\eps_0)} \sum_{k=0}^n\sum_{J,M} \big (   I_{s,k}^{0,R,\delta}-   J_{s,k}^{0,R,\delta } \big) (t,J,M )   \Big|  \\
&\leq  C n^2(s+n) \Big(R  \eta^{d-1} + C(\Phi, R,\eta) R^d \Big(
  \frac{a}{\eps_0} \Big)^{d-1} + C(\Phi, R,\eta)  R\Big(  \frac{\eps_0}{\delta} 
\Big)^{d-1}\Big)\|  F_{0}\|_{0,\beta_0,\mu_0} \|\varphi\|_{L^\infty (\R^{ds})} \, .
\end{aligned}
$$

\medskip

\item  Additional error terms come from the difference between truncated marginals and true marginals (namely on the initial data)~: by Lemma \ref{lem:trunc-untrunc}, there holds the convergence
 $$ f^{(s)}_{0,N} -\widetilde f^{(s)}_{0,N} \longrightarrow 0 \, ,\qquad \mbox{ for fixed $s \geq 1 \, ,$ as $N \to \infty$ with $N \e^{d-1} \equiv 1 \, ,$ uniformly in $\Omega_s \, .$}$$
 Together with Lemma \ref{data}, this implies that
 $$J^{R,\delta}_{s,k}  (t,J,M)(X_s) - \widetilde J^{R,\delta}_{s,k}  (t,J,M)(X_s) \to 0 \, .$$

\item The micro-translations between the ``good" Boltzmann and BBGKY pseudo-trajectories depend on the maximal duration of the interactions to be considered
$$ | X_{s+k}(0 )- X_{s+k}^0(0)| \leq  C(\Phi, R,\eta) k \eps\, ,\quad \mbox{and} \quad  V_{s+k}(0 )= V_{s+k}^0(0)\, ,$$
so that the convergence
$$\bar J^{R,\delta}_{s,k}  (t,J,M)(X_s) - J^{0,R,\delta}_{s,k}  (t,J,M)(X_s) \to 0$$
may be very slow.

\end{itemize}

\bigskip

Combining all estimates shows that for any fixed $s \in \N$ and any $t<T$
$$I_{s} (t)(X_s)  -  I_s^0 (t)(X_s)  \to 0$$
locally uniformly in~$\Omega_s$,
which concludes the proof of Theorem~\ref{main-thmpotential}.

%% file: conclusion.tex
\chapter{Concluding remarks}\label{conclusion}
\setcounter{equation}{0}

  \section{On the time of validity of Theorems~\ref{existence-thm} and~\ref{main-thm}}

Let us first note that, for any fixed $N$,  the BBGKY hierarchy  has a global solution since it is formally equivalent to the Liouville equation in the phase space of dimension $2Nd$, which is nothing else than a   linear transport equation. The fact that we obtain a  uniform bound on a finite life span only, is therefore due to the analytical-type functional spaces~${\bf X}_{\eps,\beta,\mu}$ we consider. Belonging to such a functional space requires indeed a strong control on the growth of marginals. 

An important point is that the time of convergence is exactly the time for which these uniform a priori estimates hold. By definition of the functional spaces, we are indeed in a situation where the high order correlations can be neglected (see (\ref{est1}) and (\ref{est1-pot})), so that we only have to study the dynamics of a finite system of particles. The term-by-term convergence relies then on geometrical properties of the transport in the whole space, which do not introduce any restriction on the time of convergence.

\bigskip
\noindent
A natural question is therefore to know whether or not it is possible to get better uniform a priori estimates  and thus to improve  the time of convergence.
Let us first remark that such a priori estimates would hold for the Boltzmann hierarchy and thus for the nonlinear non homogeneous  Boltzmann equation.
As mentioned in Chapter \ref{boltz-chapter}, Remark~\ref{Linfty-rmk}, the main difficulty is to control the possible spatial concentrations of particles, which would contradict the rarefaction assumption and lead to an uncontrolled collision process. 

\section{More general potentials}

 A first natural extension
 to this work concerns the case of a compactly supported, repulsive potential, but 
 no longer satisfying~(\ref{strangeassumption}). As explained in Chapter~\ref{scattering}, that assumption guarantees that the cross section is well defined everywhere, since  the deflection angle is   a one-to-one function of the impact parameter. If that is no longer satisfied, additional decompositions are necessary to split the integration domain in subdomains where the cross-section is well-defined~:  as mentioned in Remark~\ref{potential-rmk}, 
we then expect to be able to extend the convergence proof, up to some technical complications due to the resummation procedures (see~\cite{pulvirentiss} for an alternative method).  Note that, if the deflection angle  can be locally constant as a function of the impact parameter, the method does not apply, which is consistent with the fact that we do not expect the Boltzmann equation to be a good approximation of the dynamics (see the by now classical counterexample by Uchiyama \cite{CIP}).

 From a physical point of view it would be more interesting to study the case of 
long-range potentials. Then the cross section actually  becomes singular, so a different notion of limit must be considered, possibly in the spirit of Alexandre and Villani \cite{AV}.  One intermediate step, as in~\cite{desvillettespulvirenti}, would be to extend this work to the case when the support of the potential goes to infinity with the number of particles. Then one could try truncating the long-range potential and showing that the tail of the potential has very little effect in the convergence.
 
 Note that in the case when grazing collisions
 become predominant, then the  Boltzmann equation should be replaced by the Landau equation, whose derivation is essentially open; a first result in that direction was obtained very recently by A. Bobylev, M. Pulvirenti and C. Saffirio in~\cite{BPS}, where a time zero convergence result is established.

    \section{Other boundary conditions}
    
As it stands, our analysis is restricted to the whole space (namely~$X_N \in \R^{dN}$). It is indeed important that free flow corresponds to straight lines (see in particular Lemmas~\ref{geometric-lem1} and~\ref{geometric-lem3} as well more generally as the analysis of pathological  trajectories in Chapter~\ref{pseudotraj}).

 It would be very interesting to generalize this work to more general geometries. A first step in that direction is to study the case of periodic flows in~$X_N$. The geometric lemmas must  be adapted  to that framework, and in particular it appears that a finite life span must a priori be given before the surgery of the collision integrals may be performed (see \cite{BGSR}).
 
 The case of a general domain is again  much more complicated, and results from the theory of billiards would probably need to be  used.

%% file: notation.tex
\chapter*{Notation Index}
 \setlength{\columnseprule}{0.001cm}
 \begin{multicols}{2}


$B_R$, ball of radius~$R$ centered at zero in~$\R^{d}$, page~\pageref{indexdefBR}

$B_R^s$, ball of radius~$R$ centered at zero in~$\R^{ds}$, page~\pageref{indexdefBR}  

$B_R(x)$, ball of radius~$R$ centered at $x$ in $ \R^{d}$, page~\pageref{indexdefBepsx}  

$\cB_k(\overline Z_k)$ a  small set of angles and velocities of a particle adjoined to~$\overline Z_k$ (or a neighboring configuration), leading to pathological trajectories, page~\pageref{geometric-prop}

%
%
%

$b(w,\omega)/ |w|$,  cross-section, page~\pageref{defcross-section}  

\medskip

${\mathbf C}_N$,  BBGKY hierarchy collision operator, page~\pageref{def:collisionop0} for the hard-spheres case and
 page~\pageref{bH-def}    for the potential case

${\mathbf C^0} $,  Boltzmann hierarchy collision operator,   page~\pageref{mild-Boltzmann} for the hard-spheres case and 
 page~\pageref{bH-defboltzmannpotential}   for the potential case

${\mathcal C}_{s,s+1}$, BBGKY collision operator, page~\pageref{def:collisionop0} for the hard-spheres case and
page~\pageref{css+k}   for the potential case

${\mathcal C}_{s,s+m}$, BBGKY collision operator involving~$m$ additional particles, page~\pageref{css+k}

${\mathcal C}^0_{s,s+1}$, Boltzmann collision operator, page~\pageref{Boltzmann-ophardspheres} for the hard-spheres case   and 
 page~\pageref{Boltzmann-opscatchapter}   for the potential case

\medskip

${\mathcal D}_N$, domain on which the hard-spheres dynamics take place, page~\pageref{def:DNHS}

${\mathcal D}_N^s$, artificial set in~$X_N$ variables on which the Hamiltonian dynamics take place, page~\pageref{def:DNS}



$\Delta_m(X_s)$, $m$-particle cluster based on~$X_s$,  page~\pageref{definitionclusterking}

$\Delta_s$,  well-separated initial configurations,  page~\pageref{initialization}

$\Delta_s^X$,   well-separated initial positions,  page~\pageref{initialization}

$d\sigma_N^{i,j}$, surface measure on~$\Sigma_N^s(i,j)$,   page~\pageref{defSigmaNij}

$d\sigma $,  surface measure on~$S_\eps(x_i)$,   page~\pageref{dsigmaSeps}  



$dZ_{(i,j)}$,  $2d(j-i+1)$-dimensional Lebesgue measure, page \pageref{page:ij-measure}

\medskip

$E(X_s,X_n)$, $\eps$-closure of~$X_s$ in~$X_N$, page~\pageref{definitionclusterking}

$E_{<i_0,j_0>}(X_s,X_n)$, $\eps$-closure of~$X_s$ in~$X_N$ with a weak link at~$(i_0,j_0)$,  page~\pageref{definitionclusterking}

$E_\eps(Z_s)$, $s$-particle Hamiltonian, page~\pageref{indexdefPhieps}  

$E_0(Z_s)$, $s$-particle free Hamiltonian, page~\pageref{def:freehamiltonian}  

\medskip

$f_N^{(s)}$, marginal of order~$s$ of the $N$-particle distribution function,  page~\pageref{marginalHS}  for the hard-spheres case, page~\pageref{marginal}  for the potential case

$\widetilde f_N^{(s)}$, truncated marginal of order~$s$ of the $N$-particle distribution function,  page~\pageref{t-marginal}

$  f^{(s)}$,   marginal of order~$s$ associated with the Boltzmann hierarchy, page~\pageref{mild-Boltzmann}  

${\mathbf \Phi}_\eps$, rescaled potential, page~\pageref{indexdefPhieps}

\medskip

${\mathcal G}_k$,  set of good configurations of~$k$ particles,   page~\pageref{indexdefGk}

\medskip

${\mathbf H}_s(t)$, $s$-particle  flow in the potential case, page~\pageref{solutionoperator}  

${\mathbf H}(t)$,    BBGKY hierarchy flow in the potential case, page~\pageref{bH-def}  
\medskip

$I_\varphi$, observable (average with respect to momentum variables), page~\pageref{indexdefobservable}

$I_s(t)(X_s)$ BBGKY observable, page~\pageref{observablesIsHS} for the hard-spheres case, page~\pageref{observablepotential} for the potential case

$I_s^0(t)(X_s)$ Boltzmann observable, page~\pageref{observablesHS}

$I^{R,\delta}_{s,k}(t )(X_s) $ reduced BBGKY observable, page~\pageref{formulafEHS} for the hard-spheres case, page~\pageref{defIsnRdelta} for the potential case

$I^{0,R,\delta}_{s,k}(t )(X_s) $ reduced Boltzmann observable, page~\pageref{formulafEHS}

\medskip
$K(w,y,\rho)$, cylinder of origin $w \in \R^d$, of axis $y\in \R^d$ and radius $\rho>0$, page~\pageref{defindexKyeta}

\medskip

$\k_d,$ volume of the unit ball in $\R^d,$ page \pageref{defkappad}




\medskip

$n^{i,j}$, outward normal to~$\Sigma_N(i,j)$, page~\pageref{indexdefnij}

$\nu^{i,j}$, direction of~$x_i-x_j$, page~\pageref{HS-BC}  

\medskip
 
 ${\mathcal M}_s(X_s)$, good set of initial velocities associated with well separated positions, page~\pageref{straightHs}

\medskip

${\mathcal P}$,    the set of continuous densities of probability in $\R^{2d}$, page~\pageref{def:P} 

\medskip

$\rho_*$, distance of minimal approach, page~\pageref{reduced-lem} 

\medskip

${\mathbf S}_s(t)$, $s$-particle free  flow,   page~\pageref{mild-Boltzmann}

${\mathbf S}(t)$, total free flow,   page~\pageref{mild-Boltzmann}

${\mathbf S}_1^{d-1}$, unit sphere in~$\R^d$, page~\pageref{unitsphere}

$S_\eps(x_i)$,  sphere in~$\R^d$ of radius~$\eps$, centered at~$x_i$,  page~\pageref{dsigmaSeps}

$\sigma$,  scattering operator in the hard-spheres case,  page~\pageref{defscatteringhardspheres}  

$\sigma_\eps$,  scattering operator in the case of a potential,  page~\pageref{scatteringbbgky}

$\sigma_0$,  Boltzmann scattering operator,  page~\pageref{scatteringsigma0}  

$\Sigma_N(i,j)$, boundary of~${\mathcal D}_N$, page~\pageref{indexdefnij} 

$\Sigma_N^s(i,j)$,  boundary of the artificial set~${\mathcal D}_N^s$,  page~\pageref{defSigmaNij}

\medskip

$ {\bf T}_s(t)$, $s$-particle flow for hard spheres, page~\pageref{solutionoperatorHS}

$ {\bf T}(t)$, total flow for hard spheres, page~\pageref{bH-defHS}

$t_\eps = \eps \tau_*$, nonlinear interaction time, page~\pageref{reduced-lem}

${\mathcal T}_n(t)$, set of collision times, page~\pageref{defTnotdeltaHS}

${\mathcal T}_{n,\delta}(t) $, set of well-separated collision times, page~\pageref{defTdeltaHS}

\medskip

$X_{\eps,s,\beta}$ function space for  BBGKY   marginals, page~\pageref
{def:functional-spacesHS}  for the hard-spheres case and 
 page~\pageref{def:functional-spaces} for the potential case

$X_{0,s,\beta}$ function space for  Boltzmann   marginals, page~\pageref
{def:functional-spacesHS}

${\bf{X}}_{\eps,\beta,\mu}$ function space for the BBGKY  hierarchies, 
page~\pageref{def:functional-spaces2HS}  for the hard-spheres case and 
page~\pageref{def:functional-spaces} for the potential case

${\bf{X}}_{0,\beta,\mu}$ function space for the Boltzmann  hierarchies,  page~\pageref{def:functional-spaces2HS}

${\bf X}_{\e,{\boldsymbol\beta},{\boldsymbol\mu}}$ function space for the uniform existence to the   BBGKY  hierarchies, page~\pageref{deffunctionspacesexistenceHS} for the hard-spheres case and 
page~\pageref{deffunctionspacesexistence} for the potential case

${\bf X}_{0,{\boldsymbol\beta},{\boldsymbol\mu}}$ function space for the uniform existence to the   Boltzmann  hierarchies, page~\pageref{deffunctionspacesexistenceHS}

\medskip

${\bf \Psi}_s(t)$, $s$-particle hard-spheres flow, page~\pageref{HSflow}

\medskip

$\omega$, direction of the apse line, page~\pageref{indexdefomega}   

$\Omega_N$, phase space for the Liouville equation, page~\pageref{defOmegaN}

\medskip

${\mathcal Z}_N$, partition function, page~\pageref{def:Z-f} 
\medskip

$| \cdot |_{\eps,s,\beta}$ norm for the BBGKY marginal of order~$s$, page~\pageref
{def:functional-spacesHS}  for the hard-spheres case and page~\pageref{def:functional-spaces}  for the potential case

$| \cdot |_{0,s,\beta}$ norm for the Boltzmann  marginal of order~$s$, page~\pageref{norm:e-b}  

$\|{\cdot}\|_{\eps,\beta,\mu}$ norm   for the BBGKY hierarchy,
page~\pageref{def:functional-spaces2HS}  for the hard-spheres case and 
 page~\pageref{normepsbetamu}     for the potential case

$\|{\cdot}\|_{0,\beta,\mu}$ norm   for the Boltzmann hierarchy, page~\pageref{def:functional-spaces2HS}  

$  | \! \|\cdot     | \! \| _{ \e,{\boldsymbol\beta},{\boldsymbol\mu}}$, norm in~${\bf X}_{\e,{\boldsymbol\beta},{\boldsymbol\mu}}$,  page~\pageref{deffunctionspacesexistenceHS} for the hard-spheres case
and page~\pageref{deffunctionspacesexistence}   for the potential case

$  | \! \|\cdot     | \! \| _{ 0,{\boldsymbol\beta},{\boldsymbol\mu}}$, norm in~${\bf X}_{0,{\boldsymbol\beta},{\boldsymbol\mu}}$,  page~\pageref{deffunctionspacesexistenceHS}

\end{multicols} 